\documentclass[reqno,10pt]{amsart}
\usepackage{amsmath,amsthm}
\usepackage{xcolor}
\colorlet{BLUE}{blue}

\usepackage{amssymb}
\usepackage{esint}

    \usepackage{a4wide}
\usepackage[hidelinks,breaklinks=true,bookmarks=false]{hyperref}

\allowdisplaybreaks

 \usepackage[margin=1in]{geometry}

\newcommand{\Mtwo}{\mathfrak M_2}

\newtheorem{theorem}{Theorem}
\newtheorem{corollary}[theorem]{Corollary}
\newtheorem{definition}{Definition}
\newtheorem{lemma}{Lemma}
\newtheorem{proposition}[theorem]{Proposition}
\newtheorem{remark}{Remark}
\newtheorem*{maintheoremA}{Theorem A}
\newtheorem*{maintheoremB}{Theorem B}
\newtheorem*{maintheoremC}{Theorem C}
\newtheorem*{maintheoremD}{Theorem D}

\let\e=\varepsilon
\let\p=\partial
\newcommand{\Hh}{\mathbb H}

\def\bPhi{{\bf \Phi}}

\numberwithin{equation}{section}

\newcommand{\Db}{\mathbf{D}}
\newcommand{\Lb}{\mathbf{L}}
\newcommand{\Bb}{\mathbf{B}}
\newcommand{\R}{\mathbb{R}}
\newcommand{\dd}{\mathrm{d}}
\newcommand{\Bcal}{\mathcal B}
\newtheorem{assumption}{Assumption}
\def\bega{\begin{aligned}}
\def\enda{\end{aligned}}
\def\lw{\left}
\def\rw{\right}

\begin{document}
\title{Rough Axisymmetric Euler Dynamics:\\
Global Energy Compactness and Kinetic Passage Across Regularity Breakdown 
}
\author{Gi-Chan Bae}
\address{School of Mathematical Sciences, Dalian University of Technology, Dalian 116024, P. R. China}
\email{gcbae02@gmail.com}
\author{Hongxu Chen}
\address{School of Mathematical Sciences, Shenzhen University, Shenzhen, Guangdong 518060, China}
\email{hongxuchen.math@gmail.com}
\author{Chanwoo Kim}
\address{Department of Mathematics, University of Wisconsin--Madison, Madison, WI 53717, USA}
\email{chanwookim.math@gmail.com; chanwoo.kim@wisc.edu}

\begin{abstract}
We establish global energy compactness for rough three-dimensional
axisymmetric Euler flow without swirl and use it to construct
hydrodynamic limits of collisional kinetic equations beyond finite-time
Euler regularity breakdown.  For relative vorticity in
$L^1\cap L^p$, regular approximate Euler flows are globally compact in
$C_tL_x^2$ for $p\ge4/3$, without a sign condition or any
spatial-moment assumption.  Below $4/3$, the same conclusion
holds under uniform control of the absolute impulse.  We identify the
sharp radial dynamics governing this confinement: the absolute impulse
is propagated without a sign condition for $p\ge5/3$, while for
$1<p<5/3$ it is controlled under one-sided $L^{p^\sharp}$
integrability, where $p^\sharp= {4p} /  (3p-1)$,
and this exponent is sharp.  In particular, for one-sign finite-impulse
vorticity the resulting global energy compactness holds throughout the full
range $p>1$.

For the Coulomb Landau equation and a class of non-cutoff soft-potential
Boltzmann equations, we construct strong solutions on lifespans
$T^\varepsilon\to\infty$ realizing prescribed rough finite-energy Euler
initial data and obtain energy-conserving axisymmetric Euler limits on every
fixed time interval.  If the corresponding regular Euler evolution loses
regularity in finite time, the hydrodynamic fields converge to the regular Euler solution up to the singular time
and have a common strong trace there, while subsequences yield global
energy-conserving Euler continuations past that time.

\end{abstract}

\maketitle
 \setcounter{tocdepth}{1}
\tableofcontents

\section{Introduction}
 Collisions relax microscopic non-equilibrium toward local equilibrium.  When
this relaxation survives at the macroscopic level as a dissipative mechanism,
the passage from kinetic theory to fluid dynamics is by now well developed
\cite{bardosgolselevermore1991,bardosgolselevermore1993,
lionsmasmoudi2001I,lionsmasmoudi2001II,
golsesaintraymond2004,golsesaintraymond2009,guo2006diffusive}.
The situation is fundamentally different when no macroscopic dissipation
remains, as in the incompressible Euler limit.  Classical
kinetic derivations of Euler are built around a regular macroscopic reference
flow: spatial Lipschitz control is central to the weak--strong stability
mechanism, while Hilbert-expansion methods require still higher regularity
\cite{saintraymond2003euler,saintraymond2009entropy,
demasiespositolebowitz1989}.  What is not understood is whether the
kinetic-to-Euler limit can persist through a finite-time loss of regularity
of the limiting three-dimensional Euler flow.

At the rough regularity relevant here, this limitation is not merely
technical.  Although sufficiently regular axisymmetric no-swirl Euler flows
are globally regular \cite{UkhovskiiIudovich1968}, finite-time loss of
regularity is known within the same symmetry class.  Elgindi established
finite-time singularity formation in the $C^{1,\alpha}$ class
\cite{Elgindi2021} (see also the recent work of Chen
\cite{Chen2026EulerII}); finite-energy singular solutions were subsequently
obtained by Elgindi--Ghoul--Masmoudi
\cite{ElgindiGhoulMasmoudi2021}, while
C\'ordoba--Mart\'inez-Zoroa--Zheng constructed a distinct finite-energy,
non-self-similar blow-up mechanism
\cite{CordobaMartinezZoroaZheng2025}.  Thus the regular Euler reference
required by classical kinetic-to-Euler arguments can genuinely break down
in finite time.

Passing beyond such a breakdown raises a second, independent difficulty.
For rough Euler solutions, conservation of kinetic energy is itself
nontrivial; see the broader Onsager framework
\cite{Onsager1949,ConstantinETiti1994}.  The obstruction relevant here,
however, is not small-scale anomalous dissipation but large-scale loss of
energy: even under strong local compactness, an order-one amount of energy
may escape to spatial infinity.  A singularity-crossing kinetic limit must
therefore control both the loss of the regular Euler reference flow and the
possible escape of energy at infinity.  This leads to the central question
of the paper:
\begin{quote}
\itshape
Can strong solutions of collisional kinetic equations realize prescribed rough finite-energy Euler data, remain globally energy-compact across a
finite-time loss of Euler regularity, and generate energy-conserving Euler
continuations beyond it?
\end{quote}

We answer this question affirmatively.  Our construction is not confined to
the lifespan of a regular Euler reference flow.  We construct strong
solutions of collisional kinetic equations realizing prescribed rough
three-dimensional finite-energy Euler initial data and producing global
energy-conserving Euler limits along subsequences on every finite time interval. If the corresponding
regular Euler solution loses regularity at $T_*$, then, for each fixed
admissible kinetic realization, the hydrodynamic fields converge to that
solution without subsequence extraction on $[0,T_*]$. Beyond $T_*$, these limits realize energy-conserving Euler continuations through the same trace at $T_*$. Thus
the kinetic approximation is rigid up to the singular time, while compactness
survives across and beyond it.

The fluid mechanism behind this passage is a global energy-compactness theory
for rough axisymmetric Euler flow.  The mechanism changes at the threshold
$p=4/3$.  For $p\ge4/3$, strong $L^1\cap L^p$ compactness of the transported
relative vorticity, together with the natural energy bound, rules out escape
of energy to infinity without any sign condition or spatial-moment
assumption.  For $1<p<4/3$, global energy compactness instead requires
control of the absolute impulse.  To propagate this quantity, it is enough
to control one sign component in $L^{p^\sharp}$, where
\begin{equation}
    p^\sharp=\frac{4p}{3p-1}, \label{eq:psharp}
\end{equation}
and this exponent is sharp.  For one-sign finite-impulse vorticity, the
additional condition is automatic. Therefore, we extend the energy-conservation range
of \cite{nobili2022renormalization} from $p>3/2$ to every $p>1$ for the one-sign vorticity.  We also
identify the sharp sign-free thresholds $3/2$ and $5/3$ for propagation of
the first radial moment and the absolute impulse, respectively.

We then realize this fluid compactness by three-dimensional kinetic dynamics
for both the Coulomb Landau equation and non-cutoff Boltzmann equations with
soft potentials.  Bae--Kim \cite{bae2026quantitative} developed a rough
kinetic-to-Euler construction for the two-dimensional hard-sphere Boltzmann
equation.  In the present three-dimensional setting, an additional global
obstruction appears: compactness of the transported vorticity does not by
itself control the global $L^2$ energy of the fluid velocity, since energy
can escape to large cylindrical radius.  The energy-compactness and
radial-confinement theory provides the global tail control needed to prevent
this escape.  Combined with the kinetic estimates, it yields strong kinetic
solutions on lifespans $T^\varepsilon\to\infty$ without requiring a smooth
Euler reference flow throughout those lifespans.

We first isolate the fluid mechanism responsible for global energy
compactness.

\subsection{Global energy compactness and radial dynamics}

The first obstruction is already visible at the Euler level.  Write
$x=(x_h,z)\in\mathbb R^2\times\mathbb R$ and $r=|x_h|$.  For an
axisymmetric velocity without swirl,
\[
    u=u^r e_r+u^z e_z,\qquad
    \omega=\omega^\vartheta e_\vartheta,\qquad
    \varpi=\frac{\omega^\vartheta}{r}.
\]
Then
\begin{equation}\label{eq:vor-eqtn-preview}
    \partial_t\varpi+u\cdot\nabla\varpi=0,
    \qquad
    u=\mathcal B\varpi
    :=\nabla\times(-\Delta)^{-1}(r\varpi e_\vartheta).
\end{equation}
Axisymmetry therefore restores a two-dimensional transport structure,
but not a two-dimensional energy theory.  The scalar $\varpi$ is
transported in the meridional variables $(r,z)$, whereas the
Biot--Savart reconstruction and the $L^2$ energy
\[
    \frac12\int_{\mathbb R^3}|u(t,x)|^2\,dx
\]
retain the geometry of three dimensions.  This mismatch is largely
invisible to local compactness, but becomes decisive for global $L^2$
compactness.

The smooth axisymmetric no-swirl theory goes back to
Ukhovskii--Iudovich \cite{UkhovskiiIudovich1968}, with later Sobolev,
bounded-vorticity, and critical developments
\cite{ShirotaYanagisawa1994,danchin2007axisymmetric,
AbidiHmidiKeraani2010,ShaoWeiZhang2026}.  At low regularity, Delort
established axisymmetric vortex-sheet compactness \cite{Delort1992}.
For relative vorticity in $L^1\cap L^p$, $p>1$, natural approximations
yield strong local $L^2$ compactness of the velocity
\cite{JiuXin2004,JiuXin2006,nobili2022renormalization}, while global weak
Euler solutions exist in the same range \cite{JiuWuYang2015,JiuLiuNiu2021};
strong relative-vorticity convergence is also known for vanishing
viscosity \cite{brkic2024vanishing}.

Local compactness, however, does not control the global $L^2$ energy.
An order-one amount of this energy may escape to arbitrarily large
cylindrical radius while remaining invisible on every fixed region.
Thus $u_n\to u
    \quad\text{strongly in }L^2_{\mathrm{loc}}$ does not by itself imply $\|u_n(t)\|_2\to\|u(t)\|_2$. The missing issue is therefore compactness at spatial infinity.

A classical way to restore such compactness uses a sign condition
together with finite impulse.  The signed impulse
\[
    \frac1{2\pi}\int_{\mathbb R^3}r^2\varpi(t,x)\,dx
\]
is conserved by the axisymmetric Euler evolution.  When $\varpi$ has
one sign, this invariant becomes a coercive radial moment: it prevents
a fixed amount of vorticity from moving to arbitrarily large
cylindrical radius, and the Biot--Savart law transfers this radial
tightness to the velocity tail.  This is the confinement mechanism underlying the one-sign
finite-impulse theory of Nobili--Seis \cite{nobili2022renormalization}.
This leaves open whether radial confinement is needed at all.  More
precisely:

\begin{quote}
\itshape
Can the transported $L^1\cap L^p$ structure alone enforce global
$L^2$ compactness, without a sign condition or an a priori radial
moment?
\end{quote}

The answer is moment-free for every $p\ge4/3$.  Below $4/3$,
global energy compactness is recovered from control of the absolute impulse
\begin{equation}\label{eq:absolute-impulse}
    \mathfrak M_2(\varpi)
    :=\frac1{2\pi}\int_{\mathbb R^3}r^2|\varpi(x)|\,dx.
\end{equation}
The precise statement is the following.

\begin{maintheoremA}[Global energy compactness]
Fix $T<\infty$ and let $1<p<\infty$.  Let
$(u_n,\varpi_n,f_n,\pi_n)$ be regular axisymmetric no-swirl approximate
Euler flows satisfying
\[
    \partial_tu_n+u_n\cdot\nabla u_n+\nabla\pi_n=f_n,
    \qquad
    \nabla\cdot u_n=\nabla\cdot f_n=0,
    \qquad
    \nabla\times u_n=r\varpi_ne_\vartheta,
\]
with
\[
    \sup_n\|u_n\|_{L^\infty(0,T;L^2)}<\infty,
    \qquad
    f_n\to0
    \quad\text{in }L^1(0,T;L^2),
\]
and
\[
    \varpi_n\to\varpi
    \quad\text{in }C([0,T];L^1\cap L^p),
    \qquad
    u_n(0)\to u_0
    \quad\text{in }L^2.
\]
Then, with $u=\mathcal B\varpi$,
\[
    u_n\to u
    \quad\text{in }C([0,T];L^2(\mathbb R^3)),
    \qquad
    \|u(t)\|_2=\|u_0\|_2
    \quad (0\le t\le T),
\]
provided either
\[
    p\ge\frac43
    \qquad\text{or}\qquad
    \sup_n\sup_{0\le t\le T}
    \mathfrak M_2(\varpi_n(t))<\infty.
\]
\end{maintheoremA}

The two results are proved separately in the body: the moment-free
$p\ge4/3$ result is Theorem~\ref{thm:mfreecomp}, while the
moment-confined result is Corollary~\ref{cor:moment-energy-comp}. The direct $p\ge4/3$ result has a matching passive-transport obstruction
below $4/3$: Proposition~\ref{prop:passive-shell-obstruction} constructs,
for every $1<p<4/3$, smooth incompressible transports with vanishing
$L^1\cap L^p$ scalar norm and vanishing initial Biot--Savart energy but
order-one terminal Biot--Savart energy.  %This is a sharpness statement for the moment-free no-escape mechanism, not an active Euler counterexample.

The two alternatives in Theorem~A correspond to different mechanisms.
For $p\ge4/3$, energy escape is ruled out directly, without first
confining the vorticity.  Below $4/3$, the theorem reduces global energy
compactness to propagation of the absolute impulse. For sign-changing
vorticity the latter leads to the mixed exponent $p^\sharp$ in \eqref{eq:psharp}, which enters the next result.

\begin{maintheoremB}[Sharp absolute-impulse dynamics]
Fix $T<\infty$ and let $1<p<\infty$.  Let $(u,\varpi)$ be a regular
finite-energy axisymmetric no-swirl Euler flow on $[0,T]$ such that 
\[
    \varpi\in L^\infty(0,T;L^1\cap L^p),
    \qquad
    \mathfrak M_2(\varpi(0))<\infty.
\]

If $p\ge5/3$, then
\[
    \bigl(1+\mathfrak M_2(\varpi(t))\bigr)^{1/4}
    \le
    \bigl(1+\mathfrak M_2(\varpi(0))\bigr)^{1/4}
    +Ct,
    \qquad 0\le t\le T,
\]
without a sign condition.

If $1<p<5/3$, the same estimate holds provided
\[
    \min\Bigl\{
        \|\varpi_+(0)\|_{L^{p^\sharp}},
        \|\varpi_-(0)\|_{L^{p^\sharp}}
    \Bigr\}<\infty.
\]
Here $C$ depends only on $p$, the
$L^\infty(0,T;L^1\cap L^p)$ bound of $\varpi$, and, in the latter case,
the displayed one-sided $L^{p^\sharp}$ bound.
\end{maintheoremB}

The positive statement is Corollary~\ref{cor:smooth-impulse}, obtained
from the general mixed criterion in Theorem~\ref{thm:miximpulse}.  The
$p^\sharp$ sharpness is contained in the latter theorem, while the
fixed-time sign-free sharpness of $5/3$ is Theorem~\ref{thm:impobstr}
and Corollary~\ref{cor:signfree}.

The exponent $p^\sharp=4p/(3p-1)$ is optimal for this mixed
criterion: if $p\le q<p^\sharp$, uniform $L^1\cap L^p$, absolute-impulse,
and one-sided $L^q$ bounds do not yield a class-uniform bound on the
instantaneous growth of $\mathfrak M_2$.  The sign-free threshold $5/3$ is also
dynamically sharp: below it, uniformly controlled smooth sign-changing
families can develop arbitrarily large absolute impulse at a fixed positive
time.

We also derive sign-free propagation of the first radial moment for \(p\ge 3/2\) in Theorem \ref{thm:firstmom}, with \(3/2\) sharp for class-uniform instantaneous control by Theorem \ref{thm:thrsharp}.

Theorems A and B identify the fluid mechanisms that convert transported
vorticity control into global $L^2$ compactness.  We now show that the
kinetic dynamics generates precisely the structure needed to apply
these mechanisms, without prescribing a smooth Euler trajectory.

\subsection{Kinetic realization beyond Euler regularity breakdown}

We consider
\begin{equation}\label{eq:kinetic-model}
    \varepsilon\partial_tF^\varepsilon
    +\xi\cdot\nabla_xF^\varepsilon
    =
    \frac1{\kappa\varepsilon}
    Q(F^\varepsilon,F^\varepsilon),
    \qquad
    \kappa=\kappa(\varepsilon)\to0.
\end{equation}
Here $Q$ is either the Coulomb Landau operator or a non-cutoff
Boltzmann operator with soft potentials in the parameter range
considered here.

Let $\rho^\varepsilon$, $u^\varepsilon$, and $\theta^\varepsilon$
denote the scaled macroscopic fluctuations determined by the
collision-invariant moments of $F^\varepsilon$, let $\mathbb P$ denote the Leray projector, and set
\[
    v^\varepsilon:=\mathbb P u^\varepsilon,\qquad
    \mathfrak s^\varepsilon:=\frac32\theta^\varepsilon-\rho^\varepsilon,
    \qquad
    \varpi^\varepsilon :=  \frac{(\nabla \times u^\e)^\vartheta}{r} =\varpi_A^\varepsilon+\varpi_B^\varepsilon.
\]
Let $M^\varepsilon=M[F^\varepsilon]$ be the exact local Maxwellian and
set $G^\varepsilon:=F^\varepsilon-M^\varepsilon$.   \(\varpi_B^\varepsilon\) denotes the component generated by the forcing terms,
and the component $\varpi_A^\varepsilon$ is transported exactly by
$v^\varepsilon$:
\[
    \partial_t\varpi_A^\varepsilon
      +v^\varepsilon\cdot\nabla_x\varpi_A^\varepsilon=0,
    \qquad
    \varpi_A^\varepsilon(0)=\varpi_0^\varepsilon.
\]
Since $v^\varepsilon$ is divergence free, the $L^q$ norms of
$\varpi_A^\varepsilon$ and of its positive and negative parts are
conserved.

\begin{maintheoremC}[Kinetic realization of prescribed rough finite-energy Euler data]
Let $1<p<\infty$ and let
$\varpi_0\in L^1(\mathbb R^3)\cap L^p(\mathbb R^3)$ and $\mathfrak s_0\in L^2(\mathbb R^3)$ be axisymmetric,
with $u_0:=\mathcal B\varpi_0\in L^2(\mathbb R^3).$

If $1<p<4/3$, assume in addition that
\[
    \mathfrak M_2(\varpi_0)<\infty,
    \qquad
    \min\Bigl\{
        \|(\varpi_0)_+\|_{L^{p^\sharp}},
        \|(\varpi_0)_-\|_{L^{p^\sharp}}
    \Bigr\}<\infty.
\]

For either the Coulomb Landau equation or the non-cutoff
soft-potential Boltzmann equation considered here, there exists a family
of smooth nonnegative axisymmetric no-swirl kinetic initial data
$F_0^\varepsilon$ whose corresponding strong solutions $F^\varepsilon$
exist uniquely on
\[
    [0,T^\varepsilon],
    \qquad T^\varepsilon\to\infty.
\]
Their macroscopic initial data satisfy
\[
    \varpi_0^\varepsilon\to\varpi_0
        \quad\text{in }L^1\cap L^p,
    \qquad
    v_0^\varepsilon\to u_0
        \quad\text{in }L^2,
    \qquad
    \mathfrak s_0^\varepsilon\to\mathfrak s_0
        \quad\text{in }L^2.
\]
No vanishing acoustic preparation is required in this realization.

Along the whole kinetic lifespan, the same solutions satisfy the
microscopic relaxation $G^\e \to 0$. For every finite $T>0$ and every sequence
$\varepsilon_j\downarrow0$, one may extract a subsequence and a
renormalized finite-energy axisymmetric no-swirl Euler pair
$(u,\varpi)$ on $[0,T]$ with initial data $(u_0,\varpi_0)$ such that
\[
    (\varpi_A^{\varepsilon_j},v^{\varepsilon_j},\mathfrak s^{\varepsilon_j})
    \longrightarrow
    (\varpi,u,\mathfrak s)
    \quad\text{strongly in }C\bigl([0,T];
    (L^1\cap L^p)\times L^2\times L^2\bigr).
\]
The limiting thermodynamic mode solves $\partial_t\mathfrak s+u\cdot\nabla_x\mathfrak s=0$ with $\mathfrak s(0)=\mathfrak s_0,
$
and the limiting Euler energy is conserved:
\[
    \|u(t)\|_2=\|u_0\|_2,
    \qquad 0\le t\le T.
\]
The forcing vorticity $\varpi_B^{\varepsilon_j}$
%\[
 %   \varpi_B^{\varepsilon_j}\to0
 %   \quad\text{in }L^\infty(0,T;L^q(\mathbb R^3)),
 %   \qquad 2\le q\le6,
%\]
as well as $\rho^{\varepsilon_j}+\theta^{\varepsilon_j}$ and
$\mathbb P^\perp u^{\varepsilon_j}$ vanish as in Theorem~\ref{thm:roughlimit}.

See Theorems~\ref{thm:longtime}, \ref{thm:roughlimit}, and
\ref{thm:globenergy}, and Corollary~\ref{cor:roughreal}, for the precise statements.
\end{maintheoremC}

%The long kinetic lifespan is Theorem~\ref{thm:longtime}; microscopic relaxation and the all-$p$ rough limit are Theorem~\ref{thm:roughlimit}; the forcing-part decay is Lemma~\ref{lem:splitconv}; the global $C_tL_x^2$ upgrade is Theorem~\ref{thm:globenergy}; and the prescribed-data realization is Corollary~\ref{cor:roughreal}.

%More generally, the compactness argument applies to arbitrary admissible kinetic families with strongly convergent transported vorticity and thermodynamic data, yielding a subsequential rough Euler limit for every $p>1$.  The additional assumptions below $4/3$ enter only in upgrading the velocity convergence to the global energy topology.  The initial acoustic modes need not vanish: they may approach nonzero compatible $L^2$ targets and disappear dynamically through dispersion.

 \begin{maintheoremD}[Kinetic passage through Euler regularity breakdown]
Assume the setting of Theorem~C, and suppose that the prescribed Euler
initial datum generates a regular finite-energy axisymmetric no-swirl
solution on $[0,T_*)$ that loses regularity at some $0<T_*<\infty$.
Then, for each fixed admissible kinetic realization, the projected velocity
and transported relative vorticity converge, without subsequence extraction,
to the regular Euler solution on $[0,T_*)$ and reach the same strong traces at
$T_*$.  Moreover, from every sequence $\varepsilon_j\downarrow0$ one may
extract a subsequence converging on every finite time interval to a global
energy-conserving Euler solution that agrees with the regular solution for
$t<T_*$ and attains the same trace at $T_*$.  No uniqueness beyond $T_*$
is asserted. 

See Corollaries~\ref{cor:beyondbreakdown}--\ref{cor:global-passage-invariants} for the precise statements. Corollary~\ref{cor:three-singularity-passage} applies these results to three
finite-time singularity constructions
\cite{ElgindiGhoulMasmoudi2021,CordobaMartinezZoroaZheng2025,Chen2026EulerII}.
\end{maintheoremD}

 \iffalse The kinetic passage through $T_*$ is quantitative.  At each fixed kinetic
scale, the solenoidal velocity remains spatially Lipschitz throughout the
kinetic lifespan, with explicit logarithmic control.  More importantly, the
kinetic family itself approaches the same Euler trace from both sides of the
breakdown in shrinking time windows around $T_*$.  Thus the approximation
does not merely identify the regular Euler solution before singularity and
extract a continuation afterwards: the two regimes are joined through the
common trace attained by the kinetic family.  In the $p=2$ branch, the
limiting velocity additionally satisfies a local-in-space almost-Lipschitz
temporal modulus.\fi

We also identify several quantitative features of the kinetic passage through
$T_*$. In shrinking time windows around the breakdown, the kinetic family
approaches the same Euler trace from both sides of $T_*$. This control is obtained from explicit fixed-scale estimates, including spatial Lipschitz bounds for the solenoidal velocity with a quantitatively controlled blow-up rate. Thus the approximation quantitatively joins the two regimes through the common trace. In the $p=2$ branch, Proposition \ref{prop:quantitative-singular-passage} shows that the limiting velocity satisfies a local-in-space almost-Lipschitz temporal modulus. 

Combining the global $C_tL_x^2$ control supplied by the $4/3$ shell
mechanism and the radial-moment theory with the long kinetic lifespan of
Theorem~\ref{thm:longtime}, the rough-limit theory of
Theorem~\ref{thm:roughlimit}, and the prescribed-data realization of
Corollary~\ref{cor:roughreal} yields a kinetic-to-Euler passage that persists
up to, across, and beyond the loss of Euler regularity.

\subsection{Mechanisms: radial geometry and the kinetic--fluid bridge}

Three mechanisms organize the argument: two fluid-geometric estimates
and a kinetic--fluid bridge implementing the structural separation above.

\medskip
\paragraph{\textbf{The $4/3$ no-escape mechanism.}} The high-$p$ result of Theorem~A controls energy transport directly,
rather than deriving energy tightness from radial confinement.  A radial
localization of the energy identity shows that energy reaching large
cylindrical radius must cross a shell $r\sim R$.  For
$4/3\le s\le3/2$, the nonlinear flux satisfies the sharp shell estimate
\[
    \frac1R\int_{\mathcal C_R}|u|^3\,dx
    \lesssim
    R^{-\delta_s}
    \|u\|_2^{2\alpha_s}
    \left\|\frac{\nabla u}{r}\right\|_{L^s(\mathcal C_R')}^{\gamma_s}
    +R^{-5/2}\|u\|_2^3,
    \qquad
    \delta_s=\frac{3s-4}{4(s-1)},
\]
where $\mathcal C_R\subset\mathcal C_R'$ denote comparable cylindrical shells \(r\sim R\), and $\alpha_s,\gamma_s>0$ are interpolation exponents. The radial
gain is positive precisely for $s>4/3$.  It reflects the interaction
between the two-dimensional meridional geometry, the three-dimensional
cylindrical measure, and the weighted derivative $r^{-1}\nabla u$.
The weighted axisymmetric Calder\'on--Zygmund estimate of Nobili--Seis
\cite{nobili2022renormalization} transfers the transported $L^s$ control
of the relative vorticity to this weighted derivative.  At the endpoint
$s=4/3$ the power gain vanishes, but strong $C_tL_x^{4/3}$ convergence
supplies uniform far-field tightness of $r^{-1}\nabla u$.  Together with
the corresponding pressure localization, this rules out energy escape to
$r=\infty$ and upgrades interior compactness to global $C_tL_x^2$
compactness.  Below $4/3$, Proposition~\ref{prop:passive-shell-obstruction}
realizes the corresponding subcritical shell scaling by a fixed smooth
incompressible passive transport; it is not a self-consistent Euler
counterexample.  Thus, for the direct moment-free shell mechanism,
$4/3$ is a large-scale no-escape threshold at spatial infinity rather
than a small-scale Onsager regularity threshold.

\medskip
\paragraph{\textbf{Opposite signs and the $5/3$ threshold.}}
The absolute impulse is governed by a different geometry.  Set
\[
    \mathfrak X(\varpi,\eta)
    :=\int_{\mathbb R^3}(\mathcal B\eta)^r\,r\varpi\,dx .
\]
The meridional kernel is skew:
$\mathfrak X(\varpi,\eta)=-\mathfrak X(\eta,\varpi)$.  Hence the same-sign
interactions cancel exactly and, for a smooth Euler flow,
\[
    \frac d{dt}\mathfrak M_2(\varpi(t))
    =\frac2\pi\,\mathfrak X(\varpi_-,\varpi_+).
\]
The key mixed estimate is
\[
    |\mathfrak X(\varpi,\eta)|
    \lesssim_p
    \|\varpi\|_1\|\eta\|_1
    +\|\varpi\|_p^{1/4}
      \mathfrak M_2(\varpi)^{3/4}
      \|\eta\|_{p^\sharp},
    \qquad
    p^\sharp=\frac{4p}{3p-1}.
\]
The far interaction is controlled at the $L^1$ level.  Near the diagonal,
redistributing the radial weight reduces the singular interaction to a
two-dimensional Hardy--Littlewood--Sobolev pairing on the meridional
half-plane, and the balance forces $p^\sharp$.  With only the ambient
$L^1\cap L^p$ control, the estimate closes exactly when
$p^\sharp\le p$, whose fixed point is $p=5/3$.  Below $5/3$, one-sided
$L^{p^\sharp}$ control supplies the missing integrability.  The sharpness
constructions show that neither the mixed exponent nor the sign-free
$5/3$ threshold is merely an artifact of the estimate.

\medskip
\paragraph{\textbf{The kinetic--fluid bridge.}}
The structural separation above is implemented by the exact
local-Maxwellian decomposition and the transport--forcing split.  Writing
\[
    F^\varepsilon=M^\varepsilon+G^\varepsilon,
    \qquad
    \varpi^\varepsilon
      =\varpi_A^\varepsilon+\varpi_B^\varepsilon,
    \qquad
    v^\varepsilon
      =u_A^\varepsilon+u_B^\varepsilon,
    \qquad
    \nabla\times u_A^\varepsilon
      =r\varpi_A^\varepsilon e_\vartheta,
\]
exact moment matching removes the five collision-invariant directions
from $G^\varepsilon$, while the principal relative vorticity satisfies
\[
    \partial_t\varpi_A^\varepsilon
      +v^\varepsilon\cdot\nabla_x\varpi_A^\varepsilon=0.
\]
Thus its $L^q$ norms and sign structure are transported exactly.  The
acoustic, baroclinic, and microscopic-stress effects are carried by the
complementary component $\varpi_B^\e$, and the corresponding velocity
$u_B^\varepsilon$ tends to zero on every fixed time interval.  At the velocity level, $u_A^\varepsilon$ satisfies an approximate Euler
equation,
\[
    \partial_tu_A^\varepsilon
    +u_A^\varepsilon\cdot\nabla u_A^\varepsilon
    +\nabla\widetilde\pi_A^\varepsilon
    =f_A^\varepsilon,
    \qquad
    \sup_{\varepsilon\ll1}
    \|u_A^\varepsilon\|_{L_t^\infty L_x^2}<\infty,
    \qquad
    \|f_A^\varepsilon\|_{L_t^1L_x^2}\to0
\]
on every fixed interval.  The kinetic estimates therefore produce the
approximate Euler sequence required by Theorem~A rather than taking one
as an external input.  For $p\ge4/3$, the no-escape result applies
directly.  Below $4/3$, the same splitting permits the absolute-impulse
bounds to be propagated while the $u_B^\varepsilon$ errors are controlled,
and finite-impulse Biot--Savart stability supplies the global $L^2$
upgrade.  In this way the fluid thresholds enter only after the
transported rough component has been isolated.

Before $T_*$, weak--strong uniqueness identifies every cluster path with
the regular Euler solution; compactness then fixes the common trace at $T_*$,
without invoking any uniqueness theory beyond the singular time.

The local-Maxwellian structure closes the strong kinetic estimates.
Exact moment matching exposes the stress--heat-flux cancellation in the
combined energy, collision coercivity controls the microscopic defect,
acoustic dispersion controls the fast compressible modes, and the
axisymmetric estimates control the solenoidal dynamics.  Together these
bounds yield $T^\varepsilon\to\infty$.  A slow diagonal approximation then
realizes prescribed rough finite-energy data while allowing the high-order
kinetic norms to grow with $\varepsilon^{-1}$.  Thus collisional dissipation
remains active in the microscopic sector while the rough macroscopic
component retains the transport structure needed by the fluid compactness
theory.

Together, these estimates connect the long-time kinetic construction to
global finite-energy Euler compactness without requiring a regular Euler
reference flow throughout the limiting time interval.

\subsection{Proof architecture and organization.}
Sections~\ref{sec:bscompact}--\ref{sec:oppsign} develop the fluid theory: Section~\ref{sec:bscompact} proves the
moment-free $4/3$ global-energy compactness mechanism, Section~\ref{sec:radmom} establishes
the sharp $3/2$ first-moment threshold, and Section~\ref{sec:oppsign} establishes the sharp \(5/3\) sign-free threshold and the sharp \(p^\sharp\) mixed criterion for absolute impulse. Sections~\ref{sec:kinframe}--\ref{sec:longconv} develop the kinetic theory,
beginning with the exact local-Maxwellian framework and energy estimates
and leading to the long-time bounds and rough kinetic-to-Euler
convergence. Section~\ref{sec:globlimit} brings the fluid and kinetic analyses together,
establishing radial-moment bounds for the transported kinetic vorticity and
upgrading the rough kinetic limit to global $C_tL_x^2$ convergence.
Section~\ref{sec:roughdata} constructs kinetic realizations of prescribed
rough finite-energy data and gives sign-changing examples.  Finally,
Section~\ref{sec:singularpassage} proves full-family rigidity up to a common
singular-time trace, quantitative temporal control of the passage, and global
subsequential Euler continuations for later times, with applications to
finite-time singularity scenarios.

\section{The $4/3$ shell mechanism for global energy compactness}\label{sec:bscompact}
The finite-energy axisymmetric Biot--Savart map converts relative-vorticity
compactness into velocity compactness.
Theorem~\ref{thm:mfreecomp} combines a cubic radial-shell estimate with
infinite-cylinder div--curl compactness to prevent loss of fluid $L^2$ energy
at large radius when $p\ge4/3$, without a sign condition or spatial moment.
After establishing the sharp radial power, we also record a passive-transport
realization of the subcritical scaling below $4/3$; it is a sharpness
mechanism, not an active Euler counterexample.
The radial-moment Biot--Savart estimates and their sharp endpoints are treated separately
in Sections~\ref{sec:radmom}--\ref{sec:oppsign}.
Throughout Sections~\ref{sec:bscompact}--\ref{sec:oppsign}, $\Hh:=\{(r,z)\in\R^2:r>0\}$.
For an axisymmetric vector field $u$ and an axisymmetric scalar $\varpi$, we use the Cartesian identities
\begin{equation}\label{eq:axisym-cart}
 ru^r=x_1u_1+x_2u_2,
 \qquad
 r\varpi e_\vartheta=(-x_2\varpi,x_1\varpi,0).
\end{equation}
\subsection{Div--curl identification and the finite-energy Biot--Savart map}\label{sec:divcurlmap}
The moment-free argument below needs only the finite-energy div--curl
identification.  The radial-moment estimates, kernel formulas, and stability
bounds used in the $3/2$ and $5/3$ mechanisms are collected instead at the
beginning of Section~\ref{sec:radmom}.

\begin{lemma}[Div--curl energy map]\label{lem:divcurl}
If $\Omega\in\dot H^{-1}(\R^3;\R^3)$ and
$\nabla\cdot\Omega=0$, then
\begin{equation}\label{eq:curl-hm}
 \bigl\|\nabla\times(-\Delta)^{-1}\Omega\bigr\|_2
 \le\|\Omega\|_{\dot H^{-1}}.
\end{equation}
If $u\in L^2$ is divergence free and $\Omega=\nabla\times u$, then
\begin{equation}\label{eq:curl-isom}
 u=\nabla\times(-\Delta)^{-1}\Omega,
 \qquad
 \|u\|_2=\|\Omega\|_{\dot H^{-1}}.
\end{equation}
Moreover, for every $G \in L^{6/5}(\R^3;\R^3)$,
\begin{equation}\label{eq:hminus-lp}
 \|G\|_{\dot H^{-1}}\lesssim\|G\|_{6/5}.
\end{equation}
\end{lemma}

The proof is recorded in Appendix~\ref{sec:auxana}.

For every axisymmetric scalar $\varpi$ such that
$r\varpi e_\vartheta\in\dot H^{-1}(\R^3;\R^3)$, define
\begin{equation}\label{eq:biot-basic}
 \Bcal\varpi:=\nabla\times(-\Delta)^{-1}(r\varpi e_\vartheta).
\end{equation}
Thus, whenever $u\in L^2(\R^3)$ is divergence free and
$\nabla\times u=r\varpi e_\vartheta$, Lemma~\ref{lem:divcurl} gives
$u=\Bcal\varpi$ uniquely in $L^2$.
\subsection{Moment-free energy compactness at the $L^{4/3}$ endpoint}
\begin{theorem}[Moment-free energy compactness at and above $4/3$]\label{thm:mfreecomp}
Fix $T<\infty$ and $4/3\le p<\infty$.  Let
$(a_n,\varpi_n,f_n,\pi_n)$ be axisymmetric no-swirl fields on
$[0,T]\times\R^3$ such that, for each $n$,
\begin{equation}\label{eq:td-id}
 \partial_ta_n+a_n\cdot\nabla a_n+\nabla\pi_n=f_n,
 \qquad \nabla\cdot a_n=\nabla\cdot f_n=0,
 \qquad \nabla\times a_n=r\varpi_ne_\vartheta,
\end{equation}
with
\[
 a_n\in C([0,T];H^{m_n}(\R^3))\cap
 C^1([0,T];H^{m_n-1}(\R^3)),\qquad m_n>\frac52.
\]
No uniform bound on $m_n$-th order norms is assumed.  Suppose
\begin{align}
 &\sup_n\|a_n\|_{L^\infty(0,T;L^2)}<\infty,
 \qquad
 \|f_n\|_{L^1(0,T;L^2)}\to0,
 \label{eq:free-bd}\\
 &\varpi_n\to\varpi
 \quad\text{strongly in }
 C([0,T];L^1(\R^3))\cap C([0,T];L^p(\R^3)),
 \label{eq:mom-conv}\\
 &a_n(0)\to u_0
 \quad\text{strongly in }L^2(\R^3).
 \label{eq:sys-conv}
\end{align}
Then the div--curl system associated with $\varpi$ has a unique solution
$u=\Bcal\varpi\in C([0,T];L^2(\R^3))$, and
\begin{equation}\label{eq:sol-conv}
 a_n\to u
 \qquad\text{strongly in }C([0,T];L^2(\R^3)).
\end{equation}
Moreover,
\begin{equation}\label{eq:comp-bd}
 \|u(t)\|_2=\|u_0\|_2,
 \qquad 0\le t\le T.
\end{equation}
Neither a sign condition nor a spatial moment assumption is required.
\end{theorem}
\begin{remark}
If \(6/5\le p<4/3 \), then \(a_n \) still converges to \(u \) locally in the radial variable. More precisely, for every finite \(R>0\),
\[
a_n\to u
\qquad\text{strongly in }C([0,T];L^2(\mathfrak C_R)),
\quad
\mathfrak C_R:=\{x\in\mathbb R^3, r(x)<R\}.
\]
This follows directly from Lemma \ref{lem:cylcurl} and interpolation.

\end{remark}

For $0<c<C<\infty$ and $R>0$, set $\mathcal C_R(c,C):=\{x\in\R^3:cR<r(x)<CR\}$.
We use the weighted axisymmetric Calder\'on--Zygmund estimate
\begin{equation}\label{eq:bs-bd}
 \left\|\frac{\nabla_x\Bcal\zeta}{r}\right\|_{L^s(\R^3)}
 \lesssim_s \|\zeta\|_{L^s(\R^3)},
 \qquad 1<s<2,
\end{equation}
proved in \cite[Lemma~4]{nobili2022renormalization}. Let $A_s(\R^3),1<s<\infty$ denote the Muckenhoupt class of weights \(w\) satisfying
\begin{align}
    &    
[w]_{A_s}
:=
\sup_{B \text{ ball in } \R^3}
\left(\frac1{|B|}\int_B w\,dx\right)
\left(\frac1{|B|}\int_B w^{-\frac1{s-1}}\,dx\right)^{s-1}
<\infty. \label{As}
\end{align}
The proof of \eqref{eq:bs-bd} uses the
three-dimensional Calder\'on--Zygmund estimate with the Muckenhoupt weight
$r^{-s}$, which belongs to $A_s(\R^3)$,
$1<s<2$. Hence, once the $L^2$ div--curl uniqueness in
Lemma~\ref{lem:divcurl} identifies a field with $\Bcal\zeta$, \eqref{eq:bs-bd}
applies directly; below we use it only in this form.
\begin{lemma}[Cubic shell estimate]\label{lem:cubshell}
Let $4/3\le s\le3/2$ and define
\begin{equation}\label{eq:est-def}
 \alpha_s:=\frac{5s-6}{4(s-1)},\qquad
 \gamma_s:=\frac{s}{2(s-1)},\qquad
 \delta_s:=\frac{3s-4}{4(s-1)}.
\end{equation}
For fixed $0<c<C<\infty$ and every smooth axisymmetric no-swirl field
$v\in L^2(\R^3)$ with $r^{-1}\nabla_xv\in L^s(\R^3)$,
\begin{align}\label{eq:shell-dx}
 \frac1R\int_{\mathcal C_R(c,C)}|v|^3\,dx
 &\lesssim_{c,C,s}
 R^{-\delta_s}\|v\|_2^{2\alpha_s}
 \left\|\frac{\nabla_xv}{r}\right\|_{L^s(\mathcal C_R(c/2,2C))}^{\gamma_s}
 +R^{-5/2}\|v\|_2^3.
\end{align}
At $s=4/3$ this becomes
\begin{equation}\label{eq:shell-xv}
 \frac1R\int_{\mathcal C_R(c,C)}|v|^3\,dx
 \lesssim_{c,C}
 \|v\|_2
 \left\|\frac{\nabla_xv}{r}\right\|_{L^{4/3}(\mathcal C_R(c/2,2C))}^{2}
 +R^{-5/2}\|v\|_2^3.
\end{equation}
\end{lemma}
\begin{proof}
A finite radial partition reduces the proof to $\mathcal C_R(1,2)$.  Write the
two meridional components of $v$ on $S_R=(R,2R)\times\R_z$ and partition it into
$Q_j=(R,2R)\times(jR,(j+1)R)$, $j\in\mathbb Z$.
All norms on $Q_j$ below are taken with respect to $dr\,dz$.  Set
$A_j:=\|v\|_{L^2(Q_j)}$, $B_j:=\|\nabla_{r,z}v\|_{L^s(Q_j)}$, and
$\vartheta_s:=s/[6(s-1)]$.  Rescaling the two-dimensional
Gagliardo--Nirenberg inequality from the unit square gives
\[
 \|v\|_{L^3(Q_j)}^3
 \lesssim_s
 A_j^{3(1-\vartheta_s)}B_j^{3\vartheta_s}
 +R^{-1}A_j^3.
\]
The first term is scale invariant because
$1/3=(1-\vartheta_s)/2+\vartheta_s(1/s-1/2)$.

Since $3(1-\vartheta_s)=2\alpha_s$ and
$3\vartheta_s=\gamma_s$, H\"older's inequality in $j$, followed by the
embedding $\ell^s\hookrightarrow\ell^{2s/(2-s)}$, gives
\[
 \sum_j A_j^{2\alpha_s}B_j^{\gamma_s}
 \lesssim_s
 \left(\sum_jA_j^2\right)^{\alpha_s}
 \left(\sum_jB_j^s\right)^{\gamma_s/s}.
\]
Also $\sum_jA_j^3\le(\sum_jA_j^2)^{3/2}$.  Since
$dx=2\pi r\,dr\,dz$ and $r\simeq R$ on $S_R$,
\[
 \sum_jA_j^2
 \lesssim R^{-1}\|v\|_{L^2(\mathcal C_R(1,2))}^2,
 \qquad
 \sum_jB_j^s
 \lesssim R^{s-1}
 \left\|\frac{\nabla_xv}{r}\right\|_{L^s(\mathcal C_R(1,2))}^s.
\]
Summing the local Gagliardo--Nirenberg estimate and substituting the two bounds above gives exactly the factor $R^{-\delta_s}$ in
\eqref{eq:shell-dx}. At $s=4/3$,
$\alpha_s=1/2$, $\gamma_s=2$, and $\delta_s=0$, which yields
\eqref{eq:shell-xv}.
\end{proof}
\begin{proposition}[Power sharpness of the cubic shell estimate]\label{prop:cubsharp}
Let $4/3\le s\le3/2$, and let $\alpha_s$, $\gamma_s$, and $\delta_s$
be as in \eqref{eq:est-def}.  There exists $R_0\ge1$ and,
for every $R\ge R_0$, an axisymmetric
$\zeta_R\in C_c^\infty(\R^3)$ such that
$v_R:=\Bcal\zeta_R\in C_c^\infty(\R^3;\R^3)$ is no-swirl,
\begin{equation}\label{eq:supp-shell}
 \operatorname{supp}v_R\subset\mathcal C_R(1,2),
\end{equation}
and
\begin{align}
 \frac1R\int_{\mathcal C_R(1,2)}|v_R|^3\,dx
 &\to c_*>0,
 \label{eq:shell-conv}\\
 \|v_R\|_2
 &\lesssim R^{1/2},
 \label{eq:mom-free}\\
 \left\|\frac{\nabla_xv_R}{r}\right\|_{L^s(\R^3)}
 &\lesssim_s R^{1/s-1}.
 \label{eq:xv-eta}
\end{align}
Consequently, for no $\eta>0$ and $K<\infty$ can the estimate
\begin{align}\label{eq:shell-bd}
 \frac1R\int_{\mathcal C_R(1,2)}|v|^3\,dx
 &\le K R^{-\delta_s-\eta}\|v\|_2^{2\alpha_s}
 \left\|\frac{\nabla_xv}{r}\right\|_{L^s(\mathcal C_R(1/2,4))}^{\gamma_s}
 +K R^{-5/2}\|v\|_2^3
\end{align}
hold for all $R\ge R_0$ and all smooth axisymmetric no-swirl
$v\in L^2(\R^3)$ with $r^{-1}\nabla_xv\in L^s(\R^3)$.
In particular, since $\delta_{4/3}=0$, the endpoint estimate
\eqref{eq:shell-xv} admits no uniform gain $R^{-\eta}$ with
$\eta>0$.
\end{proposition}
\begin{proof}
Choose a nonconstant
$\Psi\in C_c^\infty((-1/8,1/8)\times(-1/8,1/8))$ and put
$\lambda:=3/2$.  For $R\ge1$, set $\rho:=r-\lambda R$ and $\psi_R(r,z):=R\Psi(\rho,z)$,
and define the axisymmetric no-swirl field
\begin{equation}\label{eq:psi-def}
 v_R^r:=-\frac1r\partial_z\psi_R,
 \qquad
 v_R^z:=\frac1r\partial_r\psi_R,
 \qquad
 v_R:=v_R^re_r+v_R^ze_z.
\end{equation}
For $R$ sufficiently large, the support of $\psi_R$ is contained in
$\{R<r<2R\}$, so \eqref{eq:supp-shell} holds.  Since the support
stays away from the axis, $v_R$ is a smooth compactly supported vector
field.  Moreover,
\[
 \frac1r\partial_r(rv_R^r)+\partial_zv_R^z
 =-\frac1r\partial_r\partial_z\psi_R
 +\frac1r\partial_z\partial_r\psi_R=0.
\]
Thus $v_R$ is divergence free.  Define $\zeta_R:=(\nabla\times v_R)^\vartheta/r$.  Then $\zeta_R\in C_c^\infty(\R^3)$ is axisymmetric and
$\nabla\times v_R=r\zeta_Re_\vartheta$.  The $L^2$ div--curl uniqueness
of Lemma~\ref{lem:divcurl} therefore identifies
$v_R=\Bcal\zeta_R$.
On the support of $\Psi$, write $q_R(\rho):=R/(\lambda R+\rho)$.
Equation \eqref{eq:psi-def} becomes
\begin{equation}\label{eq:psi-id}
 v_R^r=-q_R\partial_z\Psi,
 \qquad
 v_R^z=q_R\partial_\rho\Psi.
\end{equation}
Since $q_R\to\lambda^{-1}$ uniformly on the fixed support of $\Psi$,
$0<c_0\le q_R\le C_0$ for all sufficiently large $R$.  Using
$dx=2\pi r\,dr\,dz$ and the change of variables $r=\lambda R+\rho$,
\begin{align}
 \frac1R\int_{\R^3}|v_R|^3\,dx
 &=2\pi\int_{\R^2}
 \left(\frac{R}{\lambda R+\rho}\right)^2
 |\nabla_{\rho,z}\Psi(\rho,z)|^3\,d\rho\,dz
 \notag\\
 &\to
 2\pi\lambda^{-2}\int_{\R^2}
 |\nabla_{\rho,z}\Psi|^3\,d\rho\,dz
 =:c_*>0. \notag
\end{align}
Together with \eqref{eq:supp-shell}, this is
\eqref{eq:shell-conv}.  The same change of variables gives
\begin{align}
 \|v_R\|_2^2
 &=2\pi R\int_{\R^2}
 \frac{R}{\lambda R+\rho}
 |\nabla_{\rho,z}\Psi(\rho,z)|^2\,d\rho\,dz
 \lesssim R, \notag
\end{align}
which proves \eqref{eq:mom-free}.
It remains to estimate the weighted gradient.  Differentiating \eqref{eq:psi-id} in $r$ and $z$, and using $|q_R|\lesssim1$ and $|\partial_rq_R|=R/(\lambda R+\rho)^2\lesssim R^{-1}$, gives
\begin{equation*}
 |\partial_rv_R|+|\partial_zv_R|+\frac{|v_R^r|}{r}
 \le C_\Psi
 \qquad\text{on }\operatorname{supp}v_R.
\end{equation*}
For an axisymmetric no-swirl field, the left-hand side controls
$|\nabla_xv_R|$ up to an absolute constant.  Since $r\simeq R$ on the
support and $|\operatorname{supp}v_R|\lesssim R$,
\[
 \left\|\frac{\nabla_xv_R}{r}\right\|_s^s
 \lesssim_{\Psi,s} R\,R^{-s},
\]
which is \eqref{eq:xv-eta}.
Suppose now that \eqref{eq:shell-bd} holds for some
$\eta>0$.  Apply it to $v=v_R$.  By
\eqref{eq:mom-free}--\eqref{eq:xv-eta}, the first term on
its right-hand side is bounded by
\[
 C R^{-\delta_s-\eta}
 R^{\alpha_s}R^{\gamma_s(1/s-1)}.
\]
By the definitions in \eqref{eq:est-def}, the displayed power of $R$ is
$-\eta$.  Hence this term is $O(R^{-\eta})$.  The second term is
$O(R^{-1})$ by \eqref{eq:mom-free}.  Both tend to zero, contradicting
\eqref{eq:shell-conv}.  Therefore no such improvement is possible.
\end{proof}

\begin{proposition}[Passive-transport realization of the subcritical shell obstruction]
\label{prop:passive-shell-obstruction}
Let $1<p<4/3$.  There exists a fixed smooth axisymmetric divergence-free
no-swirl vector field $b$ on $\R^3$ and, for every sufficiently large
$R$, a time $T_R\to\infty$ and a smooth compactly supported axisymmetric
solution
\begin{equation}\label{eq:passive-transport}
 \partial_t\theta_R+b\cdot\nabla\theta_R=0
 \qquad\text{on }[0,T_R]\times\R^3
\end{equation}
such that
\begin{equation}\label{eq:passive-small}
 \sup_{0\le t\le T_R}
 \|\theta_R(t)\|_{L^1\cap L^p}\longrightarrow0,
 \qquad
 \|\Bcal\theta_R(0)\|_2\longrightarrow0,
\end{equation}
whereas
\begin{equation}\label{eq:passive-terminal}
 \|\Bcal\theta_R(T_R)\|_2\longrightarrow c_2>0,
 \qquad
 \frac1R\int_{\R^3}|\Bcal\theta_R(T_R)|^3\,dx
 \longrightarrow c_3>0.
\end{equation}
Moreover, for every $1\le q<\infty$,
\begin{equation}\label{eq:passive-lq-scaling}
 \|\theta_R(t)\|_q
 =\|\theta_R(T_R)\|_q
 \asymp_q R^{\frac32-\frac2q}
 \qquad(0\le t\le T_R),
\end{equation}
and, for all sufficiently large $R$,
\begin{equation}\label{eq:passive-terminal-support}
 \operatorname{supp}\theta_R(T_R)
 \cup\operatorname{supp}\Bcal\theta_R(T_R)
 \subset\mathcal C_R(1,2).
\end{equation}
In particular,
$\|\theta_R(T_R)\|_{4/3}\asymp1$.
\end{proposition}

\begin{remark}
    Proposition \ref{prop:passive-shell-obstruction} identifies a passive-transport obstruction below the moment-free \(4/3\) threshold, but leaves open whether the same large-radius energy escape can occur self-consistently under the Euler dynamics. More precisely, can one construct conservative smooth or vanishing-forcing Euler approximations
whose energy escapes to spatial infinity in the weak limit?  A positive answer would also raise the kinetic question of whether such a family can
be lifted to collisional dynamics so as to produce a renormalized Boltzmann solution \cite{diperna1989cauchy} with a strict global energy inequality.
\end{remark}

\begin{proof}
\emph{Step 1: terminal shell packet.}
Fix $\lambda=3/2$ and
$0\not\equiv\Psi\in C_c^\infty((-1/8,1/8)^2)$.  Put
\[
 \ell_R:=R^{-3/2},\qquad A_R:=R^{1/2},\qquad
 \rho:=\frac{r-\lambda R}{\ell_R},\qquad
 \zeta:=\frac z{\ell_R},
\]
and define the axisymmetric stream function
\[
 \psi_R(r,z):=A_R\Psi(\rho,\zeta).
\]
Set
\begin{equation}\label{eq:passive-stream-field}
 u_R^r:=-\frac1r\partial_z\psi_R,
 \qquad
 u_R^z:=\frac1r\partial_r\psi_R,
 \qquad
 u_R:=u_R^re_r+u_R^ze_z.
\end{equation}
For all sufficiently large $R$, both $\psi_R$ and $u_R$ are supported
in $\mathcal C_R(1,2)$ and hence away from the axis.  Moreover,
\[
 \frac1r\partial_r(ru_R^r)+\partial_zu_R^z
 =-\frac1r\partial_r\partial_z\psi_R
  +\frac1r\partial_z\partial_r\psi_R=0.
\]
Define the relative vorticity
\begin{equation}\label{eq:passive-vartheta-def}
 \vartheta_R:=\frac{(\nabla\times u_R)^\vartheta}{r}.
\end{equation}
Then $\vartheta_R\in C_c^\infty(\R^3)$ is axisymmetric,
$\nabla\times u_R=r\vartheta_Re_\vartheta$, and Lemma~\ref{lem:divcurl}
gives the exact identity
\begin{equation}\label{eq:passive-biot-id}
 u_R=\Bcal\vartheta_R.
\end{equation}

From \eqref{eq:passive-stream-field},
\[
 u_R=\frac{A_R}{r\ell_R}
 \bigl(-\Psi_\zeta e_r+\Psi_\rho e_z\bigr).
\]
Using $dx=2\pi r\ell_R^2\,d\rho\,d\zeta$ and
$r=\lambda R+\ell_R\rho$, we obtain
\begin{align}
 \|u_R\|_2^2
 &=2\pi R\int_{\R^2}
   \frac{|\nabla\Psi(\rho,\zeta)|^2}
        {\lambda R+\ell_R\rho}\,d\rho\,d\zeta
 \longrightarrow
 \frac{2\pi}{\lambda}\int_{\R^2}|\nabla\Psi|^2
 =:c_2^2>0,
 \label{eq:passive-energy-limit}\\
 \frac1R\int_{\R^3}|u_R|^3\,dx
 &=2\pi\int_{\R^2}
   \left(\frac{R}{\lambda R+\ell_R\rho}\right)^2
   |\nabla\Psi(\rho,\zeta)|^3\,d\rho\,d\zeta
 \longrightarrow
 \frac{2\pi}{\lambda^2}\int_{\R^2}|\nabla\Psi|^3
 =:c_3>0.
 \label{eq:passive-cubic-limit}
\end{align}
The positivity follows because a nonzero compactly supported smooth
function cannot have identically vanishing gradient.

The stream-function identity also yields
\begin{equation}\label{eq:passive-vartheta-formula}
 \vartheta_R
 =-\frac{A_R}{r^2\ell_R^2}\Delta_{\rho,\zeta}\Psi(\rho,\zeta)
  +\frac{A_R}{r^3\ell_R}\Psi_\rho(\rho,\zeta).
\end{equation}
Equivalently,
\begin{equation}\label{eq:passive-vartheta-scaled}
 R^{-3/2}\vartheta_R
 =-\left(\frac Rr\right)^2\Delta_{\rho,\zeta}\Psi
  +R^{-5/2}\left(\frac Rr\right)^3\Psi_\rho.
\end{equation}
Since $r/R\to\lambda$ uniformly on the fixed $(\rho,\zeta)$ support and
$R^2dx=2\pi(r/R)\,d\rho\,d\zeta$, dominated convergence gives, for every
$1\le q<\infty$,
\begin{equation}\label{eq:passive-vartheta-scaling}
 R^{2-\frac{3q}{2}}\|\vartheta_R\|_q^q
 \longrightarrow
 2\pi\lambda^{1-2q}
 \int_{\R^2}|\Delta_{\rho,\zeta}\Psi|^q>0.
\end{equation}
Here $\Delta_{\rho,\zeta}\Psi\not\equiv0$: otherwise $\Psi$ would be compactly
supported and harmonic, hence zero.  Thus
\begin{equation}\label{eq:passive-vartheta-asymp}
 \|\vartheta_R\|_q\asymp_q R^{\frac32-\frac2q}.
\end{equation}

\emph{Step 2: realization by one fixed incompressible transport.}
Choose $\eta\in C_c^\infty(\R)$ with $\eta=1$ on $[-1/2,1/2]$, and
choose $a\in C^\infty([0,\infty))$ so that
\[
 a(r)=r^2\quad(0\le r\le1),
 \qquad
 a(r)=1\quad(r\ge2),
\]
with $a$ a smooth function of $r^2$ near $r=0$.  Define
\[
 \psi_b(r,z):=-a(r)z\eta(z),
 \qquad
 b^r:=-\frac1r\partial_z\psi_b,
 \qquad
 b^z:=\frac1r\partial_r\psi_b,
 \qquad b:=b^re_r+b^ze_z.
\]
For $r>0$ the stream-function identity gives $\nabla\cdot b=0$.  Near
$r=0$, where $a(r)=r^2$,
\[
 b^r=r(\eta+z\eta'),\qquad b^z=-2z\eta,
\]
so the Cartesian vector field extends smoothly across the axis.  Hence
$b$ is a fixed smooth axisymmetric divergence-free no-swirl field on
$\R^3$.  On the transport core
\begin{equation}\label{eq:passive-core-flow}
 \{r\ge2,\ |z|\le1/2\}
\end{equation}
one has
\begin{equation}\label{eq:passive-core-identity}
 b^r=\frac1r,\qquad b^z=0.
\end{equation}

Set
\[
 s:=\frac{r^2}{2},\qquad
 s_R:=\frac{\lambda^2R^2}{2},\qquad
 s_*:=\frac92,\qquad
 T_R:=s_R-s_*.
\]
Then $T_R\to\infty$.  Let
\[
 \widetilde\vartheta_R(s,z):=\vartheta_R(\sqrt{2s},z)
 \qquad(s>0),
\]
and extend $\widetilde\vartheta_R$ by zero to $s\le0$.  This extension
is smooth because $\vartheta_R$ is supported at $r\simeq R$.  Since
$|r-\lambda R|\lesssim\ell_R$ on its support,
\begin{equation}\label{eq:passive-s-support}
 |s-s_R|\lesssim R\ell_R=R^{-1/2},
 \qquad |z|\lesssim\ell_R=R^{-3/2}.
\end{equation}
Define
\begin{equation}\label{eq:passive-theta-def}
 \widetilde\theta_R(t,s,z)
 :=\widetilde\vartheta_R(s-t+T_R,z),
 \qquad
 \theta_R(t,r,z)
 :=\widetilde\theta_R\!\left(t,\frac{r^2}{2},z\right).
\end{equation}
By \eqref{eq:passive-s-support}, the support of $\theta_R(t)$ satisfies
\[
 |s-(s_*+t)|\lesssim R^{-1/2},
 \qquad |z|\lesssim R^{-3/2}
 \qquad(0\le t\le T_R).
\]
Consequently, for all sufficiently large $R$ the entire transported
packet remains in the core \eqref{eq:passive-core-flow}.  There
$b\cdot\nabla=r^{-1}\partial_r=\partial_s$, while
\eqref{eq:passive-theta-def} gives
$\partial_t\widetilde\theta_R+\partial_s\widetilde\theta_R=0$.
Thus \eqref{eq:passive-transport} holds, and
\begin{equation}\label{eq:passive-terminal-id}
 \theta_R(T_R)=\vartheta_R.
\end{equation}
This also proves \eqref{eq:passive-terminal-support} in view of
\eqref{eq:passive-biot-id}.

For axisymmetric integrals, $dx=2\pi\,ds\,dz$.  Translation in $s$ in
\eqref{eq:passive-theta-def}, together with the fact that the support
stays in $s>0$, therefore gives
\begin{equation}\label{eq:passive-norm-preserve}
 \|\theta_R(t)\|_q=\|\vartheta_R\|_q
 \qquad(1\le q<\infty,\ 0\le t\le T_R).
\end{equation}
Combining this identity with \eqref{eq:passive-vartheta-asymp} proves
\eqref{eq:passive-lq-scaling}.  In particular,
\[
 \sup_{0\le t\le T_R}\|\theta_R(t)\|_1
 \lesssim R^{-1/2}\to0,
 \qquad
 \sup_{0\le t\le T_R}\|\theta_R(t)\|_p
 \lesssim R^{\frac32-\frac2p}\to0
\]
because $p<4/3$.

\emph{Step 3: small initial Biot--Savart energy and terminal escape.}
At $t=0$, \eqref{eq:passive-s-support}--\eqref{eq:passive-theta-def}
place the support of $\theta_R(0)$ in a fixed annulus $2<r<4$ for all
sufficiently large $R$.  Lemma~\ref{lem:divcurl}, specifically
\eqref{eq:curl-hm}--\eqref{eq:hminus-lp}, therefore gives
\[
 \|\Bcal\theta_R(0)\|_2
 \le \|r\theta_R(0)e_\vartheta\|_{\dot H^{-1}}
 \lesssim \|r\theta_R(0)\|_{6/5}
 \lesssim \|\theta_R(0)\|_{6/5}
 \asymp R^{-1/6}\longrightarrow0.
\]
This proves the second assertion in \eqref{eq:passive-small}.  Finally,
\eqref{eq:passive-terminal-id}, \eqref{eq:passive-biot-id},
\eqref{eq:passive-energy-limit}, and \eqref{eq:passive-cubic-limit}
give \eqref{eq:passive-terminal}.  Taking $q=4/3$ in
\eqref{eq:passive-lq-scaling} gives
$\|\theta_R(T_R)\|_{4/3}\asymp1$ and completes the proof.
\end{proof}

\begin{remark}[Shell sharpness, passive transport, and endpoint compactness]
\label{rem:passive-shell-obstruction}
Proposition~\ref{prop:cubsharp} shows that the radial power in
Lemma~\ref{lem:cubshell} cannot be improved within the axisymmetric
Biot--Savart class.  Proposition~\ref{prop:passive-shell-obstruction}
shows, in addition, that below $4/3$ the corresponding subcritical
scaling can be realized dynamically by one fixed smooth incompressible
transport: arbitrarily small transported $L^1\cap L^p$ data with
vanishing initial Biot--Savart energy can be carried to a far-field
packet with order-one Biot--Savart energy and an order-one cubic shell quantity.

This is not a counterexample to Theorem~\ref{thm:mfreecomp}, nor to any
active Euler compactness statement below $4/3$.  The construction lies
strictly below the theorem's range, uses a prescribed transport field
$b\neq\Bcal\theta_R$ in general, and reaches the far field only over
times $T_R\to\infty$.  At the endpoint $p=4/3$, the same packet is
scale-invariant in $L^{4/3}$ rather than small.  Theorem~\ref{thm:mfreecomp}
closes exactly there because strong $C_tL_x^{4/3}$ convergence yields the
weighted-gradient tightness of Lemma~\ref{lem:endpttight}.
\end{remark}
\begin{lemma}[Endpoint weighted-gradient tightness]\label{lem:endpttight}
Under the hypotheses of Theorem~\ref{thm:mfreecomp} with $p=4/3$,
\begin{equation}\label{eq:xa-conv}
 \lim_{R\to\infty}\limsup_{n\to\infty}
 \sup_{0\le t\le T}
 \left\|\frac{\nabla_xa_n(t)}r\right\|_{L^{4/3}(\{r>R\})}=0.
\end{equation}
\end{lemma}
\begin{proof}
The $L^2$ div--curl uniqueness in Lemma~\ref{lem:divcurl} identifies
$a_n-a_m=\Bcal(\varpi_n-\varpi_m)$.  Apply
\eqref{eq:bs-bd} to this difference.  Since
\eqref{eq:mom-conv} is strong in
$C([0,T];L^{4/3})$,
\[
 \left\{\frac{\nabla_xa_n}{r}\right\}_{n\ge1}
 \quad\text{is Cauchy in }C([0,T];L^{4/3}(\R^3)).
\]
Let $G$ be its limit.  The image $G([0,T])$ is compact in $L^{4/3}$ and
is therefore uniformly spatially tight.  Uniform convergence then implies
\eqref{eq:xa-conv}.
\end{proof}
\begin{proposition}[Vanishing cubic shell flux]\label{prop:fluxvan}
Under the hypotheses of Theorem~\ref{thm:mfreecomp}, for every fixed
$0<c<C<\infty$,
\begin{equation}\label{eq:flux-conv}
 \lim_{R\to\infty}\limsup_{n\to\infty}
 \sup_{0\le t\le T}
 \frac1R\int_{\mathcal C_R(c,C)}|a_n(t)|^3\,dx=0.
\end{equation}
\end{proposition}
\begin{proof}
If $p=4/3$, use \eqref{eq:shell-xv},
Lemma~\ref{lem:endpttight}, and the uniform $L^2$ bound in
\eqref{eq:free-bd}.  If $p>4/3$, choose $s=p$ for $4/3<p\le3/2$ and $s=3/2$ for $p>3/2$.  The strong convergence in \eqref{eq:mom-conv} and
interpolation give a uniform $C_tL^s_x$ bound.  The $L^2$ div--curl
uniqueness in Lemma~\ref{lem:divcurl} identifies
$a_n=\Bcal\varpi_n$, so \eqref{eq:bs-bd} yields a uniform bound
for $r^{-1}\nabla_xa_n$ in $C_tL^s_x$.  Since $s>4/3$,
$\delta_s>0$ in \eqref{eq:est-def}, and
Lemma~\ref{lem:cubshell} proves \eqref{eq:flux-conv}.
\end{proof}
\begin{lemma}[Pressure localization]\label{lem:preslocal}
Under the hypotheses of Theorem~\ref{thm:mfreecomp}, normalize the pressure by
\begin{equation}\label{eq:ir-def}
 \pi_n:=\sum_{i,j=1}^3R_iR_j(a_{n,i}a_{n,j}).
\end{equation}
Here $R_j$ denotes the $j$-th Riesz transform. Then
\begin{equation}\label{eq:free-conv}
 \lim_{R\to\infty}\limsup_{n\to\infty}
 \sup_{0\le t\le T}
 \frac1R\int_{\mathcal C_R(1,2)}|\pi_n(t)|\,|a_n(t)|\,dx=0.
\end{equation}
\end{lemma}
\begin{proof}
Since $\nabla\cdot a_n=\nabla\cdot f_n=0$, applying the complementary
Leray projection to \eqref{eq:td-id} shows that the equation may
be written with the canonical pressure \eqref{eq:ir-def}; changing a
pressure by a function of time is immaterial.
Choose a radial cutoff $\psi_R=\psi_R(r)$ such that
$\psi_R=1$ on $\{R/2<r<4R\}$ and
$\operatorname{supp}\psi_R\subset\{R/4<r<8R\}$.  Write
\[
 \pi_{n,\mathrm{near}}
 :=\sum_{i,j}R_iR_j(\psi_Ra_{n,i}a_{n,j}),
 \qquad
 \pi_{n,\mathrm{far}}:=\pi_n-\pi_{n,\mathrm{near}}.
\]
Calder\'on--Zygmund boundedness on $L^{3/2}$ and H\"older give
\[
 \frac1R\int_{\mathcal C_R(1,2)}
 |\pi_{n,\mathrm{near}}|\,|a_n|\,dx
 \lesssim
 \frac1R
 \|a_n\|_{L^3(\mathcal C_R(1/4,8))}^2
 \|a_n\|_{L^3(\mathcal C_R(1,2))}.
\]
With $J_{n,R}(c,C):=R^{-1}\int_{\mathcal C_R(c,C)}|a_n|^3\,dx$, the right-hand side is bounded by
$CJ_{n,R}(1/4,8)^{2/3}J_{n,R}(1,2)^{1/3}$ and therefore vanishes by
Proposition~\ref{prop:fluxvan}.
For the far part, source and target have Euclidean distance at least
$cR$.  Indeed, $(1-\psi_R)a_{n,i}a_{n,j}$ vanishes on the open annulus
$\{R/2<r<4R\}$ containing the target cylinder
$\mathcal C_R(1,2)$.  Thus any distributional kernel term of $R_iR_j$
supported on the diagonal (in particular the local term when $i=j$)
vanishes on the target, and only the off-diagonal Calder\'on--Zygmund
kernel is present.  This kernel is $O(|x-y|^{-3})$.  Minkowski's
inequality gives
\begin{equation*}
 \|\pi_{n,\mathrm{far}}\|_{L^2(\mathcal C_R(1,2))}
 \lesssim R^{-3/2}\|a_n\|_2^2,
\end{equation*}
because $\sup_y(\int_{|x-y|\ge cR}|x-y|^{-6}\,dx)^{1/2}\lesssim R^{-3/2}$.  Thus
\[
 \frac1R\int_{\mathcal C_R(1,2)}
 |\pi_{n,\mathrm{far}}|\,|a_n|\,dx
 \lesssim R^{-5/2}\|a_n\|_2^3.
\]
Together with \eqref{eq:free-bd}, this proves
\eqref{eq:free-conv}.
\end{proof}
\begin{proposition}[Radial no-escape]\label{prop:radnoesc}
Under the hypotheses of Theorem~\ref{thm:mfreecomp},
\begin{equation}\label{eq:rad-conv}
 \lim_{R\to\infty}\limsup_{n\to\infty}
 \sup_{0\le t\le T}\int_{\{r>R\}}|a_n(t)|^2\,dx=0.
\end{equation}
\end{proposition}
\begin{proof}
Let $\chi_R=\chi_R(r)$ satisfy
$0\le\chi_R\le1$, $\chi_R=0$ on $\{r\le R\}$,
$\chi_R=1$ on $\{r\ge2R\}$, and
$|\nabla\chi_R|\lesssim R^{-1}$.  Since
$a_n\in L^2\cap L^\infty$, one has $|a_n|^3\in L^1$; moreover
$a_n\in L^4$, so \eqref{eq:ir-def} gives
$\pi_n\in L^2$ and $\pi_na_n\in L^1$.
Thus the local energy identity with the bounded radial cutoff $\chi_R$ is
obtained by multiplying first by an additional compact cutoff and sending
its radius to infinity.  It gives, for $0\le t\le T$,
\begin{align}
 \frac12\int\chi_R|a_n(t)|^2\,dx
 &=\frac12\int\chi_R|a_n(0)|^2\,dx   +\int_0^t\!\int
 \left(\frac{|a_n|^2}{2}+\pi_n\right)
 a_n\cdot\nabla\chi_R\,dx\,ds
 +\int_0^t\!\int\chi_R f_n\cdot a_n\,dx\,ds.\notag
\end{align}
The cubic and pressure contributions vanish in the iterated limit
$n\to\infty$, $R\to\infty$ by Proposition~\ref{prop:fluxvan} and
Lemma~\ref{lem:preslocal}.  The forcing contribution is bounded by $\|a_n\|_{L^\infty_tL^2_x}\|f_n\|_{L^1_tL^2_x}=o_n(1)$.  Strong convergence in \eqref{eq:sys-conv} makes $\{a_n(0)\}$ uniformly $L^2$-tight.  Since $\int_{\{r>2R\}}|a_n(t)|^2\,dx\le\int\chi_R|a_n(t)|^2\,dx$, \eqref{eq:rad-conv} follows.
\end{proof}
For $R>0$, define the infinite radial cylinder $\mathfrak C_R:=\{x\in\R^3:r(x)<R\}$.  Unlike a ball, $\mathfrak C_R$ is unbounded in the axial variable.
\begin{lemma}[Div--curl control on an infinite cylinder]\label{lem:cylcurl}
Let $1<q<6/5$.  Let $v\in L^2(\R^3)$ be axisymmetric without swirl and
suppose
\[
 \nabla\cdot v=0,
 \qquad
 \nabla\times v=r\eta e_\vartheta,
 \qquad
 \eta\in L^q(\R^3)\cap L^{6/5}(\R^3).
\]
Then
\begin{equation}\label{eq:eta-qr}
 \|\mathbf1_{\mathfrak C_R}v\|_2
 \lesssim_R\|\eta\|_{6/5}
 +C_qR^{7/2-3/q}\|\eta\|_q.
\end{equation}
In particular, if $(v_n,\eta_n)$ satisfy the same div--curl relation and
$\eta_n\to0$ in $L^q\cap L^{6/5}$, then
\begin{equation*}
 v_n\to0
 \qquad\text{in }L^2(\mathfrak C_R).
\end{equation*}
\end{lemma}
The proof is recorded in Appendix~\ref{sec:auxana}.

\begin{proof}[Proof of Theorem~\ref{thm:mfreecomp}]
Fix $q\in(1,6/5)$.  Since $p\ge4/3$, interpolation in
\eqref{eq:mom-conv} gives
\[
 \varpi_n-\varpi\to0
 \qquad\text{in }C([0,T];L^q(\R^3)\cap L^{6/5}(\R^3)).
\]
For $n,m$, apply Lemma~\ref{lem:cylcurl} to
$v=a_n-a_m$ and $\eta=\varpi_n-\varpi_m$.  This convergence and
\eqref{eq:mom-conv} show
that, for every fixed $R$,
\[
 \{a_n\}_{n\ge1}
 \quad\text{is Cauchy in }C([0,T];L^2(\mathfrak C_R)).
\]
The limits on nested cylinders agree.  Hence there is an axisymmetric
no-swirl field $u$ such that, for every $R>0$,
\begin{equation}\label{eq:fld-conv}
 a_n\to u
 \qquad\text{in }C([0,T];L^2(\mathfrak C_R)).
\end{equation}
Fix $0<R<S$.  By \eqref{eq:fld-conv}, for every $t$,
\[
 \int_{\{R<r<S\}}|u(t)|^2\,dx
 =\lim_{n\to\infty}\int_{\{R<r<S\}}|a_n(t)|^2\,dx
 \le
 \limsup_{n\to\infty}\sup_{0\le s\le T}
 \int_{\{r>R\}}|a_n(s)|^2\,dx.
\]
Letting $S\to\infty$ gives
\begin{equation}\label{eq:en-conv}
 \sup_{0\le t\le T}\int_{\{r>R\}}|u(t)|^2\,dx
 \le
 \limsup_{n\to\infty}\sup_{0\le s\le T}
 \int_{\{r>R\}}|a_n(s)|^2\,dx.
\end{equation}
Thus Proposition~\ref{prop:radnoesc} gives $u(t)\in L^2$ and uniform
radial $L^2$-tightness of $u$.
For every $R$,
\begin{align}
 \sup_{0\le t\le T}\|a_n(t)-u(t)\|_2
 &\le
 \sup_t\|\mathbf1_{\mathfrak C_R}(a_n-u)(t)\|_2   +\sup_t\|\mathbf1_{\{r>R\}}a_n(t)\|_2
 +\sup_t\|\mathbf1_{\{r>R\}}u(t)\|_2.\notag
\end{align}
First let $n\to\infty$ using \eqref{eq:fld-conv};
then let $R\to\infty$ using \eqref{eq:rad-conv} and
\eqref{eq:en-conv}. Together with the continuity of $a_n(t)$ in $t$, this proves \eqref{eq:sol-conv}
and $u\in C([0,T];L^2)$.  At $t=0$,
\eqref{eq:sys-conv} gives $u(0)=u_0$.
For every $\Phi\in C_c^\infty(\R^3;\R^3)$, choose $R$ with
$\operatorname{supp}\Phi\subset\mathfrak C_R$.  Passing to the limit in
\[
 \int a_n\cdot(\nabla\times\Phi)\,dx
 =\int r\varpi_ne_\vartheta\cdot\Phi\,dx
\]
by \eqref{eq:fld-conv} and
\eqref{eq:mom-conv}, and passing to the limit in the divergence identity gives
\[
 \nabla\cdot u=0,
 \qquad
 \nabla\times u=r\varpi e_\vartheta.
\]
The $L^2$ div--curl solution is unique by Lemma~\ref{lem:divcurl}; hence
$u=\Bcal\varpi$.
Finally, $m_n>5/2$ implies $a_n\in L^\infty_x\cap L^2_x$, hence
$a_n\in L^4_x$ and $\pi_n\in L^2_x$.  A compact-cutoff argument in
\eqref{eq:td-id} gives
\begin{equation*}
 \frac12\|a_n(t)\|_2^2-\frac12\|a_n(0)\|_2^2
 =\int_0^t\!\int_{\R^3}f_n\cdot a_n\,dx\,ds.
\end{equation*}
Consequently,
\[
 \sup_{0\le t\le T}
 \left|\|a_n(t)\|_2^2-\|a_n(0)\|_2^2\right|
 \le2\|a_n\|_{L^\infty_tL^2_x}\|f_n\|_{L^1_tL^2_x}
 \to0.
\]
Together with \eqref{eq:sol-conv} and \eqref{eq:sys-conv}, this gives
\eqref{eq:comp-bd}.
\end{proof}
\section{The first radial moment and the $3/2$ threshold}\label{sec:radmom}
For the first radial moment, $p=3/2$ is the endpoint at which
$L^1\cap L^p$ control and a finite initial first moment yield a sign-free
class-uniform differential estimate.  For $p<3/2$, finite absolute impulse
still controls the first moment through \eqref{eq:mominterp}. This first-radial-moment theory is logically distinct from the moment-free
energy-compactness theorem of Section~\ref{sec:bscompact}.  The dynamical statements
below use only the stated div--curl relations, integrability conditions,
and bounded-weight balance identities; no additional weak-solution
formulation is imposed.

\subsection{Radial-moment Biot--Savart estimates}\label{sec:radbiot}
For an axisymmetric relative vorticity $\varpi=\omega^\vartheta/r$, define
\begin{equation}\label{eq:radmom-def}
\mathfrak M_k(\varpi)
:=\frac1{2\pi}\int_{\mathbb R^3}r^k|\varpi(x)|dx,
\qquad k=1,2.
\end{equation}
Thus $\mathfrak M_1$ is the \textit{first radial moment}, while
$\mathfrak M_2$ is the \textit{absolute impulse}. By Cauchy--Schwarz,
\begin{equation}\label{eq:mominterp}
\mathfrak M_1(\varpi)^2
\le \frac1{2\pi}\|\varpi\|_1\mathfrak M_2(\varpi).
\end{equation}

Equation~\eqref{eq:radmom-def} and cylindrical coordinates give
\begin{equation}\label{eq:mtwo-cyl}
 \mathfrak M_2(\varpi)
 =\int_{\Hh}r^3|\varpi(r,z)|\,dr\,dz
 =\int_{\Hh}r^2|\omega^\vartheta(r,z)|\,dr\,dz.
\end{equation}
The cylindrical identity is
\begin{equation}\label{eq:lp-cyl}
 \|\varpi\|_{L^p(\R^3)}^p
 =2\pi\int_{\Hh}r|\varpi(r,z)|^p\,dr\,dz.
\end{equation}
By cylindrical coordinates, \eqref{eq:radmom-def} is equivalently
\begin{equation}\label{eq:mone-cyl}
 \mathfrak M_1(\varpi)=\int_{\Hh}r^2|\varpi(r,z)|\,dr\,dz.
\end{equation}
\begin{definition}[Radial-moment classes]
For $1<p<\infty$, set
\begin{align}
 \mathcal Y_p
 &:=\{\varpi:\ \varpi\text{ is axisymmetric},\ \varpi\in L^p(\R^3),\ \mathfrak M_1(\varpi)<\infty\}, \notag
 \\
 \mathcal X_p
 &:=\{\varpi:\ \varpi\text{ is axisymmetric},\ \varpi\in L^p(\R^3),\ \mathfrak M_2(\varpi)<\infty\}. \notag
\end{align}
\end{definition}

\begin{proposition}[Radial-moment energy estimates]\label{prop:radmom}
Let $1<p<\infty$ and define
\begin{equation}\label{eq:radexp}
 A_p=\frac{3p}{5p-3},\qquad
 B_p=\frac{7p-6}{5p-3},\qquad
 A_{1,p}=\frac{p}{4p-3},\qquad
 B_{1,p}=\frac{7p-6}{4p-3}.
\end{equation}
If $\varpi\in\mathcal X_p$, there is a unique $u\in L^2(\R^3;\R^3)$ satisfying
\begin{equation}\label{eq:divcurl-sys}
 \nabla\cdot u=0,\qquad
 \nabla\times u=r\varpi e_\vartheta
 \quad\text{in }\mathcal D'(\R^3).
\end{equation}
Writing $u=\Bcal\varpi$,
\begin{equation}\label{eq:bs-xbound}
 \|\Bcal\varpi\|_{L^2_x}^2
 \lesssim_p
 \|\varpi\|_{L^p}^{A_p}
 \mathfrak M_2(\varpi)^{B_p}.
\end{equation}
For $\varpi,\eta\in\mathcal X_p$,
\begin{equation}\label{eq:bs-xstab}
 \|\Bcal\varpi-\Bcal\eta\|_2^2
 \lesssim_p\|\varpi-\eta\|_p^{A_p}
 \bigl(\mathfrak M_2(\varpi)+\mathfrak M_2(\eta)\bigr)^{B_p}.
\end{equation}
If $\varpi\in\mathcal Y_p$, the same div--curl problem
\eqref{eq:divcurl-sys} has a unique solution in $L^2$, and
\begin{equation}\label{eq:bs-ybound}
 \|\Bcal\varpi\|_2^2
 \lesssim_p
 \|\varpi\|_p^{A_{1,p}}
 \mathfrak M_1(\varpi)^{B_{1,p}}.
\end{equation}
For $\varpi,\eta\in\mathcal Y_p$,
\begin{equation}\label{eq:bs-ystab}
 \|\Bcal\varpi-\Bcal\eta\|_2^2
 \lesssim_p
 \|\varpi-\eta\|_p^{A_{1,p}}
 \bigl(\mathfrak M_1(\varpi)+\mathfrak M_1(\eta)\bigr)^{B_{1,p}}.
\end{equation}
In both radial-moment classes, $\Bcal\varpi$ is axisymmetric without swirl.
\end{proposition}
\begin{lemma}[Axisymmetric kernel bounds]\label{lem:axiskern}
For $s>0$, set
\begin{equation}\label{eq:kern-def}
 \mathcal F(s)
 :=\int_0^\pi
 \frac{\cos\alpha}{\bigl(s+2(1-\cos\alpha)\bigr)^{1/2}}\,d\alpha.
\end{equation}
For $0<\sigma<3/2$ and $0\le\sigma'\le1/2$,
\begin{equation}\label{eq:kern-bound}
 |\mathcal F(s)|\lesssim_\sigma s^{-\sigma},\qquad s>0,
\end{equation}
and
\begin{equation}\label{eq:kern-der}
 |\mathcal F'(s)|\lesssim_{\sigma'} s^{-1-\sigma'},\qquad s>0.
\end{equation}
\end{lemma}
The proof is given in Appendix~\ref{sec:auxana}.

\begin{lemma}\label{lem:bsformula}
Let $\varpi\in C_c^\infty(\R^3)$ be axisymmetric, set
\begin{equation}\label{eq:vortlift}
 \omega=r\varpi e_\vartheta,
\end{equation}
and define
\begin{equation}\label{eq:bs-def}
 u=\nabla\times(-\Delta)^{-1}\omega.
\end{equation}
Then $u$ is axisymmetric without swirl, $u\in L^2(\R^3)$, and
\begin{align}
 \|u\|_2^2
 &=\iint_{\Hh\times\Hh}
 (r\rho)^{3/2}
 \mathcal F\left(
 \frac{(r-\rho)^2+(z-\zeta)^2}{r\rho}
 \right)
 \varpi(r,z)\varpi(\rho,\zeta)
 \,dr\,dz\,d\rho\,d\zeta.
 \label{eq:bs-energy}
\end{align}
The value of the displayed integrand on the diagonal
$(r,z)=(\rho,\zeta)$ is immaterial and may be defined arbitrarily; the
diagonal has four-dimensional Lebesgue measure zero.
\end{lemma}
The proof is given in Appendix~\ref{sec:auxana}.

\begin{proposition}
Let $1<p<\infty$ and $\varpi\in C_c^\infty(\R^3)$ be axisymmetric, and set
$u_\varpi:=\nabla\times(-\Delta)^{-1}(r\varpi e_\vartheta)$.  Then
\begin{align}
 \|u_\varpi\|_2^2
 &\lesssim_p\|\varpi\|_p^{A_p}\mathfrak M_2(\varpi)^{B_p},
 \label{eq:mtwo-mone}\\
 \|u_\varpi\|_2^2
 &\lesssim_p\|\varpi\|_p^{A_{1,p}}
 \mathfrak M_1(\varpi)^{B_{1,p}}.
 \label{eq:mone-en}
\end{align}
\end{proposition}
\begin{proof}
For the absolute-impulse estimate, set
$\sigma_p:=3(p-1)/(5p-3)$ and $q_p:=2(5p-3)/(7p-3)$.
Then $0<2\sigma_p<2$ and $q_p=4/(4-2\sigma_p)$.  Lemmas~\ref{lem:axiskern}
and \ref{lem:bsformula}, with
$g(r,z):=\mathbf 1_{\{r>0\}}r^{3/2+\sigma_p}|\varpi(r,z)|$, give
\[
 \|u_\varpi\|_2^2
 \lesssim_p
 \iint_{\R^2\times\R^2}
 \frac{g(X)g(Y)}{|X-Y|^{2\sigma_p}}\,dX\,dY.
\]
The choice of $q_p$ and $\sigma_p$ gives the factorization
\[
 |g|^{q_p}
 =
 (r|\varpi|^p)^{\frac{3}{7p-3}}
 (r^3|\varpi|)^{\frac{7p-6}{7p-3}}.
\]
Hence H\"older's inequality, \eqref{eq:lp-cyl}, and \eqref{eq:mtwo-cyl} imply
\[
 \|g\|_{q_p}^{q_p}
 \lesssim_p
 \|\varpi\|_p^{\frac{3p}{7p-3}}
 \mathfrak M_2(\varpi)^{\frac{7p-6}{7p-3}}.
\]
The two-dimensional Hardy--Littlewood--Sobolev inequality in Lemma \ref{lem:hls} therefore yields
\[
 \|u_\varpi\|_2^2
 \lesssim_p \|g\|_{q_p}^2
 \lesssim_p
 \|\varpi\|_p^{A_p}\mathfrak M_2(\varpi)^{B_p},
\]
which is \eqref{eq:mtwo-mone}.

For the first-radial-moment estimate, set
\[
 \sigma_{1,p}:=\frac{p-1}{4p-3},
 \qquad
 q_{1,p}:=\frac{2(4p-3)}{7p-5},
 \qquad
 g_1(r,z):=\mathbf 1_{\{r>0\}}
 r^{3/2+\sigma_{1,p}}|\varpi(r,z)|.
\]
Again $0<2\sigma_{1,p}<2$ and
$q_{1,p}=4/(4-2\sigma_{1,p})$, while
\[
 |g_1|^{q_{1,p}}
 =
 (r|\varpi|^p)^{\frac1{7p-5}}
 (r^2|\varpi|)^{\frac{7p-6}{7p-5}}.
\]
Thus
\[
 \|g_1\|_{q_{1,p}}^{q_{1,p}}
 \lesssim_p
 \|\varpi\|_p^{\frac{p}{7p-5}}
 \mathfrak M_1(\varpi)^{\frac{7p-6}{7p-5}}.
\]
Applying the same kernel estimate and Hardy--Littlewood--Sobolev inequality gives
\[
 \|u_\varpi\|_2^2
 \lesssim_p \|g_1\|_{q_{1,p}}^2
 \lesssim_p
 \|\varpi\|_p^{A_{1,p}}
 \mathfrak M_1(\varpi)^{B_{1,p}},
\]
which is \eqref{eq:mone-en}.
\end{proof}
\begin{lemma}\label{lem:raddense}
Let $1<p<\infty$.
If $\varpi\in\mathcal X_p$, there are axisymmetric
$\varpi_n\in C_c^\infty(\R^3)$ such that
\begin{equation}\label{eq:mtwo-axis}
 \|\varpi_n-\varpi\|_p\to0,
 \qquad
 \mathfrak M_2(\varpi_n-\varpi)\to0.
\end{equation}
If $\varpi\in\mathcal Y_p$, the sequence may be chosen so that
\begin{equation}\label{eq:mone-ycls}
 \|\varpi_n-\varpi\|_p\to0,
 \qquad
 \mathfrak M_1(\varpi_n-\varpi)\to0.
\end{equation}
In either case, if $\varpi\in L^1(\R^3)$, one may also impose
\begin{equation}\label{eq:lone-mone}
 \|\varpi_n-\varpi\|_1\to0.
\end{equation}
\end{lemma}
The proof is given in Appendix~\ref{sec:auxana}.

\begin{proof}[Proof of Proposition~\ref{prop:radmom}]
Assume first that $\varpi\in\mathcal X_p$.  Choose $\varpi_n$ by
\eqref{eq:mtwo-axis} and set
$u_n:=\nabla\times(-\Delta)^{-1}(r\varpi_ne_\vartheta)$.  By linearity and \eqref{eq:mtwo-mone},
\[
 \|u_n-u_m\|_2^2
 \lesssim_p
 \|\varpi_n-\varpi_m\|_p^{A_p}
 \mathfrak M_2(\varpi_n-\varpi_m)^{B_p}\to0.
\]
Thus $u_n\to u$ in $L^2$.  For every
$\Phi\in C_c^\infty(\R^3;\R^3)$,
\[
 \int u_n\cdot(\nabla\times\Phi)\,dx
 =\int r\varpi_ne_\vartheta\cdot\Phi\,dx.
\]
The right-hand side converges to the corresponding expression with
$\varpi$, because $r$ is bounded on $\operatorname{supp}\Phi$ and
$\varpi_n\to\varpi$ locally in $L^1$.  Passing to the limit gives
$\nabla\times u=r\varpi e_\vartheta$.  Similarly,
$\int u\cdot\nabla\phi=0$ for every $\phi\in C_c^\infty$, so
$\nabla\cdot u=0$.

If $v\in L^2$ satisfies the same div--curl system, then $w=u-v$ has zero
divergence and curl.  Hence
$k\cdot\widehat w(k)=k\times\widehat w(k)=0$ for almost every $k$, and
therefore $w=0$.  This proves uniqueness.  Since each $u_n$ is
axisymmetric without swirl by Lemma~\ref{lem:bsformula}, the $L^2$ limit
has the same symmetries.  Passing to the limit in \eqref{eq:mtwo-mone}
gives \eqref{eq:bs-xbound}; linearity and
$\mathfrak M_2(\varpi-\eta)\le
\mathfrak M_2(\varpi)+\mathfrak M_2(\eta)$ give
\eqref{eq:bs-xstab}.

If instead $\varpi\in\mathcal Y_p$, choose $\varpi_n$ by
\eqref{eq:mone-ycls}.  The estimate \eqref{eq:mone-en} makes the
corresponding Biot--Savart fields Cauchy in $L^2$, and the same div--curl
argument identifies their limit with the unique solution of
\eqref{eq:divcurl-sys}.  Passing to the limit in \eqref{eq:mone-en} gives
\eqref{eq:bs-ybound}, while linearity and
$\mathfrak M_1(\varpi-\eta)\le
\mathfrak M_1(\varpi)+\mathfrak M_1(\eta)$ give
\eqref{eq:bs-ystab}.
\end{proof}
\begin{corollary}[Uniform-in-time $L^2$ stability]\label{cor:unifstab}
Let $T>0$, $1<p<\infty$, and let $\varpi_n,\varpi\in
C([0,T];L^p(\R^3))$ be axisymmetric with
\begin{equation}\label{eq:axis-conv}
 \varpi_n\to\varpi
 \qquad\text{in }C([0,T];L^p(\R^3)).
\end{equation}
If
\begin{equation}\label{eq:mtwo-bd}
 \sup_n\sup_{t\in[0,T]}\mathfrak M_2(\varpi_n(t))\le M<\infty,
\end{equation}
then
\begin{equation}\label{eq:mom-bd}
 \sup_{t\in[0,T]}\mathfrak M_2(\varpi(t))\le M
\end{equation}
and
\begin{equation}\label{eq:bs-mtwo}
 \sup_{t\in[0,T]}
 \|\Bcal\varpi_n(t)-\Bcal\varpi(t)\|_2^2
 \lesssim_p M^{B_p}
 \left(\sup_{t\in[0,T]}\|\varpi_n(t)-\varpi(t)\|_p\right)^{A_p}.
\end{equation}
If instead
\begin{equation}\label{eq:fin-bd}
 \sup_n\sup_{t\in[0,T]}\mathfrak M_1(\varpi_n(t))\le M<\infty,
\end{equation}
then
\begin{equation}\label{eq:biot-bd}
 \sup_{t\in[0,T]}\mathfrak M_1(\varpi(t))\le M
\end{equation}
and
\begin{equation}\label{eq:bs-mone}
 \sup_{t\in[0,T]}
 \|\Bcal\varpi_n(t)-\Bcal\varpi(t)\|_2^2
 \lesssim_p M^{B_{1,p}}
 \left(\sup_{t\in[0,T]}\|\varpi_n(t)-\varpi(t)\|_p\right)^{A_{1,p}}.
\end{equation}
In either case,
\begin{equation}\label{eq:bs-conv}
 \Bcal\varpi_n\to\Bcal\varpi
 \qquad\text{in }C([0,T];L^2(\R^3)).
\end{equation}
\end{corollary}
\begin{proof}
Fix $t\in[0,T]$ and $R>0$.  Since $p>1$ and $B_R$ has finite measure,
\eqref{eq:axis-conv} implies convergence in $L^1(B_R)$.  Therefore
\[
 \frac1{2\pi}\int_{B_R}r^k|\varpi(t)|\,dx
 =\lim_{n\to\infty}
 \frac1{2\pi}\int_{B_R}r^k|\varpi_n(t)|\,dx,
 \qquad k=1,2.
\]
Under \eqref{eq:mtwo-bd}, monotone convergence as $R\to\infty$ gives
\eqref{eq:mom-bd}; under \eqref{eq:fin-bd} it gives
\eqref{eq:biot-bd}.  The stability estimates
\eqref{eq:bs-xstab} and \eqref{eq:bs-ystab} then yield
\eqref{eq:bs-mtwo} and \eqref{eq:bs-mone}, respectively.

The same estimates applied to $\varpi(t)-\varpi(s)$ show that
$\Bcal\varpi\in C([0,T];L^2)$; applied to
$\varpi_n(t)-\varpi_n(s)$ they give the same continuity for each
$\Bcal\varpi_n$.  Finally, \eqref{eq:bs-mtwo} or
\eqref{eq:bs-mone} gives \eqref{eq:bs-conv} uniformly in time.
\end{proof}

\begin{corollary}[Moment-confined global energy compactness]
\label{cor:moment-energy-comp}
Fix $T<\infty$ and let $1<p<\infty$.  Let
$(a_n,\varpi_n,f_n,\pi_n)$ be regular axisymmetric no-swirl approximate
Euler flows satisfying \eqref{eq:td-id}, with
\[
 a_n\in C([0,T];H^{m_n}(\mathbb R^3))\cap
 C^1([0,T];H^{m_n-1}(\mathbb R^3)),\qquad m_n>\frac52,
\]
and assume
\[
 \sup_n\|a_n\|_{L^\infty(0,T;L^2)}<\infty,
 \qquad
 \|f_n\|_{L^1(0,T;L^2)}\to0,
\]
\[
 \varpi_n\to\varpi
 \quad\text{in }C([0,T];L^1(\mathbb R^3)\cap L^p(\mathbb R^3)),
 \qquad
 a_n(0)\to u_0\quad\text{in }L^2(\mathbb R^3),
\]
and
\[
 \sup_n\sup_{0\le t\le T}\mathfrak M_2(\varpi_n(t))<\infty.
\]
Then $u:=\Bcal\varpi$ belongs to $C([0,T];L^2(\mathbb R^3))$,
\[
 a_n\to u\quad\text{in }C([0,T];L^2(\mathbb R^3)),
\]
and
\[
 \|u(t)\|_2=\|u_0\|_2,
 \qquad 0\le t\le T.
\]
\end{corollary}
\begin{proof}
The uniform absolute-impulse bound and
Corollary~\ref{cor:unifstab} give
\[
 \Bcal\varpi_n\to\Bcal\varpi
 \quad\text{in }C([0,T];L^2).
\]
For every $n$ and $t$, Proposition~\ref{prop:radmom} and the div--curl
relations in \eqref{eq:td-id} identify
$a_n(t)=\Bcal\varpi_n(t)$ in $L^2$.  Hence the displayed convergence is
exactly $a_n\to u:=\Bcal\varpi$ in $C_tL_x^2$.

The regular approximate Euler equation has the energy identity
\[
 \frac12\|a_n(t)\|_2^2-\frac12\|a_n(0)\|_2^2
 =\int_0^t\!\int_{\mathbb R^3} f_n\cdot a_n\,dx\,ds.
\]
Therefore
\[
 \sup_{0\le t\le T}
 \left|\|a_n(t)\|_2^2-\|a_n(0)\|_2^2\right|
 \le
 2\|a_n\|_{L_t^\infty L_x^2}
 \|f_n\|_{L_t^1L_x^2}\to0.
\]
Passing to the limit with the strong $C_tL_x^2$ convergence and the
initial $L^2$ convergence proves the energy identity for $u$.
\end{proof}

\subsection{Sign-free first-radial-moment control}
Fix a nonincreasing cutoff
$\chi_{\rm rad}\in C_c^\infty([0,\infty))$ with
$0\le\chi_{\rm rad}\le1$, $\chi_{\rm rad}=1$ on $[0,1]$, and
$\chi_{\rm rad}=0$ on $[2,\infty)$, and set
\begin{equation}\label{eq:rad-def}
\Psi_L(s):=\int_0^s\chi_{\rm rad}(\tau/L)\,d\tau,
\qquad L>0,\ s\ge0.
\end{equation}
Then
\begin{equation}\label{eq:psi-bd}
0\le\Psi_L'\le1,\qquad
0\le\Psi_L(s)\le\min\{s,2L\},\qquad
\Psi_L(s)\uparrow s,\qquad
\Psi_L'(s)\to1
\end{equation}
for every $s\ge0$ as $L\to\infty$.

\begin{theorem}[Sign-free first-radial-moment control]
\label{thm:firstmom}
Let $3/2\le p<\infty$, and let $(u,\varpi)$ be a measurable axisymmetric
no-swirl pair on $[0,T]\times\R^3$.  Assume
\begin{equation}\label{eq:axis-bd}
 u(t)\in L^2(\R^3)\quad\text{for a.e. }t\in(0,T),
 \qquad
 \varpi(t)\in L^1(\R^3)\cap L^p(\R^3)\quad\text{for }0\le t\le T,
\end{equation}
and, for almost every $t\in(0,T)$,
\begin{equation}\label{eq:diss-id}
 \nabla\cdot u(t)=0,
 \qquad
 \nabla\times u(t)=r\varpi(t)e_\vartheta
 \quad\text{in }\mathcal D'(\R^3).
\end{equation}
No boundedness or trace condition on $\varpi$ at $r=0$ is imposed; the
Cartesian identities in \eqref{eq:axisym-cart} are used below.  Set
\begin{equation}\label{eq:cond-def}
 N_1:=\sup_{0\le t\le T}\|\varpi(t)\|_1,
 \qquad
 N_p:=\sup_{0\le t\le T}\|\varpi(t)\|_p,
\end{equation}
and assume
\begin{equation}\label{eq:urad-bd}
 N_1+N_p<\infty,
 \qquad
 \mathfrak M_1(\varpi(0))<\infty,
 \qquad
 \int_0^T\int_{\R^3}|u^r|\,|\varpi|\,dx\,dt<\infty.
\end{equation}
Assume also that, for every $L>0$ and $0\le t\le T$,
\begin{align}
 &\frac1{2\pi}\int_{\R^3}\Psi_L(r)|\varpi(t)|\,dx
 -\frac1{2\pi}\int_{\R^3}\Psi_L(r)|\varpi(0)|\,dx
  =
 \frac1{2\pi}\int_0^t\int_{\R^3}
 \Psi_L'(r)u^r(\tau,x)|\varpi(\tau,x)|\,dx\,d\tau.
 \label{eq:urad-id}
\end{align}
Set
\begin{equation*}
 \Lambda_{p,T}:=
 \begin{cases}
  N_p^{3/2}, & p=3/2,\\[1mm]
  N_p^{\frac{p}{2(p-1)}}
  N_1^{\frac{2p-3}{2(p-1)}}, & p>3/2.
 \end{cases}
\end{equation*}
Then $t\mapsto\mathfrak M_1(\varpi(t))$ is absolutely continuous, and
\begin{equation}\label{eq:dt-mone}
 \left|\frac{d}{dt}\mathfrak M_1(\varpi(t))\right|
 \lesssim_p
 \Lambda_{p,T}\,\mathfrak M_1(\varpi(t))^{1/2}
 \qquad\text{for a.e. }t\in(0,T).
\end{equation}
Moreover,
\begin{equation}\label{eq:sign-bd}
 \mathfrak M_1(\varpi(t))^{1/2}
 \le \mathfrak M_1(\varpi(0))^{1/2}
 +C_p\Lambda_{p,T}t,
 \qquad 0\le t\le T.
\end{equation}
\end{theorem}
\begin{proof}
For $L>0$, set
$\mathfrak M_{1,L}(t):=(2\pi)^{-1}\int_{\R^3}\Psi_L(r(x))|\varpi(t,x)|\,dx$.
The balance hypothesis \eqref{eq:urad-id}, the bound $0\le\Psi_L'\le1$, and
\eqref{eq:urad-bd} give
\[
 \mathfrak M_{1,L}(t)
 =\mathfrak M_{1,L}(0)
 +\frac1{2\pi}\int_0^t\int_{\R^3}
 \Psi_L'(r)u^r|\varpi|\,dx\,d\tau,
\]
with an integrable right-hand side.  Letting $L\to\infty$, monotone convergence
on the moment terms and dominated convergence on the flux term yield
\[
 \mathfrak M_1(\varpi(t))-
 \mathfrak M_1(\varpi(0))
 =\frac1{2\pi}\int_0^t\int_{\R^3}
 u^r(\tau,x)|\varpi(\tau,x)|\,dx\,d\tau.
\]
Consequently $\mathfrak M_1(\varpi(\cdot))\in W^{1,1}(0,T)$ and
\[
 \frac d{dt}\mathfrak M_1(\varpi(t))
 =\frac1{2\pi}\int_{\R^3}u^r(t,x)|\varpi(t,x)|\,dx
\]
for almost every $t$.

Fix such a time $t$, write $f:=|\varpi(t)|$, and extend
$h(r,z):=\mathbf 1_{\{r>0\}}rf(r,z)$ by zero to $\R^2$.  If $p>3/2$, then
\[
 \|h\|_{4/3}^{4/3}
 =\int_{\Hh}
 (rf^p)^{\frac1{3(p-1)}}
 (r^2f)^{1/3}
 (rf)^{\frac{2p-3}{3(p-1)}}\,dr\,dz,
\]
so H\"older's inequality gives
\[
 \|h\|_{4/3}^2
 \lesssim_p
 \|\varpi(t)\|_p^{\frac{p}{2(p-1)}}
 \|\varpi(t)\|_1^{\frac{2p-3}{2(p-1)}}
 \mathfrak M_1(\varpi(t))^{1/2}.
\]
At $p=3/2$ the endpoint factorization
\[
 \|h\|_{4/3}^{4/3}
 =\int_{\Hh}(rf^{3/2})^{2/3}(r^2f)^{1/3}\,dr\,dz
\]
gives
\[
 \|h\|_{4/3}^2
 \lesssim
 \|\varpi(t)\|_{3/2}^{3/2}
 \mathfrak M_1(\varpi(t))^{1/2}.
\]
Thus in both cases
\[
 \|h\|_{4/3}^2
 \lesssim_p
 \Lambda_{p,T}\mathfrak M_1(\varpi(t))^{1/2}.
\]

It remains to justify the radial Biot--Savart bound for the present rough pair.
Choose $\varpi_n\in C_c^\infty$ by Lemma~\ref{lem:raddense}, with convergence
in $L^1\cap L^p$ and in $\mathfrak M_1$, and set
$u_n:=\Bcal\varpi_n$ and $h_n:=r|\varpi_n|$.  Applying this interpolation to $|\varpi_n-\varpi(t)|$ gives $h_n\to h$ in $L^{4/3}(\R^2)$,
and Lemma~\ref{lem:hls} then gives $I_1h_n\to I_1h$ in $L^4(\R^2)$.
Proposition~\ref{prop:radmom} also gives
$u_n\to\Bcal\varpi(t)$ in $L^2(\R^3)$.  The div--curl assumptions and
uniqueness in \eqref{eq:divcurl-sys} identify $u(t)=\Bcal\varpi(t)$.
After passing to a common subsequence, both convergences hold almost everywhere, and the pointwise axisymmetric kernel bound
\cite[Proposition~2.3, estimate~(2.11)]{gallay2015remarks} therefore yields
\[
 |u^r(t,r,z)|\lesssim I_1h(r,z)
 \qquad\text{for a.e. }(r,z)\in\Hh.
\]
Lemma~\ref{lem:hls} now gives
\[
 \frac1{2\pi}\int_{\R^3}|u^r(t,x)|\,|\varpi(t,x)|\,dx
 \le C\int_{\R^2}hI_1h
 \lesssim \|h\|_{4/3}^2
 \lesssim_p
 \Lambda_{p,T}\mathfrak M_1(\varpi(t))^{1/2}.
\]
Combining this estimate with the derivative identity proves \eqref{eq:dt-mone}.
Finally, for $\eta>0$,
\[
 \left|\frac d{dt}
 (\mathfrak M_1(\varpi(t))+\eta)^{1/2}\right|
 \le C_p\Lambda_{p,T}
\]
for almost every $t$.  Integrating and then letting $\eta\downarrow0$ proves
\eqref{eq:sign-bd}.
\end{proof}
\subsection{Sharpness at $p=3/2$}
Theorem~\ref{thm:firstmom} propagates the first radial moment without a sign
condition for $p\ge3/2$.  Sharpness of this class-uniform instantaneous
estimate is produced by a thin anti-parallel pair whose meridional kernel
converges at the core scale to the planar Biot--Savart kernel.
Fix a nonzero radial $\phi\in C_c^\infty(B_{1/8}(0))$ with $\phi\ge0$ and $\int_{\mathbb R^2}\phi=1/2$, and, with $e_2=(0,1)$, set $W(Y):=\phi(Y+e_2)-\phi(Y-e_2)$.
Then
\begin{align}
 W(y_1,-y_2) =-W(y_1,y_2), \   \ 
 W\le0\quad\text{on }\{y_2>0\}, \notag
 \\
 \int_{\mathbb R^2}|W|\,dY =1,
  \   \ \int_{\mathbb R^2}y_1|W(Y)|\,dY=0.
 \label{eq:dy-id}
\end{align}
\begin{theorem}[Sharpness at $3/2$]\label{thm:thrsharp}
Let $1<p<3/2$ and define
\begin{align}
 \alpha_p:=\frac{2p-1}{2(p-1)}, \   \ 
 \gamma_p:=\alpha_p-2=\frac{3-2p}{2(p-1)}>0, 
 \ \
 \Gamma_R:=R^{-1}, \   \ 
 \ell_R:=R^{-\alpha_p}.
 \label{eq:large-def}
\end{align}
For all sufficiently large $R$, set
\begin{equation}\label{eq:sharp-def}
 \varpi_{0,R}(r,z)
 :=\frac{\Gamma_R}{r\ell_R^2}
 W\!\left(\frac{r-R}{\ell_R},\frac z{\ell_R}\right).
\end{equation}
Then $\varpi_{0,R}\in C_c^\infty(\mathbb R^3)$ is axisymmetric,
\begin{equation*}
 \varpi_{0,R}(r,-z)=-\varpi_{0,R}(r,z),
 \qquad
 \varpi_{0,R}(r,z)\le0\quad(z>0),
\end{equation*}
and
\begin{equation}\label{eq:sharp-bd}
 \sup_{R\gg1}\left(
 \|\varpi_{0,R}\|_1+\|\varpi_{0,R}\|_p
 +\mathfrak M_1(\varpi_{0,R})+\|\Bcal\varpi_{0,R}\|_2
 \right)<\infty.
\end{equation}
If $\varpi_R(t)$ denotes the corresponding global smooth axisymmetric no-swirl
Euler solution, then
\begin{equation}\label{eq:dt-conv}
 \left.\frac d{dt}\mathfrak M_1(\varpi_R(t))\right|_{t=0}
 =c_W R^{\gamma_p}(1+o(1)),
 \qquad R\to\infty.
\end{equation}
There are also times $t_R\downarrow0$ such that
\begin{equation}\label{eq:sharp-conv}
 \frac{\mathfrak M_1(\varpi_R(t_R))-\mathfrak M_1(\varpi_R(0))}{t_R}
 \to+\infty.
\end{equation}
Here
\begin{equation}\label{eq:dy-def}
 c_W:=\frac1{2\pi}
 \iint_{\mathbb R^2\times\mathbb R^2}
 |W(X)|W(Y)\frac{x_2-y_2}{|X-Y|^2}\,dY\,dX>0.
\end{equation}
Consequently, there exists $K_p<\infty$ such that
\begin{equation*}
 \sup\left\{
 \left.\frac d{dt}\mathfrak M_1(\varpi(t))\right|_{t=0}:
 \begin{array}{l}
  \varpi(0)\in C_c^\infty(\mathbb R^3)\ \text{axisymmetric},\\[-1mm]
  \|\varpi(0)\|_1+\|\varpi(0)\|_p+\mathfrak M_1(\varpi(0))
  +\|\Bcal\varpi(0)\|_2\le K_p
 \end{array}
 \right\}=\infty.
\end{equation*}
Thus $p=3/2$ is sharp for class-uniform sign-free instantaneous control of
$\mathfrak M_1$ by the quantities in
\eqref{eq:sharp-bd}.
\end{theorem}
\begin{proof}
Cylindrical coordinates and
\eqref{eq:dy-id} give
\begin{align}
 \|\varpi_{0,R}\|_1
 &=2\pi\Gamma_R=2\pi R^{-1},
 \\
 \mathfrak M_1(\varpi_{0,R})
 &=\Gamma_R\int_{\mathbb R^2}(R+\ell_Ry_1)|W(Y)|\,dY
 =\Gamma_RR=1,
 \label{eq:ry-id}\\
 \|\varpi_{0,R}\|_p^p
 &=2\pi\Gamma_R^p\ell_R^{2-2p}R^{1-p}
 \int_{\mathbb R^2}(1+\varepsilon_Ry_1)^{1-p}|W(Y)|^p\,dY,
 \qquad
 \varepsilon_R:=\frac{\ell_R}{R}.
 \label{eq:ry-def}
\end{align}
The definition of $\alpha_p$ makes the power of $R$ in \eqref{eq:ry-def} equal to zero.  Since $\varepsilon_R\to0$, \eqref{eq:ry-def} therefore implies
\begin{equation*}
 0<\inf_{R\gg1}\|\varpi_{0,R}\|_p
 \le\sup_{R\gg1}\|\varpi_{0,R}\|_p<\infty.
\end{equation*}
The first-radial-moment estimate \eqref{eq:bs-ybound}, together with
\eqref{eq:ry-id} and this uniform $L^p$ bound, yields
\[
 \sup_{R\gg1}\|\Bcal\varpi_{0,R}\|_2<\infty.
\]
Thus \eqref{eq:sharp-bd} holds.
It remains to compute the initial flux.  Write
$\rho_R(X):=1+\varepsilon_Rx_1$ and
$s_R(X,Y):=\varepsilon_R^2|X-Y|^2/[\rho_R(X)\rho_R(Y)]$.
For the function $\mathcal F$ in \eqref{eq:kern-def},
\cite[Lemma~2.1]{gallay2015remarks} gives
\begin{equation*}
 \widetilde{\mathcal F}(s):=-2s\mathcal F'(s),
 \qquad
 \sup_{s>0}|\widetilde{\mathcal F}(s)|<\infty,
 \qquad
 \widetilde{\mathcal F}(s)\to1\quad(s\downarrow0).
\end{equation*}
Let \(G^r(r,z;\rho,\zeta)\) denote the meridional radial-velocity kernel \cite[Eq.~(2.7)]{gallay2015remarks}
\[G^r(r,z;\rho,\zeta) = -\frac{1}{\pi}\frac{z-\zeta}{r^{3/2}\rho^{1/2}}\,
 \mathcal F'\!\left(
 \frac{(r-\rho)^2+(z-\zeta)^2}{r\rho}
 \right),\]
where $\mathcal{F}$ is defined in \eqref{eq:kern-def}; see also Lemma~\ref{lem:crossform}.  Set $G_R^r(X,Y):=G^r(R+\ell_Rx_1,\ell_Rx_2;\,R+\ell_Ry_1,\ell_Ry_2)$ and make the change of variables
\begin{equation}\label{eq:ell-rx}
 r=R+\ell_Rx_1,
 \quad z=\ell_Rx_2,
 \quad \rho=R+\ell_Ry_1,
 \quad \zeta=\ell_Ry_2,
\end{equation}
the radial Biot--Savart kernel satisfies
\begin{equation*}
 \ell_R G_R^r(X,Y)
 =\frac1{2\pi}
 \left(\frac{\rho_R(Y)}{\rho_R(X)}\right)^{1/2}
 \frac{x_2-y_2}{|X-Y|^2}
 \widetilde{\mathcal F}\!\left(s_R(X,Y)\right).
\end{equation*}
On the fixed support of $W$,
\begin{equation*}
 \ell_RG_R^r(X,Y)
 \to\frac1{2\pi}\frac{x_2-y_2}{|X-Y|^2}
 \quad(X\ne Y),
 \qquad
 |\ell_RG_R^r(X,Y)|\lesssim|X-Y|^{-1}.
\end{equation*}
The majorant is locally integrable in the two-dimensional difference
variable.  Hence dominated convergence gives
\begin{equation}\label{eq:ell-rg}
 \iint |W(X)|\ell_RG_R^r(X,Y)W(Y)\,dY\,dX
 \to c_W.
\end{equation}
To determine the sign, set
\begin{equation}
 W_-(X):=\phi(X+e_2),
 \qquad
 W_+(X):=\phi(X-e_2).
\end{equation}
Then $W=W_--W_+$ and $|W|=W_-+W_+$.  The two self-interactions in
\eqref{eq:dy-def} vanish after exchanging $X$ and $Y$, while
the cross terms give
\begin{equation}
 c_W=\frac1\pi\iint W_+(X)W_-(Y)
 \frac{x_2-y_2}{|X-Y|^2}\,dY\,dX>0,
\end{equation}
because $x_2-y_2>0$ on
$\operatorname{supp}W_+\times\operatorname{supp}W_-$.
For the classical Euler solution, transport of $|\varpi_R|$ and
incompressibility give
\begin{equation*}
 \left.\frac d{dt}\mathfrak M_1(\varpi_R(t))\right|_{t=0}
 =\int_{\Hh}u_{0,R}^r(r,z)\,r|\varpi_{0,R}(r,z)|\,dr\,dz.
\end{equation*}
Using \eqref{eq:sharp-def}, \eqref{eq:ell-rx}, and \eqref{eq:ell-rg} in this flux identity, we obtain
\begin{align}
 \left.\frac d{dt}\mathfrak M_1(\varpi_R(t))\right|_{t=0}
 &=\frac{\Gamma_R^2}{\ell_R}(c_W+o(1))
 =c_WR^{\gamma_p}(1+o(1)), \notag
\end{align}
where $\Gamma_R^2/\ell_R=R^{\alpha_p-2}=R^{(3-2p)/(2(p-1))}$.  This proves \eqref{eq:dt-conv}.  Since
$t\mapsto\mathfrak M_1(\varpi_R(t))$ is differentiable at $t=0$, for each
sufficiently large $R$ one may choose $0<t_R\le R^{-1}$ so that
\[
 \frac{\mathfrak M_1(\varpi_R(t_R))-\mathfrak M_1(\varpi_R(0))}{t_R}
 \ge \frac12
 \left.\frac d{dt}\mathfrak M_1(\varpi_R(t))\right|_{t=0}.
\]
Together with \eqref{eq:dt-conv}, this proves
\eqref{eq:sharp-conv}.  The data are smooth and supported
away from the axis for each $R$, and \cite{UkhovskiiIudovich1968} gives the
corresponding global smooth solution; see also the exact modern formulation
of this classical class in \cite[Section~1.2]{LimJeong2025}.
\end{proof}
\begin{remark}[Sharpness of the first-radial-moment route]
For $h_R=r|\varpi_{0,R}|$,
\begin{equation*}
 \|h_R\|_{4/3}^2
 =\left(\int_{\mathbb R^2}|W|^{4/3}\right)^{3/2}
 \frac{\Gamma_R^2}{\ell_R}
 \asymp R^{\gamma_p},
 \qquad
 \gamma_p>0\iff p<\frac32,
 \qquad
 \gamma_{3/2}=0.
\end{equation*}
Thus the construction saturates the HLS flux scale in
Theorem~\ref{thm:firstmom}, and $p=3/2$ is the sharp endpoint for
class-uniform sign-free differential control of the first radial moment.
The obstruction is genuinely dynamical: Theorem~\ref{thm:thrsharp} rules out such
control below $3/2$ under uniformly bounded $L^1$, $L^p$, first-radial-moment,
and kinetic-energy data.

The exponent $3/2$ therefore marks the sharp endpoint of the sign-free
first-radial-moment differential estimate.  Theorem~\ref{thm:mfreecomp} provides a
separate moment-free route to global energy compactness for $p\ge4/3$.
\end{remark}
\section{Opposite-sign interaction and the $5/3$ threshold}\label{sec:oppsign}
For the absolute impulse, the relevant mixed exponent is
$\displaystyle p^\sharp=4p/(3p-1)$ for every $p>1$, and
$p^\sharp=p$ holds exactly at $p=5/3$.  The results in this section identify
the corresponding sharp class-uniform mixed-integrability criterion and its
fixed-time obstruction.  As in Section~\ref{sec:radmom}, the dynamical statements
use only the stated div--curl relations, integrability conditions, and
bounded-weight balance identities.
We use throughout the radial-moment Biot--Savart and kernel estimates collected in Section~\ref{sec:radbiot}.

\subsection{Mixed opposite-sign criterion}
For axisymmetric $\varpi,\eta\in C_c^\infty(\R^3)$, define
\begin{equation}\label{eq:axis-def}
 \mathfrak X(\varpi,\eta)
 :=\int_{\R^3}(\Bcal\eta)^r\,r\varpi\,dx.
\end{equation}
For the mixed-integrability classes below, $\mathfrak X$ denotes the
absolutely convergent kernel extension constructed in Proposition~\ref{prop:crossint};
it agrees with \eqref{eq:axis-def} whenever the latter pairing is defined.  We write
$h_+=\max\{h,0\}$ and $h_-=\max\{-h,0\}$.
For axisymmetric $\varpi,\eta\in C_c^\infty(\R^3)$, the cross interaction
introduced in \eqref{eq:axis-def} has the following exact kernel
representation.
\begin{lemma}\label{lem:crossform}
For axisymmetric $\varpi,\eta\in C_c^\infty(\R^3)$, let
$u_\eta=\Bcal\eta$.  Then
\begin{equation}\label{eq:kern-id}
 u_\eta^r(r,z)
 =-\frac1\pi\int_{\Hh}
 \frac{z-\zeta}{r^{3/2}\rho^{1/2}}\,
 \mathcal F'\!\left(
 \frac{(r-\rho)^2+(z-\zeta)^2}{r\rho}
 \right)\rho\eta(\rho,\zeta)\,d\rho\,d\zeta,
\end{equation}
and
\begin{align}
 \mathfrak X(\varpi,\eta)
 &=-2\iint_{\Hh\times\Hh}
 (r\rho)^{1/2}(z-\zeta)\,
 \mathcal F'\!\left(
 \frac{(r-\rho)^2+(z-\zeta)^2}{r\rho}
 \right)
 \varpi(r,z)\eta(\rho,\zeta)
 \,dr\,dz\,d\rho\,d\zeta.
 \label{eq:dz-id}
\end{align}
Consequently,
\begin{equation}\label{eq:mix-id}
 \mathfrak X(\varpi,\eta)=-\mathfrak X(\eta,\varpi).
\end{equation}
\end{lemma}
\begin{proof}
Let $u_\eta$ be the field in \eqref{eq:bs-def} with $\varpi=\eta$.
Lemma~\ref{lem:bsformula} and uniqueness in
Proposition~\ref{prop:radmom} identify $u_\eta=\Bcal\eta$.  For
$x=(r,0,z)$ and
$y=(\rho\cos\varphi,\rho\sin\varphi,\zeta)$, write
\[
 D^2=(r-\rho)^2+(z-\zeta)^2+2r\rho(1-\cos\varphi),
 \qquad
 s=\frac{(r-\rho)^2+(z-\zeta)^2}{r\rho}.
\]
The three-dimensional Biot--Savart formula gives
\[
 u_\eta^r(r,z)
 =\frac1{4\pi}\int_{\Hh}\int_0^{2\pi}
 \frac{(z-\zeta)\cos\varphi}{D^3}\,
 \rho^2\eta(\rho,\zeta)\,d\varphi\,d\rho\,d\zeta.
\]
Differentiating \eqref{eq:kern-def} yields
\[
 \int_0^{2\pi}\frac{\cos\varphi}{D^3}\,d\varphi
 =-\frac4{(r\rho)^{3/2}}\mathcal F'(s),
\]
which proves \eqref{eq:kern-id}.  Cylindrical integration in
\eqref{eq:axis-def} then gives \eqref{eq:dz-id}.

By \eqref{eq:kern-der} with $\sigma'=0$,
\[
 \left|(r\rho)^{1/2}(z-\zeta)\mathcal F'(s)\right|
 \lesssim
 \frac{r^{3/2}\rho^{3/2}}{|(r,z)-(\rho,\zeta)|}.
\]
For compactly supported smooth data this is integrable on
$\Hh\times\Hh$, since $|Z|^{-1}$ is locally integrable in two
variables.  Thus \eqref{eq:dz-id} is absolutely convergent, and
interchanging $(r,z)$ with $(\rho,\zeta)$ gives \eqref{eq:mix-id}.
\end{proof}

\begin{lemma}[Rough kernel pairing]\label{lem:kernel-identity}
Let $\varpi\in\mathcal X_p\cap L^1(\R^3)$ and
$\eta\in\mathcal X_q\cap L^1(\R^3)$.  Assume that the kernel integral defining
$\mathfrak{X}(\varpi,\eta)$ in \eqref{eq:dz-id} is absolutely convergent.  Then
\[
 (\Bcal \eta)^r r\varpi,\ (\Bcal \varpi)^r r\eta\in L^1(\R^3),
\]
and
\begin{equation}\label{eq:bs-id}
 \int_{\R^3}(\Bcal \eta)^r r\varpi\,dx
 =\mathfrak{X}(\varpi,\eta)=-\mathfrak{X}(\eta,\varpi)
 =-\int_{\R^3}(\Bcal \varpi)^r r\eta\,dx .
\end{equation}

\end{lemma}

\begin{proof}
Since $h\in L^1$ and $\Mtwo(h)<\infty$ imply
$\mathfrak M_1(h)<\infty$, the lifted vorticity $rh e_\vartheta$ belongs to
$L^1(\R^3)$.  Choose axisymmetric $h_n\in C_c^\infty$ with
\[
 h_n\to h\quad\text{in }L^1\cap L^s,
 \qquad \Mtwo(h_n-h)\to0.
\]
Then, by $(\mathfrak M_1(k))^2\lesssim\|k\|_1\Mtwo(k)$,
\[
 rh_ne_\vartheta\to rhe_\vartheta\quad\text{in }L^1,
 \qquad \Bcal h_n\to\Bcal h\quad\text{in }L^2.
\]
The weak $L^1\to L^{3/2,\infty}$ bound for the three-dimensional
Biot--Savart kernel identifies $\Bcal h$ a.e. with the usual Biot--Savart
integral.  Hence, for a.e. $(r,z)$,
\[
 (\Bcal h)^r(r,z)
 =-\frac1\pi\int_{\Hh}
 \frac{z-\zeta}{r^{3/2}\rho^{1/2}}
 \mathcal F'\!\left(\frac{(r-\rho)^2+(z-\zeta)^2}{r\rho}\right)
 \rho h(\rho,\zeta)\,d\rho\,d\zeta .
\]
Absolute convergence of the double kernel integral and Tonelli's theorem now give
\eqref{eq:bs-id}.  Skew-symmetry follows by interchanging the two meridional variables.
\end{proof}

\begin{proposition}[Symmetric cross interaction]
Let $5/3\le p<\infty$.  Define
\begin{equation}\label{eq:inter-id}
 \alpha_p=\frac{2p}{5p-3},
 \qquad
 \beta_p=\frac{3(p-1)}{5p-3}.
\end{equation}
For $\varpi,\eta\in\mathcal X_p$, let
$\mathfrak X(\varpi,\eta)$ denote the right-hand side of
\eqref{eq:dz-id}.  This integral is absolutely convergent,
agrees with \eqref{eq:axis-def} for axisymmetric
$C_c^\infty$ data, satisfies \eqref{eq:mix-id}, and obeys
\begin{equation}\label{eq:eta-alp}
 |\mathfrak X(\varpi,\eta)|
 \lesssim_p
 \|\varpi\|_p^{\alpha_p}\mathfrak M_2(\varpi)^{\beta_p}
 \|\eta\|_p^{\alpha_p}\mathfrak M_2(\eta)^{\beta_p}.
\end{equation}
\end{proposition}

\begin{proof}
Set
\[
 \sigma_p^\times:=\frac{3p-5}{2(5p-3)},
 \qquad
 \lambda:=1+2\sigma_p^\times=\frac{8(p-1)}{5p-3},
 \qquad
 q:=\frac{5p-3}{3p-1}.
\]
For $p\ge5/3$, one has $0\le\sigma_p^\times<3/10$, $0<\lambda<2$, and
$q=4/(4-\lambda)$.  By \eqref{eq:kern-der} with $\sigma'=\sigma_p^\times$ and
$|z-\zeta|\le |(r,z)-(\rho,\zeta)|$,
\begin{equation*}
 2\left|(r\rho)^{1/2}(z-\zeta)
 \mathcal F'\!\left(
 \frac{(r-\rho)^2+(z-\zeta)^2}{r\rho}
 \right)\right|
 \lesssim_p
 \frac{r^{3/2+\sigma_p^\times}
       \rho^{3/2+\sigma_p^\times}}
 {|(r,z)-(\rho,\zeta)|^\lambda}.
\end{equation*}

For $h\in\mathcal X_p$, set
$g_h(r,z):=\mathbf 1_{\{r>0\}}r^{3/2+\sigma_p^\times}|h(r,z)|$.
The choice of $q$ and $\sigma_p^\times$ gives
\[
 |g_h|^q
 =
 (r|h|^p)^{\frac{2}{3p-1}}
 (r^3|h|)^{\frac{3(p-1)}{3p-1}}.
\]
Hence H\"older's inequality, \eqref{eq:lp-cyl}, and
\eqref{eq:mtwo-cyl} yield
\[
 \|g_h\|_{L^q(\R^2)}
 \lesssim_p
 \|h\|_p^{\frac{2p}{5p-3}}
 \mathfrak M_2(h)^{\frac{3(p-1)}{5p-3}}
 =
 \|h\|_p^{\alpha_p}\mathfrak M_2(h)^{\beta_p}.
\]
Since $0<\lambda<2$ and $q=4/(4-\lambda)$, Lemma~\ref{lem:hls} and
the kernel bound above give
\begin{align*}
 &2\iint_{\Hh\times\Hh}
 \left|(r\rho)^{1/2}(z-\zeta)
 \mathcal F'\!\left(
 \frac{(r-\rho)^2+(z-\zeta)^2}{r\rho}
 \right)
 \varpi(r,z)\eta(\rho,\zeta)\right|
 \,dr\,dz\,d\rho\,d\zeta
 \\
 &\qquad\lesssim_p
 \|g_\varpi\|_{L^q(\R^2)}
 \|g_\eta\|_{L^q(\R^2)}
\lesssim_p
 \|\varpi\|_p^{\alpha_p}\mathfrak M_2(\varpi)^{\beta_p}
 \|\eta\|_p^{\alpha_p}\mathfrak M_2(\eta)^{\beta_p}.
\end{align*}
This proves absolute convergence and \eqref{eq:eta-alp}.
Lemma~\ref{lem:crossform} identifies the resulting bilinear form with
\eqref{eq:axis-def} for axisymmetric $C_c^\infty$ data.  Absolute
convergence also permits interchange of the two meridional variables,
which gives \eqref{eq:mix-id}.
\end{proof}
 
This estimate is symmetric in the two factors.  The next estimate allows different integrability on the two signs and yields the sharp mixed threshold.
\begin{proposition}[Asymmetric cross interaction]\label{prop:crossint}
Let $1<p<\infty$, let $p^\sharp$ be defined by
\eqref{eq:psharp}, and set
\begin{equation}\label{eq:inter-def}
 \tau_p:=\frac{4p}{3p+1}.
\end{equation}
Let $\varpi\in\mathcal X_p\cap L^1(\R^3)$ and
$\eta\in\mathcal X_{p^\sharp}\cap L^1(\R^3)$.  Then the right-hand
side of \eqref{eq:dz-id} is absolutely convergent, defines
$\mathfrak X(\varpi,\eta)$, and satisfies
\begin{equation}\label{eq:eta-mtwo}
 |\mathfrak X(\varpi,\eta)|
 \lesssim_p  
 \|\varpi\|_1\|\eta\|_1
 +\|\varpi\|_p^{1/4}\mathfrak M_2(\varpi)^{3/4}
  \|\eta\|_{p^\sharp} .
\end{equation}
\end{proposition}
\begin{proof}
Write $X=(r,z)$, $Y=(\rho,\zeta)$, $d=|X-Y|$, and split
$\Hh\times\Hh$ into $\mathcal U_{\rm near}:=\{d^2<r\rho\}$ and
$\mathcal U_{\rm far}:=\{d^2\ge r\rho\}$.  On $\mathcal U_{\rm far}$, \eqref{eq:kern-der} with $\sigma'=1/2$ and
$|z-\zeta|\le d$ gives
\[
 2\left|(r\rho)^{1/2}(z-\zeta)
 \mathcal F'\!\left(\frac{d^2}{r\rho}\right)\right|
 \lesssim r\rho.
\]
Hence the far contribution is bounded by
$C\|\varpi\|_1\|\eta\|_1$.

On $\mathcal U_{\rm near}$, $(r-\rho)^2<r\rho$, so $r/\rho$ stays between
the two roots of $t^2-3t+1=0$; in particular $r$ and $\rho$ are comparable.
Using \eqref{eq:kern-der} with $\sigma'=0$ and the definition of
$p^\sharp$, we may distribute the radial weight between the two comparable
variables to obtain
\[
 2\left|(r\rho)^{1/2}(z-\zeta)
 \mathcal F'\!\left(\frac{d^2}{r\rho}\right)\right|
 \lesssim_p
 \frac{r^{(9p+1)/(4p)}\rho^{1/p^\sharp}}d.
\]
Set
\[
 G_\varpi(r,z):=\mathbf1_{\{r>0\}}r^{(9p+1)/(4p)}|\varpi(r,z)|,
 \qquad
 H_\eta(\rho,\zeta):=\mathbf1_{\{\rho>0\}}
 \rho^{1/p^\sharp}|\eta(\rho,\zeta)|.
\]
With $\tau_p=4p/(3p+1)$,
\[
 |G_\varpi|^{\tau_p}
 =(r|\varpi|^p)^{\frac1{3p+1}}
  (r^3|\varpi|)^{\frac{3p}{3p+1}},
\]
and therefore
\[
 \|G_\varpi\|_{\tau_p}
 \lesssim_p
 \|\varpi\|_p^{1/4}\mathfrak M_2(\varpi)^{3/4},
 \qquad
 \|H_\eta\|_{p^\sharp}
 =(2\pi)^{-1/p^\sharp}\|\eta\|_{p^\sharp}.
\]
Since $1/\tau_p+1/p^\sharp=3/2$, the two-dimensional
Hardy--Littlewood--Sobolev inequality in Lemma \ref{lem:hls} with kernel $|X-Y|^{-1}$ bounds the near
contribution by
\[
 C_p\|\varpi\|_p^{1/4}
 \mathfrak M_2(\varpi)^{3/4}\|\eta\|_{p^\sharp}.
\]
Together with the far estimate this proves absolute convergence and
\eqref{eq:eta-mtwo}.  Absolute convergence also permits interchange of the
meridional variables, hence
$\mathfrak X(\varpi,\eta)=-\mathfrak X(\eta,\varpi)$ whenever both kernel
forms are defined.
\end{proof}

\begin{lemma}[Diagonal cancellation]\label{lem:diagonal-cancel}
Let $\varpi\in\mathcal X_p\cap L^1(\R^3)$, $p>1$, and set $u=\Bcal \varpi$.
Let $\chi_R$ be a standard radial cutoff function defined by
\begin{align*}
    &  \chi_R(x) := \chi(|x|/R), \ \ \chi\in C_c^\infty([0,\infty))  \ \text{ s.t } \ \chi(\rho) = \begin{cases}
        1, \ 0\leq \rho\leq 1\\
        0, \ 2 \leq \rho
    \end{cases},
\end{align*}
then
\begin{align}
 \lim_{R\to\infty}\int_{\R^3}\chi_R u^r r\varpi\,dx=0
 \label{eq:bs-dx}
\end{align}
provided either
\begin{align}
 p\ge\frac43,
 \qquad\text{or}\qquad
 1<p<\frac43\ \text{ and }\ r\varpi\in L^2_{\rm loc}(\R^3). \label{eq:diagonal-cancel-cases}
\end{align}
\end{lemma}

\begin{proof}
Assume first $4/3\le p<2$.  The weighted axisymmetric
Calder\'on--Zygmund estimate \eqref{As} gives
\begin{align}
 \left\|\frac{\nabla_x u}{r}\right\|_{L^p(\R^3)}
 \lesssim_p\|\varpi\|_{L^p(\R^3)}.
 \label{eq:diagonal-cancel-weighted-cz}
\end{align}
Let $Q\subset\overline{\Hh}$ be a bounded, connected Lipschitz domain.  Since $r\le C_Q$ on $Q$,
\eqref{eq:diagonal-cancel-weighted-cz} implies $\nabla_{r,z}u\in L^p(Q)$.  Using $u\in L^2(\R^3)$ on a
subrectangle of $Q$ away from the axis and an anchored Poincar\'e inequality, we obtain $u\in W^{1,p}(Q)$.  Writing $p'=\frac{p}{p-1}$ for the H\"older conjugate, we have
\begin{align}
 W^{1,p}(Q)\hookrightarrow L^{p'}(Q),
 \qquad p'\le\frac{2p}{2-p}\iff p\ge\frac43. \label{eq:diagonal-cancel-embedding}
\end{align}
Thus $u\nabla u\in L^1_{\rm loc}(dr\,dz)$.  Using
\[
 \partial_zu^r-\partial_ru^z=r\varpi,
 \qquad \frac1r\partial_r(ru^r)+\partial_zu^z=0,
\]
we obtain distributionally on $\{r>0\}$
\begin{align}
 r u^r(r\varpi)
 =\partial_z\!\left[\frac r2\bigl((u^r)^2-(u^z)^2\bigr)\right]
 -\partial_r(ru^ru^z).  \label{eq:diagonal-cancel-divergence}
\end{align}
Denote an axis cutoff function $\theta_\e(r) := 1-\chi(r/\e) $. Fix $R$ and insert $\theta_\varepsilon(r)$ together with
$\chi_R$.  The axis error is bounded by
\[
 C\int_{\substack{0<r<2\varepsilon\\ (r,z)\in\operatorname{supp}\chi_R}}
 |u|^2\,dr\,dz\longrightarrow0
 \qquad(\varepsilon\downarrow0),
\]
by the local Sobolev control above.  Letting first $\varepsilon\downarrow0$
in \eqref{eq:diagonal-cancel-divergence} and then $R\to\infty$ gives
\begin{align}
 \left|\int_{\R^3}\chi_R u^r r\varpi\,dx\right|
 \lesssim \frac1R\int_{R\lesssim|x|\lesssim2R}|u|^2\,dx\longrightarrow0.
 \label{eq:diagonal-cancel-limit}
\end{align}

If $p\ge2$, then $\varpi\in L^1\cap L^p$ implies $\varpi\in L^2_{\rm loc}$, hence
$r\varpi\in L^2_{\rm loc}$.  Thus it remains to treat the second alternative in
\eqref{eq:diagonal-cancel-cases}.  Since $u\in L^2$, $\nabla\cdot u=0$, and
$\nabla\times u=r\varpi e_\vartheta\in L^2_{\rm loc}$, the local div--curl
estimate yields $u\in H^1_{\rm loc}$.  Therefore
\begin{align}
 u^r r\varpi =(u\times\nabla\times u)_3
 =\partial_3\frac{|u|^2}{2}-\nabla\cdot(u_3u)
 \label{eq:diagonal-cancel-divcurl}
\end{align}
holds distributionally.  Testing \eqref{eq:diagonal-cancel-divcurl} with $\chi_R$ gives \eqref{eq:diagonal-cancel-limit}, hence \eqref{eq:bs-dx}.
\end{proof}

\begin{corollary}[Mixed sign cancellation]\label{cor:sign-cancel}
Let $p>1$, $p^\sharp=4p/(3p-1)$, and
\[
 \varpi=a-b,
 \qquad a,b\ge0, \qquad ab=0 \text{ a.e.},
 \qquad a\in\mathcal X_{p^\sharp}\cap L^1,
 \qquad b\in\mathcal X_p\cap L^1.
\]
Set $u=\Bcal a-\Bcal b$.  Assume $u^r r(a+b)\in L^1$.  If
\begin{align}
 p\ge\frac43,
 \qquad\text{or}\qquad
 1<p<\frac43\ \text{ and }\ rb\in L^2_{\rm loc},
 \label{eq:sign-cancel-cases}
\end{align}
then
\begin{align}
 \int_{\R^3}u^r r|\varpi|\,dx=2\mathfrak{X}(b,a),
 \label{eq:sign-cancel-identity}
\end{align}
and from \eqref{eq:eta-mtwo},
\begin{align}
 \left|\int_{\R^3}u^r r|\varpi|\,dx\right|
 \lesssim_p
 \|b\|_1\|a\|_1
 +\|b\|_p^{1/4}(\Mtwo(b))^{3/4}\|a\|_{p^\sharp}.
 \label{eq:sign-cancel-bound}
\end{align}
\end{corollary}

\begin{proof}
The asymmetric kernel estimate \eqref{eq:eta-mtwo} makes $\mathfrak{X}(b,a)$ absolutely convergent, so
Lemma \ref{lem:kernel-identity} identifies both cross pairings. Since $p^\sharp\ge4/3$, Lemma \ref{lem:diagonal-cancel}
applies to $a$; it applies to $b$ by \eqref{eq:sign-cancel-cases}. Hence
\[
 \begin{aligned}
 \int u^r r|\varpi|\,dx
 &=\lim_{R\to\infty}
 \int\chi_R(\Bcal a-\Bcal b)^r r(a+b)\,dx =\mathfrak{X}(b,a)-\mathfrak{X}(a,b)=2\mathfrak{X}(b,a).
 \end{aligned}
\]
\end{proof}

\begin{proposition}[Sharpness of the asymmetric exponent]\label{prop:asymsharp}
Let $1<p<5/3$, let $p^\sharp$ be defined by
\eqref{eq:psharp}, and let
\begin{equation}\label{eq:exp-bd}
 p\le q<p^\sharp.
\end{equation}
There are nonnegative axisymmetric
$\varpi_R,\eta_R\in C_c^\infty(\R^3)$, $R\ge R_0$, with disjoint supports, such that
\begin{equation}\label{eq:mtwo-eta}
 \sup_{R\ge R_0}\left(
 \|\varpi_R\|_1+\|\varpi_R\|_p+\mathfrak M_2(\varpi_R)
 +\|\eta_R\|_1+\|\eta_R\|_q+\mathfrak M_2(\eta_R)
 \right)<\infty,
\end{equation}
whereas
\begin{equation}\label{eq:mtwo-conv}
 \mathfrak X(\varpi_R,\eta_R)\to\infty
 \qquad\text{as }R\to\infty.
\end{equation}
Hence the exponent $p^\sharp=4p/(3p-1)$ in
Proposition~\ref{prop:crossint} cannot be lowered
using only the norms appearing in \eqref{eq:mtwo-eta}.
\end{proposition}
\begin{proof}
For $0<s\le1$, differentiation of \eqref{eq:kern-def} gives
\[
 -2\mathcal F'(s)
 =\int_0^\pi
 \frac{\cos\alpha}
 {(s+2(1-\cos\alpha))^{3/2}}\,d\alpha.
\]
The part $0\le\alpha\le\sqrt s$ is bounded below by
$c s^{-1}$, the contribution from $[\sqrt s,\pi/2]$ is nonnegative, and the
part $[\pi/2,\pi]$ is uniformly bounded.  Hence there are $s_0,c_0>0$ such
that
\[
 -\mathcal F'(s)\ge c_0s^{-1},
 \qquad 0<s\le s_0.
\]

Choose nonzero nonnegative $\phi,\psi\in C_c^\infty(\R^2)$ supported in
$B_{1/4}((0,2))$ and $B_{1/4}((0,-2))$, respectively.  Set
\[
 \mu_p:=\frac{3p-1}{2(p-1)},
 \qquad
 \ell_R:=R^{-\mu_p},
 \qquad
 A_R:=R^{-3}\ell_R^{-2},
 \qquad
 B_R:=(R\ell_R^2)^{-1/q},
\]
and define
\[
 \varpi_R(r,z)
 :=A_R\phi\!\left(\frac{r-R}{\ell_R},\frac z{\ell_R}\right),
 \qquad
 \eta_R(r,z)
 :=B_R\psi\!\left(\frac{r-R}{\ell_R},\frac z{\ell_R}\right).
\]
For large $R$ these functions are smooth, axisymmetric, supported in
$\{R/2<r<2R\}$, and have disjoint supports.  Cylindrical coordinates and the
change of variables $r=R+\ell_Ru$, $z=\ell_Rv$ give
\[
 \|\varpi_R\|_p^p\asymp R A_R^p\ell_R^2=1,
 \qquad
 \|\varpi_R\|_1\lesssim R A_R\ell_R^2=R^{-2},
 \qquad
 \mathfrak M_2(\varpi_R)\lesssim R^3A_R\ell_R^2=1,
\]
and
\[
 \|\eta_R\|_q^q\asymp R B_R^q\ell_R^2=1,
 \qquad
 \|\eta_R\|_1\lesssim(R\ell_R^2)^{1-1/q},
 \qquad
 \mathfrak M_2(\eta_R)
 \lesssim R^{\frac{2(p-q)}{q(p-1)}}.
\]
Since $\mu_p>1/2$ and $q\ge p$, these bounds prove
\eqref{eq:mtwo-eta}.

If $X=(r,z)\in\operatorname{supp}\varpi_R$ and
$Y=(\rho,\zeta)\in\operatorname{supp}\eta_R$, then
$r,\rho\asymp R$, $z-\zeta\asymp\ell_R>0$, and $|X-Y|\asymp\ell_R$.
Consequently
$s:=|X-Y|^2/(r\rho)\to0$, and for large $R$ the lower bound just proved for
$-\mathcal F'$ implies
\[
 -2(r\rho)^{1/2}(z-\zeta)\mathcal F'(s)
 \gtrsim \frac{R^3}{\ell_R}
\]
on the product of the two supports.  Since $\phi$ and $\psi$ are nonnegative,
\[
 \mathfrak X(\varpi_R,\eta_R)
 \gtrsim A_RB_RR^3\ell_R^3
 =R^{\frac{4p+q-3pq}{2q(p-1)}}.
\]
The exponent is positive exactly when $q(3p-1)<4p$, equivalently
$q<p^\sharp$.  Hence \eqref{eq:mtwo-conv} follows.
\end{proof}

Recall the truncations $\Psi_L$ from \eqref{eq:rad-def}.  For the absolute impulse we use the weight $\Psi_L(r^2)$.  Then $\Psi_L(r^2)\uparrow r^2$ as $L\to\infty$, while for every $L>0$,
$u\cdot\nabla\Psi_L(r^2)=2\Psi_L'(r^2)(x_1u_1+x_2u_2)$.

\begin{theorem}[Sharp mixed-integrability criterion for absolute impulse]\label{thm:miximpulse}
Let $1<p<\infty$, and let $p^\sharp$ be defined by
\eqref{eq:psharp}.  Let $(u,\varpi)$ be a measurable
axisymmetric no-swirl pair on $[0,T]\times\R^3$.  Assume
\begin{equation}\label{eq:imp-bd}
 u(t)\in L^2(\R^3)\quad\text{for a.e. }t\in(0,T),
 \qquad
 \varpi(t)\in L^1(\R^3)\cap L^p(\R^3)\quad\text{for }0\le t\le T,
\end{equation}
and, for almost every $t\in(0,T)$,
\begin{equation}\label{eq:div-curl}
 \nabla\cdot u(t)=0,
 \qquad
 \nabla\times u(t)=r\varpi(t)e_\vartheta
 \quad\text{in }\mathcal D'(\R^3).
\end{equation}
No boundedness or trace condition on $\varpi$ at $r=0$ is imposed; here
$r\varpi e_\vartheta=(-x_2\varpi,x_1\varpi,0)$.  Set
\begin{equation}\label{eq:trace-def}
 N_1:=\sup_{0\le t\le T}\|\varpi(t)\|_1,
 \qquad
 N_p:=\sup_{0\le t\le T}\|\varpi(t)\|_p,
 \qquad
 Y(t):=\mathfrak M_2(\varpi(t)),
\end{equation}
and assume
\begin{equation}\label{eq:ltwo-dx}
 N_1+N_p<\infty,
 \qquad
 Y(0)<\infty,
 \qquad
 \int_0^T\int_{\R^3}
 |x_1u_1+x_2u_2|\,|\varpi|\,dx\,dt<\infty.
\end{equation}
Assume also that, for every $L>0$ and $0\le t\le T$,
\begin{align}
 &\frac1{2\pi}\int_{\R^3}\Psi_L(r^2)|\varpi(t)|\,dx
 -\frac1{2\pi}\int_{\R^3}\Psi_L(r^2)|\varpi(0)|\,dx
 =
 \frac1\pi\int_0^t\int_{\R^3}
 \Psi_L'(r^2)(x_1u_1+x_2u_2)|\varpi(\tau,x)|\,dx\,d\tau.
 \label{eq:psi-tau}
\end{align}
Then $Y\in W^{1,1}(0,T)$.  If, in addition,
\begin{equation}\label{eq:mix-def}
\begin{split}
    & m_\sharp(t):=\min_{\varsigma\in\{+,-\}}\|\varpi_\varsigma(t)\|_{p^\sharp}, \qquad 
 N_\sharp:=\sup_{0\le t\le T}m_\sharp(t)<\infty,
\end{split}
\end{equation}
and only when $1<p<4/3$, assume additionally that
\begin{equation}
\text{for a.e. }t\in(0,T),\ r\varpi_{\overline{\sigma}(t)}(t)
\in L^2_{\rm loc}(\mathbb R^3) , \text{ where }
\sigma(t)=\operatorname*{arg\,min}_{\varsigma\in\{+,-\}}
\|\varpi_\varsigma(t)\|_{p^\sharp} ,\label{eq:mix-local-l2}
\end{equation}
and \(\overline{\sigma}(t)\) denotes the opposite sign, then
\begin{equation}\label{eq:sharp-mix}
 |Y'(t)|
 \lesssim_p
 N_1^2+N_p^{1/4}m_\sharp(t)Y(t)^{3/4}
 \qquad\text{for a.e. }t\in(0,T).
\end{equation}
Consequently,
\begin{equation}\label{eq:sharp-sign}
 (1+Y(t))^{1/4}
 \le
 (1+Y(0))^{1/4}
 +C_p\left(N_1^2+N_p^{1/4}N_\sharp\right)t,
 \qquad 0\le t\le T.
\end{equation}
Assume $1<p<5/3$ and let $p\le q<\infty$.  For $M>0$, consider
axisymmetric no-swirl pairs satisfying
\eqref{eq:imp-bd}--\eqref{eq:psi-tau}, and when \(1<p<4/3\), the one-sided local condition \eqref{eq:mix-local-l2}, together with
\begin{equation}\label{eq:sig-axis}
 N_1+N_p+Y(0)
 +\sup_{0\le t\le T}\min_{\sigma\in\{+,-\}}
 \|\varpi_\sigma(t)\|_q\le M.
\end{equation}
If $q\ge p^\sharp$, then there is $C_{p,M,T}<\infty$ such that every
such pair obeys
\begin{equation}\label{eq:sharp-mom}
 \sup_{0\le t\le T}Y(t)
 +\|Y'\|_{L^\infty(0,T)}
 \le C_{p,M,T}.
\end{equation}
If $p\le q<p^\sharp$, there are classical finite-energy axisymmetric
no-swirl Euler solutions $U_R$, with relative vorticities $\widetilde\varpi_R$, and
times $t_R\downarrow0$ such that \eqref{eq:sig-axis} holds
on $[0,t_R]$ and, with $Y_R(t):=\mathfrak M_2(\widetilde\varpi_R(t))$,
\begin{equation}\label{eq:sol-bd}
 \sup_R\sup_{0\le t\le t_R}\left(
 \|U_R(t)\|_2+\|\widetilde\varpi_R(t)\|_1+\|\widetilde\varpi_R(t)\|_p+Y_R(t)
 +\min\{\|(\widetilde\varpi_R(t))_+\|_q,\|(\widetilde\varpi_R(t))_-\|_q\}
 \right)<\infty,
\end{equation}
whereas
\begin{equation}\label{eq:mix-sign}
 Y_R'(0)\to+\infty
\end{equation}
and
\begin{equation}\label{eq:thr-conv}
 \frac{Y_R(t_R)-Y_R(0)}{t_R}\to+\infty.
\end{equation}
Thus $p^\sharp=4p/(3p-1)$ is the exact threshold, for $q\ge p$, for
class-uniform local Lipschitz control of the absolute impulse under
the bounds in \eqref{eq:sig-axis}.
\end{theorem}
\begin{proof}
By \eqref{eq:ltwo-dx} and $ru^r=x_1u_1+x_2u_2$,
$|ru^r||\varpi|$ is integrable on $(0,T)\times\R^3$.  Define
\[
 Y^{(L)}(t):=\frac1{2\pi}\int_{\R^3}
 \Psi_L(r(x)^2)|\varpi(t,x)|\,dx.
\]
The balance hypothesis \eqref{eq:psi-tau} gives
\[
 Y^{(L)}(t)
 =Y^{(L)}(0)
 +\frac1\pi\int_0^t\int_{\R^3}
 \Psi_L'(r^2)ru^r|\varpi|\,dx\,d\tau.
\]
Letting $L\to\infty$, monotone convergence on the moment terms and dominated
convergence on the flux term yield
\[
 Y(t)-Y(0)
 =\frac1\pi\int_0^t\int_{\R^3}ru^r|\varpi|\,dx\,d\tau.
\]
Thus $Y\in W^{1,1}(0,T)$ and
\[
 \pi Y'(t)=\int_{\R^3}u^r(t,x)r|\varpi(t,x)|\,dx
\]
for almost every $t$.

Fix such a time and set $\sigma(t)=\operatorname*{arg\,min}_{\varsigma\in\{+,-\}}
\|\varpi_\varsigma(t)\|_{p^\sharp}$. Then
$\|\varpi_\sigma(t)\|_{p^\sharp}=m_\sharp(t)$, with $\bar\sigma$ denoting the
opposite sign. Without loss of generality, let $\sigma= +$. Set $a=\varpi_\sigma$ and
$b=\varpi_{\bar\sigma}$. Corollary \ref{cor:sign-cancel} gives
\[
 |Y'(t)|
 =\frac2\pi
 \left|\mathfrak X(\varpi_{\bar\sigma}(t),\varpi_\sigma(t))\right|
 \lesssim_p
 N_1^2+N_p^{1/4}m_\sharp(t)Y(t)^{3/4},
\]
which is \eqref{eq:sharp-mix}.  With $Z:=1+Y$, multiplication by
$\frac14Z^{-3/4}$ and integration give \eqref{eq:sharp-sign}.

Assume now $1<p<5/3$ and $q\ge p^\sharp$.  Since
$1<p^\sharp\le q$, interpolation between $L^1$ and $L^q$ gives, for each
sign component,
\[
 \|\varpi_\sigma(t)\|_{p^\sharp}
 \le
 \|\varpi_\sigma(t)\|_1^\theta
 \|\varpi_\sigma(t)\|_q^{1-\theta}
\]
with the unique $\theta\in[0,1)$ satisfying
$1/p^\sharp=\theta+(1-\theta)/q$.  The bound \eqref{eq:sig-axis} therefore
implies $N_\sharp\le M$.  Equations \eqref{eq:sharp-mix} and
\eqref{eq:sharp-sign} then give \eqref{eq:sharp-mom}.

It remains to prove sharpness when $p\le q<p^\sharp$.  Let
$\varpi_R,\eta_R$ be given by Proposition~\ref{prop:asymsharp}.  Because
$q\ge p$, interpolation between the uniform $L^1$ and $L^q$ bounds also gives
$\sup_R\|\eta_R\|_p<\infty$.  Hence there is a constant $C_{p,q}$ such that
\[
 \|\eta_R-\varpi_R\|_1+
 \|\eta_R-\varpi_R\|_p+
 \mathfrak M_2(\eta_R-\varpi_R)
 +\min\{\|\eta_R\|_q,\|\varpi_R\|_q\}
 \le C_{p,q}
\]
for all $R$.  Set
\[
 \lambda_M:=\min\left\{1,\frac{M}{1+C_{p,q}}\right\},
 \qquad
 \widetilde\varpi_{0,R}:=\lambda_M(\eta_R-\varpi_R).
\]
The supports are disjoint, so
$(\widetilde\varpi_{0,R})_+=\lambda_M\eta_R$ and
$(\widetilde\varpi_{0,R})_-=\lambda_M\varpi_R$, and the initial quantities
in \eqref{eq:sig-axis} are bounded by $M$ after this fixed scaling.  By
Proposition~\ref{prop:radmom}, the velocities
$U_{0,R}:=\Bcal\widetilde\varpi_{0,R}$ are uniformly bounded in $L^2$.
The corresponding vorticities are smooth and compactly supported; axisymmetry
gives zero spatial mean for the vorticity, so the Biot--Savart velocities are
smooth and belong to every $H^m$.  Classical Euler theory therefore gives a
smooth axisymmetric no-swirl solution $U_R$ on some interval $[0,T_R]$.
Transport preserves all $L^s$ norms of the relative vorticity and of its sign
components, and the Euler energy identity preserves $\|U_R(t)\|_2$.
Smoothness and compact support also justify the truncated balance
\eqref{eq:psi-tau} on every compact subinterval of $[0,T_R]$.

Writing $Y_R(t):=\mathfrak M_2(\widetilde\varpi_R(t))$ and differentiating at
$t=0$ gives
\[
 \pi Y_R'(0)
 =\int_{\R^3}U_{0,R}^r r|\widetilde\varpi_{0,R}|\,dx.
\]
By linearity of $\Bcal$, vanishing of the two self-interactions, and the skew
identity for $\mathfrak X$,
\[
 Y_R'(0)
 =\frac{2\lambda_M^2}{\pi}
 \mathfrak X(\varpi_R,\eta_R)
 \to+\infty
\]
by \eqref{eq:mtwo-conv}.  This proves \eqref{eq:mix-sign}.  Since $Y_R$ is differentiable at $0$, there is $s_R\in(0,T_R)$ such that for every $0<t\le s_R$,
\[
 \frac{Y_R(t)-Y_R(0)}{t}\ge\frac12Y_R'(0),
 \qquad
 |Y_R(t)-Y_R(0)|\le1.
\]
Choose $0<t_R\le\min\{s_R,R^{-1}\}$.  Then $t_R\to 0$ and \eqref{eq:thr-conv} follows.  The transport and energy identities, together with the uniform bound on $Y_R$ throughout $[0,t_R]$, give \eqref{eq:sol-bd} and verify \eqref{eq:sig-axis} on that interval.
\end{proof}

\begin{corollary}[Absolute impulse for classical Euler flows]\label{cor:smooth-impulse}
Let $1<p<\infty$, and let $(u,\varpi)$ be a classical finite-energy
axisymmetric no-swirl Euler flow on $[0,T]$ such that
\[
    u\in C([0,T];H^m(\mathbb R^3)),\qquad m>\frac52,
\]
and
\[
    \varpi\in L^\infty(0,T;L^1(\mathbb R^3)\cap L^p(\mathbb R^3)),
    \qquad
    \mathfrak M_2(\varpi(0))<\infty.
\]
If $p\ge5/3$, then
\[
    \bigl(1+\mathfrak M_2(\varpi(t))\bigr)^{1/4}
    \le
    \bigl(1+\mathfrak M_2(\varpi(0))\bigr)^{1/4}
    +Ct,
    \qquad 0\le t\le T,
\]
without a sign condition.

If $1<p<5/3$, the same estimate holds provided
\[
    \min\Bigl\{
        \|\varpi_+(0)\|_{L^{p^\sharp}},
        \|\varpi_-(0)\|_{L^{p^\sharp}}
    \Bigr\}<\infty,
    \qquad
    p^\sharp=\frac{4p}{3p-1}.
\]
Here $C$ depends only on $p$, the
$L^\infty(0,T;L^1\cap L^p)$ bound of $\varpi$, and, in the latter case,
the displayed one-sided $L^{p^\sharp}$ norm.
\end{corollary}
\begin{proof}
Since $m>5/2$, the velocity is Lipschitz in space on $[0,T]$, and its
classical flow $X(t,\cdot)$ is volume preserving.  The relative vorticity
satisfies
\[
    \partial_t\varpi+u\cdot\nabla\varpi=0,
    \qquad \nabla\cdot u=0,
\]
so
\[
    \varpi(t,X(t,a))=\varpi(0,a).
\]
In particular, each sign component is transported separately and, whenever
the right-hand side is finite,
\[
    \|\varpi_\pm(t)\|_{L^q}
    =\|\varpi_\pm(0)\|_{L^q},
    \qquad 1\le q\le\infty,\quad 0\le t\le T.
\]
Thus, if $1<p<5/3$, the one-sided initial $L^{p^\sharp}$ assumption gives
\eqref{eq:mix-def} uniformly on $[0,T]$.  If $p\ge5/3$, then
$p^\sharp\le p$, and interpolation between $L^1$ and $L^p$ gives
\eqref{eq:mix-def} for both sign components.

It remains only to verify the balance hypotheses of
Theorem~\ref{thm:miximpulse}.  Since
\[
    \nabla\times u=r\varpi e_\vartheta
\]
and $u\in C([0,T];H^m)$, one has $r\varpi\in C([0,T];L^2)$; hence, when
$1<p<4/3$, the local condition \eqref{eq:mix-local-l2} holds for both sign
components.  Moreover,
\[
 \int_0^T\!\int_{\mathbb R^3}
 |x_1u_1+x_2u_2|\,|\varpi|\,dx\,dt
 \le
 \int_0^T \|u(t)\|_2\,\|r\varpi(t)\|_2\,dt<\infty.
\]
For fixed $L>0$, use the flow identity and the chain rule along
characteristics:
\[
 \Psi_L(r(X(t,a))^2)-\Psi_L(r(a)^2)
 =2\int_0^t
 \Psi_L'(r(X(s,a))^2)\,r(X(s,a))\,u^r(s,X(s,a))\,ds.
\]
Multiplying by $|\varpi(0,a)|/(2\pi)$, integrating in $a$, and using
volume preservation together with the preceding flux integrability yields
exactly \eqref{eq:psi-tau}.  The div--curl assumptions
\eqref{eq:div-curl} are the Euler identities.  Therefore all hypotheses of
Theorem~\ref{thm:miximpulse} are satisfied, and \eqref{eq:sharp-sign}
gives the asserted estimate.  The $H^m$ regularity is used only to
justify the classical transport identities and the admissibility conditions
of Theorem~\ref{thm:miximpulse}; it does not enter the constant $C$.
\end{proof}

\begin{corollary}[Conditional first-radial-moment propagation below $3/2$]\label{cor:momprop}
Let $1<p<3/2$, and assume the hypotheses of Theorem~\ref{thm:miximpulse}, including
\eqref{eq:mix-def}.  Then
\begin{equation}\label{eq:prop-bd}
 \sup_{0\le t\le T}\mathfrak M_1(\varpi(t))<\infty.
\end{equation}
More precisely, with $N_1$ and $Y$ defined in \eqref{eq:trace-def},
\begin{equation}\label{eq:mone-mix}
 \mathfrak M_1(\varpi(t))
 \le \left(\frac{N_1Y(t)}{2\pi}\right)^{1/2},
 \qquad0\le t\le T,
\end{equation}
and $Y$ is bounded by \eqref{eq:sharp-sign}.
\end{corollary}
\begin{proof}
Equation~\eqref{eq:mominterp} applied at time $t$ gives
\[
 \mathfrak M_1(\varpi(t))^2
 \le \frac1{2\pi}\|\varpi(t)\|_1\mathfrak M_2(\varpi(t))
 \le \frac{N_1Y(t)}{2\pi}.
\]
Theorem~\ref{thm:miximpulse}, through \eqref{eq:sharp-sign}, bounds $Y$ uniformly on
$[0,T]$, which proves \eqref{eq:prop-bd}--\eqref{eq:mone-mix}.
\end{proof}
Thus Theorem~\ref{thm:thrsharp} identifies $p=3/2$ as the endpoint for
class-uniform sign-free differential control based on the first moment alone,
while Theorem~\ref{thm:miximpulse} yields first-radial-moment propagation below $3/2$ under finite initial absolute impulse and the mixed sign-component condition.
\subsection{The symmetric impulse threshold at $p=5/3$}
A stronger obstruction holds on a fixed time interval.  Unless an \(R\)-dependence is indicated explicitly, all implicit constants below are independent of the large parameter $R$; subscripts record additional
dependence.  The anti-parallel monotonicity and localization estimates are
adapted from Choi--Jeong \cite{ChoiJeong2025}, Gustafson--Miller--Tsai
\cite{GustafsonMillerTsai2026}, Egamberganov--Yao
\cite{EgamberganovYao2025}, and Cao--Fan--Qin \cite{CaoFanQin2026}.  Only
estimates uniform in the datum are used.
\begin{theorem}[Fixed-time obstruction to symmetric impulse control below $5/3$]\label{thm:impobstr}
Let $1<p\le q<5/3$.  There exists smooth, compactly supported, axisymmetric no-swirl initial data
$\varpi_{0,R}$, with $u_{0,R} = \mathcal{B}\varpi_{0,R}$ such that
\begin{equation}\label{eq:data-bd}
 \varpi_{0,R}(r,-z)=-\varpi_{0,R}(r,z),
 \qquad
 \varpi_{0,R}(r,z)\le0\quad(z>0),
\end{equation}
\begin{equation}\label{eq:sym-bd}
 \sup_{R\gg1}\Bigl(
 \|\varpi_{0,R}\|_{L^1}+\|\varpi_{0,R}\|_{L^p}+\|\varpi_{0,R}\|_{L^q}
 +\mathfrak M_2(\varpi_{0,R})+\|u_{0,R}\|_{L^2}
 \Bigr)<\infty,
\end{equation}
and, for the corresponding global smooth solutions,
\begin{equation}\label{eq:sm-conv}
 \mathfrak M_2(\varpi_R(1))\to\infty.
\end{equation}
More precisely, if
\begin{equation}\label{eq:sym-def}
 \beta_q:=\frac{5-3q}{2(q-1)},
\end{equation}
then
\begin{equation*}
 R^{-\beta}\mathfrak M_2(\varpi_R(1))\to\infty
 \qquad(0\le\beta<\beta_q).
\end{equation*}
\end{theorem}
\subsubsection{The singular family}
Fix $W\in C_c^\infty(\mathbb R^2)$ such that
\begin{equation}\label{eq:dy-bd}
 W(y_1,-y_2)=-W(y_1,y_2),\qquad
 W\ge0\ \text{on }\{y_2>0\},\qquad
 \operatorname{supp}W\subset B_1,
 \qquad
 \int_{y_2>0}W\,dy=1.
\end{equation}
In particular,
\begin{equation}\label{eq:sym-id}
 \int_{\mathbb R^2}W\,dy=0.
\end{equation}
Define
\begin{equation}\label{eq:imp-def}
 a_q:=\frac{3q-1}{2(q-1)},
 \qquad
 \Gamma_R:=R^{-2},
 \qquad
 \ell_R:=R^{-a_q},
\end{equation}
and
\begin{equation}\label{eq:thr-def}
 \varpi_{0,R}(r,z)
 :=-\frac{\Gamma_R}{r\ell_R^2}
 W\!\left(\frac{r-R}{\ell_R},\frac z{\ell_R}\right).
\end{equation}
The sign assertion \eqref{eq:data-bd} follows from
\eqref{eq:dy-bd}--\eqref{eq:thr-def}.
\begin{lemma}[Initial scales]\label{lem:initscale}
The family \eqref{eq:thr-def} satisfies
\begin{align}
 \|\varpi_{0,R}\|_{L^s(\mathbb R^3)}^s
 &\asymp_{W,s}
 \Gamma_R^s R^{1-s}\ell_R^{-2(s-1)},
   \ \ 1\le s<\infty,
 \label{eq:thr-bd}\\
 \mathfrak M_2(\varpi_{0,R})&\asymp_W \Gamma_RR^2\asymp_W1,
 \label{eq:mtwo-rr}\\
 \|u_{0,R}\|_{L^2(\mathbb R^3)}^2&\asymp_W\Gamma_R^2R=R^{-3}.
 \label{eq:imp-id}
\end{align}
Consequently, \eqref{eq:sym-bd} holds.
\end{lemma}
\begin{proof}
With $(r,z)=(R+\ell_Ry_1,\ell_Ry_2)$ and
$dx=2\pi r\,dr\,dz$,
\begin{equation*}
 \|\varpi_{0,R}\|_{L^s}^s
 =2\pi\Gamma_R^s\ell_R^{-2(s-1)}
 \int |W(y)|^s(R+\ell_Ry_1)^{1-s}\,dy,
\end{equation*}
which proves \eqref{eq:thr-bd}.  By the choice of $a_q$, the $L^q$
scale is invariant, while the $L^p$ scale is bounded for $p\le q$.
Equation \eqref{eq:mtwo-rr} follows from the same change of variables.
To prove \eqref{eq:imp-id}, let $\psi_R$ be the axisymmetric
stream function,
\begin{equation}\label{eq:axis-id}
 u_{0,R}^r=-\frac1r\partial_z\psi_R,
 \qquad
 u_{0,R}^z=\frac1r\partial_r\psi_R.
\end{equation}
Set $\varepsilon_R=\ell_R/R$, $a_R(y)=1+\varepsilon_Ry_1$,
$\Omega_R=\{y_1>-\varepsilon_R^{-1}\}$, and
\begin{equation}\label{eq:rr-def}
 \Phi_R(y):=\frac{\psi_R(R+\ell_Ry_1,\ell_Ry_2)}{\Gamma_RR}.
\end{equation}
The stream equation and \eqref{eq:sym-id} give
\begin{equation}\label{eq:phi-dy}
 \int_{\Omega_R}\frac{\nabla\Phi_R\cdot\nabla\varphi}{a_R}\,dy
 =-\int_{\mathbb R^2}W\varphi\,dy,
 \qquad
 \varphi\in C_c^\infty(\Omega_R).
\end{equation}
The axisymmetric vector identity $\int_{\mathbb R^3}\omega_{0,R}(y)\,dy=0$ and the three-dimensional Biot--Savart formula imply, for $|x|$ larger
than twice the vorticity-support radius,
\begin{equation}\label{eq:supp-id}
 u_{0,R}(x)=\frac1{4\pi}\int_{\mathbb R^3}\omega_{0,R}(y)\times
 \left(\frac{x-y}{|x-y|^3}-\frac{x}{|x|^3}\right)dy,
 \qquad |u_{0,R}(x)|\lesssim_R|x|^{-3}.
\end{equation}
Local boundedness and \eqref{eq:supp-id} already give $u_{0,R}\in L^2$.
Normalize the stream function by $\psi_R(0,z)=0$.  Then $\psi_R(r,z)=\int_0^r s u_{0,R}^z(s,z)\,ds$.  Moreover, $\nabla\cdot u_{0,R}=0$ and the decay in \eqref{eq:supp-id} imply $\frac d{dz}\int_0^\infty s u_{0,R}^z(s,z)\,ds=0$ and $\int_0^\infty s u_{0,R}^z(s,z)\,ds=0$.  Hence also $\psi_R(r,z)=-\int_r^\infty s u_{0,R}^z(s,z)\,ds$.
These two representations, smoothness, and \eqref{eq:supp-id} give
\begin{equation}\label{eq:psi-sym}
 |\psi_R(r,z)|\lesssim_R r^2(1+z^2)^{-3/2}\quad(0<r\le1),
 \qquad
 |\psi_R(r,z)|\lesssim_R(1+r^2+z^2)^{-1/2}.
\end{equation}
If $\psi_{R,n}$ is obtained by cutting $\psi_R$ off on
$\{r<1/n\}\cup\{(r^2+z^2)^{1/2}>n\}$, then
\begin{align}
 \int\frac{|\nabla(\psi_{R,n}-\psi_R)|^2}{r}\,dr\,dz
 &\lesssim
 \int_{\{r<2/n\}\cup\{(r^2+z^2)^{1/2}>n\}}
 \frac{|\nabla\psi_R|^2}{r}\,dr\,dz+O_R(n^{-2})
  =o_{n\to\infty}(1). \notag
\end{align}
The first term tends to zero because $u_{0,R}\in L^2$ and
$2\pi\int |\nabla\psi_R|^2r^{-1}\,dr\,dz=\|u_{0,R}\|_2^2$.
The $O_R(n^{-2})$ term follows from \eqref{eq:psi-sym}, after splitting the
outer annulus into $r<1$ and $r\ge1$.  After the scaling \eqref{eq:rr-def}, these cutoffs are compactly supported
inside $\Omega_R$ and converge to $\Phi_R$ in the weighted energy norm.
Thus $\Phi_R$ is an admissible energy test in \eqref{eq:phi-dy}.
Since $\operatorname{supp}W\subset B_1$, $a_R\asymp1$ on $B_2$, and
\eqref{eq:sym-id} removes constants, Poincar\'e's inequality on
$B_2$ yields
\begin{equation*}
 \mathcal A_R:=\int_{\Omega_R}\frac{|\nabla\Phi_R|^2}{a_R}\,dy
 =-\int W\Phi_R\,dy\lesssim_W \mathcal A_R^{1/2}.
\end{equation*}
Testing \eqref{eq:phi-dy} with $W$ gives
\begin{equation*}
 \|W\|_{L^2}^2
 \lesssim_W \mathcal A_R^{1/2}.
\end{equation*}
Thus $\mathcal A_R\asymp_W1$.  Finally, \eqref{eq:axis-id}--\eqref{eq:rr-def}
give
\begin{equation*}
 \|u_{0,R}\|_{L^2}^2=2\pi\Gamma_R^2R \mathcal A_R,
\end{equation*}
which proves \eqref{eq:imp-id}.
\end{proof}
For each fixed $R$, the datum \eqref{eq:thr-def} is smooth, compactly
supported away from the symmetry axis, and satisfies
$\omega_{0,R}:=(r\varpi_{0,R})e_\vartheta \in H^m(\mathbb R^3)$ for every $m$ and
$\omega_{0,R}^\vartheta/r=\varpi_{0,R}\in L^2\cap L^\infty$.  It therefore lies
in the classical global axisymmetric no-swirl class of
\cite{UkhovskiiIudovich1968}; see also the modern formulation with the same
$H^m$ and $L^2\cap L^\infty$ hypotheses in
\cite[Section~1.2]{LimJeong2025}.  Denote the solution by
$(u_R,\varpi_R)$.  No estimate below uses a high-Sobolev norm uniformly in $R$.
\subsubsection{Anti-parallel localization estimates and their $R$-scales}
On $\Pi_+:=\{r>0,z>0\}$ set
\begin{equation*}
 \eta_R:=-\omega_R^\vartheta\ge0,
 \qquad
 P_R(t):=\iint_{\Pi_+}r^2\eta_R\,dr\,dz,
 \qquad
 \mathcal K_R:=\iint_{\Pi_+}r|u_R|^2\,dr\,dz.
\end{equation*}
Odd symmetry gives
\begin{equation}\label{eq:bsk-id}
 \mathfrak M_2(\varpi_R(t))=2P_R(t),
 \qquad
 \|u_R(t)\|_{L^2}^2=4\pi\mathcal K_R.
\end{equation}
The sign and symmetry in \eqref{eq:data-bd} are preserved by the Euler flow; see Choi--Jeong \cite{ChoiJeong2025} and Cao--Fan--Qin \cite{CaoFanQin2026}.
The exact axis identity of Egamberganov--Yao \cite[(4.3)]{EgamberganovYao2025} is
\begin{equation}\label{eq:dr-eul}
 P_R'(t)=\int_0^\infty r|u_R^r(r,0,t)|^2\,dr\ge0.
\end{equation}
Here the scalar vorticity denoted by $\omega$ in
\cite{EgamberganovYao2025} equals $-\omega^\vartheta$ in our convention,
namely $\eta_R$ on $\Pi_+$, so no sign change is hidden in
\eqref{eq:dr-eul}.  Let $Q_R:=\|\eta_R(0)/r\|_{L^\infty(\Pi_+)}$.  Transport of $\eta_R/r$ gives
$\eta_R(t,r,z)\le Q_Rr$.  Using the first equality in the derivation of
\eqref{eq:dr-eul}, Cauchy--Schwarz, and energy conservation,
\begin{align}
 P_R'(t)
 &\le 2\left(\iint_{\Pi_+}r|u_R^r|^2\right)^{1/2}
       \left(\iint_{\Pi_+}r\eta_R^2\right)^{1/2} \le 2(\mathcal K_RQ_RP_R(t))^{1/2}. \notag
\end{align}
Hence
\begin{equation*}
 P_R(t)^{1/2}
 \le P_R(0)^{1/2}+t(\mathcal K_RQ_R)^{1/2}.
\end{equation*}
This is the same finite-time moment estimate recorded in
\cite[Section~2.1, equations~(2.9)--(2.10)]{CaoFanQin2026}; here it is
rederived to keep its dependence on the degenerating family explicit.
By Lemma~\ref{lem:initscale},
\begin{equation}\label{eq:bsk-wr}
 P_R(0)\asymp_W1,
 \qquad
 \mathcal K_R\asymp_WR^{-3},
 \qquad
 Q_R\lesssim_WR^{2a_q-3},
\end{equation}
and therefore
\begin{equation}\label{eq:local-bd}
 \sup_{0\le t\le1}P_R(t)\lesssim_W R^{2a_q-6}.
\end{equation}
Cao--Fan--Qin \cite[Proposition~3.2, equation~(3.5)]{CaoFanQin2026} introduced the mixed
weight
\begin{equation*}
 \Upsilon(r,z):=\int_0^z\sqrt{r^2+s^2}\,ds,
 \qquad
 M_R(t):=\iint_{\Pi_+}\Upsilon(r,z)\eta_R(r,z,t)\,dr\,dz,
\end{equation*}
and proved the exact dissipation identity
\begin{align}
 M_R'(t)
 &=-\frac12\int_0^\infty z|u_R^z(0,z,t)|^2\,dz
 -\iint_{\Pi_+}\frac{(zu_R^r-ru_R^z)^2}{r\sqrt{r^2+z^2}}\,dr\,dz
 \le0. \notag
\end{align}
Since $rz\le\Upsilon(r,z)$,
\begin{equation}\label{eq:pr-bd}
 \iint_{\Pi_+}rz\eta_R(t)\,dr\,dz\le M_R(0).
\end{equation}
The support of \eqref{eq:thr-def} gives
\begin{equation}\label{eq:pr-def}
 m_R:=\iint_{\Pi_+}\eta_R\,dr\,dz=\Gamma_R=R^{-2},
 \qquad
 M_R(0)\lesssim_W R\ell_R\Gamma_R=R^{-1-a_q}.
\end{equation}
To control both meridional components of the velocity, we use the second mixed moment of Cao--Fan--Qin \cite[Propositions~3.3--3.4]{CaoFanQin2026}.  Put $\varrho=(r^2+z^2)^{1/2}$ and $g(\varrho):=\varrho/\operatorname{arsinh}\varrho$, with $g(0)=1$.  There is an absolute constant $\lambda_0>0$ such that,
for
\[
 \widetilde\Upsilon(r,z)
 :=\lambda_0\Upsilon(r,z)+z g(\varrho),
 \qquad
 \widetilde M_R(t)
 :=\iint_{\Pi_+}\widetilde\Upsilon(r,z)\eta_R(r,z,t)\,dr\,dz,
\]
the function $\widetilde M_R$ is nonnegative and nonincreasing.  If $\mathfrak D_R(t):=-\widetilde M_R'(t)\ge0$, then
\begin{equation*}
 \int_0^\infty\mathfrak D_R(t)\,dt\le \widetilde M_R(0),
 \qquad
 \iint_{\Pi_+\cap\{\varrho\le L\}}r|u_R|^2\,dr\,dz
 \lesssim L\log^2(2+L)\mathfrak D_R(t),
 \quad L\ge2.
\end{equation*}
On the support of the upper-half-plane part of \eqref{eq:thr-def}, one has
$r\asymp R$, $0<z\lesssim\ell_R$, and $\varrho\asymp R$.  Since
$\Upsilon(r,z)\lesssim R\ell_R$ and
$zg(\varrho)\lesssim R\ell_R$, \eqref{eq:pr-def} gives $\widetilde M_R(0)\lesssim_W R\ell_R\Gamma_R=R^{-1-a_q}$.  Consequently,
\begin{equation}\label{eq:dz-bd}
 \int_0^1\iint_{\Pi_+\cap\{\sqrt{r^2+z^2}\le L\}}r|u_R|^2\,dr\,dz\,dt
 \lesssim_W L\log^2(2+L)R^{-1-a_q},
 \qquad L\ge2.
\end{equation}
Fix $0<b<a_q-2$ and $A>1$.  Taking $L=AR^b$ in \eqref{eq:dz-bd}, dividing by
$\mathcal K_R\asymp_W R^{-3}$ from \eqref{eq:bsk-wr}, and using
$b+2-a_q<0$, we obtain
\[
 \frac{1}{\mathcal K_R}
 \int_0^1\iint_{\Pi_+\cap\{\sqrt{r^2+z^2}\le AR^b\}}r|u_R|^2\,dr\,dz\,dt
 \lesssim_{W,A}R^{b+2-a_q}\log^2 R\to0.
\]
Energy conservation therefore gives, for all sufficiently large $R$, a time
$t_R\in[0,1]$ such that
\begin{equation}\label{eq:ar-bd}
 \iint_{\Pi_+\cap\{\sqrt{r^2+z^2}>AR^b\}}r|u_R(r,z,t_R)|^2\,dr\,dz
 \ge\frac34\mathcal K_R.
\end{equation}
\subsubsection{Uniform logarithmic energy control}
For a measurable $0\le g\le\eta_R(t)$, let $u_g$ be the velocity induced
by the odd extension of $-g$ from $\Pi_+$.  For measurable truncations, the positive energy representation and the
pointwise logarithmic energy-kernel bound are given in
\cite[Section~2.2, equations~(2.14), (2.22)--(2.23)]{CaoFanQin2026}; the
underlying three-dimensional energy estimate goes back, in this anti-parallel
framework, to \cite[Section~3.3, Lemma~3.5]{GustafsonMillerTsai2026}.  The
source logarithm is
$L(X,Y)=\log(2+rs/|X-Y|^2)$.  Since, for
$a=\sqrt{rs}/|X-Y|$,
$\log(2+a^2)\le2\log(2+a)$, its bound implies the following form.  In the notation
$X=(r,z)$, $Y=(s,w)$,
\begin{equation}\label{eq:rs-def}
 \|u_g\|_{L^2}^2
 \lesssim
 \iint_{\Pi_+\times\Pi_+}
 \sqrt{rs}\,\Lambda(X,Y)g(X)g(Y)\,dX\,dY,
 \qquad
 \Lambda(X,Y):=\log\!\left(2+\frac{\sqrt{rs}}{|X-Y|}\right).
\end{equation}
Moreover, the kernel in the exact representation
\cite[Eq.~(2.14)]{CaoFanQin2026} is nonnegative on
$\Pi_+\times\Pi_+$; hence truncation monotonicity gives
\begin{equation}\label{eq:trunc-bd}
 0\le g\le\eta_R(t)
 \quad\Longrightarrow\quad
 \|u_g\|_{L^2}\le\|u_R(t)\|_{L^2}.
\end{equation}
The logarithmic $L^p$ kernel estimates in
\cite[Lemma~3.6]{GustafsonMillerTsai2026} and
\cite[Lemma~2.1]{CaoFanQin2026} have constants depending on fixed-data
quantities.  For the degenerating family \eqref{eq:thr-def}, the required dependence on the geometric parameters is given by the following estimate.
\begin{lemma}[Uniform logarithmic bound]\label{lem:logbound}
Let $\eta\ge0$ on $\Pi_+$ and set $m:=\iint\eta$, $Q:=\|\eta/r\|_\infty$, and $P:=\iint r^2\eta$.  If $m>0$, then every $0\le g\le\eta$ satisfies
\begin{equation}\label{eq:qm-bd}
 \left(\iint\Lambda(X,Y)^2g(X)g(Y)\,dX\,dY\right)^{1/2}
 \lesssim
 m\left[1+\log\!\left(2+\frac Qm+\frac Pm\right)\right].
\end{equation}
\end{lemma}
\begin{proof}
Fix $Y=(s,w)\in\Pi_+$ and put $d=|X-Y|$.  If $d\ge s/2$, then
$r\le3d$ and $\sqrt{rs}/d\lesssim1$; if $d<s/2$, then $r<3s/2$.
Hence
\[
 \Lambda(X,Y)^2
 \lesssim1+\mathbf1_{\{d<s/2\}}
 \log_+^2\!\left(\frac{s}{d}\right).
\]
Since $g\le\eta\le Qr$, for every $t\ge0$,
\[
 \int_{\{|X-Y|<se^{-t}\}}g(X)\,dX
 \le\min\{m,CQs^3e^{-2t}\}.
\]
Layer cake therefore gives
\[
 \int\log_+^2\!\left(\frac{s}{|X-Y|}\right)g(X)\,dX
 \lesssim
 m\left[1+\log_+^2\!\left(\frac{CQs^3}{m}\right)\right].
\]
Here we used the elementary bound
$\int_0^\infty2t\min\{1,Ae^{-2t}\}\,dt
\lesssim1+\log_+^2A$.

Now set $d\nu(Y)=m^{-1}\eta(Y)\,dY$.  Since
$\int s^2\,d\nu=P/m$, Chebyshev's inequality and the same layer-cake
bound give
\[
 \int\log^2(1+s)\,d\nu
 \lesssim1+\log_+^2(P/m).
\]
Finally,
\[
 \log_+\!\left(\frac{CQs^3}{m}\right)
 \lesssim
 1+\log\!\left(1+\frac Qm\right)+\log(1+s).
\]
Integrating the previous pointwise bound for $\Lambda^2$ first in $X$
and then against $g(Y)\,dY\le\eta(Y)\,dY$ proves the square of
\eqref{eq:qm-bd}.
\end{proof}
Lemma~\ref{lem:logbound}, \eqref{eq:bsk-wr},
\eqref{eq:local-bd}, and \eqref{eq:pr-def} give, uniformly for
$t\in[0,1]$,
\begin{equation*}
 \left(\iint\Lambda^2g(X)g(Y)\,dX\,dY\right)^{1/2}
 \lesssim_{W,q}R^{-2}\log R,
 \qquad 0\le g\le\eta_R(t).
\end{equation*}
By Cauchy--Schwarz in \eqref{eq:rs-def},
\begin{equation}\label{eq:ltwo-pr}
 \|u_g\|_{L^2}^2
 \lesssim_{W,q}R^{-2}\log R
 \iint_{\Pi_+}r g\,dr\,dz.
\end{equation}
\subsubsection{Radial escape and the $5/3$ window}
Fix $A\ge4$, let $L_R:=R^b$, and let $t_R$ satisfy
\eqref{eq:ar-bd}.  Split
\begin{equation*}
 \eta_R(t_R)=\eta_R^{\rm in}+\eta_R^{\rm out},
 \qquad
 \eta_R^{\rm out}=\eta_R^{\rm cap}+\eta_R^{\rm far},
\end{equation*}
where
\begin{equation*}
 \eta_R^{\rm in}:=\eta_R\mathbf1_{\{\sqrt{r^2+z^2}\le L_R\}},
 \quad
 \eta_R^{\rm cap}:=\eta_R\mathbf1_{\{\sqrt{r^2+z^2}>L_R,\ r<L_R/2\}},
 \quad
 \eta_R^{\rm far}:=\eta_R\mathbf1_{\{\sqrt{r^2+z^2}>L_R,\ r\ge L_R/2\}}.
\end{equation*}
The odd azimuthal vorticity obtained from the hard truncation
$\eta_R^{\rm in}$ is compactly supported and belongs to $L^1\cap L^2$.
Moreover it is distributionally divergence-free: for an axisymmetric scalar
$f(r,z)$, $\nabla\cdot(f e_\vartheta)=r^{-1}\partial_\theta f=0$, and this
identity does not differentiate the truncating indicator.  Hence the
hypotheses of Cao--Fan--Qin's exterior harmonic estimate
\cite[Lemma~3.5, equation~(3.29)]{CaoFanQin2026} are satisfied, and
\begin{equation}\label{eq:al-bd}
 \|u_{\eta_R^{\rm in}}\|_{L^2(\mathbb R^3\setminus B_{AL_R})}
 \le A^{-3/2}\|u_{\eta_R^{\rm in}}\|_{L^2}
 \le A^{-3/2}\|u_R(t_R)\|_{L^2},
\end{equation}
where the last inequality is \eqref{eq:trunc-bd}.
Combining \eqref{eq:ar-bd}, \eqref{eq:bsk-id},
and \eqref{eq:al-bd} yields
\begin{equation}\label{eq:out-bd}
 \|u_{\eta_R^{\rm out}}\|_{L^2}\ge\frac12\|u_R(t_R)\|_{L^2}
\end{equation}
for $A\ge4$.
On the support of $\eta_R^{\rm cap}$,
$z>\sqrt3L_R/2$.  Hence \eqref{eq:pr-bd} and
\eqref{eq:pr-def} give
\begin{equation*}
 \iint r\eta_R^{\rm cap}
 \le\frac{2M_R(0)}{\sqrt3L_R}
 \lesssim_W R^{-1-a_q-b}.
\end{equation*}
By \eqref{eq:ltwo-pr},
\begin{equation}\label{eq:thr-id}
 \|u_{\eta_R^{\rm cap}}\|_{L^2}^2
 \lesssim_{W,q}R^{-3-a_q-b}\log R=o(R^{-3}).
\end{equation}
Since \eqref{eq:imp-id} and energy conservation give
$\|u_R(t_R)\|_2^2\asymp_WR^{-3}$, equations
\eqref{eq:out-bd}--\eqref{eq:thr-id} imply
\begin{equation}\label{eq:eta-far}
 \|u_{\eta_R^{\rm far}}\|_{L^2}^2\gtrsim_{W,q}R^{-3}.
\end{equation}
This step uses only linearity of the Biot--Savart map and the triangle
inequality; no additivity of kinetic energy is used.
Since $r\ge L_R/2$ on $\operatorname{supp}\eta_R^{\rm far}$,
\begin{equation}\label{eq:far-bd}
 \iint r\eta_R^{\rm far}
 \le\frac2{L_R}\iint r^2\eta_R(t_R)
 =\frac{2P_R(t_R)}{L_R}.
\end{equation}
Equations \eqref{eq:ltwo-pr},
\eqref{eq:eta-far}, and \eqref{eq:far-bd} give
\begin{equation}\label{eq:rad-sym}
 P_R(t_R)\gtrsim_{W,q}\frac{R^{b-1}}{\log R}.
\end{equation}
By \eqref{eq:dr-eul}, $P_R(1)\ge P_R(t_R)$.
By \eqref{eq:imp-def} and \eqref{eq:sym-def},
\begin{equation}\label{eq:bet-rad}
 a_q-3
 =\frac{3q-1-6(q-1)}{2(q-1)}
 =\frac{5-3q}{2(q-1)}
 =\beta_q.
\end{equation}
For $0\le\beta<\beta_q$ choose $\beta+1<b<a_q-2$; this is possible precisely by \eqref{eq:bet-rad}.
Then \eqref{eq:bsk-id} and \eqref{eq:rad-sym} imply
\begin{equation*}
 R^{-\beta}\mathfrak M_2(\varpi_R(1))
 \gtrsim_{W,q}\frac{R^{b-1-\beta}}{\log R}
 \to\infty.
\end{equation*}
This proves Theorem~\ref{thm:impobstr}.
\begin{corollary}[Symmetric sign-free absolute-impulse threshold]\label{cor:signfree}
Let $5/3\le p<\infty$ and let $(u,\varpi)$ satisfy
\eqref{eq:imp-bd}--\eqref{eq:psi-tau}.  For every $M,T>0$ there is
$C_{p,M,T}<\infty$ such that
\begin{equation}\label{eq:thr-imp}
 N_1+N_p+Y(0)\le M
 \quad\Longrightarrow\quad
 \sup_{0\le t\le T}Y(t)+\|Y'\|_{L^\infty(0,T)}
 \le C_{p,M,T}.
\end{equation}
No sign condition is imposed.  Conversely, for every $1<p<5/3$ there
are smooth classical solutions satisfying the balance hypotheses of
Theorem~\ref{thm:miximpulse} with uniformly bounded $N_1$, $N_p$, $Y(0)$, and
Euler energy on $[0,1]$, but with $Y(1)\to\infty$.  Thus $5/3$ is the
exact sign-free threshold for class-uniform fixed-time absolute-impulse
control in this balance class.
\end{corollary}
\begin{proof}
For $p\ge5/3$, one has $p^\sharp\le p$.  Interpolation between
$L^1$ and $L^p$ therefore bounds each sign component in $L^{p^\sharp}$
uniformly in time by a constant depending only on $p$, $N_1$, and $N_p$.
Hence \eqref{eq:mix-def} holds without a sign hypothesis.  Equation
\eqref{eq:sharp-sign} first bounds $\sup_{0\le t\le T}Y(t)$ in terms of
$p,M,T$; substituting this bound into \eqref{eq:sharp-mix} then bounds
$\|Y'\|_{L^\infty(0,T)}$.  This proves \eqref{eq:thr-imp}.
If $1<p<5/3$, apply Theorem~\ref{thm:impobstr} with $q=p$.  Relative vorticity is
transported by the smooth incompressible Euler flow, so its $L^1$ and
$L^p$ norms are conserved.  Smoothness, compact support on finite time intervals, and the classical
transport identity verify \eqref{eq:imp-bd}--\eqref{eq:psi-tau}.  The solutions in Theorem~\ref{thm:impobstr}
therefore have uniformly bounded $N_1$, $N_p$, $Y(0)$, and energy on
$[0,1]$, while \eqref{eq:sm-conv} gives $Y(1)\to\infty$.
\end{proof}

\section{Kinetic framework, exact Maxwellian geometry, and main theorems}\label{sec:kinframe}
We first fix the collision models and parameter ranges in \eqref{eq:kinetic-model}.  For Coulomb Landau,
\[
 Q(G,H)
 =\nabla_\xi\!\cdot\!\int_{\mathbb R^3}
 \Phi(\xi-\xi_*)\bigl(G_*\nabla_\xi H-H\nabla_{\xi_*}G_*\bigr)\,d\xi_*,
 \qquad
 \Phi(\mathbf k)=|\mathbf k|^{-1}\!\left(I-\frac{\mathbf k\otimes\mathbf k}{|\mathbf k|^2}\right).
\]
For non-cutoff Boltzmann,
\[
 Q(G,H)(\xi)
 =\int_{\mathbb R^3\times\mathbb S^2}
 B(\xi-\xi_*,\mathbf n)(G'H'_*-GH_*)\,d\mathbf n\,d\xi_*,
\]
where
\[
 \xi'=\frac{\xi+\xi_*}{2}+\frac{|\xi-\xi_*|}{2}\mathbf n,
 \qquad
 \xi_*'=\frac{\xi+\xi_*}{2}-\frac{|\xi-\xi_*|}{2}\mathbf n,
\]
and
\[
 B(\mathbf k,\mathbf n)=C_v|\mathbf k|^\gamma b(\cos\sigma),
 \qquad
 \sin\sigma\,b(\cos\sigma)\asymp\sigma^{-1-2s}, \qquad \cos \sigma = \frac{\mathbf k}{|\mathbf k|}\cdot \mathbf n,
 \qquad 0<s<1,
\]
with
\[
 \max\left\{-3,-\frac32-2s\right\}<\gamma<0,
 \qquad
 \gamma+2s\in
 \left(\max\left\{-3+2s,-\frac32\right\},0\right].
\]
For either model set
\[
 \mathcal N(G,H):=\frac12\bigl(Q(G,H)+Q(H,G)\bigr),
\]
so that $\mathcal N(F,F)=Q(F,F)$.  The kinetic equation is therefore
\[
 \e\,\partial_tF^\e+\xi\cdot\nabla_xF^\e
 =\frac1{\kappa\e}\mathcal N(F^\e,F^\e),
\]
or, after division by $\e$,
$\mathcal T_\e F^\e=(\kappa\e^2)^{-1}\mathcal N(F^\e,F^\e)$ with
$\mathcal T_\e:=\partial_t+\e^{-1}\xi\cdot\nabla_x$.

\begin{equation}\label{eq:xi-id}
\int_{\mathbb{R}^3} F^\varepsilon
\begin{pmatrix}
1\\ \xi\\ |\xi|^2/2
\end{pmatrix}\, \dd \xi
=
\begin{pmatrix}
\mathrm P^\varepsilon\\ \mathrm P^\varepsilon \mathrm U^\varepsilon\\
\tfrac{3k_B}{2}\mathrm P^\varepsilon\Theta^\varepsilon + \tfrac{1}{2}\mathrm P^\varepsilon|\mathrm U^\varepsilon|^2
\end{pmatrix}.
\end{equation}
The collision operator preserves the collision invariants:
\begin{equation*}
\int_{\mathbb{R}^3}\mathcal{N}(F^\e,F^\e)(\xi)\,(1,\ \xi,\ |\xi|^2/2)^{\mathrm T}\, \dd \xi =(0,0,0)^T.
\end{equation*}
These collision invariants yield exact balance laws:
\begin{align}\label{eq:invar-id}
\bega
&\e\p_t\mathrm{P}^{\e} +\nabla_x \cdot (\mathrm{P}^{\e}\mathrm{U}^{\e})=0, \cr
&\e\p_t(\mathrm{P}^{\e} \mathrm{U}^{\e}) + \nabla_x\cdot\big(\mathrm{P}^{\e}\,\mathrm{U}^{\e}\otimes\mathrm{U}^{\e}\big) +k_B\nabla_x(\mathrm{P}^{\e} \mathrm{\Theta}^{\e}) + \nabla_x\cdot \mathbf{r}^{\e} =0, \cr  
&\e\p_t\bigg(\frac{3k_B}{2}\mathrm{P}^{\e}\mathrm{\Theta}^{\e}\bigg)
+  \nabla_x\cdot\bigg(\frac{3k_B}{2}\mathrm{P}^{\e}\mathrm{\Theta}^{\e} \mathrm{U}^{\e}\bigg) + k_B
\mathrm{P}^{\e} \mathrm{\Theta}^{\e}(\nabla_x \cdot \mathrm{U}^{\e})  + \nabla_x\cdot \mathfrak{q}^{\e} +\nabla \mathrm U^{\e} \! : \mathbf{r}^{\e} =0.
\enda
\end{align}
The stress deviator and heat flux are the Burnett moments.
\begin{definition}[Burnett stress and heat-flux moments]
Suppose $F^{\e}$ is given and the macroscopic fields in \eqref{eq:xi-id} are
well-defined. Set
\begin{align}
\mathbf r_{ij}^{\e}
&:=\langle F^{\e},\mathfrak R_{ij}^{\e}\rangle,
&
\mathfrak q_i^{\e}
&:=\langle F^{\e},\mathcal Q_i^{\e}\rangle, \notag
\end{align}
where
\begin{equation}\label{eq:ij-def}
\mathfrak R_{ij}^{\e}
:=(\xi_i-\mathrm U_i^{\e})(\xi_j-\mathrm U_j^{\e})
-\frac{|\xi-\mathrm U^{\e}|^2}{3}\delta_{ij},
\qquad
\mathcal Q_i^{\e}
:=(\xi_i-\mathrm U_i^{\e})
\frac{|\xi-\mathrm U^{\e}|^2-5k_B\Theta^{\e}}{2}.
\end{equation}
\end{definition}
For spatial vector fields, let $\mathbb P$ denote the Leray projection:
\begin{equation}\label{eq:pq-def}
 \mathbb P^\perp f:=\nabla_x\Delta_x^{-1}\nabla_x\cdot f,
 \qquad
 \mathbb Pf:=f-\mathbb P^\perp f,
\end{equation}
where $\Delta_x^{-1}$ is the multiplier $-|k|^{-2}$.  Thus
$\nabla_x\cdot\mathbb Pf=0$ and $\mathbb P$ is the $L_x^2$-orthogonal
projection onto divergence-free fields.
\subsection{Local and global Maxwellian decompositions}
Write the local Maxwellian as
\begin{align}
    & M_{[\mathrm P,\mathrm U,\Theta]} 
=
\frac{\mathrm P(t,x)}{\big(2\pi k_B \Theta(t,x)\big)^{3/2}}
\exp\!\Big(-\frac{|\xi-\mathrm U(t,x)|^2}{2k_B\Theta(t,x)}\Big). \label{eq:loc-id}
\end{align}
\begin{align*}
    & M^\varepsilon  = M_{[\mathrm P^\e, \mathrm U^\e, \Theta^\e]}
=
\frac{\mathrm P^\varepsilon(t,x)}{\big(2\pi k_B \Theta^\varepsilon(t,x)\big)^{3/2}}
\exp\!\Big(-\frac{|\xi-\mathrm U^\varepsilon(t,x)|^2}{2k_B\Theta^\varepsilon(t,x)}\Big). 
\end{align*}
For fixed $(t,x)$, put $c^\e=\xi-\mathrm U^\e$ and
$\chi^\e=|c^\e|^2/(2k_B\Theta^\e)-3/2$.  The local null space in the
distribution variable is
\[
 \mathcal N_M
 :=\operatorname{span}\{M^\e,\ c_1^\e M^\e,\ c_2^\e M^\e,
 c_3^\e M^\e,\ \chi^\e M^\e\}.
\]
For any integrable $h=h(\xi)$ define the local macroscopic projection by
\begin{align}\label{eq:mp-def}
 \Pi_Mh
 :=M^\e\bigg[&\frac1{\mathrm P^\e}\int h\,d\xi
 +\frac{c^\e}{k_B\Theta^\e}\cdot
   \frac1{\mathrm P^\e}\int c^\e h\,d\xi +\frac2{3\mathrm P^\e}\chi^\e
   \int\chi^\e h\,d\xi\bigg].
\end{align}
The Gaussian identities
\[
 \int c_i^\e c_j^\e M^\e\,d\xi
 =\mathrm P^\e k_B\Theta^\e\delta_{ij},
 \qquad
 \int(\chi^\e)^2M^\e\,d\xi=\frac32\mathrm P^\e
\]
show that \eqref{eq:mp-def} is the orthogonal projection onto
$\mathcal N_M$ for the pairing
$\langle h,k\rangle_{M^{-1}}=\int hk(M^\e)^{-1}d\xi$.  Equivalently,
$\Pi_M^\perp h$ is characterized by the five vanishing moments
\[
 \int \Pi_M^\perp h\,(1,c^\e,|c^\e|^2)\,d\xi=0.
\]
For the moment-matching Maxwellian $M^\e=M[F^\e]$ this gives
\begin{equation}\label{eq:dist-def}
 \Pi_M^{\perp}:=I-\Pi_M,
 \qquad
 \Pi_M F^\e=M^\e.
\end{equation}
Consequently, with $G^\e:=F^\e-M^\e$, one has for every spatial
multi-index $\alpha$
\begin{equation}\label{eq:xd-id}
 \int_{\mathbb R^3}\partial_x^\alpha G^\e
 (1,\xi,|\xi|^2)\,d\xi=0,
 \qquad
 \Pi_M(\partial_x^\alpha G^\e)=0.
\end{equation}
Indeed, the first identity follows by differentiating the exact five moment
matching relations, and the second follows from the characterization of
$\operatorname{Ran}\Pi_M^\perp$ above.  In particular,
\begin{equation*}
 F^\e=\Pi_M F^\e+\Pi_M^{\perp} F^\e
 =M^\e+\Pi_M^{\perp} F^\e.
\end{equation*}
The Burnett functions \eqref{eq:ij-def} are orthogonal to the local Maxwellian:
\begin{align}
    & \int_{\R^3} M^\e \mathfrak{R}_{ij}^\e \dd \xi = \int_{\R^3} M^\e \mathcal
    Q^\e_i \dd \xi = 0.   \notag
\end{align}
For either collision operator, we denote the local linearized collision operator around $M^\e$ by
\begin{align*}
    &
\mathcal{L}_M(\Pi_M^{\perp}F^{\varepsilon}) := -2\,\mathcal{N}\big(M^{\varepsilon},\Pi_M^{\perp}F^{\varepsilon}\big).
\end{align*}
For high velocity moments, use the global Maxwellians 
\begin{align}
 &   \mu(\xi) = \frac{1}{\big(2\pi k_B \big)^{3/2}}
\exp\!\Big(-\frac{|\xi|^2}{2k_B}\Big), \qquad \mu_0(\xi) = \frac{1}{\big(2\pi k_B(1-\e \e^{-c_0}) \big)^{3/2}}
\exp\!\Big(-\frac{|\xi|^2}{2k_B(1-\e \e^{-c_0})}\Big)  \label{eq:glob-id}.
\end{align}
Here $c_0\ll 1$ will be specified in \eqref{eq:boot-bd}.
The expansion around the global Maxwellian reads
\begin{equation}\label{eq:dist-id}
\begin{split}
    & F^\e = \mu + \e \sqrt{\mu}f^\e, \text{ for Landau equation,} \\
    & F^\e = \mu_0 + \e \sqrt{\mu_0}f^\e, \ \  F^\e = M^\e + \e \sqrt{\mu_0}g^\e, \text{ for the non-cutoff Boltzmann equation}. 
\end{split}
\end{equation}
The linearized Landau operator around $\mu$ and linearized Boltzmann operator around $\mu_0$ are denoted as
\begin{align*}
    & \mathcal{L} f^\e := -2\frac{\mathcal{N}(\mu,\sqrt{\mu}f^\e)}{\sqrt{\mu}}, \ \  \mathcal{L}_0 f^\e := -2 \frac{\mathcal{N}(\mu_0,\sqrt{\mu_0} f^\e)}{\sqrt{\mu_0}}.
\end{align*}
For the non-cutoff Boltzmann equation, define also the difference operator
\begin{equation}\label{eq:lin-def}
 \mathcal L_d^0 h
 :=-\frac{2\mathcal N(M^\e-\mu_0,\sqrt{\mu_0}h)}{\sqrt{\mu_0}}
 =-\frac{Q(M^\e-\mu_0,\sqrt{\mu_0}h)
          +Q(\sqrt{\mu_0}h,M^\e-\mu_0)}{\sqrt{\mu_0}}.
\end{equation}
Thus, pointwise in $(t,x)$,
\begin{equation}\label{eq:coll-id}
 -\frac{2\mathcal N(M^\e,\sqrt{\mu_0}h)}{\sqrt{\mu_0}}
 =\mathcal L_0h+\mathcal L_d^0h.
\end{equation}
The nonlinear Landau and Boltzmann operators are the bilinear maps
\begin{equation*}
 \Gamma(h,k):=\frac{Q(\sqrt\mu\,h,\sqrt\mu\,k)}{\sqrt\mu},
 \qquad
 \Gamma_0(h,k):=\frac{Q(\sqrt{\mu_0}\,h,\sqrt{\mu_0}\,k)}{\sqrt{\mu_0}}.
\end{equation*}
Let $\Theta_0^\e:=1-\e^{1-c_0}$. The null spaces are
\begin{align*}
 \operatorname{Null}(\mathcal L)
 &=\operatorname{span}\!\left\{
 \sqrt\mu,\ \xi_1\sqrt\mu,\ \xi_2\sqrt\mu,\ \xi_3\sqrt\mu,
 (|\xi|^2-3k_B)\sqrt\mu\right\},\\
 \operatorname{Null}(\mathcal L_0)
 &=\operatorname{span}\!\left\{
 \sqrt{\mu_0},\ \xi_1\sqrt{\mu_0},\ \xi_2\sqrt{\mu_0},\
 \xi_3\sqrt{\mu_0},
 (|\xi|^2-3k_B\Theta_0^\e)\sqrt{\mu_0}\right\}.
\end{align*}
Let $\Pi$ and $\Pi_0$ denote the $L^2_\xi$-orthogonal projections onto
these spaces:
\begin{equation}\label{eq:pr-conv}
 \Pi:L^2_\xi\to\operatorname{Null}(\mathcal L),
 \qquad
 \Pi_0:L^2_\xi\to\operatorname{Null}(\mathcal L_0).
\end{equation}
Because $M^\e$ has the same mass, momentum, and energy moments as $F^\e$,
\eqref{eq:dist-id} gives the exact zero-order orthogonality
\begin{equation}\label{eq:pr-id}
 \Pi_0 g^\e=0.
\end{equation}
The exact-local decomposition $F^\e=M^\e+\Pi_M^\perp F^\e$ is common to
both collision models and carries the macroscopic closure.  Only the global
high-velocity representation is model-dependent.  For Landau it is taken
around the unshifted equilibrium $\mu$.  For non-cutoff Boltzmann the global
reference $\mu_0$ is shifted to a lower temperature.  Exact moment matching
then gives the zero-order orthogonality \eqref{eq:pr-id}.  The homogeneous
temperature mismatch is not perturbative after division by $\e$, so the
zero spatial order is measured by $g^\e$; for $|\alpha|\ge1$,
$\partial_x^\alpha\mu_0=0$ removes that mismatch and $f^\e$ is the natural
unknown.  The quantitative comparison between $M^\e$ and $\mu_0$ is proved
after the scale and bootstrap assumptions below.
With the projections defined in \eqref{eq:pr-conv}, the linearized operators satisfy the coercive estimates
\begin{align*}
    &    \int_{\R^3} f^\e\mathcal{L}f^\e  \dd \xi \gtrsim \Vert (I-\Pi)f^\e\Vert_{L^2_{\mathbf{L}}}^2, \qquad \int_{\mathbb{R}^3} f^\e \mathcal{L}_0 f^\e \dd \xi  \gtrsim \Vert (I-\Pi_0)f^\e\Vert_{L^2_{\mathbf{B}}}^2.
\end{align*}
Define the collision dissipation norms as follows:
\begin{align*}
   \Vert f^\e\Vert_{L^2_\Lb}^2 : = &   \int_{\R^3} \sigma^{ij}(\xi) \p_{\xi_i} f^\e (t,x,\xi) \p_{\xi_j} f^\e(t,x,\xi) \dd \xi + \frac{1}{4k_B^2} \int_{\R^3} \sigma^{ij}(\xi) \xi_i \xi_j |f^\e(t,x,\xi)|^2 \dd \xi, 
\end{align*}
where
\begin{align*}
    & \sigma^{ij}(\xi) := \int_{\R^3} \Phi^{ij}(\xi-\xi')\mu(\xi') \dd \xi'.
\end{align*}
\begin{align*}
 \Vert f^\e\Vert_{L^2_{\Bb}}^2 :=   &   \iint_{\R^6\times \mathbb S^2} B(\xi-\xi_*,\mathbf n)\mu_0(\xi_*) (f^\e(\xi')-f^\e(\xi))^2 \dd \xi_* \dd \xi \dd \mathbf n \\
    & + \iint_{\R^6\times \mathbb S^2} B(\xi-\xi_*,\mathbf n) |f^\e(\xi_*)|^2 (\sqrt{\mu_0(\xi')}-\sqrt{\mu_0(\xi)})^2 \dd \xi \dd \xi_* \dd \mathbf n.
\end{align*}
The dissipation norms satisfy the following coercive estimates:
\begin{equation}\label{eq:xi-bd}
 \begin{split}
     & C_1 [\Vert f^\e\Vert_{H^1_{-3}(\R_\xi^3)}^2 +\Vert f^\e\Vert_{L^2_{-1}(\R_\xi^3)}^2] \leq  \Vert f^\e\Vert_{L^2_{\Lb}}^2 \leq C_2 \Vert f^\e\Vert_{H^1_{-1}(\R_\xi^3)}^2, \\
    &C_1 [\Vert f^\e\Vert_{H^s_{\gamma}(\R_\xi^3)}^2 + \Vert f^\e\Vert_{L^2_{2s+\gamma}(\R_\xi^3)}^2] \leq  \Vert  f^\e\Vert_{L^2_{\Bb}}^2 \leq C_2 \Vert f^\e\Vert_{H^s_{2s+\gamma}(\R_\xi^3)}^2.
 \end{split}    
\end{equation}
Fix the polynomial velocity exponent as follows: $\ell\ge0$ in the
Landau case, while in the non-cutoff Boltzmann case we require
$\ell\ge-(\gamma+2s)$.  Since the collision range above has $\gamma+2s>-3/2$, the fixed choice
$\ell=2$ is admissible uniformly throughout that range.  Choose $q>0$
sufficiently small, in particular $q<1/8$.  Set $a_L:=k_B^{1/2}$ and
$a_\e:=(k_B\Theta_0^\e)^{1/2}$.  We use the velocity weights in the corresponding thermal variables,
\begin{align*}
 w(\xi):=\left\langle\frac{\xi}{a_L}\right\rangle^\ell
 e^{q\left\langle\xi/a_L\right\rangle^2}
 & \   \   \  \ \text{for Landau},\\
 w_0(\xi)=w_0^\e(\xi)
 :=\left\langle\frac{\xi}{a_\e}\right\rangle^\ell
 e^{q\left\langle\xi/a_\e\right\rangle}
 & \  \ \  \ \text{for Boltzmann}.
\end{align*}
We suppress the superscript $\e$ on $w_0$.
\subsection{Norms and energy functionals}
Throughout, $\delta>0$ is sufficiently small and independent of $\e$.
The notation $a+$ and $a-$ means $a+\eta$ and $a-\eta$, respectively,
for an arbitrarily small fixed $\eta>0$; the value of $\eta$ may change between estimates.
The energy multipliers are fixed with $0<\delta_2\ll\delta_1\ll1$.
In the kinetic estimates below, $o(1)=o_{\e\to0}(1)$ always means a
quantity tending to zero uniformly for $0\le t\le T_*$ under the bootstrap
assumption \eqref{eq:tot-bd} and the scale relations \eqref{eq:choice-bd}--\eqref{eq:choice-boot}.
We write $L^2_{x,\Lb}=L^2_xL^2_{\Lb}$ and
$L^2_{x,\Bb}=L^2_xL^2_{\Bb}$, and set
\begin{align*}
 \langle\xi\rangle&:=(1+|\xi|^2)^{1/2},\qquad
 \|f\|_{L^2_\gamma}:=\|\langle\xi\rangle^{\gamma/2}f\|_{L^2_\xi},\\
 \|f\|_{H^s_\gamma}^2&:=\|f\|_{L^2_\gamma}^2
 +\iint_{\R^3\times\R^3}\langle\xi\rangle^\gamma
 \frac{|f(\xi)-f(\xi^*)|^2}{|\xi-\xi^*|^{3+2s}}
 \chi_{|\xi-\xi^*|\le1}\,d\xi^*d\xi, \ \ s\in (0,1), \\
  \Vert f\Vert_{H^1_\gamma}^2 &:= \Vert f\Vert_{L^2_\gamma}^2 + \sum_{j=1}^3 \Vert \p_{\xi_j} f\Vert_{L^2_\gamma}^2.
\end{align*}
\paragraph{Energy and dissipation functionals.}
We work on $\mathbb R^3$ and recall the low-Mach parametrization
\begin{equation}\label{eq:param-id}
 \mathrm P^\e=e^{\e\rho^\e},\qquad
 \mathrm U^\e=\e u^\e,\qquad
 \Theta^\e=e^{\e\theta^\e},
\end{equation}
used throughout the kinetic estimates.
For $\alpha,\beta\in\mathbb N_0^3$, write
$\partial^\alpha:=\partial_x^\alpha$, $\partial_\beta:=\partial_\xi^\beta$, and
$\partial_\beta^\alpha:=\partial_x^\alpha\partial_\xi^\beta$.
For $N\geq0$, we define the lower-order energy and dissipation as follows:
\begin{align}\label{eq:diss-def}
\bega
\mathcal{E}_{M}^{N}(F^\e(t))=
\mathcal{E}_{M}^{N}(t) &:= \bigg(\|\rho^{\e}(t)\|_{L^2_x}^2+\|u^{\e}(t)\|_{L^2_x}^2 +\|\theta^{\e}(t)\|_{L^2_x}^2\bigg) \\
&+ \sum_{1\leq|\alpha|\leq N} \bigg(\|\p^{\alpha}\rho^{\e}(t)\|_{L^2_x}^2+\|\p^{\alpha}u^{\e}(t)\|_{L^2_x}^2+\|\p^{\alpha}\theta^{\e}(t)\|_{L^2_x}^2\bigg), \cr
\mathcal{E}_G^{N}(F^\e (t)) =\mathcal{E}_G^{N}(t)&:= \sum_{0\leq|\alpha|\leq N}\frac{1}{\e^2} \iint_{\R^3\times\R^3}|\p^{\alpha}\Pi_M^{\perp}F^{\e}|^2|M^{\e}|^{-1}\dd \xi \dd x , \\ 
\mathcal{D}_G^{N}(F^\e(t))=
\mathcal{D}_G^{N}(t)&:= \begin{cases}
    \sum_{0\leq|\alpha|\leq N} \frac{1}{\kappa\e^4} \big\Vert (\p^\alpha \Pi_M^{\perp} F^\e) |M^\e|^{-1/2} \big\Vert_{L^2_{x,\Lb}}^2 , \text{ for Landau}\\
    \sum_{0\leq|\alpha|\leq N} \frac{1}{\kappa\e^4} \big\Vert (\p^\alpha \Pi_M^{\perp} F^\e) |M^\e|^{-1/2} \big\Vert_{L^2_{x,\Bb}}^2, \text{ for Boltzmann}
\end{cases} .
\enda
\end{align}
For the top spatial order we use the intrinsic full-distribution energy
\begin{align}\label{eq:top-def}
\bega
\mathcal H_{\rm top}(t)
&:=\sum_{|\alpha|=N+1}\frac{\kappa}{2\e^2}
\iint_{\R^3\times\R^3}\frac{|\p^\alpha F^\e(t,x,\xi)|^2}{M^\e(t,x,\xi)}\,\dd\xi\dd x,\cr
\mathcal{D}_{top}^{N}(F^\e (t)) =  \mathcal{D}_{top}^{N}(t)&:=\begin{cases}
    \sum_{|\alpha|=N+1} \frac{1}{\e^4}\big\Vert (\Pi_M^{\perp}\p^\alpha F^\e) |M^\e|^{-1/2} \big\Vert_{L^2_{x,\Lb}}^2 , \text{ for Landau}\\
     \sum_{|\alpha|=N+1} \frac{1}{\e^4}\big\Vert (\Pi_M^{\perp}\p^\alpha F^\e) |M^\e|^{-1/2} \big\Vert_{L^2_{x,\Bb}}^2, \text{ for Boltzmann}.
\end{cases}
\enda
\end{align}
The full energy and dissipation are
\begin{align}\label{eq:kin-def}
\bega
\mathcal{E}^{N}(F^\e(t)) =\mathcal{E}^{N}(t) &:= \mathcal{E}_M^{N}(t) + \mathcal{E}_G^{N}(t) + 2\mathcal H_{\rm top}(t), \\
\mathcal{D}^{N}(F^\e(t)) = \mathcal{D}^{N}(t) &:= \mathcal{D}_G^{N}(t) + \mathcal{D}_{top}^{N}(t).
\enda
\end{align}
The high-moment functionals differ by the collision operator.
\begin{itemize}
    \item  Landau equation.
\begin{align}\label{eq:land-id}
\bega
 &\mathcal{E}_H^{N}(F^\e(t)) = \mathcal{E}_H^{N}(t)  \\
 &:  = \sum_{0\le|\alpha|\le N}  [\kappa^{2c} \Vert \p^\alpha f^\e\Vert_{L^2_{x,\xi}}^2 + \kappa^{2c} \Vert w \p^\alpha (I-\Pi)f^\e\Vert_{L^2_{x,\xi}}^2]  + \sum_{\substack{|\beta|=1\\|\alpha|\le\lfloor N/2\rfloor}}   \kappa^{2c} \Vert w^{\frac{1}{2}} \p^\alpha_\beta (I-\Pi)f^\e\Vert_{L^2_{x,\xi}}^2 \\
 &+   \kappa^{2c+1} \e^2 \Vert w\p^{N+1} f^\e\Vert_{L^2_{x,\xi}}^2 + \kappa^{2c+1} \Vert \p^{N+1} f^\e\Vert_{L^2_{x,\xi}}^2 .
\enda
\end{align}
\begin{align*}
  &\mathcal{D}_H^{N}(F^\e(t)) = \mathcal{D}_H^{N}(t)  \\
  &  := \sum_{0\le|\alpha|\le N}  \frac{\kappa^{2c-1}}{\e^2}[\Vert (I-\Pi)\p^\alpha f^\e\Vert_{L^2_{x,\Lb}}^2+  \Vert w(I-\Pi)\p^\alpha f^\e\Vert_{L^2_{x,\Lb}}^2]   +  \sum_{|\beta|=1, |\alpha|\leq \frac{N}{2}}  \frac{\kappa^{2c-1}}{\e^2}  \Vert w^{\frac{1}{2}}\p^\alpha_\beta(I-\Pi) f^\e\Vert_{L^2_{x,\Lb}}^2  \\
  &+  \kappa^{2c} \Vert w \p^{N+1} f^\e\Vert_{L^2_{x,\Lb}}^2 + \frac{\kappa^{2c}}{\e^2}\Vert (I-\Pi)\p^{N+1} f^\e\Vert_{L^2_{x,\Lb}}^2.
\end{align*}
\end{itemize}
\begin{itemize}
    \item Boltzmann equation
\end{itemize}
\begin{align}\label{eq:xi-def}
 \mathcal E_{H,1}^N
 &:={\kappa^{2c}}\Big(
 \|g^\e\|_{L^2_{x,\xi}}^2
 +\|w_0g^\e\|_{L^2_{x,\xi}}^2
 +\sum_{|\beta|=1}
  \|w_0^{1/2}\partial_\beta g^\e\|_{L^2_{x,\xi}}^2\Big),\notag\\
 \mathcal E_{H,2}^N
 &:=\sum_{1\le|\alpha|\le N}\kappa^{2c}
 \Big(\|\partial^\alpha f^\e\|_{L^2_{x,\xi}}^2
 +\|w_0\partial^\alpha(I-\Pi_0)f^\e\|_{L^2_{x,\xi}}^2\Big)\notag\\
 &\quad+
 \sum_{\substack{|\beta|=1\\1\le|\alpha|\le\lfloor N/2\rfloor}}
 \kappa^{2c}\|w_0^{1/2}\partial_\beta^\alpha(I-\Pi_0)f^\e\|_{L^2_{x,\xi}}^2
 +\kappa^{2c+1}\e^2\|w_0\partial^{N+1}f^\e\|_{L^2_{x,\xi}}^2
 +\kappa^{2c+1}\|\partial^{N+1}f^\e\|_{L^2_{x,\xi}}^2,\notag\\
 \mathcal E_H^N&:=\mathcal E_{H,1}^N+\mathcal E_{H,2}^N.
\end{align}
\begin{align}\label{eq:hn-def}
 \mathcal D_{H,1}^N
 &:=\frac{\kappa^{2c-1}}{\e^2}\Big(
 \|g^\e\|_{L^2_{x,\Bb}}^2
 +\|w_0g^\e\|_{L^2_{x,\Bb}}^2
 +\sum_{|\beta|\le1}
  \|w_0^{1/2}\partial_\beta g^\e\|_{L^2_{x,\Bb}}^2\Big),\notag\\
 \mathcal D_{H,2}^N
 &:=\sum_{1\le|\alpha|\le N}\frac{\kappa^{2c-1}}{\e^2}
 \left(\|(I-\Pi_0)\partial^\alpha f^\e\|_{L^2_{x,\Bb}}^2
 +\|w_0(I-\Pi_0)\partial^\alpha f^\e\|_{L^2_{x,\Bb}}^2\right)\notag\\
 &\quad+
 \sum_{\substack{|\beta|=1\\1\le|\alpha|\le\lfloor N/2\rfloor}}
 \frac{\kappa^{2c-1}}{\e^2}
 \|w_0^{1/2}\partial_\beta^\alpha(I-\Pi_0)f^\e\|_{L^2_{x,\Bb}}^2
 +\kappa^{2c}\|w_0\partial^{N+1}f^\e\|_{L^2_{x,\Bb}}^2
 +\frac{\kappa^{2c}}{\e^2}
 \|(I-\Pi_0)\partial^{N+1}f^\e\|_{L^2_{x,\Bb}}^2,\notag\\
 \mathcal D_H^N&:=\mathcal D_{H,1}^N+\mathcal D_{H,2}^N.
\end{align}
The total energy and dissipation are defined by
\begin{align}
 \mathcal{E}_{tot}^{N}(F^\e(t)) = \mathcal{E}_{tot}^{N}(t)  &:= \mathcal{E}^{N}(t) + \mathcal{E}_H^{N}(t), \notag\\
 \mathcal{D}_{tot}^{N}(F^\e(t)) = \mathcal D_{tot}^{N}(t) & := \mathcal{D}^{N}(t) + \mathcal{D}_H^{N}(t). \label{eq:tot-def}
\end{align}
\subsection{Scale choice and bootstrap}
\begin{assumption}[Bootstrap Assumption]
For $\mathfrak{y} \ll 1, N\geq 4$, we impose the following bootstrap assumption.
\begin{align}\label{eq:tot-bd}
\sup_{0\leq t \leq T_*}
\big(\e^{\mathfrak{y}}+\e\kappa^{-\frac{1}{2}}\big)\big(1+t\big)
\bigg(\mathcal{E}_{tot}^{N}(t)
+\int_0^t\mathcal{D}_{tot}^{N}(s)\,ds\bigg)
\lesssim 1.
\end{align}
We choose $\e \ll 1, \kappa \ll 1$ such that
\begin{align}
    &  \frac{\e}{\kappa^c} \sim \e^{\mathfrak{y}},     \ \ \frac{1}{2}<c \leq 1 , \ \ \e \kappa^{-\frac{1}{2}} = \frac{\e}{\kappa^{c}} \kappa^{c-\frac{1}{2}} < \e^{\mathfrak{y}}.  \label{eq:choice-bd}
\end{align}
We further let $|c-\frac{1}{2}| \ll 1$, and let $c_0\ll 1$ in \eqref{eq:glob-id} satisfy
\begin{equation}\label{eq:boot-bd}
 c_0=\mathfrak y,
 \qquad \frac{2c-1}{c}(1-c_0)\ge 6c_0.
\end{equation}
We also assume
\begin{equation}\label{eq:choice-boot}
 \e^{c_0}T_*\lesssim1.
\end{equation}
Such a choice of $c_0$ and the bootstrap assumption \eqref{eq:tot-bd} guarantee
\begin{align}
    &   \Theta^\e = e^{\e \theta^\e} \geq   e^{-C \e \e^{-\frac{c_0}{2}}}  \geq 1- \e \e^{-c_0}.   \label{eq:assm-bd}
\end{align}
We impose the same scale relation on both collision operators.
\end{assumption}
\begin{lemma}[Elementary consequences of the scale choice]\label{lem:scalecons}
Assume \eqref{eq:choice-bd}--\eqref{eq:boot-bd} and set $c_0=\mathfrak y$.  Then, after reducing $\e_0>0$ if necessary,
\begin{align}
 \kappa^c \asymp \e^{1-c_0},
 & \ \ \ \kappa\asymp \e^{(1-c_0)/c},
 \label{eq:kap-choice}\\
 \e\kappa^{-1/2}\lesssim \e^{c_0},
 & \ \ \ \kappa^2\lesssim \e,
 \label{eq:kin-bd}\\
 \kappa^{2c-1}\lesssim \e^{6c_0},
 &\ \ \ \frac{\e^2}{\kappa}\lesssim \e^{8c_0}.
 \label{eq:maxw-bd}
\end{align}
In particular, every quantity on the left-hand side of
\eqref{eq:kin-bd}--\eqref{eq:maxw-bd} tends to zero.
\end{lemma}

The proof is recorded in Appendix~\ref{sec:auxana}.

\subsection{Long-time kinetic estimate}\label{sec:kinlimit}
Let
\[
 O(2)_z:=\{\operatorname{diag}(G_h,1):G_h\in O(2)\}.
\]
A kinetic function is axisymmetric no-swirl if
$F(t,Gx,G\xi)=F(t,x,\xi)$ for every $G\in O(2)_z$.
A vector field is $O(2)_z$-equivariant if $A(Gx)=GA(x)$; equivariance under
reflections implies the no-swirl form $A=A^re_r+A^ze_z$.
Assume throughout that $F_0^\e$ has this symmetry.
Recall $v^\e=\mathbb Pu^\e$.  By axisymmetry and absence of swirl,
\begin{align}
 v^\e&=v^{\e,r}e_r+v^{\e,z}e_z,\label{eq:pp-id}\\
 \omega^\e&=\nabla_x\times u^\e=\nabla_x\times\mathbb Pu^\e
       =\omega^{\e,\vartheta}e_\vartheta, \qquad \varpi^\e=\frac{\omega^{\e,\vartheta}}r. \notag 
\end{align}
We use the standard Lorentz space $L^{3,1}(\R^3)$ and the homogeneous
Besov spaces $\dot B_{p,q}^s$ with the standard homogeneous Littlewood--Paley convention.  In particular,
\[
 \|f\|_{L^{3,1}}:=\int_0^\infty
 |\{x:|f(x)|>\lambda\}|^{1/3}\,d\lambda.
\]
Fix $\delta>0$ so small that
\begin{equation}\label{eq:small-bd}
 \frac{\delta}{2(1+\delta)}
 \left(1+\frac{1-\mathfrak y}{2c}\right)<\frac1{16}.
\end{equation}
For $\e>0$ sufficiently small, set
\begin{equation}\label{eq:small-def}
 \Lambda_1^\e:=\log\frac1\e,
 \qquad
 \Lambda_{j+1}^\e:=\log\Lambda_j^\e\quad(1\le j\le4),
 \qquad
 T^\e:=(\Lambda_5^\e)^{1/32}.
\end{equation}
Then $T^\e\to\infty$, every power of \(T^\varepsilon\) is dominated by every negative power of $\e$.
\begin{theorem}[Long-time kinetic estimate]\label{thm:longtime}
Consider either the Landau or the non-cutoff Boltzmann equation. Let $N\ge5$. Assume the scale relations \eqref{eq:choice-bd}--\eqref{eq:boot-bd}, and let $T^\e$ be defined by \eqref{eq:small-def}. For smooth nonnegative axisymmetric no-swirl initial data, there exists $\e_0>0$ such that for $0<\e\le\e_0$, if
\begin{align}
    & \mathcal{E}_{tot}^{N}(0) + \Vert \varpi_0^\e\Vert^2_{L^{3,1}_x}  \leq \Lambda_5^\e,  \label{eq:init-bd}
\end{align}
then the corresponding strong solution exists uniquely on $[0,T^\e]$. Moreover, for every $0\leq t\leq T^\e$,
\begin{align}
    & \mathcal{E}_{tot}^{N}(t) + \int_0^t \mathcal{D}_{tot}^{N}(s) \dd s \leq \Big|  \log \frac{1}{\e} \Big|^{\frac{15}{16}}.  \label{eq:ac-bd}
\end{align}
For every $0<T\le T^\e$ and $s\in[0,N-1)$, the acoustic components satisfy
\begin{align}\label{eq:bega-rho}
\bega
\Big\|\Big((\rho^{\e}+\theta^{\e}),\mathbb{P}^{\perp}u^{\e}\Big)&\Big\|_{L^\mathfrak{r}_T\dot{B}_{p,1}^{s+3(\frac{1}{p}-\frac{1}{2})+\frac{1}{\mathfrak{r}}}} \leq C \e^{\frac{1}{\mathfrak{r}}}\Big\|\Big((\rho^{\e}_0+\theta^{\e}_0),\mathbb{P}^{\perp}u^{\e}_0\Big)\Big\|_{\dot{B}_{2,1}^s}  + C\e^{\frac{1}{\mathfrak{r}}}\int_0^{T}\Big(\mathcal{E}_{tot}^{N}(s) +\kappa^{\frac{1}{2}} \sqrt{\mathcal{D}_{tot}^{N}(s)}\Big) ds,
\enda
\end{align}
where $2\leq p\leq\infty$, $2\leq \mathfrak{r}\leq\infty$, and
$\frac{1}{\mathfrak{r}}\leq\frac{1}{2}-\frac{1}{p}$, with
$(\mathfrak{r},p)\neq(2,\infty)$. The Besov convention is the standard homogeneous Littlewood--Paley convention.
\end{theorem}
The proof of Theorem \ref{thm:longtime} is given in Section \ref{sec:closure}.

\subsection{Main kinetic-to-Euler conclusions}
From this point on, all families indexed by $\e$ are understood on the
small-parameter range $0<\e\le\e_0$ supplied by Theorem~\ref{thm:longtime};
for brevity, suprema written as $\sup_{\e>0}$ are taken over this range.
\begin{definition}[Renormalized axisymmetric Euler pair]\label{def:eulerpair}
Fix $1\le p<\infty$. Following
\cite{diperna1989ordinary}, a pair $(u^E,\varpi^E)$ is a
\emph{renormalized axisymmetric Euler pair in $L^p$} on $[0,T]$ with initial
relative vorticity $\varpi_0^E$ if:
\begin{itemize}
\item $\varpi^E\in L^\infty(0,T;L^p(\R^3))$ and
$u^E\in L^\infty(0,T;L^2(\R^3;\R^3))$ is axisymmetric without swirl, with
\begin{equation*}
 \nabla_x\cdot u^E=0,
 \qquad
 \nabla_x\times u^E=r\varpi^E e_\vartheta
 \quad\text{in }\mathcal D'((0,T)\times\R^3);
\end{equation*}
\item for every $\psi\in C_c^\infty([0,T)\times\R^3)$ and every $\beta\in C^1(\R)\cap L^\infty(\R)$ vanishing in a neighborhood of $0$,
\[
\int_0^T\int_{\R^3}
\beta(\varpi^E)
\bigl(\partial_t\psi+u^E\cdot\nabla_x\psi\bigr)
\,dx\,dt
+
\int_{\R^3}
\beta(\varpi^E_0)\psi(0,x)
\,dx
=0.
\]
\end{itemize}
\end{definition}
The transported relative-vorticity and velocity components $\varpi_A^\e$ and $u_A^\e$ appearing below are defined in Subsection~\ref{sec:transport}.
\begin{theorem}[Rough kinetic-to-Euler limit for all $p>1$]\label{thm:roughlimit}
Consider either collision operator under the hypotheses of Theorem~\ref{thm:longtime}.
Assume in addition that, for some $1<p<\infty$,
\begin{align}
&\sup_{\e>0}\|\varpi_0^{\e}\|_{L^p\cap L^1(\R^3)}<\infty,
\label{eq:lp-lone}\\
&\sup_{\e>0}\|\mathbb{P}u_0^\e\|_{L^2(\R^3)}<\infty,
\label{eq:rough-bd}
\end{align}
and that, as $\e\to0$,
\begin{align}
\varpi_0^\e\to\varpi_0^E
&\text{ in }L^p(\R^3)\cap L^1(\R^3),\label{eq:kin-conv}\\
\mathfrak{s}_0^\e:=\tfrac32\theta_0^\e-\rho_0^\e
\to
\mathfrak{s}_0^E:= \tfrac32\theta_0^E-\rho_0^E
&\text{ in }L^2(\R^3).
\label{eq:mic-conv}
\end{align}
Then the microscopic defect $G^\e=F^\e-M^\e$ satisfies
\begin{equation}\label{eq:mic-maxw}
 \sup_{0\le t\le T^\e}
 \left\|\frac{G^\e(t)}{\sqrt{M^\e(t)}}\right\|_{L^2_{x,\xi}}
 \lesssim \e\left|\log\frac1\e\right|^{15/32}
 \to0.
\end{equation}
Moreover, for every finite $T>0$, up to a subsequence, the following statements hold on $[0,T]$:
\begin{enumerate}
\item[(1)] The transported relative vorticity and the incompressible velocity
satisfy
\begin{align}
 \varpi_A^\e\to\varpi^E
 &\text{ strongly in }C([0,T];L^1(\R^3))\cap C([0,T];L^p(\R^3)),\label{eq:vel-conv}\\
 u_A^\e\to u^E
 &\text{ strongly in }C([0,T];L^2_{\rm loc}(\R^3)),\label{eq:loc-conv}\\
 \mathbb Pu^\e\to u^E
 &\text{ strongly in }C([0,T];L^2_{\rm loc}(\R^3)).\label{eq:pp-conv}
\end{align}
If $p\geq 6/5$, then the local in $\R^3$ convergence can be upgraded to local in the radial variable: for $\mathfrak{C}_R=\{x\in \R^3:r(x)<R\}$ with any finite $R>0$
\begin{align}
 u_A^\e\to u^E &\text{ strongly in }C([0,T];L^2(\mathfrak{C}_R)),\label{eq:loc-conv_r}\\
 \mathbb Pu^\e\to u^E
 &\text{ strongly in }C([0,T];L^2(\mathfrak{C}_R)).\label{eq:pp-conv_r}
\end{align}
\item[(2)] The pair $(u^E,\varpi^E)$ is a renormalized axisymmetric Euler pair
in $L^p$ in the sense of Definition~\ref{def:eulerpair}, with
\[
 \varpi^E\in L^\infty(0,T;L^1(\R^3)\cap L^p(\R^3)),
 \qquad
 u^E\in L^\infty(0,T;W^{1,p}_{\rm loc}(\R^3)).
\]
In particular, the limit in (1) is identified by
$\nabla\cdot u^E=0$ and $\nabla\times u^E=r\varpi^Ee_\vartheta$ in
$\mathcal D'((0,T)\times\R^3)$.
\item[(3)] The thermodynamic variable
$\mathfrak{s}^{\e}:=\frac{3}{2}\theta^{\e}-\rho^{\e}$ satisfies
\begin{align}\label{eq:therm-conv}
 \mathfrak{s}^\e
 \to \mathfrak{s}^E
 \qquad\text{in }C([0,T];L^2(\R^3)),
\end{align}
where $\mathfrak{s}^E$ is the unique renormalized solution of
\[
 \partial_t\mathfrak{s}^E+u^E\cdot\nabla_x\mathfrak{s}^E=0,
 \qquad
 \mathfrak{s}^E(0)=\mathfrak{s}_0^E.
\]

\item[(4)] Fix $2<\bar p\le\infty$, $1\le \mathfrak{r}<\infty$, and $0<s<4$ such that
$1/\mathfrak{r}\le1/2-1/\bar p$ and $(\mathfrak{r},\bar p)\ne(2,\infty)$.  For every sufficiently
small $\delta>0$,
\begin{align}\label{eq:pq-rho}
 \|\mathbb P^\perp u^\e\|_{L^\mathfrak{r}(0,T;\dot B_{\bar p,1}^{s+3(1/\bar p-1/2)+1/\mathfrak{r}})}
 +\|\rho^\e+\theta^\e\|_{L^\mathfrak{r}(0,T;\dot B_{\bar p,1}^{s+3(1/\bar p-1/2)+1/\mathfrak{r}})}
 \lesssim_{T,\delta,s,\mathfrak{r},\bar p}\e^{1/\mathfrak{r}-\delta}.
\end{align}
After decreasing $\delta$ so that $\delta<1/\mathfrak{r}$, both acoustic terms
converge to zero in the displayed space.
\end{enumerate}
\end{theorem}
The proof of Theorem \ref{thm:roughlimit} is given in Section \ref{sec:proof-rough-limit}.

\begin{theorem}[Global-energy compactness at the $4/3$ threshold]\label{thm:globenergy}
Assume the hypotheses of Theorem~\ref{thm:roughlimit} for some $1<p<\infty$, and
suppose
\begin{equation}\label{eq:comp-conv}
 \mathbb Pu_0^\e\to u_0^E
 \qquad\text{strongly in }L^2(\R^3).
\end{equation}
Then, for every finite $T>0$,
\begin{equation}\label{eq:pp-fin}
 \sup_{t\in[0,T]}
 \left|\|\mathbb Pu^\e(t)\|_2-\|u_0^E\|_2\right|\to0.
\end{equation}
Let $(u^E,\varpi^E)$ be the limiting Euler fields along any subsequence
furnished by Theorem~\ref{thm:roughlimit}.  If $p\ge4/3$, no sign or spatial moment
assumption is required.  If $1<p<4/3$, assume
\begin{equation}\label{eq:fld-bd}
 \sup_{\e>0}\mathfrak M_2(\varpi_0^\e)<\infty
\end{equation}
and, with $p^\sharp$ defined by \eqref{eq:psharp},
\begin{equation}\label{eq:lp-sharp}
 \sup_{\e>0}
 \min\left\{
 \|(\varpi_0^\e)_+\|_{L^{p^\sharp}},
 \|(\varpi_0^\e)_-\|_{L^{p^\sharp}}
 \right\}<\infty.
\end{equation}
Then, along the same subsequence,
\begin{align}
 u_A^\e\to u^E
 & \text{ strongly in }C([0,T];L^2(\R^3)),\label{eq:eul-conv}\\
 \mathbb Pu^\e\to u^E
 & \text{ strongly in }C([0,T];L^2(\R^3)),\label{eq:rough-conv}
\end{align}
and
\begin{equation}\label{eq:pp-bd}
 \|u^E(t)\|_2=\|u_0^E\|_2,
 \qquad0\le t\le T.
\end{equation}
\end{theorem}
\begin{corollary}[Radial-moment criteria across the $3/2$ and $5/3$ thresholds]\label{cor:radcrit}
Assume the hypotheses of Theorem~\ref{thm:roughlimit} for some $1<p<\infty$ and fix
$T<\infty$.
If $p\ge3/2$ and
\begin{equation}\label{eq:cond-bd}
 \sup_{\e>0}\mathfrak M_1(\varpi_0^\e)<\infty,
\end{equation}
then, without a sign condition or an absolute-impulse assumption,
\begin{equation}\label{eq:lsup-mone}
 \limsup_{\e\to0}\sup_{0\le t\le T}
 \mathfrak M_1(\varpi_A^\e(t))<\infty.
\end{equation}
More generally, for every $1<p<\infty$, suppose the initial absolute
impulse bound \eqref{eq:fld-bd} holds and assume either $p\ge5/3$, or
$1<p<5/3$ together with the mixed sign-component condition \eqref{eq:lp-sharp}.  Then
Corollary~\ref{cor:unifimp} gives
\begin{equation}\label{eq:lsup-mtwo}
 \limsup_{\e\to0}\sup_{0\le t\le T}\mathfrak M_2(\varpi_A^\e(t))<\infty.
\end{equation}
Consequently, \eqref{eq:mominterp}, \eqref{eq:lp-lone}, and conservation of the
$L^1$ norm of $\varpi_A^\e$ imply \eqref{eq:lsup-mone} in this entire conditional
range as well.  In particular, neither $3/2$ nor $5/3$ is a barrier to
moment propagation: $3/2$ is the sign-free first-radial-moment endpoint, while
$5/3$ is the symmetric point at which the mixed $L^{p^\sharp}$ condition
becomes automatic from the ambient $L^p$ control.
\end{corollary}
The proof of Theorem \ref{thm:globenergy} and Corollary \ref{cor:radcrit} are given in Section \ref{sec:proof-conservation}.

\subsection{Local--global Maxwellian bridge}
The bootstrap yields uniform control of the local Maxwellian parameters.  The following bounds are used throughout the collision and energy estimates.
\begin{lemma}\label{lem:macrocoef}
Under \eqref{eq:tot-bd},
\begin{align}
 &\Vert\nabla_x[(\mathrm P^\e)^{\pm1}]\Vert_{H^{N-1}}
 +\Vert\nabla_x[(\Theta^\e)^{\pm1}]\Vert_{H^{N-1}}
 +\Vert(\mathrm P^\e)^{\pm1}\Vert_\infty
 +\Vert(\Theta^\e)^{\pm1}\Vert_\infty
 +\Vert\mathrm U^\e\Vert_{H^N}\lesssim1,\label{eq:hn-en}\\
 &\left\|\frac{\Theta^\e-1}{\e}\right\|_{H^N}
 +\left\|\frac{\nabla_x\Theta^\e}{\e}\right\|_{H^{N-1}}
 \lesssim (\mathcal E_M^N)^{1/2},\label{eq:hn-bd}\\
 &\e\| (\rho^\e,u^\e,\theta^\e)\|_{H^N_x}
 \lesssim \e(\mathcal E_M^N)^{1/2}
 \lesssim \e^{1-c_0/2}=o(1).\label{eq:loc-bd}
\end{align}
\end{lemma}
The proof, based on the bootstrap bound \eqref{eq:tot-bd}, is given in Appendix~\ref{sec:auxana}.

For the non-cutoff Boltzmann equation, the colder global Maxwellian in
\eqref{eq:glob-id} is useful only because its relation to the exact local
Maxwellian is quantitative.  The following lemma isolates that relation.
\begin{lemma}[Local--global Maxwellian bridge]\label{lem:maxbridge}
Assume \eqref{eq:tot-bd} and \eqref{eq:choice-bd}--\eqref{eq:choice-boot}, and consider the
non-cutoff Boltzmann equation.  Set
\begin{equation}\label{eq:bolt-def}
 \delta_\e(t):=\e\bigl(\e^{-c_0}+(\mathcal E^N(t))^{1/2}\bigr).
\end{equation}
Then, after decreasing $\e_0$ if necessary,
\begin{equation}\label{eq:glob-bd}
 \sup_{0\le t\le T_*}\delta_\e(t)
 \lesssim \e^{1-c_0}+\e^{1-c_0/2}
 \lesssim \e^{1-c_0}=o(1),
\end{equation}
and uniformly in $(t,x)$,
\begin{equation}\label{eq:bridge-bd}
 |\mathrm P^\e-1|+|\mathrm U^\e|+|\Theta^\e-\Theta_0^\e|
 \lesssim\delta_\e(t).
\end{equation}
Moreover, the temperature gap is strictly positive:
\begin{equation}\label{eq:pos-bd}
 \Theta^\e-\Theta_0^\e
 \ge \frac12\e^{1-c_0}.
\end{equation}
For every $\xi\in\mathbb R^3$ one has the exact identity
\begin{align}\label{eq:xi-loc}
 &-\frac{|\xi|^2}{2k_B\Theta_0^\e}
 +\frac{|\xi-\mathrm U^\e|^2}{2k_B\Theta^\e} =
 -\frac{\Theta^\e-\Theta_0^\e}
 {2k_B\Theta^\e\Theta_0^\e}
 \left|\xi+
 \frac{\Theta_0^\e}{\Theta^\e-\Theta_0^\e}\mathrm U^\e\right|^2
 +\frac{|\mathrm U^\e|^2}
 {2k_B(\Theta^\e-\Theta_0^\e)}.
\end{align}
Consequently,
\begin{equation}\label{eq:mu-maxw}
 \frac{\mu_0}{M^\e}\lesssim1,
\end{equation}
and, for every fixed polynomial exponent $m\ge0$, there are constants
$C_m<\infty$ and $c_q>0$, independent of $\e,t,x,\xi$, such that
\begin{equation}\label{eq:poly-bd}
 \langle\xi\rangle^m\mu_0(M^\e)^{-1}w_0^{-2}
 \le C_m e^{-c_q\langle\xi\rangle}.
\end{equation}
Finally, set
\[
 A^\e:=\frac{M^\e-\mu_0}{\sqrt{\mu_0}}, \qquad H^\e_\infty:= \frac{\mu-\mu_0}{\e \sqrt{\mu_0}}, \qquad \tilde{H}^\e:= \frac{M^\e-\mu}{\e\sqrt{\mu_0}},
\]
then
\begin{equation}\label{eq:vd-bd}
 \sum_{|\beta|\le2}
 \left|\partial_\xi^\beta A^\e(t,x,\xi)\right|
 \lesssim \delta_\e(t)\mu_0(\xi)^{1/3},
\end{equation}

\begin{align}
    & \|w_0\partial_\beta H_\infty^\e\|_{L_\xi^2\cap L_{\Bb}^2}
 +\|w_0H_\infty^\e\|_{L^2_{2s+\gamma}}
 \lesssim \e^{-1}(1-\Theta_0^\e)=\e^{-c_0}, \ \ |\beta|\leq 1 \label{eq:hinfty-bd}\\
 & \|w_0\widetilde H^\e\|_{H_x^2L_\xi^2}
 +\|w_0\widetilde H^\e\|_{H_x^2L_{\Bb}^2}
 +\|w_0\widetilde H^\e\|_{H_x^2L^2_{2s+\gamma}} +
 \sum_{|\beta|=1}
 \|w_0^{1/2}\partial_\beta\widetilde H^\e\|_{H_x^2L_\xi^2}
 \lesssim (\mathcal E^N)^{1/2}. \label{eq:xi-gam}
\end{align}

\end{lemma}
The proof applies Lemma~\ref{lem:macrocoef} to the Maxwellians and is given in Appendix~\ref{sec:auxana}.

\begin{lemma}[Positive-order fluctuation transfer]\label{lem:fluctxfer}
Under the hypotheses of Lemma~\ref{lem:maxbridge}, for $1\le|\alpha|\le N$,
\begin{align}
 \|w_0\partial^\alpha g^\e\|_{L^2_{x,\xi}}
 +\|w_0\partial^\alpha g^\e\|_{L^2_{x,\Bb}}
 +\|w_0\partial^\alpha g^\e\|_{L^2_xL^2_{2s+\gamma}} 
 &\lesssim
 (\mathcal E^N)^{1/2}
 +\|w_0\partial^\alpha f^\e\|_{L^2_{x,\xi}}
 +\|w_0\partial^\alpha f^\e\|_{L^2_{x,\Bb}}
 \notag\\
 &\quad+
 \|w_0\partial^\alpha f^\e\|_{L^2_xL^2_{2s+\gamma}},
 \label{eq:der-alp}\\
 \|w_0^{1/2}\partial_\beta^\alpha g^\e\|_{L^2_{x,\xi}}  
 &\lesssim
 (\mathcal E^N)^{1/2}
 +\|w_0^{1/2}\partial_\beta^\alpha f^\e\|_{L^2_{x,\xi}},
 \qquad |\beta|=1, \notag
 \\
 \|\partial^\alpha g^\e\|_{L^2_{x,\xi}}
 +\|\partial^\alpha g^\e\|_{L^2_{x,\Bb}}
 &\lesssim
 (\mathcal E^N)^{1/2}
 +\|\partial^\alpha f^\e\|_{L^2_{x,\xi}}
 +\|\partial^\alpha f^\e\|_{L^2_{x,\Bb}}.
 \label{eq:der-alp-bd}
\end{align}
\end{lemma}
The proof uses $g^\e=f^\e+(\mu_0-M^\e)/(\e\sqrt{\mu_0})$ together with Lemma~\ref{lem:maxbridge}; it is given in Appendix~\ref{sec:auxana}.

\subsection{A priori energy theorem}
\begin{theorem}[Kinetic energy estimate]\label{thm:kinenergy}
Let $N\ge4$. Consider either the Coulomb Landau
equation or the non-cutoff Boltzmann equation.  Assume
\eqref{eq:tot-bd}, \eqref{eq:choice-bd}, \eqref{eq:boot-bd}, and \eqref{eq:choice-boot}.  Then, for
$0\le t\le T_*$,
\begin{align}\label{eq:diss-bd}
 \mathcal E_{tot}^N(t)+\int_0^t\mathcal D_{tot}^N(s)\,ds
 &\lesssim \mathcal E_{tot}^N(0)
 +\int_0^t
 \bigl(1+\|\nabla_x(\rho^\e,u^\e,\theta^\e)(s)\|_{L_x^\infty}^2\bigr)
 \mathcal E_{tot}^N(s)\,ds +C.
\end{align}
\end{theorem}
The local-Maxwellian closure follows the exact moment-matching and cancellation identities developed in Bae--Kim \cite{bae2026quantitative}.  Moment matching yields the macroscopic balance laws and the stress--heat-flux cancellation in the combined energy identity, while the $(M^\e)^{-1}$ full-distribution energy controls the tangent macroscopic derivatives together with the microscopic projection at the highest spatial order; the corresponding top-order estimate appears in \cite[Lemma~2.13]{bae2026quantitative}.  Here these estimates are combined with the three-dimensional axisymmetric finite-energy bounds and the collision-specific coercive estimates.  Landau uses $w(\xi)=\langle\xi/a_L\rangle^\ell e^{q\langle\xi/a_L\rangle^2}$ and the Coulomb coercive norm, whereas non-cutoff Boltzmann uses the temperature-adapted weight $w_0$, the shifted Maxwellian $\mu_0$ from \eqref{eq:glob-id}, and the hybrid zero-order $g^\e$/positive-order $f^\e$ hierarchy established above.  The proof of Theorem~\ref{thm:kinenergy} is completed at the end of Section~\ref{sec:macenergy}.

\section{Quantitative kinetic closure}\label{sec:macroclose}
For the non-cutoff positive-order estimates we use the exact decomposition
\begin{equation}\label{eq:maxw-def}
 f^\e=g^\e+H_\infty^\e+\widetilde H^\e,
\end{equation}
with the terms defined in Lemma~\ref{lem:maxbridge}.

\subsection{Exact balance laws and the local-Maxwellian defect}\label{sec:energyhier}
With the logarithmic variables in \eqref{eq:param-id}, the collision-invariant
balances become
\begin{align}\label{eq:tu-id}
\bega
&\p_t \rho^{\e} +u^{\e}\cdot\nabla_x\rho^{\e} + \frac{1}{\e}\nabla_x\cdot u^{\e} =0 ,\cr
&\frac{1}{k_B\mathrm{\Theta}^{\e}}\bigg[\p_tu^{\e} + u^{\e}\cdot \nabla_x u^{\e} \bigg] +\frac{1}{\e}\nabla_x(\rho^{\e}+\theta^{\e})
+\frac{1}{\e^2}\frac{1}{k_B\mathrm{P}^{\e} \mathrm{\Theta}^{\e}}\sum_{j} \p_{x_j} \mathbf{r}_{ij}^{\e} =0, \cr
&\frac{3}{2}\bigg[\p_t \theta^{\e} +u^{\e}\cdot\nabla_x\theta^{\e} \bigg] + \frac{1}{\e}\nabla_x\cdot u^{\e}
 +\frac{1}{\e^2}\frac{1}{k_B\mathrm{P}^{\e}\mathrm{\Theta}^{\e} }\sum_{j} \p_{x_j} \mathfrak{q}_j^{\e}
 + \frac{1}{\e^2}\frac{1}{k_B\mathrm{P}^{\e}\mathrm{\Theta}^{\e} }\sum_{i,j} \p_{x_i}\mathrm{U}^{\e}_j \mathbf{r}_{ij}^{\e}=0.
\enda
\end{align}
The singular acoustic terms in these three equations cancel exactly in the
macroscopic energy identity.  Exact moment matching also removes the five
collision-invariant directions from the microscopic defect
$G^\e=F^\e-M^\e$.

Set $c^\e=\xi-\mathrm U^\e$ and
\[
 \chi^\e=\frac{|c^\e|^2}{2k_B\Theta^\e}-\frac32,
 \qquad \mathcal T_\e=\partial_t+\e^{-1}\xi\cdot\nabla_x.
\]
A direct differentiation of the Maxwellian, followed by substitution of
\eqref{eq:tu-id}, gives
\begin{align}\label{eq:mic-id}
 \mathcal T_\e G^\e+\frac1{\kappa\e^2}\mathcal L_MG^\e
 &=-\mathcal A^\e+\mathcal F^\e
 +\frac1{\kappa\e^2}\mathcal N(G^\e,G^\e),
\end{align}
where
\begin{align*}
 \mathcal A^\e
 &=\frac{M^\e}{k_B\Theta^\e}
 \left(\sum_{i,j}\partial_i u_j^\e\,\mathfrak R_{ij}^\e
 +\sum_i\partial_i\theta^\e\,\mathcal Q_i^\e\right),\\
 \mathcal F^\e
 &=\frac{M^\e}{\e k_B\mathrm P^\e\Theta^\e}
 \sum_{i,j}\partial_j\mathbf r_{ij}^\e c_i^\e
 +\frac{2M^\e}{3\e k_B\mathrm P^\e\Theta^\e}
 \left(\sum_j\partial_j\mathfrak q_j^\e
 +\sum_{i,j}\partial_i\mathrm U_j^\e\mathbf r_{ij}^\e\right)\chi^\e .
\end{align*}
The tensors $\mathfrak R^\e$ and $\mathcal Q^\e$ are the traceless quadratic
and cubic Burnett modes associated with $M^\e$ defined in \eqref{eq:ij-def}.  Formula \eqref{eq:mic-id} is
the only local-Maxwellian transport identity needed below.

The stress and heat flux satisfy the following collision-independent estimates.
\begin{lemma}[Stress and heat-flux bounds]\label{lem:stress}
For either collision operator and $0\le t\le T_*$,
\begin{equation}\label{eq:hn-kap}
 \|\mathbf r^\e\|_{H^N_x}+\|\mathfrak q^\e\|_{H^N_x}
 \lesssim \e^2\kappa^{1/2}(\mathcal D_G^N)^{1/2}.
\end{equation}
At the one-higher spatial order, one has the complementary estimate
\begin{equation}\label{eq:htop-bd}
 \|\mathbf r^\e\|_{H^{N+1}_x}+\|\mathfrak q^\e\|_{H^{N+1}_x}
 \lesssim
 \e^2\left((\mathcal D^N)^{1/2}+\mathcal E_M^N\right).
\end{equation}
\end{lemma}

\begin{lemma}[Top spatial macroscopic control]\label{lem:topmacro}
Let $N\ge4$ and assume \eqref{eq:tot-bd}.  With
$Z^\e=(\rho^\e,u^\e,\theta^\e)$, one has
\begin{equation}\label{eq:topmacro-bd}
 \sqrt\kappa\,\|Z^\e\|_{\dot H_x^{N+1}}
 \lesssim (\mathcal E^N)^{1/2}.
\end{equation}
\end{lemma}
\begin{proof}
Fix $|\alpha|=N+1$.  Exact moment matching \eqref{eq:xd-id} gives
$\Pi_M(\partial^\alpha F^\e)=\Pi_M(\partial^\alpha M^\e)$.
The Maxwellian chain rule has the form
\[
 \partial^\alpha M^\e
 =\e M^\e\Phi_{\alpha,1}^\e+R_{\alpha,M}^\e,
\]
where
\[
 \Phi_{\alpha,1}^\e
 =\partial^\alpha\rho^\e
 +\frac{c^\e\cdot\partial^\alpha u^\e}{k_B\Theta^\e}
 +\left(\frac{|c^\e|^2}{2k_B\Theta^\e}-\frac32\right)
   \partial^\alpha\theta^\e,
\]
and every term in $R_{\alpha,M}^\e$ contains at least two positive
spatial derivatives and a factor $\e^2$.  Hence the $H_x^2$ high--low
product estimate, Lemma~\ref{lem:macrocoef}, and \eqref{eq:loc-bd} imply
\[
 \left(\sum_{|\alpha|=N+1}
 \left\|\frac{R_{\alpha,M}^\e}{\sqrt{M^\e}}\right\|_2^2\right)^{1/2}
 \lesssim \e^2\mathcal E_M^N.
\]
Since the linear term lies in the local null space and $\Pi_M$ is
contractive in $L^2_\xi((M^\e)^{-1}d\xi)$,
\[
 \e\|\Phi_{\alpha,1}^\e\sqrt{M^\e}\|_2
 \le \left\|\frac{\partial^\alpha F^\e}{\sqrt{M^\e}}\right\|_2
 +\left\|\frac{R_{\alpha,M}^\e}{\sqrt{M^\e}}\right\|_2.
\]
Moreover,
\[
 \int|\Phi_{\alpha,1}^\e|^2M^\e\,d\xi
 =\mathrm P^\e\left(
 |\partial^\alpha\rho^\e|^2
 +\frac{|\partial^\alpha u^\e|^2}{k_B\Theta^\e}
 +\frac32|\partial^\alpha\theta^\e|^2\right).
\]
Using Lemma~\ref{lem:macrocoef}, \eqref{eq:top-def}, and the remainder estimate just obtained,
\[
 \sqrt\kappa\,\|Z^\e\|_{\dot H^{N+1}}
 \lesssim (\mathcal E^N)^{1/2}
 +\e\sqrt\kappa\,\mathcal E_M^N.
\]
The last term is $o(1)(\mathcal E^N)^{1/2}$ by \eqref{eq:loc-bd}; this proves
\eqref{eq:topmacro-bd}.
\end{proof}

\begin{proof}[Proof of Lemma \ref{lem:stress}]
We give the direct collision-independent moment argument.  It reduces the stress and heat flux to fixed velocity polynomials and then uses the dissipation norms defined in \eqref{eq:xi-bd}.

Set $G^\e:=\Pi_M^\perp F^\e=F^\e-M^\e$.  By exact moment matching
\eqref{eq:xd-id},
\[
 \int_{\mathbb R^3}G^\e(1,\xi,|\xi|^2)\,d\xi=0.
\]
Introduce the fixed velocity polynomials
\[
 A_{ij}(\xi):=\xi_i\xi_j-\frac{|\xi|^2}{3}\delta_{ij},
 \qquad
 B_i(\xi):=\frac12\xi_i|\xi|^2.
\]
Expanding the centered Burnett polynomials in \eqref{eq:ij-def} and using the
five vanishing moments gives the exact identities
\begin{equation}\label{eq:ij-mic}
 \mathbf r_{ij}^\e=\int_{\mathbb R^3}G^\e A_{ij}\,d\xi,
 \qquad
 \mathfrak q_i^\e
 =\int_{\mathbb R^3}G^\e B_i\,d\xi
   -\sum_j\mathrm U_j^\e\mathbf r_{ij}^\e.
\end{equation}
Indeed, all terms of degree at most one in $\xi$, as well as the trace
term containing $|\xi|^2$, vanish by \eqref{eq:xd-id}; in the cubic moment
only $\frac12\xi_i|\xi|^2$ and
$-\mathrm U_j^\e\xi_i\xi_j$ remain, and
$\int G^\e\xi_i\xi_j\,d\xi=\mathbf r_{ij}^\e$.

Let $\mathbf D=\Lb$ in the Landau case and $\mathbf D=\Bb$ in the
Boltzmann case.  The lower bounds in \eqref{eq:xi-bd} imply the following
uniform Gaussian dual estimate: for every fixed polynomial $P$,
\begin{equation}\label{eq:xi-poly}
 \left|\int_{\mathbb R^3}h(\xi)P(\xi)\sqrt{M^\e(t,x,\xi)}\,d\xi\right|
 \le C_P\|h\|_{L^2_{\mathbf D}},
\end{equation}
where $C_P$ is independent of $(\e,t,x)$.  Indeed, in the Landau case
$\|h\|_{L^2_{\mathbf D}}$ controls
$\|\langle\xi\rangle^{-1/2}h\|_2$, while in the Boltzmann case it controls
$\|\langle\xi\rangle^{(2s+\gamma)/2}h\|_2$; the reciprocal polynomial
weights are absorbed by the Gaussian $\sqrt{M^\e}$.  Uniformity follows
from \eqref{eq:hn-en} and \eqref{eq:assm-bd}.

For every $|\alpha|\le N$, differentiation of the fixed moments in
\eqref{eq:ij-mic} gives, by \eqref{eq:xi-poly},
\[
 \|\partial^\alpha\mathbf r^\e\|_2
 +\left\|\partial^\alpha\int G^\e B\,d\xi\right\|_2
 \lesssim
 \left\|\frac{\partial^\alpha G^\e}{\sqrt{M^\e}}
 \right\|_{L_x^2L_{\mathbf D}^2}.
\]
Summing and using \eqref{eq:diss-def} yields
\begin{equation}\label{eq:mic-xi}
 \|\mathbf r^\e\|_{H_x^N}
 +\left\|\int G^\e B\,d\xi\right\|_{H_x^N}
 \lesssim \e^2\kappa^{1/2}(\mathcal D_G^N)^{1/2}.
\end{equation}
Since $N\ge4$, $H^N(\mathbb R^3)$ is an algebra, and
\eqref{eq:hn-en} gives $\|\mathrm U^\e\|_{H^N}\lesssim1$.
The second identity in \eqref{eq:ij-mic} and \eqref{eq:mic-xi} therefore prove
\eqref{eq:hn-kap}.  Notice that this argument never differentiates
$(M^\e)^{-1/2}$.

We next prove \eqref{eq:htop-bd}.  Let $|\alpha|=N+1$.  The parameter chain
rule for
$M^\e=M_{[e^{\e\rho^\e},\e u^\e,e^{\e\theta^\e}]}$ gives
\begin{equation*}
 \partial^\alpha M^\e
 =\e M^\e\Phi_{\alpha,1}^\e+\mathcal R_{\alpha,M}^\e,
\end{equation*}
where
\[
 \Phi_{\alpha,1}^\e
 =\partial^\alpha\rho^\e
 +\frac{c^\e\cdot\partial^\alpha u^\e}{k_B\Theta^\e}
 +\left(\frac{|c^\e|^2}{2k_B\Theta^\e}-\frac32\right)
  \partial^\alpha\theta^\e,
 \qquad c^\e=\xi-\mathrm U^\e,
\]
and every term in $\mathcal R_{\alpha,M}^\e$ contains at least two
positive spatial derivatives of $Z^\e=(\rho^\e,u^\e,\theta^\e)$.
The $H^N$ product estimate, the uniform Gaussian parameter bounds
\eqref{eq:hn-en}, and \eqref{eq:loc-bd} give
\begin{equation*}
 \sum_{|\alpha|=N+1}
 \left[
 \left\|\frac{\mathcal R_{\alpha,M}^\e}{\sqrt{M^\e}}\right\|_{L^2_{x,\xi}}
 +\left\|\frac{\mathcal R_{\alpha,M}^\e}{\sqrt{M^\e}}\right\|_{L_x^2L_{\mathbf D}^2}
 \right]
 \lesssim \e^2\mathcal E_M^N.
\end{equation*}
Indeed, the quadratic terms are bounded by
$C\e^2\|\nabla Z^\e\|_\infty\|Z^\e\|_{H^N}$, and every term of degree
$m\ge3$ is bounded by $C_m\e^m(\mathcal E_M^N)^{m/2}$; the latter series
is absorbed into the right-hand side because
$\e(\mathcal E_M^N)^{1/2}=o(1)$.

Set $H_\alpha^\e:=\Pi_M^\perp(\partial^\alpha F^\e)$.  By
\eqref{eq:xd-id}, $\Pi_M(\partial^\alpha G^\e)=0$.  Since the linear term
$\e M^\e\Phi_{\alpha,1}^\e$ belongs to the local null space,
\begin{equation*}
 \partial^\alpha G^\e
 =H_\alpha^\e-\Pi_M^\perp\mathcal R_{\alpha,M}^\e.
\end{equation*}
For fixed $(t,x)$, the normalized local null space is spanned by finitely
many Gaussian polynomials with parameters in the compact set furnished by
\eqref{eq:hn-en}.  Hence \eqref{eq:xi-poly} and the upper bounds in
\eqref{eq:xi-bd} imply uniformly
\[
 \left\|\frac{\Pi_M h}{\sqrt{M^\e}}\right\|_{L^2_{\mathbf D}}
 \lesssim
 \left\|\frac{h}{\sqrt{M^\e}}\right\|_{L^2_{\mathbf D}},
 \qquad
 \left\|\frac{\Pi_M^\perp h}{\sqrt{M^\e}}\right\|_{L^2_{\mathbf D}}
 \lesssim
 \left\|\frac{h}{\sqrt{M^\e}}\right\|_{L^2_{\mathbf D}}.
\]
Therefore \eqref{eq:top-def} and these two estimates give
\begin{equation*}
 \left(\sum_{|\alpha|=N+1}
 \left\|\frac{\partial^\alpha G^\e}{\sqrt{M^\e}}
 \right\|_{L_x^2L_{\mathbf D}^2}^2\right)^{1/2}
 \lesssim
 \e^2\left((\mathcal D_{top}^N)^{1/2}+\mathcal E_M^N\right).
\end{equation*}
Applying \eqref{eq:xi-poly} to the fixed polynomials $A_{ij}$ and $B_i$ in
\eqref{eq:ij-mic}, and combining \eqref{eq:mic-xi} with the last bound, yields
\begin{equation*}
 \|\mathbf r^\e\|_{H_x^{N+1}}
 +\left\|\int G^\e B\,d\xi\right\|_{H_x^{N+1}}
 \lesssim
 \e^2\left((\mathcal D^N)^{1/2}+\mathcal E_M^N\right).
\end{equation*}

It remains only to control the product $\mathrm U^\e\mathbf r^\e$ in
the second identity of \eqref{eq:ij-mic}.  The top macroscopic estimate
\eqref{eq:topmacro-bd} gives
\[
 \|\mathrm U^\e\|_{H^{N+1}}
 \lesssim 1+\e\kappa^{-1/2}(\mathcal E^N)^{1/2}.
\]
The Moser product inequality, \eqref{eq:hn-en}, \eqref{eq:hn-kap}, and
\eqref{eq:topmacro-bd} imply
\begin{align*}
 \|\mathrm U^\e\mathbf r^\e\|_{H^{N+1}}
 &\lesssim
 \|\mathrm U^\e\|_\infty\|\mathbf r^\e\|_{H^{N+1}}
 +\|\mathrm U^\e\|_{H^{N+1}}\|\mathbf r^\e\|_\infty \lesssim
 \e^2\left((\mathcal D^N)^{1/2}+\mathcal E_M^N\right)
 +\e^3(\mathcal E^N)^{1/2}(\mathcal D_G^N)^{1/2}.
\end{align*}
Since $\e(\mathcal E^N)^{1/2}=o(1)$ by \eqref{eq:tot-bd}, the last term is
absorbed by $C\e^2(\mathcal D^N)^{1/2}$.  Combining this estimate with \eqref{eq:ij-mic} and the moment estimate just derived proves \eqref{eq:htop-bd}.
\end{proof}

\subsection{Collision hierarchy and local energy closure}\label{sec:macenergy}
\subsubsection{Fixed-global collision estimates}
The global-Maxwellian macroscopic components used in the high-order estimates
are controlled by the local macroscopic energy. For Landau,
\begin{equation}\label{eq:alp-pr}
 \sum_{|\alpha|\le N}
 \|\Pi\partial^\alpha f^\e\|_{L^2_{x,\xi}}^2
 +\kappa\|\Pi\partial^{N+1}f^\e\|_{L^2_{x,\xi}}^2
 \lesssim \mathcal E^N.
\end{equation}
For non-cutoff Boltzmann, the shifted Maxwellian affects the zero spatial
order, which is why $g^\e$ is used there. At positive spatial order,
\begin{equation}\label{eq:alp-pr-bd}
 \sum_{1\le|\alpha|\le N}
 \|\Pi_0\partial^\alpha f^\e\|_{L^2_{x,\xi}}^2
 +\kappa\|\Pi_0\partial^{N+1}f^\e\|_{L^2_{x,\xi}}^2
 \lesssim \mathcal E^N.
\end{equation}
The zero spatial order in the Boltzmann expansion is instead microscopic by
\eqref{eq:pr-id}.  Indeed, the microscopic part $F^\e-M^\e$ has zero
collision-invariant moments. Thus the global macroscopic projections in
\eqref{eq:alp-pr}--\eqref{eq:alp-pr-bd} come entirely from derivatives of
$M^\e-\mu$ or $M^\e-\mu_0$, respectively. The lower spatial orders follow from Lemma~\ref{lem:macrocoef} and the
macroscopic energy in \eqref{eq:diss-def}; the $N+1$ spatial order follows from
the top macroscopic control \eqref{eq:topmacro-bd}.
The lower-order microscopic identity, the macroscopic energy pairing,
and their stress--heat-flux cancellation are derived below directly from
the exact balance laws.  Thus the local energy estimate will not require
an external microscopic--macroscopic energy identity.  The
collision-specific coercive and trilinear estimates remain the only
model-dependent analytic input in this part of the proof.
\begin{lemma}[Collision estimates]\label{lem:collest}
Let $N\ge4$ and assume \eqref{eq:tot-bd}. Set
\[
 (\mathcal L_*,\Pi_*,w_*,\mathbf D)
 =
 \begin{cases}
 (\mathcal L,\Pi,w,\Lb),&\text{Landau},\\
 (\mathcal L_0,\Pi_0,w_0,\Bb),&\text{non-cutoff Boltzmann}.
 \end{cases}
\]
For $0\le|\alpha|\le N+1$,
\begin{equation}\label{eq:lin-bd}
 \langle\partial^\alpha\mathcal L_*h,\partial^\alpha h\rangle_{x,\xi}
 \gtrsim
 \|(I-\Pi_*)\partial^\alpha h\|_{L^2_{x,\mathbf D}}^2.
\end{equation}
Moreover,
\begin{equation}\label{eq:xi-pr}
 \|\langle\xi\rangle^{-1}(I-\Pi_*)h\|_{L^2_\xi}
 \lesssim \|(I-\Pi_*)h\|_{L^2_{\mathbf D}}.
\end{equation}
The corresponding local-Maxwellian quadratic form is uniformly coercive.
More precisely, if $h=h(x,\xi)$ has zero local mass, momentum, and energy
moments for a.e. $x$, then
\begin{equation}\label{eq:mh-bd}
 \iint \mathcal L_Mh\,\frac{h}{M^\e}\,\dd\xi\dd x
 \ge c_M
 \left\|\frac{h}{\sqrt{M^\e}}\right\|_{L^2_{x,\mathbf D}}^2,
 \qquad c_M>0,
\end{equation}
where $c_M$ is independent of $\e$ and $0\le t\le T_*$.  For arbitrary
$h$ the same estimate holds with $h$ replaced by $\Pi_M^\perp h$ in both
slots; in particular,
\begin{equation}\label{eq:mq-bd}
 \iint \mathcal L_Mh\,
 \frac{\Pi_M^\perp h}{M^\e}\,\dd\xi\dd x
 \ge c_M
 \left\|\frac{\Pi_M^\perp h}{\sqrt{M^\e}}
 \right\|_{L^2_{x,\mathbf D}}^2.
\end{equation}
Denote $B_R:=\{\xi \in \R^3 | |\xi|<R\}.$ For $0\le|\alpha|\le N$, there is a fixed $R>1$ such that
\begin{align}
 &\left\langle
 \mathcal L_*(I-\Pi_*)\partial^\alpha h,
 w_*^2(I-\Pi_*)\partial^\alpha h
 \right\rangle_{x,\xi}
  \gtrsim
 \|w_*(I-\Pi_*)\partial^\alpha h\|_{L^2_{x,\mathbf D}}^2
 -C_R\|(I-\Pi_*)\partial^\alpha h\|_{L^2(\R_x^3\times B_R)}^2.
 \label{eq:top-bd}
\end{align}
At top spatial order,
\begin{align}
 \left\langle
 \mathcal L_*\partial^{N+1}h,
 w_*^2\partial^{N+1}h
 \right\rangle_{x,\xi}
 &\gtrsim
 \|w_*\partial^{N+1}h\|_{L^2_{x,\mathbf D}}^2
 -C_R\|\partial^{N+1}h\|_{L^2(\R_x^3\times B_R)}^2.
 \label{eq:vel-bd}
\end{align}
For every integer $0\le m\le\lfloor N/2\rfloor$, the one-velocity-derivative
coercivity needed below holds after summing the finitely many components:
\begin{align}
 &\sum_{|\alpha|\le m}\sum_{|\beta|=1}
 \left\langle
 \partial_\beta^\alpha\mathcal L_*(I-\Pi_*)h,
 w_*\partial_\beta^\alpha(I-\Pi_*)h
 \right\rangle_{x,\xi}
 \notag\\
 &\qquad\gtrsim
 \sum_{|\alpha|\le m}\sum_{|\beta|=1}
 \|w_*^{1/2}\partial_\beta^\alpha(I-\Pi_*)h\|_{L^2_{x,\mathbf D}}^2
 -C\sum_{|\alpha|\le m}
 \|w_*^{1/2}\partial^\alpha(I-\Pi_*)h\|_{L^2_{x,\mathbf D}}^2.
 \label{eq:alp-bet}
\end{align}
For the local-Maxwellian collision operator,
\begin{align}\label{eq:op-bd}
 \left|\int_{\mathbb R^3}\mathcal N(F,G)H(M^\e)^{-1}\,\dd\xi\right|
 \lesssim
 &\left(
 \|F(M^\e)^{-1/2}\|_{L^2_\xi}
 \|G(M^\e)^{-1/2}\|_{L^2_{\mathbf D}}\right.\notag\\
 &\left.\quad+
 \|F(M^\e)^{-1/2}\|_{L^2_{\mathbf D}}
 \|G(M^\e)^{-1/2}\|_{L^2_\xi}
 \right)
 \|H(M^\e)^{-1/2}\|_{L^2_{\mathbf D}}.
\end{align}
For the global-Maxwellian Landau operator, we record separately the
unweighted and weighted estimates actually used below.  For each fixed
$x$ and every spatial multi-index $\alpha$,
\begin{align}\label{eq:landau-tri}
 &|\langle \partial^\alpha\Gamma(f,g),\partial^\alpha h\rangle_\xi| \lesssim
 \|\partial^\alpha h\|_{L^2_{\Lb}}
 \sum_{\alpha_1+\alpha_2=\alpha}
 \Big(
 \|\mu^{1/32}\partial^{\alpha_1}f\|_{L^2_\xi}
 \|\partial^{\alpha_2}g\|_{L^2_{\Lb}}
 +
 \|\partial^{\alpha_1}f\|_{L^2_{\Lb}}
 \|\partial^{\alpha_2}g\|_{L^2_{-1}}
 \Big).
\end{align}
Set
\[
 \vartheta_L(\xi):=e^{-q\langle\xi/a_L\rangle^2}w(\xi)
 =\langle\xi/a_L\rangle^\ell.
\]
For spatial derivatives, again pointwise in $x$,
\begin{align}\label{eq:landau-tri-wt}
 &|\langle \partial^\alpha\Gamma(f,g),
          w^2\partial^\alpha h\rangle_\xi| \lesssim
 \|w\partial^\alpha h\|_{L^2_{\Lb}}
 \sum_{\alpha_1+\alpha_2=\alpha}
 \Big(
 \|\vartheta_L\partial^{\alpha_1}f\|_{L^2_\xi}
 \|w\partial^{\alpha_2}g\|_{L^2_{\Lb}}
 +
 \|\vartheta_L\partial^{\alpha_1}f\|_{L^2_{\Lb}}
 \|w\partial^{\alpha_2}g\|_{L^2_{-1}}
 \Big).
\end{align}
For $|\beta|=1$, specialize
\cite[Lemma~2.5(ii)]{lei2024global} to the constant Maxwellian with
$(\ell_s,Q)=(1+\ell/2,2q)$.  In the normalized variable
$y=\xi/a_L$, set
\[
 A_0:=\langle y\rangle^{1+\ell/2},\qquad
 A_1:=\langle y\rangle^{\ell/2},\qquad
 W_0:=w,\qquad W_1:=w^{1/2}.
\]
The localized first-slot weights satisfy
\begin{equation}\label{eq:landau-slot-wt}
 \|A_j h\|_{L^2_\xi}+\|A_j h\|_{L^2_{\Lb}}
 \lesssim \|W_j h\|_{L^2_{\Lb}},\qquad j=0,1.
\end{equation}
The first source slot in that lemma carries $A_{|\beta_1|}$, whereas the
second and testing slots are bounded by $W_{|\beta_2|}$ and $W_1$,
respectively; moreover
$\alpha_1+\alpha_2=\alpha$ and $\beta_1+\beta_2\le\beta$.
The additional source sum there is proportional to spatial derivatives of
the Maxwellian parameters and therefore vanishes for the constant
Maxwellian used here.
For the constant Maxwellian the same bounded-velocity/Gaussian split used
in \eqref{eq:landau-tri-wt} gives the $L^2_{-1}$ refinement in the second
result.  Thus the form used below is
\begin{align}
 &|\langle\partial_\beta^\alpha\Gamma(f,g),
          w\partial_\beta^\alpha h\rangle_\xi|\notag\\
 &\quad\lesssim
 \|W_1\partial_\beta^\alpha h\|_{L^2_{\Lb}}
 \sum_{\substack{\alpha_1+\alpha_2=\alpha\\
                  \beta_1+\beta_2\le\beta}}
 \Big(
 \|A_{|\beta_1|}\partial_{\beta_1}^{\alpha_1}f\|_{L^2_\xi}
 \|W_{|\beta_2|}\partial_{\beta_2}^{\alpha_2}g\|_{L^2_{\Lb}}
 +
 \|A_{|\beta_1|}\partial_{\beta_1}^{\alpha_1}f\|_{L^2_{\Lb}}
 \|W_{|\beta_2|}\partial_{\beta_2}^{\alpha_2}g\|_{L^2_{-1}}
 \Big). \label{eq:landau-space-velocity}
\end{align}

For the non-cutoff Boltzmann operator, the unweighted estimate is
\begin{align}\label{eq:boltz-tri}
 |\langle\Gamma_0(f,g),h\rangle_\xi|
 \lesssim{}&\Big\{
 \|f\|_{L^2_{2s+\gamma}}\|g\|_{L^2_{\Bb}}
 +\|g\|_{L^2_{2s+\gamma}}\|f\|_{L^2_{\Bb}} +
 \min\big(
 \|f\|_{L^2_\xi}\|g\|_{L^2_{2s+\gamma}},
 \|g\|_{L^2_\xi}\|f\|_{L^2_{2s+\gamma}}
 \big)\Big\}\|h\|_{L^2_{\Bb}}.
\end{align}
Spatial derivatives are obtained from \eqref{eq:boltz-tri} by the
Leibniz rule.  The weighted mixed-derivative estimate, pointwise in $x$, is
\begin{align}
 &|\langle \partial^\alpha_\beta\Gamma_0(f,g),
          w_0^2\partial_\beta^\alpha h\rangle_\xi|\notag\\
 &\lesssim \|w_0\partial_\beta^\alpha h\|_{L^2_{\Bb}}
 \Big\{\sum_{\substack{\alpha_1+\alpha_2=\alpha\\
                        \beta_1+\beta_2\le\beta}}
 \Big[
 \|w_0\partial^{\alpha_1}_{\beta_1}f\|_{L^2_{2s+\gamma}}
 \|\partial^{\alpha_2}_{\beta_2}g\|_{L^2_{\Bb}}
 +\|w_0\partial^{\alpha_1}_{\beta_1}f\|_{L^2_{\Bb}}
 \|\partial^{\alpha_2}_{\beta_2}g\|_{L^2_{2s+\gamma}}\notag\\
 &\hspace{4.4cm}+
 \min\Big\{
 \|w_0\partial^{\alpha_1}_{\beta_1}f\|_{L^2_\xi}
 \|\partial^{\alpha_2}_{\beta_2}g\|_{L^2_{2s+\gamma}},
 \|w_0\partial^{\alpha_1}_{\beta_1}f\|_{L^2_{2s+\gamma}}
 \|\partial^{\alpha_2}_{\beta_2}g\|_{L^2_\xi}
 \Big\}\Big]\notag\\
 &\hspace{1.0cm}+
 \sum_{\substack{\alpha_1+\alpha_2=\alpha\\
                  \beta_1+\beta_2\le\beta}}
 \|\langle\xi\rangle^{-\ell}w_0
      \partial^{\alpha_2}_{\beta_2}g\|_{L^2_\xi}
 \|w_0\partial^{\alpha_1}_{\beta_1}f\|_{L^2_{2s+\gamma}}
 \Big\}.\label{eq:boltz-tri-wt}
\end{align}
For $|\beta|=1$ we use the same source estimate with the weight parameters
$(\ell,q)$ replaced by $(\ell/2,q/2)$; hence the test weight is $w_0$ and
the testing/coercive weight is $w_0^{1/2}$.  The extra weighted $L^2$
term in \eqref{eq:boltz-tri-wt} is also coercive.  Indeed,
$\ell\ge-(\gamma+2s)$ and \eqref{eq:xi-bd} imply
\begin{align}\label{eq:boltz-wt-coer}
 \|\langle\xi\rangle^{-\ell}w_0h\|_{L^2_\xi}
 &\lesssim \|w_0h\|_{L^2_{2s+\gamma}}
 \lesssim \|w_0h\|_{L^2_{\Bb}},\notag\\
 \|\langle\xi\rangle^{-\ell/2}w_0^{1/2}h\|_{L^2_\xi}
 &\lesssim \|w_0^{1/2}h\|_{L^2_{2s+\gamma}}
 \lesssim \|w_0^{1/2}h\|_{L^2_{\Bb}}.
\end{align}
The fixed weight hierarchy is also bounded in the non-cutoff norm:
\begin{equation}\label{eq:boltz-wt-hier}
 \|h\|_{L^2_{\Bb}}
 \lesssim \|w_0^{1/2}h\|_{L^2_{\Bb}}
 \lesssim \|w_0h\|_{L^2_{\Bb}}.
\end{equation}

\end{lemma}

\subsubsection{Large-velocity closure}

For every fixed polynomial exponent $m\ge0$ and $|\alpha|\le N+1$, set
\begin{align}
 \mathcal V_m^\alpha(t)
 &:=\frac{\kappa^{(|\alpha|-N)_+}}{\e^2}
   \iint_{\mathbb R^3\times\mathbb R^3}
   \langle\xi\rangle^m
   |\partial^\alpha\Pi_M^\perp F^\e|^2(M^\e)^{-1}\,\dd\xi dx. \label{eq:largev-def} 
\end{align}
The constants below may depend on the fixed exponent $m$, but are uniform
in $\e,t$.
\begin{lemma}[Maxwellian spatial chain rule]\label{lem:maxchain}
Let $Z^\e = (\rho^\e,u^\e,\theta^\e)$. For every $1\le|\alpha|\le N+1$,
\begin{equation}\label{eq:par-id}
 \partial^\alpha M^\e
 =\e M^\e
 \sum_{n=1}^{|\alpha|}\e^{n-1}
 \sum_{\substack{\alpha_1+\cdots+\alpha_n=\alpha\\|\alpha_j|\ge1}}
 \mathcal P_{\alpha_1,\ldots,\alpha_n}(\xi-\mathrm{U}^\e)
 \prod_{j=1}^n\partial^{\alpha_j}Z^\e,
\end{equation}
where the coefficients of the fixed polynomials $\mathcal P$ are uniformly
bounded under Lemma~\ref{lem:macrocoef}.
\end{lemma}
\begin{lemma}[Large-velocity closure]
For either collision operator, every fixed $m\ge0$, and
$|\alpha|\le N+1$,
\begin{align}
 \mathcal V_m^\alpha(t)
 &\lesssim
 \e^2\mathcal E^N
 +\e^{3/2}\kappa
 \bigl(\mathcal D^N+\mathbf1_{|\alpha|\ge2}(\mathcal E_M^N)^2\bigr)
 +\e^{10}\mathcal E_H^N,
 \label{eq:large-bd}\\
 \mathcal V_m^\alpha(t)
 &\lesssim \e^{-1/2}\mathcal E^N+\e^{10}\mathcal E_H^N.
 \label{eq:wt-bd}
\end{align}
\end{lemma}
Estimate~\eqref{eq:large-bd} is the $\e$-gain estimate used when the weighted
microscopic norm is multiplied by a singular kinetic prefactor; the rougher
bound~\eqref{eq:wt-bd} is used when only a direct weighted $L^2$ control is
needed.  The two estimates therefore serve different parts of the closure.
\begin{proof}
Set $R_\e=\e^{-1/(2(m+2))}$ and split at $|\xi|=R_\e$.
On $|\xi|\le R_\e$,
\[
 \langle\xi\rangle^m\lesssim\e^{-1/2},
 \qquad
 \langle\xi\rangle^m
 \lesssim\e^{-1/2}\langle\xi-\mathrm U^\e\rangle^{-2}.
\]
Define
\begin{align*}
 \mathcal V_{m,\mathrm{lo}}^\alpha
 &:=\frac{\kappa^{(|\alpha|-N)_+}}{\e^2}
   \iint_{|\xi|\le R_\e}\langle\xi\rangle^m
   |\partial^\alpha\Pi_M^\perp F^\e|^2(M^\e)^{-1}\,\dd\xi dx,\\
 \mathcal V_{m,\mathrm{hi}}^\alpha
 &:=\frac{\kappa^{(|\alpha|-N)_+}}{\e^2}
   \iint_{|\xi|>R_\e}\langle\xi\rangle^m
   |\partial^\alpha\Pi_M^\perp F^\e|^2(M^\e)^{-1}\,\dd\xi dx.
\end{align*}
Thus $\mathcal V_m^\alpha=\mathcal V_{m,\mathrm{lo}}^\alpha+
\mathcal V_{m,\mathrm{hi}}^\alpha$.  Using \eqref{eq:xi-bd}, \eqref{eq:diss-def}, and
\eqref{eq:top-def} gives
\begin{equation}\label{eq:lo-bd}
 \mathcal V_{m,\mathrm{lo}}^\alpha
 \lesssim
 \min\left\{
 \e^{-1/2}\mathcal E^N,
 \e^{3/2}\kappa
 \bigl(\mathcal D^N+\mathbf1_{|\alpha|\ge2}(\mathcal E_M^N)^2\bigr)
 \right\}.
\end{equation}
For Landau, use the chain rule \eqref{eq:par-id} and write
$F^\e=\mu+\e\sqrt\mu f^\e=M^\e+\Pi_M^\perp F^\e$.
For $|\alpha|\ge1$,
\[
 \partial^\alpha\Pi_M^\perp F^\e
 =\e\sqrt\mu\,\partial^\alpha f^\e-\partial^\alpha M^\e,
\]
while at zero order
$\Pi_M^\perp F^\e=\mu-M^\e+\e\sqrt\mu f^\e$.
Decompose
$\partial^\alpha f^\e=(I-\Pi)\partial^\alpha f^\e+
\Pi\partial^\alpha f^\e$.  The first part is controlled by the weighted
terms in \eqref{eq:land-id}.  Since $\Pi$ has finite-dimensional range,
there is a fixed polynomial $P_\alpha$ such that
\[
 |\Pi\partial^\alpha f^\e(\xi)|
 \le C P_\alpha(\xi)\sqrt\mu(\xi)
       \|\partial^\alpha f^\e\|_{L^2_\xi}.
\]
Consequently, for every fixed $m$ and $R\ge1$,
\begin{align*}
 \int_{|\xi|>R}\langle\xi\rangle^m
 \mu(M^\e)^{-1}|\partial^\alpha f^\e|^2\,d\xi
 &\lesssim e^{-cR^2}
 \left(\|w(I-\Pi)\partial^\alpha f^\e\|_{L^2_\xi}^2
       +\|\partial^\alpha f^\e\|_{L^2_\xi}^2\right).
\end{align*}
Using \eqref{eq:par-id} and the uniform Gaussian bounds of
Lemma~\ref{lem:macrocoef} gives
\[
 \frac{\kappa^{(|\alpha|-N)_+}}{\e^2}
 \iint_{|\xi|>R_\e}\langle\xi\rangle^m
 |\partial^\alpha M^\e|^2(M^\e)^{-1}\,d\xi dx
 \lesssim \e^2\mathcal E^N.
\]
At zero order the Taylor formula for $M^\e-\mu$ gives the same bound.
Using \eqref{eq:land-id} and allowing harmless fixed powers of
$\e^{-1}$ and $\kappa^{-1}$ therefore yields
\[
 \mathcal V_{m,\mathrm{hi}}^\alpha
 \lesssim \e^2\mathcal E^N
 +\frac{e^{-cR_\e^2}}{\e^2\kappa^{2c+2}}\mathcal E_H^N
 \lesssim \e^2\mathcal E^N+\e^{10}\mathcal E_H^N.
\]
For non-cutoff Boltzmann, Lemma~\ref{lem:maxbridge} gives directly
\begin{equation}\label{eq:cut-bd}
 \langle\xi\rangle^m\mu_0(M^\e)^{-1}w_0^{-2}
 \lesssim e^{-c_q\langle\xi\rangle}.
\end{equation}
At zero spatial order,
$\Pi_M^\perp F^\e=\e\sqrt{\mu_0}\,g^\e$, so \eqref{eq:cut-bd} gives
\[
 \mathcal V_{m,\mathrm{hi}}^0
 \lesssim
 \frac{e^{-cR_\e}}{\e^2\kappa^{2c+2}}\mathcal E_H^N.
\]
For $|\alpha|\ge1$ we use the exact decomposition
\begin{equation*}
 \partial^\alpha\Pi_M^\perp F^\e
 =\e\sqrt{\mu_0}\,\partial^\alpha f^\e
 -\partial^\alpha M^\e.
\end{equation*}
For the $f^\e$ term, first split
\[
 \partial^\alpha f^\e
 =(I-\Pi_0)\partial^\alpha f^\e
  +\Pi_0\partial^\alpha f^\e.
\]
By \eqref{eq:cut-bd},
\begin{align*}
 &\int_{|\xi|>R_\e}\langle\xi\rangle^m
 \mu_0(M^\e)^{-1}
 |(I-\Pi_0)\partial^\alpha f^\e|^2\,d\xi \lesssim e^{-c_qR_\e}
 \|w_0(I-\Pi_0)\partial^\alpha f^\e\|_{L^2_\xi}^2.
\end{align*}
On the other hand, the range of $\Pi_0$ is spanned by fixed-degree
polynomials times $\sqrt{\mu_0}$.  Since \eqref{eq:mu-maxw} gives
$\mu_0(M^\e)^{-1}\lesssim1$, finite dimensionality and Gaussian decay give
\begin{align*}
 &\int_{|\xi|>R_\e}\langle\xi\rangle^m
 \mu_0(M^\e)^{-1}|\Pi_0\partial^\alpha f^\e|^2\,d\xi \lesssim e^{-cR_\e^2}
 \|\partial^\alpha f^\e\|_{L^2_\xi}^2.
\end{align*}
After integration in $x$, \eqref{eq:xi-def} therefore gives, uniformly for
$1\le |\alpha|\le N+1$,
\begin{equation}\label{eq:kap-alp}
 \kappa^{(|\alpha|-N)_+}
 \iint_{|\xi|>R_\e}\langle\xi\rangle^m
 \mu_0(M^\e)^{-1}|\partial^\alpha f^\e|^2\,d\xi dx
 \lesssim
 \frac{e^{-c_qR_\e}}{\e^2\kappa^{2c}}\mathcal E_H^N.
\end{equation}
For $|\alpha|\le N$ the factor $\e^{-2}$ on the right is an admissible
loss.  For $|\alpha|=N+1$, the microscopic part uses the top weighted term
$\kappa^{2c+1}\e^2\|w_0\partial^{N+1}f^\e\|_{{L^2_{x,\xi}}}^2$ in \eqref{eq:xi-def}; the
finite-dimensional part uses the accompanying unweighted top-order term.
For the Maxwellian term we use the chain rule \eqref{eq:par-id} itself rather than a norm
comparison.  Since the parameters of $M^\e$ range in a
fixed compact subset of
$(0,\infty)\times\mathbb R^3\times(0,\infty)$, there are fixed $c,C>0$
such that
\[
 M^\e(t,x,\xi)\,|\mathcal P_{\alpha_1,\ldots,\alpha_n}(c^\e)|^2
 \le C e^{-c|\xi|^2}
\]
for every polynomial occurring in \eqref{eq:par-id}.  The same
$H_x^2$ product estimate used in the proof of Lemma~\ref{lem:fluctxfer},
together with
$\e\|Z^\e\|_{H^N}=o(1)$, yields
\begin{align}\label{eq:kap-alp-bd}
 &\frac{\kappa^{(|\alpha|-N)_+}}{\e^2}
 \iint_{|\xi|>R_\e}\langle\xi\rangle^m
 |\partial^\alpha M^\e|^2(M^\e)^{-1}\,d\xi dx\notag\\
 &\qquad\lesssim
 e^{-cR_\e^2}
 \begin{cases}
  \|Z^\e\|_{H^N}^2,&1\le|\alpha|\le N,\\
  \kappa\|Z^\e\|_{H^{N+1}}^2+\|Z^\e\|_{H^N}^2,&|\alpha|=N+1,
 \end{cases}
 \lesssim e^{-cR_\e^2}\mathcal E^N
 \lesssim \e^2\mathcal E^N.
\end{align}
For the top order, the term in \eqref{eq:par-id} containing the unique
$N+1$ derivative is bounded by \eqref{eq:topmacro-bd}, namely
$\kappa\|Z^\e\|_{\dot H^{N+1}}^2\lesssim\mathcal E^N$; every other
product contains only derivatives of order at most $N$ and is bounded by
the same $H_x^2$ high--low product estimate.  The last inequality in
\eqref{eq:kap-alp-bd} holds for all sufficiently small $\e$, since
$R_\e=\e^{-1/(2(m+2))}$.
Combining \eqref{eq:kap-alp} and \eqref{eq:kap-alp-bd} gives
\[
 \mathcal V_{m,\mathrm{hi}}^\alpha
 \lesssim \e^2\mathcal E^N
 +\frac{e^{-cR_\e}}{\e^2\kappa^{2c}}\mathcal E_H^N
 \lesssim \e^2\mathcal E^N+\e^{10}\mathcal E_H^N.
\]
Indeed, \eqref{eq:choice-bd} gives
$\kappa\asymp\e^{(1-\mathfrak y)/c}$, so the prefactor is a fixed power
of $\e^{-1}$, whereas $e^{-cR_\e^2}$ and $e^{-cR_\e}$ are
$o(\e^m)$ for every $m>0$.
Combining this with \eqref{eq:lo-bd} proves
\eqref{eq:large-bd}--\eqref{eq:wt-bd}.
\end{proof}
\subsubsection{Local-Maxwellian source and defect estimates}
\begin{lemma}[Local-Maxwellian transport source]\label{lem:transsrc}
For the non-cutoff Boltzmann equation, assume \eqref{eq:tot-bd} and set
\[
 \mathcal T_\e:=\partial_t+\e^{-1}\xi\cdot\nabla_x.
\]
Then
\begin{align}\label{eq:mu-ltwo}
 &\left\|\mu_0^{-3/4}\mathcal T_\e M^\e\right\|_{L^2_{x,\xi}} +
 \sum_{|\beta|=1}
 \left\|\langle\xi\rangle w_0^{1/2}\partial_\beta
 \left(\mu_0^{-1/2}\mathcal T_\e M^\e\right)
 \right\|_{L^2_{x,\xi}}
 \lesssim
 (\mathcal E^N)^{1/2}+\e\kappa^{1/2}(\mathcal D^N)^{1/2},
\end{align}
\begin{equation}\label{eq:xi-mu}
 \left\|\langle\xi\rangle\mu_0^{-1/2}\mathcal T_\e M^\e\right\|_{L^2_{x,\xi}}
 +\left\|\langle\xi\rangle w_0\mu_0^{-1/2}\mathcal T_\e M^\e\right\|_{L^2_{x,\xi}}
 \lesssim
 (\mathcal E^N)^{1/2}+\e\kappa^{1/2}(\mathcal D^N)^{1/2}.
\end{equation}
\end{lemma}
The proof is given in Appendix~\ref{sec:auxana}.

\begin{lemma}[Perturbative Maxwellian defect operator]\label{lem:maxdefect}
Under \eqref{eq:tot-bd}, let $A^\e:=\frac{M^\e-\mu_0}{\sqrt{\mu_0}}.$
Then the non-cutoff defect operator in \eqref{eq:lin-def} satisfies
\begin{align}\label{eq:op-id}
 &|\langle\mathcal L_d^0g^\e,g^\e\rangle_{x,\xi}|
 +|\langle\mathcal L_d^0g^\e,w_0^2g^\e\rangle_{x,\xi}|
 +\sum_{|\beta|=1}|\langle\partial_\beta\mathcal L_d^0g^\e,
           w_0\partial_\beta g^\e\rangle_{x,\xi}| \notag\\
 &\qquad\lesssim
 \delta_\e(t)\left(
 \|g^\e\|_{L^2_{x,\Bb}}^2
 +\|w_0g^\e\|_{L^2_{x,\Bb}}^2
 +\sum_{|\beta|=1}
  \|w_0^{1/2}\partial_\beta g^\e\|_{L^2_{x,\Bb}}^2\right)  \lesssim
 o(1)\frac{\e^2}{\kappa^{2c-1}}\mathcal D_H^N.
\end{align}
\end{lemma}
\begin{proof}
By \eqref{eq:lin-def}, every term in $\mathcal L_d^0g^\e$ is a fixed-global
Boltzmann bilinear form with one input
$A^\e=(M^\e-\mu_0)/\sqrt{\mu_0}$ and the other $g^\e$.  The explicit
Maxwellian ratio, \eqref{eq:vd-bd}, and the Gaussian comparison give, uniformly
in $x$, for $a\in\{0,1/2,1\}$ and $|\beta|\le1$,
\[
 |w_0^a\partial_\beta A^\e|+|\nabla_\xi(w_0^a\partial_\beta A^\e)|
 \le C\delta_\e(t)e^{-c|\xi|^2}.
\]
Hence
\[
 \|w_0^a\partial_\beta A^\e\|_{L^2_{2s+\gamma}}
 +\|w_0^a\partial_\beta A^\e\|_{H^s_{2s+\gamma}}
 \lesssim\delta_\e(t):
\]
indeed the local fractional part is bounded by
$C\delta_\e(t)^2\int_0^1 \tau^{1-2s}\,d \tau$, while the $|\tau|>1$ part is Gaussian
integrable.  Applying Lemma~\ref{lem:collest} with weights $(0,0)$,
$(\ell,q)$, and $(\ell/2,q/2)$, and using \eqref{eq:xi-bd}, yields
\[
 |\langle\mathcal L_d^0g^\e,g^\e\rangle_{x,\xi}|
 +|\langle\mathcal L_d^0g^\e,w_0^2g^\e\rangle_{x,\xi}|
 +\sum_{|\beta|=1}|\langle\partial_\beta\mathcal L_d^0g^\e,
 w_0\partial_\beta g^\e\rangle_{x,\xi}|
 \lesssim\delta_\e(t)\,\mathfrak D_g,
\]
where $\mathfrak D_g$ is the three-term dissipation bracket in \eqref{eq:op-id}.
By \eqref{eq:hn-def},
$\mathfrak D_g\lesssim\e^2\kappa^{-(2c-1)}\mathcal D_H^N$, and
$\delta_\e=o(1)$ by \eqref{eq:glob-bd}; this is \eqref{eq:op-id}.
\end{proof}

\subsubsection{High-order global hierarchy}
\begin{proposition}[High-order moment estimate]
\label{prop:highord}
Under the bootstrap and scale assumptions
\eqref{eq:tot-bd}, \eqref{eq:choice-bd}, \eqref{eq:boot-bd}, and
\eqref{eq:choice-boot}, the Landau equation satisfies
\begin{equation}\label{eq:td-bd}
  \mathcal E_H^N(t)+ \int_0^t \mathcal D_H^N(s) d s
 \lesssim \mathcal{E}_H^N(0)+ \int_0^t \Big[ 
 \kappa^{2c-1}\bigl(\mathcal E^N+(\mathcal E^N)^2\bigr)
 +\frac{\e}{\kappa^c}\mathcal D_H^N(\mathcal E_H^N)^{1/2}\Big] ds.
\end{equation}
The non-cutoff Boltzmann equation satisfies
\begin{equation}\label{eq:td-en}
  \mathcal E_H^N(t)+ \int_0^t \mathcal D_H^N(s) ds
 \lesssim \mathcal{E}_H^N(0)+
 \int_0^t \Big[ o(1)\mathcal D^N
 +\kappa^{2c-1}\bigl(\e^{-2c_0}\mathcal E^N+(\mathcal E^N)^2\bigr)
 +\frac{\e}{\kappa^c}\mathcal D_H^N
 \bigl((\mathcal E_H^N)^{1/2}+(\mathcal E^N)^{1/2}\bigr)\Big] ds.
\end{equation}
\end{proposition}
\begin{proof}

We first treat Landau, and then
separate the positive and zero spatial orders for non-cutoff Boltzmann. For Landau, let $h^\e = (I-\Pi)f^\e, \ \mathcal{T}_\e = \p_t + \e^{-1}\xi \cdot \nabla_x.$ We use the full and microscopic equation
\begin{align*}
    &    \mathcal{T}_\e f^\e + \frac{1}{\kappa \e^2} \mathcal{L}f^\e = \frac{1}{\kappa \e}\Gamma(f^\e,f^\e), \\
    &   \mathcal{T}_\e h^\e + \frac{1}{\kappa \e^2} \mathcal{L}h^\e = \frac{1}{\kappa \e} \Gamma(f^\e,f^\e) + \frac{1}{\e} \Pi (\xi\cdot \nabla_x h^\e) - \frac{1}{\e} (I-\Pi) (\xi\cdot \nabla_x \Pi f^\e).
\end{align*}
We differentiate and test these equations using the five families below.
The first two rows have $|\alpha|\le N$, the next two have
$|\alpha|=N+1$, and the last has
$|\alpha|\le\lfloor N/2\rfloor$, $|\beta|=1$.
\begin{equation}
\begin{array}{c|c|c|c}
\text{family}&\text{test}&\text{energy multiplier}
&\text{nonlinear prefactor}\\ \hline
\text{spatial}&\partial^\alpha f^\varepsilon
&\kappa^{2c}&\kappa^{2c-1}/\varepsilon\\
\text{weighted spatial}&w^2\partial^\alpha h^\varepsilon
&\kappa^{2c}&\kappa^{2c-1}/\varepsilon\\
\text{top unweighted}&\partial^\alpha f^\varepsilon
&\kappa^{2c+1}&\kappa^{2c}/\varepsilon\\
\text{top weighted}&w^2\partial^\alpha f^\varepsilon
&\kappa^{2c+1}\varepsilon^2&\kappa^{2c}\varepsilon\\
\text{spatial--velocity}&w\partial_\beta^\alpha h^\varepsilon
&\kappa^{2c}&\kappa^{2c-1}/\varepsilon
\end{array}
\label{eq:high-energy-test}
\end{equation}
We will combine these identities with fixed small multipliers on the weighted lower-order and spatial--velocity families, so that their compact-velocity and lower-weight remainders are absorbed by the preceding coercive estimates.

The microscopic equation is used in the second and fifth rows.
Collision invariance removes the macroscopic part of the unweighted
collision testing factors. Set
\[
 m_*:=\left\lfloor\frac{N+1}{2}\right\rfloor,
 \qquad f_\nu^{\rm mac}:=\Pi\partial^\nu f^\e,
 \qquad f_\nu^{\rm mic}:=(I-\Pi)\partial^\nu f^\e.
\]
Since $m_*+2\le N$, Sobolev embedding in $x$, \eqref{eq:alp-pr}, and
\eqref{eq:land-id} give, for $|\nu|\le m_*$,
\begin{align}
 \|f_\nu^{\rm mac}\|_{L_x^\infty(L_\xi^2\cap L_{\Lb}^2)}
 &\lesssim (\mathcal E^N)^{1/2}, \quad
 \|f_\nu^{\rm mic}\|_{L_x^\infty L_\xi^2}
 \lesssim \kappa^{-c}(\mathcal E_H^N)^{1/2}, \quad
 \|f_\nu^{\rm mic}\|_{L_x^\infty L_{\Lb}^2}
 \lesssim \frac{\e}{\kappa^{(2c-1)/2}}(\mathcal D_H^N)^{1/2}.
\end{align}
The microscopic estimates remain valid with the weights from \eqref{eq:land-id}. The macroscopic bounds remain valid with these weights and one velocity
derivative, since the projection has finite-dimensional Gaussian range.
The corresponding $L_x^2$ bounds hold through order $N$; at order $N+1$,
\begin{align}\label{eq:mac-xi}
 \|f_{N+1}^{\rm mac}\|_{L_x^2(L_\xi^2\cap L_{\Lb}^2)}
 &\lesssim \kappa^{-1/2}(\mathcal E^N)^{1/2},\quad
 \|f_{N+1}^{\rm mic}\|_{L^2_{x,\Lb}}
 \lesssim \e\kappa^{-c}(\mathcal D_H^N)^{1/2},\quad
 \|w\partial^{N+1}f^\e\|_{L^2_{x,\Lb}}
 \lesssim \kappa^{-c}(\mathcal D_H^N)^{1/2}.
\end{align}

Expand the nonlinear term by the Leibniz rule and split each input into
its macroscopic and microscopic parts. For spatial derivatives, put the
factor of order at most $m_*$ in $L_x^\infty$ and the other in $L_x^2$.
Apply \eqref{eq:landau-tri}--\eqref{eq:landau-tri-wt}, using
\[
\|\mu^{1/32}h\|_{L^2_\xi}+\|h\|_{L_{-1}^2}\lesssim\|h\|_{L_\Lb^2},
\qquad
\|\vartheta_L h\|_{L^2_\xi}+\|\vartheta_L h\|_{L_\Lb^2}
\lesssim\|wh\|_{L_\Lb^2}.
\]
Here we note that
$\vartheta_L/w=e^{-q\langle\xi/a_L\rangle^2}$, and the second bound follows from differentiating this
Gaussian multiplier in the Landau norm.
Before multiplication by the nonlinear prefactors in \eqref{eq:high-energy-test}, the lower-order trilinear pairings are bounded, respectively, by the following macro--macro, mixed, and micro--micro expressions:
\begin{align}\label{eq:mm-def}
 \frac{\e}{\kappa^{(2c-1)/2}}
 \mathcal E^N(\mathcal D_H^N)^{1/2}, \qquad
 \frac{\e^2}{\kappa^{2c-1}}
 (\mathcal E^N)^{1/2}\mathcal D_H^N, \qquad
 \frac{\e^2}{\kappa^{3c-1}}
 \mathcal D_H^N(\mathcal E_H^N)^{1/2}.
\end{align}
For the spatial--velocity family, use \eqref{eq:landau-slot-wt} and \eqref{eq:landau-space-velocity}. Put the no-velocity-derivative input in
$L_x^\infty$ and the velocity-differentiated microscopic input in
$L_x^2$. The same three bounds follow.

At top order, use \eqref{eq:mac-xi}. When a top microscopic input is estimated by energy, pair it
with a lower microscopic collision norm; use the opposite allocation
in the other branch. The three bounds for the unweighted top pairing are
\[
\frac{\varepsilon}{\kappa^{c+1/2}}\mathcal{E}^N (\mathcal{D}_H^N)^{1/2},\qquad
\frac{\varepsilon^2}{\kappa^{2c}}(\mathcal{E}^N)^{1/2} \mathcal{D}_H^N,\qquad
\frac{\varepsilon^2}{\kappa^{3c}} \mathcal{D}_H^N (\mathcal{E}_H^N)^{1/2},
\]
whereas for the weighted top pairing they are
\[
\frac{1}{\kappa^{c+1/2}}\mathcal{E}^N (\mathcal{D}_H^N)^{1/2},\qquad
\frac{1}{\kappa^{2c}}(\mathcal{E}^N)^{1/2} \mathcal{D}_H^N,\qquad
\frac{1}{\kappa^{3c}}\mathcal{D}^N_H (\mathcal{E}_H^N)^{1/2}.
\]
After multiplication by the prefactors in \eqref{eq:high-energy-test}, the
macro--macro terms are bounded by
$C\kappa^{c-1/2}\mathcal{E}^N (\mathcal{D}_H^N)^{1/2}$, the mixed terms by
$C\varepsilon (\mathcal{E}^N)^{1/2} \mathcal{D}^N_H$, and the micro--micro terms by
$C\varepsilon\kappa^{-c} \mathcal{D}_H^N (\mathcal{E}_H^N)^{1/2}$.
Let $\mathfrak N_L$ denote the sum of the absolute nonlinear
contributions with prefactors. Young's inequality gives
\begin{equation}
\mathfrak N_L
\le C_\eta\kappa^{2c-1} (\mathcal{E}^N)^2
+(\eta+o(1)) \mathcal{D}_H^N
+C\frac{\varepsilon}{\kappa^c} \mathcal{D}_H^N (\mathcal{E}_H^N)^{1/2},
\label{eq:highmoment-landau-nonlinear}
\end{equation}
where $\varepsilon (\mathcal{E}^N){1/2}=o(1)$ follows from the bootstrap.

It remains to consider the linear operator and the remainder term in the equation of $h^\e$. For the linear operator estimates, use \eqref{eq:lin-bd}, \eqref{eq:top-bd}, and \eqref{eq:vel-bd} for the spatial families.
For the spatial--velocity family, sum over
$|\alpha|\le\lfloor N/2\rfloor$, $|\beta|=1$ before applying \eqref{eq:alp-bet}.
The top compact-velocity remainder is estimated by
\begin{align}
\kappa^{2c}\sum_{|\alpha|=N+1}
\|\partial^\alpha f^\varepsilon\|_{L^2(\mathbb R_x^3\times B_R)}^2
&\lesssim\kappa^{2c}\sum_{|\alpha|=N+1}
\left(
\|\partial^\alpha h^\varepsilon\|_{L_x^2L_\Lb^2}^2
+\|\Pi\partial^\alpha f^\varepsilon\|_{L^2_{x,\xi}}^2
\right)
 \lesssim\varepsilon^2 \mathcal{D}_H^N + \kappa^{2c-1}\mathcal{E}^N,
\label{eq:highmoment-top-compact}
\end{align}
by \eqref{eq:alp-pr} and the dissipation definition.
For the projection--transport terms, Young's inequality leaves a small
multiple of $\mathcal{D}_H^N$ and terms bounded by
\begin{align*}
&\kappa^{2c+1}\sum_{|\alpha|\le N}
\left[
\|\Pi(\xi\cdot\nabla_x\partial^\alpha h^\varepsilon)\|_{L^2_{x,\xi}}^2
+\|(I-\Pi)(\xi\cdot\nabla_x\Pi\partial^\alpha f^\varepsilon)\|_{L^2_{x,\xi}}^2
\right]
\\
&\qquad\lesssim
\kappa^{2c+1}
\left(
\|h^\varepsilon\|_{H_x^{N+1}L_\Lb^2}^2
+\|\Pi f^\varepsilon\|_{H_x^{N+1}L_\xi^2}^2
\right)
\lesssim o(1)\mathcal{D}_H^N + \kappa^{2c-1}\mathcal{E}^N.
\end{align*}
The same estimate applies after one velocity derivative to these
finite-dimensional Gaussian terms. The additional commutator
$[\partial_{\xi_j},\mathcal T_\varepsilon]
=\varepsilon^{-1}\partial_{x_j}$ contributes at most
\[
C\varepsilon\kappa \mathcal{D}_H^N
+C\kappa^{c+1/2}(\mathcal{E}^N \mathcal{D}_H^N)^{1/2}
\lesssim o(1)\mathcal{D}_H^N + C\kappa^{2c-1}\mathcal{E}^N,
\]
using the exponential weight gap and Young's inequality.

Combine the five identities with fixed multipliers
$0<\delta_2\ll\delta_1\ll1$ on the spatial--velocity and weighted
lower-order families, respectively. The lower-order compact and
lower-weight remainders are absorbed successively; the top remainder
is controlled by \eqref{eq:highmoment-top-compact}.
Next choose $\eta$ small in \eqref{eq:highmoment-landau-nonlinear} and decrease
$\varepsilon_0$. The resulting energy
$\widetilde{\mathcal E}_{H,L}^N$ is uniformly equivalent to
$\mathcal E_H^N$, and for some $\lambda_L>0$,
\begin{equation}
\frac{d}{dt}\widetilde{\mathcal E}_{H,L}^N+\lambda_L \mathcal{D}_H^N
\le C\kappa^{2c-1}(\mathcal{E}^N+(\mathcal{E}^N)^2)
+C\frac{\varepsilon}{\kappa^c}\mathcal{D}_H^N (\mathcal{E}_H^N)^{1/2}.
\label{eq:highmoment-landau-closed}
\end{equation}
Integration and energy equivalence prove \eqref{eq:td-bd}.

\textbf{Non-cutoff Boltzmann}. First consider lower spatial families  $1\le|\alpha|\le N$
and the mixed family $1\le|\alpha|\le\lfloor N/2\rfloor$. Use the same five families \eqref{eq:high-energy-test} with
$(L,\Pi,w,\Gamma)$ replaced by $(L_0,\Pi_0,w_0,\Gamma_0)$.
These are precisely the components of $\mathcal E_{H,2}^N$ in \eqref{eq:xi-def}.
Apply \eqref{eq:boltz-tri}--\eqref{eq:boltz-tri-wt}, with the $(\ell/2,q/2)$ specialization for one
velocity derivative. In a mixed macroscopic--microscopic contribution,
choose the minimum bracket so that the $L_{2s+\gamma}^2$ norm falls on
the microscopic input. The extra weighted term is controlled by \eqref{eq:boltz-wt-coer},
and \eqref{eq:boltz-wt-hier} controls the lower weight in the second collision slot.

For positive-order inputs, \eqref{eq:alp-pr-bd} and \eqref{eq:xi-def}--\eqref{eq:hn-def} give the same
high--low bounds as above. The input carrying no spatial derivative
requires the decomposition \eqref{eq:maxw-def}:
\[
f^\varepsilon=g^\varepsilon+H_\infty^\varepsilon
+\widetilde H^\varepsilon.
\]
Here $H_\infty^\varepsilon$ is independent of $x$ and therefore occurs
only in such an input. By \eqref{eq:hinfty-bd}--\eqref{eq:der-alp-bd}, \eqref{eq:alp-pr-bd}, and Sobolev embedding,
\begin{align}
\sum_{|\beta|\le1}
\|w_0^{1-|\beta|/2}\partial_\beta f^\varepsilon\|_{L_x^\infty L_\xi^2}
&\lesssim\varepsilon^{-c_0}
+\kappa^{-c}((\mathcal{E}^N)^{1/2}+(\mathcal{E}_H^N)^{1/2}),
\nonumber\\
\|\Pi_0f^\varepsilon\|_{L_x^\infty L_\xi^2}
&\lesssim\varepsilon^{-c_0}+(\mathcal{E}^N)^{1/2},
\nonumber\\
\sum_{|\beta|\le1}
\|w_0^{1-|\beta|/2}\partial_\beta f^\varepsilon\|_{L_x^\infty L_\Bb^2}
&\lesssim\varepsilon^{-c_0} + (\mathcal{E}^N)^{1/2}
+\frac{\varepsilon}{\kappa^{c-1/2}}(\mathcal{D}^N_H)^{1/2}.
\label{eq:highmoment-boltzmann-low}
\end{align}
The $\varepsilon^{-c_0}$ term comes only from $H_\infty^\varepsilon$.
After the lower-order nonlinear prefactor, its contribution paired
with a macroscopic differentiated input is bounded by
\begin{equation}
C\varepsilon^{-c_0}\kappa^{c-1/2}(\mathcal{E}^N (\mathcal{D}^N_H))^{1/2}
\le\eta \mathcal{D}_H^N + C_\eta\kappa^{2c-1}\varepsilon^{-2c_0}\mathcal{E}^N.
\label{eq:highmoment-homogeneous-term}
\end{equation}
The top unweighted prefactor gives the same bound; the weighted top
prefactor gives an additional factor $\varepsilon$.
When the other input is microscopic, the corresponding contribution
is at most $C\varepsilon^{1-c_0}\mathcal{D}^N_H=o(1)\mathcal{D}^N_H$.
The remaining terms satisfy the same estimates in the Landau case, with the microscopic
energy factor replaced by $(\mathcal{E}_H^N)^{1/2}+ (\mathcal{E}^N)^{1/2}$.

Combining the positive-order identities with the same fixed
$\delta_1,\delta_2$ hierarchy therefore gives an equivalent energy
$\widetilde{\mathcal E}_{H,2}^N\simeq\mathcal E_{H,2}^N$ such that
\begin{align}
\frac{d}{dt}\widetilde{\mathcal E}_{H,2}^N
+\lambda_2\mathcal D_{H,2}^N
\le{}&C_\eta\kappa^{2c-1}
\left(\varepsilon^{-2c_0}\mathcal{E}^N + (\mathcal{E}^N)^2\right)
+(\eta+o(1))\mathcal{D}^N_H
 +C\frac{\varepsilon}{\kappa^c}\mathcal{D}_H^N
\left((\mathcal{E}_H^N)^{1/2}+(\mathcal{E}^N)^{1/2}\right).
\label{eq:highmoment-boltzmann-positive}
\end{align}
We retain the small multiple of the full $\mathcal{D}_H^N$ until the zero-order estimate is added.

Finally, consider the zero spatial order. From $F^\varepsilon=M^\varepsilon+\varepsilon\sqrt{\mu_0}g^\varepsilon$
and \eqref{eq:coll-id},
\begin{equation}
\mathcal T_\varepsilon g^\varepsilon
+\frac{1}{\kappa\varepsilon^2}(\mathcal{L}_0+\mathcal{L}_d^0)g^\varepsilon
=-\frac{1}{\varepsilon}\mu_0^{-1/2}
  \mathcal T_\varepsilon M^\varepsilon
+\frac{1}{\kappa\varepsilon}
  \Gamma_0(g^\varepsilon,g^\varepsilon).
\label{eq:highmoment-g-equation}
\end{equation}
Test the unweighted, $w_0^2$-weighted, and one-velocity-derivative
equations with energy multiplier $\kappa^{2c}$.
Set
\begin{align*}
D_g&:=\|g^\varepsilon\|_{L_x^2L_\Bb^2}
+\|w_0g^\varepsilon\|_{L_x^2L_\Bb^2}
+\sum_{|\beta|=1}
 \|w_0^{1/2}\partial_\beta g^\varepsilon\|_{L_x^2L_\Bb^2},
\\
E_g&:=\|g^\varepsilon\|_{H_x^2L_\xi^2}
+\|w_0g^\varepsilon\|_{H_x^2L_\xi^2}
+\sum_{|\beta|=1}
 \|w_0^{1/2}\partial_\beta g^\varepsilon\|_{H_x^2L_\xi^2}.
\end{align*}
The positive-order transfer estimates in Lemma \ref{lem:fluctxfer}, and \eqref{eq:alp-pr-bd} give
\begin{equation}
D_g\lesssim\frac{\varepsilon}{\kappa^{c-1/2}}(\mathcal{D}_H^N)^{1/2},
\qquad
E_g\lesssim (\mathcal{E}^N)^{1/2}+\kappa^{-c}(\mathcal{E}_H^N)^{1/2}.
\label{eq:highmoment-g-norms}
\end{equation}

After multiplication by $\kappa^{2c-1}/\varepsilon^2$,
Lemma \ref{lem:maxdefect} bounds the defect contributions by $o(1)\mathcal{D}_H^N$.
Since $\gamma+2s\le0$, \eqref{eq:boltz-tri}--\eqref{eq:boltz-wt-hier} and Sobolev embedding imply
\begin{align*}
&|\langle\Gamma_0(g^\varepsilon,g^\varepsilon),g^\varepsilon\rangle_{x,\xi}|
+|\langle\Gamma_0(g^\varepsilon,g^\varepsilon),w_0^2g^\varepsilon\rangle_{x,\xi}|
 +
\sum_{|\beta|=1}
|\langle\partial_\beta\Gamma_0(g^\varepsilon,g^\varepsilon),
 w_0\partial_\beta g^\varepsilon\rangle_{x,\xi}|
\lesssim D_g^2E_g.
\end{align*}
Consequently, the nonlinear contribution, including its prefactor, is
bounded by
\begin{equation}
\frac{\kappa^{2c-1}}{\varepsilon}D_g^2E_g
\lesssim\frac{\varepsilon}{\kappa^c}\mathcal{D}_H^N
\left((\mathcal{E}_H^N)^{1/2}+(\mathcal{E}^N)^{1/2}\right).
\label{eq:highmoment-g-nonlinear}
\end{equation}

It remains to consider the source and commutator terms.
Let $S_g$ denote the sum of the absolute Maxwellian-source
pairings with the energy prefactor.
Lemma \ref{lem:transsrc}, \eqref{eq:xi-bd}, and \eqref{eq:highmoment-g-norms} give
\begin{align}
S_g
&\lesssim\frac{\kappa^{2c}}{\varepsilon}
\left((\mathcal{E}^N)^{1/2}+\varepsilon\kappa^{1/2}(\mathcal{D}^N)^{1/2}\right)D_g
\nonumber\\
&\lesssim\kappa^{c+1/2}(\mathcal{E}^N \mathcal{D}_H^N)^{1/2}
+\varepsilon\kappa^{c+1}(\mathcal{D}^N \mathcal{D}_H^N)^{1/2}
 \le\eta \mathcal{D}_H^N + C_\eta\kappa^{2c+1} \mathcal{E}^N
+C_\eta\varepsilon^2\kappa^{2c+2} \mathcal{D}^N.
\label{eq:highmoment-g-source}
\end{align}
For fixed $\eta$, the last term is $o(1)\mathcal{D}^N$.
The commutator
$[\partial_{\xi_j},\mathcal T_\varepsilon]g^\varepsilon
=\varepsilon^{-1}\partial_{x_j}g^\varepsilon$
is controlled using
\[
\|w_0^{3/4}h\|_{L^2_\xi}\lesssim\|w_0h\|_{L_\Bb^2},\qquad
\|w_0^{1/4}h\|_{L^2_\xi}\lesssim\|w_0^{1/2}h\|_{L_\Bb^2},
\]
and the positive-order transfer estimate in Lemma \ref{lem:fluctxfer}. Thus
\begin{align}
\frac{\kappa^{2c}}{\varepsilon}
\sum_{j=1}^3
|\langle\partial_{x_j}g^\varepsilon,
 w_0\partial_{\xi_j}g^\varepsilon\rangle|
&\lesssim\varepsilon\kappa \mathcal{D}_H^N
+\kappa^{c+1/2}(\mathcal{E}^N \mathcal{D}_H^N)^{1/2}
 \le(\eta+o(1))\mathcal{D}_H^N + C_\eta\kappa^{2c+1}\mathcal{E}^N.
\label{eq:highmoment-g-transport}
\end{align}

Since $\Pi_0g^\varepsilon=0$ by \eqref{eq:pr-id}, combine the unweighted identity,
a fixed small $\delta_1$ multiple of the weighted identity, and a fixed
$0<\delta_2\ll\delta_1$ multiple of the summed
one-velocity-derivative identities. Estimates \eqref{eq:lin-bd}, \eqref{eq:top-bd}, and \eqref{eq:alp-bet}
absorb the compact and lower-weight remainders.
Using \eqref{eq:highmoment-g-nonlinear}--\eqref{eq:highmoment-g-transport}, and renaming
the free Young parameters, yields an equivalent energy
$\widetilde{\mathcal E}_{H,1}^N\simeq\mathcal E_{H,1}^N$ satisfying
\begin{align}
\frac{d}{dt}\widetilde{\mathcal E}_{H,1}^N
+\lambda_1\mathcal D_{H,1}^N
\le{}&(\eta+o(1))\mathcal{D}^N_H+o(1) \mathcal{D}^N
+C_\eta\kappa^{2c+1} \mathcal{E}^N
 +C\frac{\varepsilon}{\kappa^c} \mathcal{D}_H^N
\left((\mathcal{E}_H^N)^{1/2}+(\mathcal{E}^N)^{1/2}\right).
\label{eq:highmoment-boltzmann-zero}
\end{align}

Finally, add \eqref{eq:highmoment-boltzmann-positive} and
\eqref{eq:highmoment-boltzmann-zero}. Choose
$\eta<\tfrac14\min\{\lambda_1,\lambda_2\}$ and then decrease
$\varepsilon_0$ to absorb the small multiples of
$\mathcal{D}_H^N = \mathcal D_{H,1}^N+\mathcal D_{H,2}^N$.
Since
\[
\kappa^{2c+1} \mathcal{E}
\le\kappa^{2c-1}\varepsilon^{-2c_0} \mathcal{E},
\qquad 0<\kappa,\varepsilon\le1,
\]
integration and the uniform equivalence
$\widetilde{\mathcal E}_{H,1}^N+
\widetilde{\mathcal E}_{H,2}^N\simeq\mathcal E_H^N$
prove \eqref{eq:td-en}. 
\end{proof}

\subsubsection{Local-Maxwellian energy closure}
The global hierarchy above controls the fixed-Maxwellian velocity weights.
We next derive the intrinsic energy associated with the moving local Maxwellian.
The lower-order and top-order estimates are stated separately.

\begin{lemma}[Lower local-Maxwellian energy block]\label{lem:lowerlocalblock}
Set
\begin{equation}\label{eq:hg-local}
 \mathcal H_G^N(t):=\sum_{|\alpha|\le N}\frac1{2\e^2}
 \iint\frac{|\partial^\alpha G^\e|^2}{M^\e}\,d\xi dx
\end{equation}
and
\begin{equation}\label{eq:mn-local}
 \mathfrak E_{\rm mac}^N(t):=\frac12\sum_{|\alpha|\le N}\int
 \left(|\partial^\alpha\rho^\e|^2+(k_B\Theta^\e)^{-1}
 |\partial^\alpha u^\e|^2+\frac32|\partial^\alpha\theta^\e|^2\right)dx.
\end{equation}
Then $\mathfrak E_{\rm mac}^N\simeq\mathcal E_M^N$ uniformly on the bootstrap interval,
and for every fixed $\eta\in(0,1)$,
\begin{align}\label{eq:lower-local-block}
 \frac d{dt}(\mathfrak E_{\rm mac}^N+\mathcal H_G^N)+c_1\mathcal D_G^N
 &\le \eta\mathcal D_{top}^N
 +C_\eta(1+\|\nabla_xZ^\e\|_\infty^2)\mathcal E^N
 +o(1)\mathcal D_H^N\notag\\
 &\quad+C_\eta\e^2(\mathcal E_H^N+1)^2
 +C_\eta\frac{\e^2}{\kappa}(\mathcal E^N+1)^3,
 \qquad Z^\e=(\rho^\e,u^\e,\theta^\e).
\end{align}
\end{lemma}
\begin{proof}
Spatial differentiation of
\eqref{eq:mic-id} gives, for $|\alpha|\le N$,
\begin{align}
 \mathcal T_\e\partial^\alpha G^\e
 +\frac1{\kappa\e^2}\mathcal L_M(\partial^\alpha G^\e)
 &=-\partial^\alpha\mathcal A^\e+\partial^\alpha\mathcal F^\e\notag\\
 &\quad+\frac2{\kappa\e^2}\sum_{0<\beta\le\alpha}\binom\alpha\beta
 \mathcal N(\partial^\beta M^\e,\partial^{\alpha-\beta}G^\e)\notag\\
 &\quad+\frac1{\kappa\e^2}\sum_{\beta\le\alpha}\binom\alpha\beta
 \mathcal N(\partial^\beta G^\e,\partial^{\alpha-\beta}G^\e).
\end{align}
Testing by $\partial^\alpha G^\e/(\e^2M^\e)$ and summing gives
\begin{equation}\label{eq:micro-local-exact}
 \frac d{dt}\mathcal H_G^N
 +\sum_{|\alpha|\le N}\frac1{\kappa\e^4}
 \iint\mathcal L_M(\partial^\alpha G^\e)
 \frac{\partial^\alpha G^\e}{M^\e}
 =J_0+J_A+J_F+J_{MG}+J_{GG},
\end{equation}
where
\begin{align*}
 J_0&:=-\sum_{|\alpha|\le N}\frac1{2\e^2}
 \iint \mathcal T_\e M^\e\frac{|\partial^\alpha G^\e|^2}{(M^\e)^2},\\
 J_A&:=-\sum_{|\alpha|\le N}\frac1{\e^2}
 \iint\partial^\alpha\mathcal A^\e\frac{\partial^\alpha G^\e}{M^\e},
 \qquad
 J_F:=\sum_{|\alpha|\le N}\frac1{\e^2}
 \iint\partial^\alpha\mathcal F^\e\frac{\partial^\alpha G^\e}{M^\e},\\
 J_{MG}&:=\sum_{|\alpha|\le N}\frac2{\kappa\e^4}
 \sum_{0<\beta\le\alpha}\binom\alpha\beta
 \iint\mathcal N(\partial^\beta M^\e,\partial^{\alpha-\beta}G^\e)
 \frac{\partial^\alpha G^\e}{M^\e},\\
 J_{GG}&:=\sum_{|\alpha|\le N}\frac1{\kappa\e^4}
 \sum_{\beta\le\alpha}\binom\alpha\beta
 \iint\mathcal N(\partial^\beta G^\e,\partial^{\alpha-\beta}G^\e)
 \frac{\partial^\alpha G^\e}{M^\e}.
\end{align*}
The identity for $J_0$ follows exactly from differentiating the moving weight:
\[
 \frac1{\e^2}\iint\mathcal T_\e(\partial^\alpha G^\e)
 \frac{\partial^\alpha G^\e}{M^\e}
 =\frac d{dt}\frac1{2\e^2}\iint\frac{|\partial^\alpha G^\e|^2}{M^\e}
 +\frac1{2\e^2}\iint\mathcal T_\e M^\e
 \frac{|\partial^\alpha G^\e|^2}{(M^\e)^2}.
\]
Because every $\partial^\alpha G^\e$ has zero local mass, momentum, and energy
moments, the coercivity estimate \eqref{eq:mh-bd} makes the second term on the
left of \eqref{eq:micro-local-exact} bounded below by $c\mathcal D_G^N$.

For the macroscopic partner, define the macroscopic stress--heat-flux pairing by
\begin{align*}
\bega
\mathfrak{S}_{\alpha}(t) &:= -\frac{1}{\e^2} \int_{\Omega} \p^{\alpha}\bigg[\frac{1}{k_B\mathrm{P}^{\e} \mathrm{\Theta}^{\e}} \nabla_x \cdot \mathbf{r}^{\e}\bigg] \cdot \p^{\alpha}u^{\e} \dd x  -\frac{1}{\e^2}\int_{\Omega} \p^{\alpha}\bigg[\frac{1}{k_B\mathrm{P}^{\e}\mathrm{\Theta}^{\e}} \Big(\nabla_x\cdot \mathfrak{q}^{\e} + \nabla_x \mathrm{U}^{\e} : \mathbf{r}^{\e}\Big)\bigg] \p^{\alpha}\theta^{\e} \dd x .
\enda
\end{align*}
Applying $\partial^\alpha$ to the three balance laws
\eqref{eq:tu-id}, testing with
$(\partial^\alpha\rho^\e, \partial^\alpha u^\e,
\partial^\alpha\theta^\e)$ and summing gives
\begin{align*}
    &   \frac{d}{dt} \mathfrak{E}_{\rm mac}^N = \sum_{0\leq |\alpha|\leq N} \mathfrak{S}_\alpha + \mathcal{R}_{\rm mac}.
\end{align*}  
The two $\e^{-1}$ acoustic pairings cancel
exactly after integration by parts. Writing \(b^\varepsilon=(k_B\Theta^\varepsilon)^{-1}\) and \(D_t=\partial_t+u^\varepsilon\cdot\nabla_x\), the remainder $\mathcal{R}_{\rm mac}$ consists of the transport commutators, the commutator $[\p^\alpha, b^\e]D_t u^\e$ and the velocity-weight contribution
\begin{align*}
    &  \frac{1}{2} \sum_{|\alpha|\leq N} \int_{\R^3} (D_t b^\e + b^\e \nabla_x \cdot u^\e)|\p^\alpha u^\e|^2 dx.
\end{align*}
The temperature equation controls \(D_tb^\varepsilon=-\varepsilon b^\varepsilon D_t\theta^\varepsilon\). For the commutator term, use $\nabla_x  b^\e= -\e b^\e \nabla_x \theta^\e$ and the equation of $D_t u^\e$ in \eqref{eq:tu-id}. Thus the factor \(\varepsilon\) from differentiating $b^\e$ compensates for the singular acoustic scale. Lemma \ref{lem:macrocoef} and Lemma~\ref{lem:stress} and the Sobolev commutator estimates therefore give
\begin{align*}
     \mathcal{R}_{\rm mac} &\lesssim \Vert \nabla_x Z^\e \Vert_\infty \mathcal{E}_M^N + \e \kappa^{1/2} \mathcal{E}_M^N (\mathcal{D}_G^N)^{1/2}  \leq C(1+ \Vert \nabla_x Z^\e\Vert_\infty^2) \mathcal{E}_M^N + \eta \mathcal{D}_G^N + C_\eta \e^2 \kappa (\mathcal{E}^N+1)^2.
\end{align*}
Collecting the estimates above concludes
\begin{equation}\label{eq:macro-local-rem}
 \frac d{dt}\mathfrak E_{\rm mac}^N
 \le \sum_{|\alpha|\le N}\mathfrak S_\alpha
 +C(1+\|\nabla_xZ^\e\|_\infty^2)\mathcal E_M^N
 +\eta\mathcal D_G^N+C_\eta\e^2\kappa(\mathcal E^N+1)^2.
\end{equation}

We next combine $\sum_\alpha \mathfrak{S}_\alpha$ with $J_A$. Exact moment matching \eqref{eq:xd-id} and \eqref{eq:ij-mic} give
\begin{align*}
    &   \int_{\R^3} \mathfrak{R}_{ij}^\e \p^\alpha G^\e d \xi = \p^\alpha \mathbf{r}_{ij}^\e, \ \ \int_{\R^3} \mathcal{Q}_i^\e \p^\alpha G^\e d \xi = \p^\alpha \mathfrak{q}_i^\e + \sum_{0<\beta\leq \alpha} \begin{pmatrix}
        \alpha \\
        \beta
    \end{pmatrix} \sum_k \p^\beta \mathrm{U}^\e_k \p^{\alpha-\beta} \mathbf{r}_{ik}^\e.
\end{align*}
Isolate the terms in \(J_A\) where all derivatives fall on
\(\nabla_xu^\varepsilon\) or \(\nabla_x\theta^\varepsilon\), and integrate the principal macroscopic flux terms by parts. The principal microscopic and macroscopic flux pairings cancel up to the density mismatch \((P^\varepsilon)^{-1}-1\), which supplies an additional factor \(\varepsilon\):
\begin{align*}
    &    J_A + \sum_{|\alpha|\leq N} \mathfrak{S}_\alpha = \frac{1}{\e^2} \sum_{|\alpha|\leq N} \int_{\R^3} \frac{1}{k_B \Theta^\e} ((\mathrm{P}^\e)^{-1}-1) (\p^\alpha \nabla_x u^\e : \p^\alpha \mathbf{r}^\e + \p^\alpha \nabla_x \theta^\e \cdot \p^\alpha \mathfrak{q}^\e) dx + \mathcal{R}_{\rm flux}.
\end{align*}
Applying Lemma \ref{lem:macrocoef}, Lemma \ref{lem:stress} and Lemma \ref{lem:topmacro} controls the first term by 
\begin{align*}
    & \frac{1}{\e^2} \Vert (\mathrm P^\e)^{-1} -1\Vert_\infty \Vert (\mathbf{r}^\e, \mathfrak{q}^\e)\Vert_{H^N} \Vert (u^\e,\theta^\e)\Vert_{H^{N+1}} \lesssim \eta \mathcal{D}_G^N + C_\eta \e^2 (\mathcal{E}^N)^2.
\end{align*}
$\mathcal{R}_{\rm flux}$ contains the differentiated-coefficient terms, the heat-flux commutator displayed above, and the stress-work contribution involving \(\nabla_xU^\varepsilon:r^\varepsilon\). By the same Sobolev product argument: differentiated Maxwellian coefficients and positive derivatives of \(U^\varepsilon=\varepsilon u^\varepsilon\) supply an additional factor \(\varepsilon\); moment derivatives are controlled by Lemma \ref{lem:stress}, and any \(N+1\) macroscopic derivative by Lemma \ref{lem:topmacro}. Thus, $     \mathcal{R}_{\rm flux} \lesssim \e \mathcal{E}^N (\mathcal{D}_G^N)^{1/2}.$ Consequently,
\begin{equation}\label{eq:flux-cancel-local}
 \left|J_A+\sum_{|\alpha|\le N}\mathfrak S_\alpha\right| \lesssim \eta \mathcal{D}^N + C_\eta \e^2 (\mathcal{E}^N)^2
 \le \eta\mathcal D^N+C_\eta\frac{\e^2}{\kappa}(\mathcal E^N+1)^3.
\end{equation}

It remains to check that the moving-weight, source, and collision remainders
have precisely the scales retained in \eqref{eq:lower-local-block}. For the low-order microscopic factors that will be placed in \(L^\infty_x\) in the collision pairings, we introduce quantities \(X_0\),\(X_D\) and collect their \(H^2_x\)-based energy and dissipation bounds. Denote
$m_*:=\lfloor(N+1)/2\rfloor$ and
\begin{align}
 X_0&:=\sum_{|\delta|\le m_*}
 \left\|\frac{\partial^\delta G^\e}{\sqrt{M^\e}}\right\|_{H_x^2L_\xi^2},\\
 X_D&:=\sum_{|\delta|\le m_*}
 \left\|\frac{\partial^\delta G^\e}{\sqrt{M^\e}}\right\|_{H_x^2L_{\mathbf D}^2}.
\end{align}
The non-trivial terms arise when the spatial derivative acts on $(M^\e)^{-1/2}$ and the Maxwellian chain rule \eqref{eq:par-id} is needed. With the macroscopic coefficients bound in Lemma \ref{lem:macrocoef}, 
\begin{align*}
    & X_0 \lesssim \e \sum_{|\delta|\le m_*+2}(\mathcal{V}_0^\delta)^{\frac{1}{2}} + \e^{3/2}  \sum_{|\delta|\leq m_*+1} (\mathcal{V}_4^\delta)^{1/2} + \e^{3/2} \sum_{|\delta|\leq m_* } (\mathcal{V}_8^\delta)^{1/2}.
\end{align*}
The large-velocity
bounds \eqref{eq:large-bd}--\eqref{eq:wt-bd} imply
\begin{align}
 X_0&\lesssim \e^{3/4}(\mathcal E^N)^{1/2}
             +\e^5(\mathcal E_H^N)^{1/2},\notag\\
    X_0         &  \lesssim \e^2 (\mathcal{E}^N)^{1/2} + \e^{7/4}\kappa^{1/2} (\mathcal{E}^N + (\mathcal{D}^N)^{1/2}) + \e^5 (\mathcal{E}_H^N)^{1/2} \lesssim \e^{2} (1+ \mathcal{E}^N + (\mathcal{D}^N)^{1/2}) + \e^5 (\mathcal{E}_H^N)^{1/2}  \label{eq:x0xd-bd-local} .
\end{align}
The first bound for \(X_0\) is used to obtain small coefficients multiplying dissipation terms. The second retains the stronger \(\varepsilon\)-gain needed when \(X_0\) is accompanied by a singular prefactor.

For $X_D$, the macroscopic coefficients bound in Lemma \ref{lem:macrocoef} leads to
\begin{align}
    &  \Big\Vert  \frac{\p^{\delta} G^\e}{|M^\e|^{\frac{1}{2}}} \Big\Vert_{H^2_x L^2_\Db}  \lesssim   \underbrace{\sum_{|\gamma|=0}^2\Big\Vert  \frac{\p^{\delta+\gamma} G^\e}{|M^\e|^{\frac{1}{2}}} \Big\Vert_{L^2_x L^2_\Db}}_{\eqref{eq:xd-h2-split}_1} + \underbrace{\e|\mathcal{E}_M^{N}|^{\frac{1}{2}}  \sum_{|\gamma|=0}^1\Big\Vert \langle \xi\rangle^4  \frac{\p^{\delta+\gamma} G^\e}{|M^\e|^{\frac{1}{2}}} \Big\Vert_{L^2_x L^2_\Db}}_{\eqref{eq:xd-h2-split}_2} . \label{eq:xd-h2-split}
\end{align}
For the Landau equation, applying the global Maxwellian expansion \eqref{eq:dist-id},
\begin{align*}
    & \eqref{eq:xd-h2-split}_2 \lesssim \e (\mathcal{E}_M^{N})^{\frac{1}{2}} \sum_{|\gamma|=0}^1\Big\Vert \langle \xi\rangle^4  \frac{\p^{\delta+\gamma} (M^\e-\mu-\e\sqrt{\mu}f^\e)}{|M^\e|^{\frac{1}{2}}} \Big\Vert_{L^2_x L^2_\Lb} \lesssim \e^2 \mathcal{E}_M^{N}  + \e^2 (\mathcal{E}_M^{N})^{\frac{1}{2}} \sum_{|\gamma|=0}^1 \Big\Vert \langle \xi\rangle^4  \frac{w^{-1}\sqrt{\mu}}{|M^\e|^{\frac{1}{2}}} w\p^{\delta+\gamma} f^\e \Big\Vert_{L^2_{x}L^2_{\Lb}} \\
    &  \lesssim \e^2 \mathcal{E}_M^{N} + \e^2 (\mathcal{E}_M^{N})^{\frac{1}{2}} [\Vert w \Pi f^\e\Vert_{H^{N}_x L^2_\Lb} + \Vert w (I-\Pi) f^\e\Vert_{H^{N}_x L^2_\Lb} ]    \lesssim  \e^2 \mathcal{E}^{N} + \e^2 (\mathcal{E}_M^{N})^{\frac{1}{2}} \frac{\e}{\kappa^{\frac{2c-1}{2}}} (\mathcal{D}_H^{N})^{\frac{1}{2}}.
\end{align*}
For the non-cutoff Boltzmann equation, applying the local-global Maxwellian bridge \eqref{eq:poly-bd} and \eqref{eq:der-alp},
\begin{align*}
    & \eqref{eq:xd-h2-split}_2 \lesssim \e|\mathcal{E}_M^{N}|^{\frac{1}{2}} \sum_{|\gamma|=0}^1\Big\Vert \langle \xi\rangle^4  \frac{\p^{\delta+\gamma} \e \sqrt{\mu_0}g^\e}{|M^\e|^{\frac{1}{2}}} \Big\Vert_{L^2_x L^2_\Bb} \lesssim   \e^2 |\mathcal{E}_M^{N}|^{\frac{1}{2}} \sum_{|\gamma|=0}^1 \Big\Vert \langle \xi\rangle^4  \frac{w^{-1}_0\sqrt{\mu_0}}{|M^\e|^{\frac{1}{2}}} w_0\p^{\delta+\gamma} g^\e \Big\Vert_{L^2_{x}L^2_{\Bb}}  \\
    &\lesssim  \e^2 |\mathcal{E}_M^{N}|^{\frac{1}{2}} \Vert w_0 g^\e\Vert_{H^{N}_x L^2_\Bb}   \lesssim \e^2 |\mathcal{E}_M^{N}|^{\frac{1}{2}} [\Vert w_0g\Vert_{L^2_{x,\Bb}} + |\mathcal{E}^{N}|^{\frac{1}{2}} + \sum_{|\delta|=1}^{N}\Vert w_0 \p^\delta f^\e\Vert_{L^2_{x,\Bb}}] \lesssim  \e^2 |\mathcal{E}^{N}| + \e^2 |\mathcal{E}^{N}|^{\frac{1}{2}}\frac{\e}{\kappa^{\frac{2c-1}{2}}} |\mathcal{D}_H^{N}|^{\frac{1}{2}}.
\end{align*}
Consequently, $X_D$ satisfies, in both cases,
\begin{align}
     X_D&\lesssim \kappa^{1/2}\e^2(\mathcal D^N)^{1/2}
             +\e^2\mathcal E^N
             +\frac{\e^3}{\kappa^{(2c-1)/2}}
               (\mathcal E^N)^{1/2}(\mathcal D_H^N)^{1/2}. \label{eq:xd-bd-local}
\end{align}

For each Leibniz splitting in $J_{MG}$ or $J_{GG}$ at least one collision
factor has at most $m_*$ spatial derivatives and is placed in $L_x^\infty$.
Using \eqref{eq:op-bd} in Lemma~\ref{lem:collest},
\begin{align*}
    &    |J_{MG}| \lesssim \varepsilon (\mathcal E^N)^{1/2}\mathcal D_G^N
    +
    \frac{(\mathcal E^N)^{1/2}}{\varepsilon\kappa^{1/2}}
    (X_0+X_D)(\mathcal D_G^N)^{1/2}, \\
    & |J_{GG}| \lesssim X_0 \mathcal{D}_G^N + \frac{X_D}{\e \kappa^{1/2}} (\mathcal{E}^N)^{1/2} (\mathcal{D}_G^N)^{1/2}.
\end{align*}
Applying \eqref{eq:x0xd-bd-local},\eqref{eq:xd-bd-local}, the estimate to $J_{GG},J_{MG}$ reduce to the following four types.
\begin{align}\label{eq:local-scale-summary}
 &\frac\e{\kappa^{1/2}}(1+(\mathcal E^N)^{3/2})(\mathcal D_G^N)^{1/2},
 \qquad
 \frac{\e^4}{\kappa^{1/2}}(\mathcal E_H^N\mathcal E^N)^{1/2}
 (\mathcal D_G^N)^{1/2},\notag\\
 &\frac\e{\kappa^{1/2}}(\mathcal E^N)^{1/2}
 (\mathcal D^N\mathcal D_G^N)^{1/2},
 \qquad
 \frac{\e^2}{\kappa^c}\mathcal E^N
 (\mathcal D_H^N\mathcal D_G^N)^{1/2}.
\end{align}
The first two are bounded by $\eta\mathcal D_G^N$ plus
$C_\eta\e^2\kappa^{-1}(1+(\mathcal E^N)^3)$ and
$C_\eta\e^8\kappa^{-1}\mathcal E_H^N\mathcal E^N$, respectively; the latter is
bounded by $C_\eta\e^2(\mathcal E_H^N+1)^2$ under the bootstrap.
The third term is $o(1)\mathcal D^N$ because
$\e\kappa^{-1/2}(\mathcal E^N)^{1/2}=o(1)$, while the fourth is
$o(1)(\mathcal D_H^N+\mathcal D_G^N)$ by the scale relations.

Finally, for \(J_0\), insert the exact identity \(\mathcal T_\varepsilon M^\varepsilon=\mathcal A^\varepsilon-\mathcal F^\varepsilon\) and apply the large-velocity estimates \eqref{eq:large-bd},\eqref{eq:wt-bd}:
\begin{align*}
    &  |J_0| \lesssim \Vert \nabla_x Z^\e\Vert_\infty \sum_{|\alpha| \leq N} \mathcal{V}_3^\alpha + \e \sqrt{\kappa}(\mathcal{D}_G^N)^{1/2} \sum_{|\alpha|\leq N} \mathcal{V}_2^\alpha.
\end{align*}
For \(J_F\), the terms in which no derivative falls on the Maxwellian coefficient vanish by moment matching \eqref{eq:xd-id}; the remaining terms are controlled by the Maxwellian chain rule \eqref{eq:par-id}, Lemma~\ref{lem:stress}, and the Gaussian dual
bound \eqref{eq:xi-poly}. Therefore,
\begin{equation*}
 |J_0|+|J_F|
 \le \eta(\mathcal D^N+\mathcal D_G^N)
 +C_\eta(1+\|\nabla_xZ^\e\|_\infty^2)\mathcal E^N
 +C_\eta\e^2(\mathcal E_H^N+1)^2+o(1)\mathcal D_H^N.
\end{equation*}
Adding \eqref{eq:micro-local-exact}, \eqref{eq:flux-cancel-local},
\eqref{eq:macro-local-rem} and \eqref{eq:local-scale-summary}, together with the $J_0,J_F$ estimate, and decreasing the free Young parameters, proves
\eqref{eq:lower-local-block}.
\end{proof}

\begin{lemma}[Top local-Maxwellian energy block]\label{lem:toplocalblock}
Recall the top energy and dissipation defined in \eqref{eq:top-def}. For every fixed $\eta\in(0,1)$, after decreasing $\e_0$ if necessary,
\begin{align}\label{eq:top-local-block}
 \frac d{dt}(2\mathcal H_{\rm top})+c_2\mathcal D_{top}^N
 &\le \eta\mathcal D^N
 +C_\eta(1+\|\nabla_xZ^\e\|_\infty^2)\mathcal E^N
 +o(1)\mathcal D_H^N\notag\\
 &\quad+C_\eta\e^2(\mathcal E_H^N+1)^2
 +C_\eta\frac{\e^2}{\kappa}(\mathcal E^N+1)^3.
\end{align}
\end{lemma}
\begin{proof}
For $|\alpha|=N+1$, set
$H_\alpha^\e:=\Pi_M^\perp(\partial^\alpha F^\e)$.
Differentiating the full kinetic equation only in $x$, testing by
$\kappa \p^\alpha F^\e /(\e^2M^\e)$, and summing gives the exact identity
\begin{equation}
 \frac d{dt}\mathcal H_{\rm top}
 +\sum_{|\alpha|=N+1}\frac1{\e^4}
 \iint\mathcal L_M(\partial^\alpha F^\e)\frac{H_\alpha^\e}{M^\e}
 =\sum_{j=1}^5 I_j^{top},
\end{equation}
where
\begin{align*}
 I_1^{top}&:=\sum_{|\alpha|=N+1}\frac1{\e^4}
 \iint\mathcal L_M(\partial^\alpha M^\e)\frac{H_\alpha^\e}{M^\e},\\
 I_2^{top}&:=-\sum_{|\alpha|=N+1}\frac\kappa{2\e^2}
 \iint\mathcal T_\e M^\e\frac{|\partial^\alpha F^\e|^2}{(M^\e)^2},\\
 I_3^{top}&:=\sum_{|\alpha|=N+1}\frac2{\e^4}
 \iint\mathcal N(\partial^\alpha F^\e,G^\e)\frac{H_\alpha^\e}{M^\e},\\
 I_4^{top}&:=\sum_{|\alpha|=N+1}\frac2{\e^4}
 \sum_{0<\beta<\alpha}\binom\alpha\beta
 \iint\mathcal N(\partial^\beta G^\e,\partial^{\alpha-\beta}M^\e)
 \frac{H_\alpha^\e}{M^\e},\\
 I_5^{top}&:=\sum_{|\alpha|=N+1}\frac1{\e^4}
 \sum_{0<\beta<\alpha}\binom\alpha\beta
 \iint\mathcal N(\partial^\beta G^\e,\partial^{\alpha-\beta}G^\e)
 \frac{H_\alpha^\e}{M^\e}.
\end{align*}
 In \(I_4^{top}\) and \(I_5^{top}\), both collision inputs have spatial order at most \(N\), and one has order at most \(m_*\). We therefore use the same high--low input estimates as in the derivation of \eqref{eq:local-scale-summary}, with the top-order testing bound $(\sum_{|\alpha|=N+1}  \Vert H^\e_\alpha/\sqrt{M^\e} \Vert_{L^2_{x,\mathbf{D}}}^2)^{1/2}\lesssim \e^2 (\mathcal{D}_{top}^N)^{1/2}.$

For $I_1^{top}$, recall the Maxwellian decomposition from the proof of Lemma \ref{lem:stress}:
\[
\partial^\alpha M^\varepsilon
=
\varepsilon M^\varepsilon\Phi_{\alpha,1}^\varepsilon
+
R_{\alpha,M}^\varepsilon,
\qquad |\alpha|=N+1.
\]
Since $M^\varepsilon\Phi_{\alpha,1}^\varepsilon$ belongs to the
local null space, $\mathcal L_M(\partial^\alpha M^\varepsilon)=
\mathcal L_M R_{\alpha,M}^\varepsilon.$ The quadratic terms in the remainder contain two positive spatial
derivatives, while all higher-degree terms contain at least three.
The Maxwellian chain rule \eqref{eq:par-id} and Sobolev product estimates therefore give
\begin{align*}
&\sum_{|\alpha|=N+1}
\left(
\left\|
\frac{R_{\alpha,M}^\varepsilon}{\sqrt{M^\varepsilon}}
\right\|_{L^2_{x,\xi}}
+
\left\|
\frac{R_{\alpha,M}^\varepsilon}{\sqrt{M^\varepsilon}}
\right\|_{L_x^2L_\mathbf{D}^2}
\right) \lesssim
\varepsilon^2
\|\nabla_x Z^\varepsilon\|_{L_x^\infty}
(\mathcal E_M^N)^{1/2}
+
\varepsilon^3(\mathcal E_M^N)^{3/2}.
\end{align*}
Here the terms of degree higher than three are absorbed using
$\varepsilon(\mathcal E_M^N)^{1/2}=o(1)$.
Using $\mathcal L_M h=-2\mathcal N(M^\varepsilon,h)$,
the collision estimate \eqref{eq:op-bd}, and the definition of
$\mathcal D_{\mathrm{top}}^N$, we obtain the desired bound \eqref{eq:top-local-block} for $I_1^{top}$.

For $I_2^{top}$, set
\[
\mathcal V_m^{top}
:=
\frac{\kappa}{\varepsilon^2}
\sum_{|\alpha|=N+1}
\iint_{\mathbb R^3\times\mathbb R^3}
\langle\xi\rangle^m
\frac{|\partial^\alpha G^\varepsilon|^2}{M^\varepsilon}
\,d\xi\,dx.
\]
The Maxwellian chain rule \eqref{eq:par-id}, Lemma \ref{lem:stress}, Lemma \ref{lem:topmacro} imply
\[
\left|
\frac{\mathcal T_\varepsilon M^\varepsilon}{M^\varepsilon}
\right|
\lesssim
\|\nabla_x Z^\varepsilon\|_{L_x^\infty}\langle\xi\rangle^3
+
\varepsilon\kappa^{1/2}
(\mathcal D_G^N)^{1/2}\langle\xi\rangle^2, \ \ \frac{\kappa}{\varepsilon^2}
\sum_{|\alpha|=N+1}
\left\|
\langle\xi\rangle^{m/2}
\frac{\partial^\alpha M^\varepsilon}{\sqrt{M^\varepsilon}}
\right\|_{L^2_{x,\xi}}^2
\lesssim \mathcal E^N.
\]
Therefore,
\begin{align}
|I_2^{top}|
\lesssim{}&
\|\nabla_x Z^\varepsilon\|_{L_x^\infty}
\left(\mathcal E^N+\mathcal V_3^{top}\right)+
\varepsilon\kappa^{1/2}(\mathcal D_G^N)^{1/2}
\left(\mathcal E^N+\mathcal V_2^{top}\right).
\label{eq:top-i2-moments}
\end{align}
We use the dissipation-dependent bound \eqref{eq:large-bd} for
$\mathcal V_3^{top}$ and the energy bound \eqref{eq:wt-bd} for
$\mathcal V_2^{top}$:
\begin{align*}
\mathcal V_3^{top}
&\lesssim
\varepsilon^2\mathcal E^N
+
\varepsilon^{3/2}\kappa
\left[
\mathcal D^N+(\mathcal E_M^N)^2
\right]
+
\varepsilon^{10}\mathcal E_H^N,
\\
\mathcal V_2^{top}
&\lesssim
\varepsilon^{-1/2}\mathcal E^N
+
\varepsilon^{10}\mathcal E_H^N.
\end{align*}
Inserting the above bounds into \eqref{eq:top-i2-moments} gives the desired bound in \eqref{eq:top-local-block}.

For $I_3^{top}$, the top energy, the decomposition
$\partial^\alpha F^\varepsilon
=\Pi_M\partial^\alpha F^\varepsilon+H_\alpha^\varepsilon$,
and the uniform Gaussian bounds on the local macroscopic projection give
\begin{align*}
\left(
\sum_{|\alpha|=N+1}
\left\|
\frac{\partial^\alpha F^\varepsilon}{\sqrt{M^\varepsilon}}
\right\|_{L^2_{x,\xi}}^2
\right)^{1/2}
&\lesssim
\frac{\varepsilon}{\kappa^{1/2}}(\mathcal E^N)^{1/2},
\\
\left(
\sum_{|\alpha|=N+1}
\left\|
\frac{\partial^\alpha F^\varepsilon}{\sqrt{M^\varepsilon}}
\right\|_{L_x^2L_{\mathbf{D}}^2}^2
\right)^{1/2}
&\lesssim
\frac{\varepsilon}{\kappa^{1/2}}(\mathcal E^N)^{1/2}
+
\varepsilon^2(\mathcal D^N)^{1/2}.
\end{align*}
Thus, applying \eqref{eq:op-bd} to $I_3^{top}$ and placing the
undifferentiated microscopic factor in $L_x^\infty$, we obtain the desired bound in \eqref{eq:top-local-block}.
\end{proof}

\begin{proposition}[Local-Maxwellian energy estimate]\label{prop:compactclosure}
For either collision operator, under the bootstrap assumptions of
Theorem~\ref{thm:kinenergy},
\begin{align}\label{eq:compact-local}
 \mathcal E^N(t)+\int_0^t\mathcal D^N(s)\,ds
 &\lesssim \mathcal E^N(0)
 +\int_0^t
 \bigl(1+\|\nabla_x Z^\e(s)\|_{L_x^\infty}^2\bigr)\mathcal E^N(s)\,ds\notag\\
 &\quad+o(1)\int_0^t\mathcal D_H^N(s)\,ds
 +\e^2t\sup_{s\le t}(\mathcal E_H^N(s)+1)^2
 +\frac{\e^2}{\kappa}t\sup_{s\le t}(\mathcal E^N(s)+1)^3,
\end{align}
where $Z^\e=(\rho^\e,u^\e,\theta^\e)$.
\end{proposition}
\begin{proof}
Choose the occurrence of $\eta$ in Lemma~\ref{lem:lowerlocalblock} so that
its $\mathcal D_{top}^N$ term is at most half the coercive constant in
Lemma~\ref{lem:toplocalblock}, and choose the occurrence in
Lemma~\ref{lem:toplocalblock} so that its $\mathcal D_G^N$ term is at most
half the coercive constant in Lemma~\ref{lem:lowerlocalblock}.  Add
\eqref{eq:lower-local-block} and \eqref{eq:top-local-block}.  By the definition
of $\mathcal E^N$ and the local Maxwellian equivalence,
\[
 \mathfrak E_{\rm mac}^N+\mathcal H_G^N+2\mathcal H_{\rm top}\simeq\mathcal E^N
\]
uniformly on the bootstrap interval.  Integrating in time gives
\eqref{eq:compact-local}.
\end{proof}

\begin{proof}[Proof of Theorem~\ref{thm:kinenergy}]
Combine Proposition~\ref{prop:compactclosure} with
Proposition~\ref{prop:highord}.  For Landau, \eqref{eq:td-bd} is stronger
than the Boltzmann-form inequality \eqref{eq:td-en}, so it is enough to
use the latter structure in both cases.  Lemma~\ref{lem:scalecons}, the
bootstrap bound, and $c_0=\mathfrak y$ give
\[
 \kappa^{2c-1}\lesssim\e^{6c_0},\qquad
 \frac{\e}{\kappa^c}\asymp\e^{c_0},\qquad
 \sup_{s\le t}(\mathcal E_{tot}^N(s)+1)\lesssim\e^{-c_0}.
\]
Hence the coefficient of $\mathcal D_H^N$ in \eqref{eq:td-en} is
$O(\e^{c_0/2})$ and is absorbed.  The remaining scale errors satisfy
\[
 \kappa^{2c-1}\e^{-2c_0}t\sup_{s\le t}
 \bigl(\mathcal E^N+(\mathcal E^N)^2\bigr)\lesssim\e^{2c_0}t,
\]
\[
 \e^2t\sup_{s\le t}(\mathcal E_H^N+1)^2\lesssim\e^{2-2c_0}t,
 \qquad
 \frac{\e^2}{\kappa}t\sup_{s\le t}(\mathcal E^N+1)^3
 \lesssim\e^{5c_0}t.
\]
By \eqref{eq:choice-boot}, $T_*\lesssim\e^{-c_0}$, so these three terms
are respectively $O(\e^{c_0})$, $O(\e^{2-3c_0})$, and
$O(\e^{4c_0})$.  They are $o(1)$ for the chosen $c_0>0$.

We make the final cross-dissipation absorption explicit.  After the
previous absorption of the nonlinear $\mathcal D_H^N$ term, the integrated
high-order estimate has the form
\[
 \mathcal E_H^N(t)+c_H\int_0^t\mathcal D_H^N(s)\,ds
 \le C\mathcal E_H^N(0)+b_\e\int_0^t\mathcal D^N(s)\,ds+C,
 \qquad b_\e\to0,
\]
whereas Proposition~\ref{prop:compactclosure} gives, for some $c_L>0$,
\[
 \mathcal E^N(t)+c_L\int_0^t\mathcal D^N(s)\,ds
 \le C\mathcal E^N(0)
 +C\int_0^t
  \bigl(1+\|\nabla_xZ^\e(s)\|_{L_x^\infty}^2\bigr)
  \mathcal E^N(s)\,ds
 +a_\e\int_0^t\mathcal D_H^N(s)\,ds+C,
 \qquad a_\e\to0.
\]
Here $a_\e$ is the $o(1)$ coefficient in \eqref{eq:compact-local}, while
$b_\e$ collects the $o(1)\mathcal D^N$ terms in the high-order hierarchy.
Decrease $\e_0$ so that $Ca_\e\le c_H/4$ and
$Cb_\e\le c_L/4$.  Adding the two inequalities then absorbs both
cross-dissipation terms simultaneously.  Since
$\mathcal E^N\le\mathcal E_{tot}^N$ and
$\mathcal D^N+\mathcal D_H^N=\mathcal D_{tot}^N$, the result is exactly
\eqref{eq:diss-bd}.
\end{proof}

\section{Long-time estimates and rough kinetic-to-Euler convergence}\label{sec:longconv}
Let $x\in\Omega=\mathbb R^3$ and assume axisymmetric no-swirl initial data.  By Theorem~\ref{thm:kinenergy}, the long-time estimate reduces to controlling
\[
 \int_0^t\|\nabla_x(\rho^\e,u^\e,\theta^\e)(s)\|_{L^\infty_x}^2\,ds.
\]
The acoustic variables $\rho^\e+\theta^\e$ and $\mathbb P^\perp u^\e$ satisfy \eqref{eq:bega-rho-bd}.  The incompressible part is controlled through the relative-vorticity equation \eqref{eq:bega-td-id} and Danchin's estimate \eqref{eq:omth-bd}.
\subsection{Acoustic and axisymmetric vorticity estimates}
\subsubsection{Acoustic dispersion and Besov estimates in $\mathbb{R}^3$}
We use the following elementary consequences of the homogeneous
Littlewood--Paley decomposition.  For the nonnegative indices used below,
\begin{align}
 \|u\|_{\dot B_{2,2}^{s}}&\lesssim\|u\|_{H^s},\label{eq:decomp-bd}\\
 \|u\|_{L^q}&\lesssim\|u\|_{\dot B_{p,1}^{3/p-3/q}},\qquad p\le q,\label{eq:lw-bd}
\end{align}
and, for $s_1<s_2$ and $0<\tau<1$,
\begin{equation}\label{eq:tau-ac}
 \|u\|_{\dot B_{p,1}^{\tau s_1+(1-\tau)s_2}}
 \lesssim_{s_1,s_2,\tau}
 \|u\|_{\dot B_{p,2}^{s_1}}^\tau
 \|u\|_{\dot B_{p,2}^{s_2}}^{1-\tau}.
\end{equation}
Indeed, \eqref{eq:decomp-bd} follows by separating low and high dyadic
frequencies, and \eqref{eq:lw-bd} is the Bernstein inequality summed in
$\ell^1$.  To see \eqref{eq:tau-ac}, set
$s=\tau s_1+(1-\tau)s_2$, $A=\|u\|_{\dot B_{p,2}^{s_1}}$, and
$B=\|u\|_{\dot B_{p,2}^{s_2}}$.  Splitting at an integer $J$ and applying
Cauchy--Schwarz to the two geometric tails gives
\[
 \sum_{j\le J}2^{js}\|\dot\Delta_j u\|_p
 \lesssim A\,2^{J(s-s_1)},\qquad
 \sum_{j>J}2^{js}\|\dot\Delta_j u\|_p
 \lesssim B\,2^{-J(s_2-s)}.
\]
Choosing $J$ so that $2^{J(s_2-s_1)}\simeq B/A$ proves
\eqref{eq:tau-ac}; the cases $A=0$ or $B=0$ follow by approximation.  We also
use the immediate dyadic derivative equivalence
$\|\nabla^j f\|_{\dot B_{p,q}^s}\simeq\|f\|_{\dot B_{p,q}^{s+j}}$.
We derive the acoustic estimate directly from \eqref{eq:tu-id}.  The only
harmonic-analysis input is the frequency-localized half-wave estimate
proved in the next lemma.
\begin{lemma}[Dyadic acoustic dispersion]\label{lem:acdisp}
Let $A,B>0$ be fixed constants, let
$a:[0,T]\times\R^3\to\R$, and let
$q:[0,T]\times\R^3\to\R^3$ satisfy
$q=\mathbb P^\perp q$ and
\begin{equation}\label{eq:pq-id}
 \partial_ta+\frac{A}{\e}\nabla\cdot q=F_a,
 \qquad
 \partial_tq+\frac{B}{\e}\nabla a=F_q,
 \qquad F_q=\mathbb P^\perp F_q.
\end{equation}
If $2\le p\le\infty$ and
\begin{equation}\label{eq:pq-bd}
 \frac{2}{\mathfrak{r}}\le \min\left\{1,2\left(\frac12-\frac1p\right)\right\},
 \qquad (\mathfrak{r},p)\ne(2,\infty),
\end{equation}
then, for every $s\in\R$ for which the right-hand side is finite,
\begin{align}\label{eq:fin-ac}
 \|(a,q)\|_{L^\mathfrak{r}(0,T;\dot B_{p,1}^{s+3(1/p-1/2)+1/\mathfrak{r}})}
 &\lesssim_{A,B,p,\mathfrak{r}}
 \e^{1/\mathfrak{r}}\|(a(0),q(0))\|_{\dot B_{2,1}^s} +
 \e^{1/\mathfrak{r}}\int_0^T
 \|(F_a,F_q)(\tau)\|_{\dot B_{2,1}^s}\,d\tau.
\end{align}
\end{lemma}

The proof is recorded in Appendix~\ref{sec:auxana}.

Recall the acoustic variables from the introduction and set the source combination
\begin{equation}\label{eq:acoustic-source-def}
 a_{\rm ac}^\e=\rho^\e+\theta^\e,
 \qquad
 q_{\rm ac}^\e=\mathbb P^\perp u^\e,
 \qquad
 \mathfrak h^\e:=\nabla\cdot\mathfrak q^\e+
                   \nabla\mathrm U^\e:\mathbf r^\e.
\end{equation}
Adding the first line of \eqref{eq:tu-id} to $2/3$ times its third line,
and multiplying the second line by $k_B\Theta^\e$ before applying
$\mathbb P^\perp$, gives the exact acoustic system
\begin{align}\label{eq:ac-id}
 \partial_ta_{\rm ac}^\e+\frac{5}{3\e}\nabla\cdot q_{\rm ac}^\e
 &=\mathcal S_a^\e,\notag\\
 \partial_tq_{\rm ac}^\e+\frac{k_B}{\e}\nabla a_{\rm ac}^\e
 &=\mathcal S_q^\e,
\end{align}
where
\begin{align}\label{eq:ac-def}
 \mathcal S_a^\e
 &:=-u^\e\cdot\nabla a_{\rm ac}^\e
 -\frac{2}{3\e^2k_B\mathrm P^\e\Theta^\e}\mathfrak h^\e,\notag\\
 \mathcal S_q^\e
 &:=-\mathbb P^\perp(u^\e\cdot\nabla u^\e)
 -\mathbb P^\perp\left(
   \frac{k_B(\Theta^\e-1)}{\e}\nabla a_{\rm ac}^\e
   +\frac1{\e^2\mathrm P^\e}\nabla\cdot\mathbf r^\e
   \right).
\end{align}
No smallness or vanishing assumption on
$(a_{\rm ac}^\e(0),q_{\rm ac}^\e(0))$ has been used.
\begin{lemma}[Acoustic source bound]\label{lem:acsource}
Let $X_N^\e(t):=\mathcal E_{tot}^N(t)+\kappa^{1/2}(\mathcal D_{tot}^N(t))^{1/2}$.
Under \eqref{eq:tot-bd}, for every $0\le s<N$,
\begin{equation}\label{eq:disp-bd}
 \|(\mathcal S_a^\e,\mathcal S_q^\e)(t)\|_{\dot B_{2,1}^s}
 \lesssim
 \mathfrak L_s(\kappa)X_N^\e(t),
\end{equation}
where, for every fixed $\eta>0$,
\begin{equation*}
 \mathfrak L_s(\kappa):=
 \begin{cases}
 1,&0\le s<N-1,\\
 C_\eta\kappa^{-\eta},&s=N-1,\\
 \kappa^{-(s+1-N)/2},&N-1<s<N.
 \end{cases}
\end{equation*}
\end{lemma}
\begin{proof}
We first establish the two integer Sobolev bounds
\begin{align}\label{eq:hn-sob}
 \|(\mathcal S_a^\e,\mathcal S_q^\e)\|_{H^{N-1}}
 &\lesssim X_N^\e, \qquad
 \|(\mathcal S_a^\e,\mathcal S_q^\e)\|_{H^N}
 \lesssim \kappa^{-1/2}X_N^\e.
\end{align}
Since $N\ge5$, $H^{N-1}(\R^3)$ is an algebra.  The convective terms
satisfy
\begin{align}\label{eq:ac-hn}
 \|u^\e\cdot\nabla a_{\rm ac}^\e\|_{H^{N-1}}
 +\|u^\e\cdot\nabla u^\e\|_{H^{N-1}}
 &\lesssim \mathcal E_M^N,\notag\\
 \|u^\e\cdot\nabla a_{\rm ac}^\e\|_{H^N}
 +\|u^\e\cdot\nabla u^\e\|_{H^N}
 &\lesssim \kappa^{-1/2}\mathcal E^N.
\end{align}
For the second line we used
$\|Z^\e\|_{H^{N+1}}
 \lesssim \kappa^{-1/2}(\mathcal E^N)^{1/2}$, which follows directly
from the top macroscopic control \eqref{eq:topmacro-bd}.
Put $b^\e:=(\Theta^\e-1)/\e$.  By \eqref{eq:hn-bd},
$\|b^\e\|_{H^N}\lesssim(\mathcal E_M^N)^{1/2}$, and hence
\begin{align}
 \|b^\e\nabla a_{\rm ac}^\e\|_{H^{N-1}}
 &\lesssim \mathcal E_M^N,\qquad 
 \|b^\e\nabla a_{\rm ac}^\e\|_{H^N}
 \lesssim \kappa^{-1/2}\mathcal E^N. \notag 
\end{align}
The Moser bounds \eqref{eq:hn-en} and the flux bounds
\eqref{eq:hn-kap}--\eqref{eq:htop-bd} give
\begin{align}
 \left\|\frac1{\e^2\mathrm P^\e}\nabla\cdot\mathbf r^\e
 \right\|_{H^{N-1}}
 +\left\|\frac1{\e^2\mathrm P^\e\Theta^\e}
 \nabla\cdot\mathfrak q^\e\right\|_{H^{N-1}}
 &\lesssim \kappa^{1/2}(\mathcal D_G^N)^{1/2},\notag\\
 \left\|\frac1{\e^2\mathrm P^\e}\nabla\cdot\mathbf r^\e
 \right\|_{H^N}
 +\left\|\frac1{\e^2\mathrm P^\e\Theta^\e}
 \nabla\cdot\mathfrak q^\e\right\|_{H^N}
 &\lesssim (\mathcal D^N)^{1/2}+\mathcal E_M^N. \notag 
\end{align}
For the remaining piece of $\mathfrak h^\e$, use
$\mathrm U^\e=\e u^\e$, \eqref{eq:hn-kap}, and the $H^{N-1}$ product
estimate to obtain
\begin{align}
 \frac1{\e^2}
 \|\nabla\mathrm U^\e:\mathbf r^\e\|_{H^{N-1}}
 &\lesssim
 \e\kappa^{1/2}(\mathcal E_M^N)^{1/2}
 (\mathcal D_G^N)^{1/2}
 \lesssim \kappa^{1/2}(\mathcal D_G^N)^{1/2}. \notag 
\end{align}
At order $H^N$, one factor may carry $N$ derivatives.  If it is
$\nabla\mathrm U^\e$, \eqref{eq:topmacro-bd} gives
$\|\nabla\mathrm U^\e\|_{H^N}
 \lesssim \e\kappa^{-1/2}(\mathcal E^N)^{1/2}$; if it is
$\mathbf r^\e$, use \eqref{eq:htop-bd}.  Thus
\begin{align*}
 \frac1{\e^2}
 \|\nabla\mathrm U^\e:\mathbf r^\e\|_{H^N}
 &\lesssim
 \e(\mathcal E^N)^{1/2}(\mathcal D_G^N)^{1/2}
 +\e(\mathcal E_M^N)^{1/2}
   \bigl((\mathcal D^N)^{1/2}+\mathcal E_M^N\bigr) \lesssim (\mathcal D^N)^{1/2}+\kappa^{-1/2}\mathcal E^N.
\end{align*}
Here \eqref{eq:loc-bd} was used in the last step.  This bound together with \eqref{eq:ac-hn} and the boundedness
of $\mathbb P^\perp$ on Sobolev spaces, proves \eqref{eq:hn-sob}.
We next control the homogeneous low frequencies explicitly.  For
$j\le0$, Bernstein's inequalities give
\begin{align}
 \|\dot\Delta_j\nabla\cdot H\|_2
 &\lesssim
 2^j\|H\|_2+2^{5j/2}\|H\|_1,\qquad  
 \|\dot\Delta_j h\|_2
 \lesssim2^{3j/2}\|h\|_1, \notag 
\end{align}
where in the first term either available norm may be used separately.
Every nonmicroscopic term in \eqref{eq:ac-def} is the sum of a
divergence of an $L^1$ or $L^2$ product and an $L^1$ product; explicitly,
\begin{align*}
 u\cdot\nabla a&=\nabla\cdot(ua)-a\nabla\cdot u,\\
 u\cdot\nabla u&=\nabla\cdot(u\otimes u)-u\nabla\cdot u,\\
 b\nabla a&=\nabla(ba)-a\nabla b.
\end{align*}
Likewise,
\[
 (\mathrm P^\e)^{-1}\nabla\cdot\mathbf r^\e
 =\nabla\cdot((\mathrm P^\e)^{-1}\mathbf r^\e)
  -\nabla(\mathrm P^\e)^{-1}:\mathbf r^\e,
\]
and the same identity applies to
$(\mathrm P^\e\Theta^\e)^{-1}\nabla\cdot\mathfrak q^\e$.
The last term
$(\mathrm P^\e\Theta^\e)^{-1}\nabla\mathrm U^\e:\mathbf r^\e$
is itself in $L^1_x$.  H\"older's inequality, \eqref{eq:hn-en},
\eqref{eq:hn-bd}, and \eqref{eq:hn-kap} therefore give
\begin{equation}\label{eq:js-bd}
 \sum_{j\le0}2^{js}
 \|\dot\Delta_j(\mathcal S_a^\e,\mathcal S_q^\e)\|_2
 \lesssim X_N^\e,
 \qquad s\ge0,
\end{equation}
because both geometric series
$\sum_{j\le0}2^{j(s+1)}$ and
$\sum_{j\le0}2^{j(s+3/2)}$ converge.
It remains to sum $j\ge1$.  If $0\le s<N-1$, Cauchy--Schwarz and the
first line of \eqref{eq:hn-sob} give
\begin{align*}
 \sum_{j\ge1}2^{js}\|\dot\Delta_j\mathcal S^\e\|_2
 &\le
 \left(\sum_{j\ge1}2^{-2j(N-1-s)}\right)^{1/2}
 \|\mathcal S^\e\|_{H^{N-1}}
 \lesssim X_N^\e.
\end{align*}
At $s=N-1$, choose
$J=\lceil\frac12\log_2(\kappa^{-1})\rceil$.  The two lines of
\eqref{eq:hn-sob} imply
\begin{align}\label{eq:kap-ac}
 \sum_{j\ge1}2^{j(N-1)}\|\dot\Delta_j\mathcal S^\e\|_2
 &\lesssim J^{1/2}X_N^\e
 +2^{-J}\kappa^{-1/2}X_N^\e
 \lesssim (1+|\log\kappa|)^{1/2}X_N^\e \lesssim_\eta\kappa^{-\eta}X_N^\e.
\end{align}
Finally, if $N-1<s<N$, put $\vartheta=s+1-N\in(0,1)$.  Dyadic
H\"older and \eqref{eq:hn-sob} give
\begin{align*}
 \sum_{j\ge1}2^{js}\|\dot\Delta_j\mathcal S^\e\|_2
 &\lesssim 
 \|\mathcal S^\e\|_{H^{N-1}}^{1-\vartheta}
 \|\mathcal S^\e\|_{H^N}^{\vartheta}
 \lesssim \kappa^{-\vartheta/2}X_N^\e.
\end{align*}
Combining \eqref{eq:js-bd}, \eqref{eq:kap-ac}, and this dyadic estimate proves \eqref{eq:disp-bd}.
\end{proof}
Apply Lemma~\ref{lem:acdisp} to \eqref{eq:ac-id} with
$A=5/3$ and $B=k_B$, then use Lemma~\ref{lem:acsource}.  For
$0\le s<N$ we obtain
\begin{align}\label{eq:bega-rho-bd}
\bega
\Big\|&\Big((\rho^{\e}+\theta^{\e}),\mathbb{P}^{\perp}u^{\e}\Big)\Big\|_{L^{\mathfrak{r}}_T\dot{B}_{p,1}^{s+3(\frac{1}{p}-\frac{1}{2})+\frac{1}{\mathfrak{r}}}} \leq C \e^{\frac{1}{\mathfrak{r}}}\Big\|\Big((\rho^{\e}_0+\theta^{\e}_0),(\mathbb{P}^{\perp}u^{\e}_0)\Big)\Big\|_{\dot{B}_{2,1}^s} \cr 
&+ \begin{cases}
C\e^{\frac{1}{\mathfrak{r}}}\int_0^T\Big(\mathcal{E}_{tot}^{N}(
F^\e(\tau)
) +\kappa^{\frac{1}{2}}\Big(\mathcal{D}_{tot}^{N}(F^\e(\tau))\Big)^{\frac{1}{2}}\Big) d\tau, \quad &0\leq s < N-1, \\
C\e^{\frac{1}{\mathfrak{r}}}\kappa^{-\eta}\int_0^T\Big(\mathcal{E}_{tot}^{N}(F^\e(\tau)) +\kappa^{\frac{1}{2}}\Big(\mathcal{D}_{tot}^{N}(F^\e(\tau))\Big)^{\frac{1}{2}}\Big) d\tau, \quad & s = N-1, \\
C\e^{\frac{1}{\mathfrak{r}}}\kappa^{-\frac{s+1-N}{2}}\int_0^T\Big(\mathcal{E}_{tot}^{N}(F^\e(\tau)) +\kappa^{\frac{1}{2}}\Big(\mathcal{D}_{tot}^{N}(F^\e(\tau))\Big)^{\frac{1}{2}}\Big) d\tau, \quad &N-1 < s < N,
\end{cases} 
\enda
\end{align}
where
\begin{align*}
\bega
2 \leq p \leq \infty, \qquad \frac{2}{\mathfrak{r}} \leq \min \left\{ 1, 2\bigg(\frac{1}{2}-\frac{1}{p}\bigg)\right\}, \qquad (\mathfrak{r},p) \neq (2,\infty),
\enda
\end{align*}
for any fixed $0<\eta\ll1$.
\subsubsection{Symmetry and relative-vorticity equations}
The kinetic and relative-vorticity equations are recorded below.
\begin{lemma}\label{lem:sympres}
For either collision operator, axisymmetric no-swirl initial data remain axisymmetric no-swirl. The macroscopic fields $u^\e$ and $\mathbb P u^\e$ have the same symmetry.
\end{lemma}
The proof is recorded in Appendix \ref{sec:auxana}.

Taking the curl of the momentum equation in \eqref{eq:tu-id} gives
\begin{align}\label{eq:xu-id}
\partial_t\omega^\e+\mathbb Pu^\e\cdot\nabla_x\omega^\e
={}&\omega^\e\cdot\nabla_xu^\e
-\mathbb P^\perp u^\e\cdot\nabla_x\omega^\e
-(\nabla_x\cdot u^\e)\omega^\e \notag\\
&-k_B\Theta^\e\nabla_x\theta^\e\times\nabla_x(\rho^\e+\theta^\e)
-\frac1{\e^2}\nabla_x\times\left(\frac1{\mathrm P^\e}\nabla_x\cdot\mathbf r^\e\right).
\end{align}
Indeed, $\nabla_x\Theta^\e=\e\Theta^\e\nabla_x\theta^\e$ and
$\nabla\times(u\cdot\nabla u)=u\cdot\nabla\omega+(\nabla\cdot u)\omega-\omega\cdot\nabla u$.
\begin{lemma}
Let $F^\e$ be an axisymmetric no-swirl solution. Then the azimuthal
vorticity $\omega^{\e,\vartheta}$ satisfies
\begin{align}\label{eq:bega-td}
\bega
\partial_t\omega^{\e,\vartheta}
+
v^{\e,r}\partial_r\omega^{\e,\vartheta}
&+
v^{\e,z}\partial_z\omega^{\e,\vartheta}
=
\frac{v^{\e,r}}{r}\omega^{\e,\vartheta}
+
\big(
\nabla_x\times(\mathbb P^\perp u^\e\times\omega^\e)
\big)^\vartheta \cr
&-
k_B\Theta^\e
\big(
\nabla_x\theta^\e\times\nabla_x(\rho^\e+\theta^\e)
\big)^\vartheta
-
\frac{1}{\e^2}
\bigg(
\nabla_x\times
\bigg(
\frac{1}{\mathrm P^\e}
\nabla_x\cdot\mathbf r^\e
\bigg)
\bigg)^\vartheta.
\enda
\end{align}
The relative vorticity $\varpi^{\e}= \frac{\omega^{\e,\vartheta}}{r}$ satisfies
\begin{align}\label{eq:bega-td-id}
\bega
\partial_t \varpi^{\e}
+
v^{\e,r}\partial_r\varpi^{\e}
+
v^{\e,z}\partial_z\varpi^{\e}
=
\bPhi^{\e}_{\varpi},
\enda
\end{align}
where
\begin{align}\label{eq:pq-th}
\bPhi^{\e}_{\varpi}(t,r,z) &:= \frac{1}{r} \big(\nabla_x\times(\mathbb P^\perp u^\e\times\omega^\e)\big)^\vartheta 
-\frac{1}{r}k_B\mathrm{\Theta}^{\e}
\big(\nabla_x\theta^{\e}\times\nabla_x(\rho^{\e}+\theta^{\e})\big)^{\vartheta}
\cr
& -\frac{1}{\e^2}\frac{1}{r}
\bigg(
\nabla_x \times
\bigg(
\frac{1}{\mathrm{P}^{\e}}
\nabla_x \cdot \mathbf{r}^{\e}
\bigg)
\bigg)^{\vartheta}.
\end{align}
\end{lemma}
\begin{proof}
Recall $q_{\rm ac}^\e$ defined in \eqref{eq:acoustic-source-def}. By Lemma~\ref{lem:sympres},
$\mathbb Pu^\e$ and $q_{\rm ac}^\e$ are axisymmetric without swirl and
$\omega^\e=\omega^{\e,\vartheta}e_\vartheta$.  Hence
\[
 (\omega^\e\cdot\nabla_x)\mathbb Pu^\e
 =\frac{v^{\e,r}}{ r}\omega^{\e,\vartheta}e_\vartheta.
\]
Since $\nabla\cdot\omega^\e=0$,
\[
 (\omega^\e\cdot\nabla)q_{\rm ac}^\e-(q_{\rm ac}^\e\cdot\nabla)\omega^\e
 -(\nabla\cdot q_{\rm ac}^\e)\omega^\e
 =\nabla\times(q_{\rm ac}^\e\times\omega^\e).
\]
Taking the azimuthal component of \eqref{eq:xu-id} gives \eqref{eq:bega-td}.
Because
$(\partial_t+v^{\e,r}\partial_r+v^{\e,z}\partial_z)r=v^{\e,r}$,
division by $r$ cancels the stretching term and gives
\eqref{eq:bega-td-id}--\eqref{eq:pq-th}.
\end{proof}
Although $\varpi^\e(t,r,z)$ is defined in the $(r,z)$ variables, we identify it
with its axisymmetric extension $\bar\varpi^\e(t,x):=\varpi^\e(t,r(x),z)$ to
$\R^3$.  Then \eqref{eq:bega-td-id} can be rewritten in $\R^3$ as
\begin{align}
\bega
\partial_t\bar\varpi^\e
+
v^\e\cdot\nabla_x\bar\varpi^\e
=
\bPhi_\varpi^\e,
\enda \notag 
\end{align}
where $\bPhi_\varpi^\e$ is understood as its axisymmetric extension to
$\R^3$.
Write $\varpi^\e$ for
$\bar\varpi^\e$.
\subsubsection{Kinetic vorticity estimates}
Danchin's axisymmetric Biot--Savart estimate
\cite[Lemma~2]{danchin2007axisymmetric} gives
\begin{equation}\label{eq:omth-bd}
 \left\|\frac{u^r}{r}\right\|_{L_x^\infty}
 \lesssim \left\|\frac{\omega^\vartheta}{r}\right\|_{L_x^{3,1}}.
\end{equation}
\begin{lemma}\label{lem:wenergy}
Let $\omega^\e$ satisfy \eqref{eq:xu-id}. Then
\begin{align}\label{eq:bega-dt}
\bega
\frac{d}{dt}\|\omega^{\e}\|_{H^2_x}^2 
&\lesssim \Big(
\|\nabla_x \mathbb{P}u^{\e}\|_{L^\infty_x}
+\|\nabla_x \mathbb{P}^{\perp}u^{\e}\|_{L^\infty_x}
+\|\omega^{\e}\|_{L^\infty_x}
+\|\nabla_x^2 \mathbb{P}^{\perp}u^{\e}\|_{L^\infty_x}
\Big)
\|\omega^{\e}\|_{H^2_x}^2 \cr 
&+\bigg(
\|\nabla_x\nabla_x\cdot u^{\e}\|_{L^\infty_x}
\|\omega^{\e}\|_{L^2_x}
+\big\|
k_B\mathrm{\Theta}^{\e}\nabla_x\theta^{\e}\times\nabla_x(\rho^{\e}+\theta^{\e})
\big\|_{H^1(\R^3)}
\bigg)
\|\omega^{\e}\|_{H^3(\R^3)} \cr 
&+\bigg\|
\frac{1}{\e^2}\nabla_x \times
\bigg(
\frac{1}{\mathrm{P}^{\e}}
\nabla_x \cdot \mathbf{r}^{\e}
\bigg)
\bigg\|_{H^2_x}
\|\omega^{\e}\|_{H^2_x}.
\enda
\end{align}
\end{lemma}
\begin{proof}
Recall $q_{\rm ac}^\e=\mathbb P^\perp u^\e$.  Since
$\nabla\cdot\mathbb Pu^\e=0$, $\nabla\times q_{\rm ac}^\e=0$, and
$\omega^\e=\nabla\times\mathbb Pu^\e$,
Calder\'on--Zygmund estimates give
\[
 \|\nabla^2\mathbb Pu^\e\|_4\lesssim\|\nabla\omega^\e\|_4,
 \qquad
 \|\nabla^3\mathbb Pu^\e\|_2\lesssim\|\nabla^2\omega^\e\|_2.
\]
Moreover, integration by parts yields
\[
 \|\nabla\omega^\e\|_4^2
 \lesssim \|\omega^\e\|_\infty\|\nabla^2\omega^\e\|_2.
\]
Applying $\partial^\alpha$, $|\alpha|\le2$, to \eqref{eq:xu-id}, the
solenoidal transport and stretching terms are therefore bounded by
\[
 C(\|\nabla\mathbb Pu^\e\|_\infty+\|\omega^\e\|_\infty)
 \|\omega^\e\|_{H^2}^2.
\]
The $q_{\rm ac}^\e$ terms are bounded by
$C(\|\nabla q_{\rm ac}^\e\|_\infty+\|\nabla^2q_{\rm ac}^\e\|_\infty)\|\omega^\e\|_{H^2}^2$,
except for two apparent top-order terms.  Since
$\partial_j(q_{\rm ac}^\e)_i=\partial_i(q_{\rm ac}^\e)_j$ and $\nabla\cdot\omega^\e=0$,
\[
 \int \omega_j^\e\partial_j\partial^\alpha(q_{\rm ac}^\e)_i
       \partial^\alpha\omega_i^\e
 =-\int \partial_i\omega_j^\e\partial^\alpha(q_{\rm ac}^\e)_j
       \partial^\alpha\omega_i^\e,
 \qquad |\alpha|=2,
\]
which is controlled by $\|\nabla^2q_{\rm ac}^\e\|_\infty\|\omega^\e\|_{H^2}^2$.
For $d^\e=\nabla\cdot q_{\rm ac}^\e$, one integration by parts gives
\[
 \left|\int \partial^\alpha d^\e\,\omega^\e\cdot
       \partial^\alpha\omega^\e\right|
 \lesssim \|\nabla d^\e\|_\infty
 \big(\|\nabla\omega^\e\|_2\|\nabla^2\omega^\e\|_2
      +\|\omega^\e\|_2\|\nabla^3\omega^\e\|_2\big).
\]
Finally set
$\mathcal F^\e=k_B\Theta^\e\nabla\theta^\e\times\nabla(\rho^\e+\theta^\e)$ and
$\mathcal R^\e=\e^{-2}\nabla\times((\mathrm P^\e)^{-1}\nabla\cdot\mathbf r^\e)$.
Cauchy--Schwarz, with one integration by parts at $|\alpha|=2$ for
$\mathcal F^\e$, gives
\[
 \sum_{|\alpha|\le2}|\langle\partial^\alpha\mathcal F^\e,
 \partial^\alpha\omega^\e\rangle|
 \lesssim\|\mathcal F^\e\|_{H^1}\|\omega^\e\|_{H^3},
 \qquad
 \sum_{|\alpha|\le2}|\langle\partial^\alpha\mathcal R^\e,
 \partial^\alpha\omega^\e\rangle|
 \lesssim\|\mathcal R^\e\|_{H^2}\|\omega^\e\|_{H^2}.
\]
Summing the $H^2$ energy identities proves \eqref{eq:bega-dt}.
\end{proof}
Let $X^\e(t;s,x)$ denote the flow generated by $v^\e$:
\begin{align}\label{eq:kin-id}
\frac{d}{dt}X^\e(t;s,x)
=
 v^\e\bigl(t,X^\e(t;s,x)\bigr),
\qquad
X^\e(s;s,x)=x.
\end{align}
Since $\nabla_x\cdot v^\e=0$, the flow $X^\e(t;s,\cdot)$ is measure-preserving in $\R^3$. Moreover, since the axisymmetric no-swirl field $v^\e$ has the form in \eqref{eq:pp-id}, the angular variable remains constant along the flow, while its radial and axial components satisfy
\begin{align*}
\bega
\frac{d}{dt}X^{\e,r}(t;s,x)
&=
v^{\e,r}\bigl(t,X^{\e,r}(t;s,x),X^{\e,z}(t;s,x)\bigr),
\cr
\frac{d}{dt}X^{\e,z}(t;s,x)
&=
v^{\e,z}\bigl(t,X^{\e,r}(t;s,x),X^{\e,z}(t;s,x)\bigr).
\enda
\end{align*}
\begin{lemma}\label{lem:angenergy}
Let $\omega^{\e,\vartheta}$ satisfy \eqref{eq:bega-td} and let $N\ge4$.
For every sufficiently small $\zeta>0$, there is $C_\zeta$ such that
\begin{align}\label{eq:bega-om}
\bega
\|\omega^{\e,\vartheta}(t)&\|_{L^\infty_x}
\leq
\exp\bigg(
\int_0^t \left\|\frac{v^{\e,r}(s)}{r}\right\|_{L^\infty_x}\,ds \bigg)
\|\omega^{\e,\vartheta}_0\|_{L^\infty_x}
\cr
&+
C_\zeta\exp\bigg(
\int_0^t \left\|\frac{v^{\e,r}(s)}{r}\right\|_{L^\infty_x}\,ds \bigg)
\bigg[
\kappa^{\frac{1}{2}}\int_0^t |\mathcal{D}_G^{N}(s)|^{\frac{1}{2}}ds +\e^{1/2-\zeta}(1+t) \sup_{0\leq s \leq t}|\mathcal{E}_M^{N}(s)|^{\frac{1}{2}} \cr 
&\quad \times \bigg(|\mathcal{E}_M^{N}(0)|^{\frac{1}{2}}+ \int_0^t\Big(\mathcal{E}_{tot}^{N}(s) +\kappa^{\frac{1}{2}}\Big(\mathcal{D}_{tot}^{N}(s)\Big)^{\frac{1}{2}}\Big) ds \bigg) \bigg].
\enda
\end{align}
\end{lemma}
\begin{proof}
Integrate \eqref{eq:bega-td} along the measure-preserving flow \eqref{eq:kin-id}
and apply Gronwall.  The three forcing terms satisfy
\begin{align*}
 \int_0^t\|\nabla\times(\mathbb P^\perp u^\e\times\omega^\e)\|_\infty ds
 &\lesssim \e^{1/2-\zeta}(1+t)\sup_{s\le t}(\mathcal E_M^N(s))^{1/2}\,\mathfrak A(t),\\
 \int_0^t\|\Theta^\e\nabla\theta^\e\times\nabla(\rho^\e+\theta^\e)\|_\infty ds
 &\lesssim \e^{1/2-\zeta}(1+t)\sup_{s\le t}(\mathcal E_M^N(s))^{1/2}\,\mathfrak A(t),\\
 \int_0^t\left\|\e^{-2}\nabla\times((\mathrm P^\e)^{-1}
 \nabla\cdot\mathbf r^\e)\right\|_\infty ds
 &\lesssim \kappa^{1/2}\int_0^t(\mathcal D_G^N(s))^{1/2}ds,
\end{align*}
where
\[
 \mathfrak A(t):=(\mathcal E_M^N(0))^{1/2}
 +\int_0^t\big(\mathcal E_{tot}^N(s)
 +\kappa^{1/2}(\mathcal D_{tot}^N(s))^{1/2}\big)ds.
\]
Choose $0<\sigma<\zeta$ and apply \eqref{eq:bega-rho-bd} with $p=\infty$, $\mathfrak{r}=(1/2-\sigma)^{-1}$, $s=1+\sigma$ and
$p=\infty$, $\mathfrak{r}=(1/2-\sigma)^{-1}$, $s=2+\sigma$. Then $s-\frac{3}{2}+\frac{1}{\mathfrak r} = 0$ and $s-\frac32+\frac{1}{\mathfrak{r}}=1$, respectively. This yields control in $L^\mathfrak{r}_t \dot{B}^0_{\infty,1}$ and $L^\mathfrak{r}_t \dot{B}^1_{\infty,1}$ of the acoustic factor. Since
$N\ge4$ and $\sigma$ is small, thus the acoustic factor is bounded in $L_t^\mathfrak{r} W_x^{1,\infty}$ with
coefficient $\e^{1/2-\sigma}\le\e^{1/2-\zeta}$.
Sobolev products and H\"older in time give the first two bounds.  The
last follows from \eqref{eq:hn-en} and \eqref{eq:hn-kap}.  Substitution gives
\eqref{eq:bega-om}.
\end{proof}
For the relative-vorticity forcing in \eqref{eq:bega-td-id}, we use the following axisymmetric estimate.
For an axisymmetric vector field
$A=A^re_r+A^\vartheta e_\vartheta+A^ze_z$,
$\partial_\vartheta A=A^re_\vartheta-A^\vartheta e_r$; hence
\begin{equation}\label{eq:xa-bd}
 \left|\frac{A^\vartheta}{r}\right|\le |\nabla_xA|,
 \qquad
 \left\|\frac{A^\vartheta}{r}\right\|_{L^p}\le\|\nabla_xA\|_{L^p}.
\end{equation}
\begin{lemma}\label{lem:vortenergy}
Let $\varpi^{\e}$ satisfy \eqref{eq:bega-td-id} and let $N\ge5$.
For every sufficiently small $\zeta>0$, there is $C_\zeta$ such that
\begin{align*}
\|\varpi^{\e}(t)\|_{L^{3,1}_x}
&\leq \|\varpi^{\e}_0\|_{L^{3,1}_x}
+ C_\zeta \bigg[\kappa^{\frac{1}{2}}\int_0^t|\mathcal{D}_G^{N}(s)|^{\frac{1}{2}} ds
+ \e^{1/2-\zeta}(1+t) \sup_{0\leq s \leq t}|\mathcal{E}_M^{N}(s)|^{\frac{1}{2}} \cr 
&\quad \times \bigg(|\mathcal{E}_M^{N}(0)|^{\frac{1}{2}}+ \int_0^t\Big(\mathcal{E}_{tot}^{N}(s) +\kappa^{\frac{1}{2}}\Big(\mathcal{D}_{tot}^{N}(s)\Big)^{\frac{1}{2}}\Big) ds \bigg) \bigg] .
\end{align*}
\end{lemma}
\begin{proof}
The flow \eqref{eq:kin-id} preserves Lebesgue measure, hence also Lorentz
norms.  Integrating \eqref{eq:bega-td-id} along the flow gives
$\|\varpi^\e(t)\|_{3,1}\le\|\varpi_0^\e\|_{3,1}+
\int_0^t\|\Phi_\varpi^\e(s)\|_{3,1}ds$.
For each vector field $A$ occurring in \eqref{eq:pq-th},
\eqref{eq:xa-bd}, real interpolation between $L^2$ and $L^6$, and
$H^1\hookrightarrow L^6$ give
\[
 \left\|\frac{A^\vartheta}{r}\right\|_{L^{3,1}}
 \lesssim \|\nabla A\|_2^{1/2}\|\nabla A\|_6^{1/2}
 \lesssim \|A\|_{H^2}.
\]
Thus $\|\Phi_\varpi^\e\|_{3,1}$ is bounded by
\[
 \|\nabla\times(\mathbb P^\perp u^\e\times\omega^\e)\|_{H^2}
 +\|\Theta^\e\nabla\theta^\e\times\nabla(\rho^\e+\theta^\e)\|_{H^2}
 +\e^{-2}\|\nabla\times((\mathrm P^\e)^{-1}\nabla\cdot\mathbf r^\e)\|_{H^2}.
\]
Choose $0<\sigma\ll1$ so that
$\sigma(1+(1-\mathfrak y)/(2c))<\zeta$.  For $0\le j\le3$, apply
\eqref{eq:bega-rho-bd} with $p=\infty$, $\mathfrak{r}=(1/2-\sigma)^{-1}$, and
$s=j+1+\sigma$.  The largest index is $4+\sigma$.  If $N=5$, the
last case in \eqref{eq:bega-rho-bd} contributes $\kappa^{-\sigma/2}$; if
$N>5$, insert this factor since $0<\kappa<1$.  By \eqref{eq:choice-bd}, $\e^{1/2-\sigma}\kappa^{-\sigma/2}\asymp\e^{1/2-\sigma(1+(1-\mathfrak y)/(2c))}\le \e^{1/2-\zeta}$.
Sobolev products, \eqref{eq:hn-en}, \eqref{eq:hn-kap}, and H\"older in
time give the stated bound.
\end{proof}

\subsubsection{Entropy estimate}
Set
\begin{equation}\label{eq:par-def}
 S^\e:=\frac{\nabla_x\cdot\mathfrak q^\e+\nabla_x\mathrm U^\e:\mathbf r^\e}
 {\e^2k_B\mathrm P^\e\Theta^\e}.
\end{equation}
The first and third equations in \eqref{eq:tu-id} give
\begin{equation}\label{eq:fst-id}
 (\partial_t+u^\e\cdot\nabla_x)\mathfrak{s}^\e=-S^\e.
\end{equation}
\begin{proposition}[Entropy-gradient estimate]\label{prop:entgrad}
Let $\Omega=\mathbb R^3$, let $N\ge5$, and assume
\eqref{eq:tot-bd} on $[0,T]$.  Then, for $0\le t\le T$,
\begin{align}
\|\nabla_x\mathfrak{s}^\e(t)\|_{L^\infty_x}
&\le
\exp\left(\int_0^t\|\nabla_xu^\e(\tau)\|_{L^\infty_x}\,d\tau\right)
\|\nabla_x\mathfrak{s}^\e_0\|_{L^\infty_x}
\notag\\
&\quad
+C\kappa^{1/2}
\int_0^t
\exp\left(\int_s^t\|\nabla_xu^\e(\tau)\|_{L^\infty_x}\,d\tau\right)
\big(\mathcal D_G^{N}(s)\big)^{1/2}\,ds .
\label{eq:kap-xu}
\end{align}
\end{proposition}
\begin{proof}
By \eqref{eq:hn-en}, \eqref{eq:hn-kap}, Sobolev embedding, and $N\ge5$,
\begin{equation*}
 \|\nabla_xS^\e(t)\|_{L_x^\infty}
 \lesssim \kappa^{1/2}(\mathcal D_G^N(t))^{1/2}.
\end{equation*}
Differentiating \eqref{eq:fst-id} gives
\[
 (\partial_t+u^\e\cdot\nabla_x)\nabla_x\mathfrak{s}^\e
 =-(\nabla_xu^\e)^T\nabla_x\mathfrak{s}^\e-\nabla_xS^\e.
\]
Taking the Euclidean norm along the flow of $u^\e$ gives, for a.e. $t$,
\[
 \frac{d}{dt}|\nabla_x\mathfrak{s}^\e|
 \le
 \|\nabla_xu^\e(t)\|_\infty|\nabla_x\mathfrak{s}^\e|
 +\|\nabla_xS^\e(t)\|_\infty.
\]
Gronwall's inequality and this source bound give \eqref{eq:kap-xu}.
\end{proof}

\subsection{Long-time kinetic estimates}
The acoustic dispersion estimate and the axisymmetric vorticity bounds close the long-time bootstrap of Theorem~\ref{thm:longtime}.
\begin{definition}
For a sufficiently large $C>0$, define
\begin{align}\label{eq:ct-def}
\bega
\Psi[\varpi^{\e}_0,\omega^{\e}_0,\mathfrak{s}^{\e}_0](t)
:={}&\exp\bigg[Ct
\exp\bigg\{(1+t)
\exp\bigg(
C\big(\|\omega^{\e,\vartheta}_0\|_{L^\infty_x}+1\big)
e^{C(\|\varpi^{\e}_0\|_{L^{3,1}_x}+1)t}
\bigg)
\cr
&\hspace{35mm}\times
\big(1+\ln^+\|\omega^{\e}_0\|_{H^2_x}\big)
\bigg\}
\big(\|\nabla_x\mathfrak{s}^{\e}_0\|_{L^\infty_x}+1\big)^2
\bigg].
\enda
\end{align}
Write $\Psi(t):=\Psi[\varpi^{\e}_0,\omega^{\e}_0,\mathfrak{s}^{\e}_0](t)$.
\end{definition}
\begin{lemma}[Logarithmic Biot--Savart estimate]\label{lem:logbs}
Let $v\in L^2(\R^3)$ be a smooth divergence-free vector field and set
$\omega:=\nabla_x\times v$.  If $\omega\in H^2(\R^3)\cap
L^\infty(\R^3)$, then
\begin{equation}\label{eq:xv-bd}
 \|\nabla_xv\|_{L^\infty_x}
 \lesssim
 (1+\|\omega\|_{L^\infty_x})
 (1+\ln^+\|\omega\|_{H^2_x}).
\end{equation}
\end{lemma}
The proof is recorded in Appendix~\ref{sec:auxana}.

\subsubsection{Long-time bootstrap closure}\label{sec:closure}
\begin{proof}[Proof of Theorem \ref{thm:longtime}]
Let $T_{\max}^\e>0$ be the maximal lifespan furnished by the standard
perturbative classical local theory for the given smooth datum; for the
non-cutoff Boltzmann equation see \cite{alexandre2012boltzmann}; for a
whole-space Coulomb Landau local theory in the more strongly coupled
Vlasov--Poisson--Landau setting, see \cite{strainzhu2013vpl}.  Since
$\mathcal E_{\rm tot}^N(0)\le\Lambda_5^\e<\Lambda_1^\e$ for all sufficiently
small $\e$, continuity of the local solution makes the following bootstrap
time positive:
\begin{equation*}
 T_b^\e:=\sup\left\{0<T<\min\{T_{\max}^\e,T^\e\}:\
 \sup_{0\le t\le T}\mathcal E_{\rm tot}^N(t)
 +\int_0^T\mathcal D_{\rm tot}^N(s)\,ds<\Lambda_1^\e\right\}.
\end{equation*}
Fix $0<T<T_b^\e$.  On $[0,T]$ we therefore have the bootstrap bound
\begin{align}\label{eq:tot-ds}
\sup_{0\leq t\leq T}
\mathcal{E}^{N}_{tot}(t)
+\int_0^T\mathcal{D}^{N}_{tot}(s)\,ds
< \log  \frac{1}{\e}.
\end{align}
\noindent\textbf{Acoustic estimate.}
Under \eqref{eq:tot-ds}, \eqref{eq:bega-rho-bd} yields
\eqref{eq:bega-rho} on $[0,T]$. It suffices to prove \eqref{eq:ac-bd} there, with a
bound independent of the chosen $T<T_b^\e$.
By \eqref{eq:choice-bd},
$\e\kappa^{-1/2}<\e^{\mathfrak y}=\e^{c_0}$.  Since
$T\le T^\e=(\Lambda_5^\e)^{1/32}$ and
$\Lambda_1^\e=\log(1/\e)$, one has the stronger limit
\begin{equation*}
 \bigl(\e^{c_0}+\e\kappa^{-1/2}\bigr)
 (1+T^\e)\Lambda_1^\e\to0.
\end{equation*}
Thus \eqref{eq:tot-ds} implies \eqref{eq:tot-bd} on $[0,T]$ for all sufficiently
small $\e$.  The same iterated-log comparison gives
$\e^{c_0}T^\e\to0$, hence \eqref{eq:choice-boot}.  Therefore
Theorem~\ref{thm:kinenergy} applies with $T_*=T$.
To control $\int_0^t \Vert \nabla_x (\rho^\e,u^\e,\theta^\e)(s)\Vert_{L^\infty_x}^2 \dd s$, we decompose
\begin{align*}
\bega
\|\nabla_xu^{\e}\|_{L^\infty_x}^2
&\lesssim
\|\nabla_x\mathbb{P}u^{\e}\|_{L^\infty_x}^2
+\|\nabla_x\mathbb{P}^{\perp}u^{\e}\|_{L^\infty_x}^2,
\cr
\|\nabla_x\rho^{\e}\|_{L^\infty_x}^2
+\|\nabla_x\theta^{\e}\|_{L^\infty_x}^2
&\lesssim
\|\nabla_xa_{\rm ac}^{\e}\|_{L^\infty_x}^2
+\|\nabla_x\mathfrak{s}^{\e}\|_{L^\infty_x}^2,
\enda
\end{align*}
where $a_{\rm ac}^{\e}=\rho^{\e}+\theta^{\e}$ and
$\mathfrak{s}^{\e}=\frac{3}{2}\theta^{\e}-\rho^{\e}$.
To estimate the acoustic part, choose $\sigma_\delta>0$ so small that
\[
 c_\delta:=\frac12-
 \sigma_\delta\left(1+\frac{1-\mathfrak y}{2c}\right)>0,
 \qquad
 r_\delta:=\left(\frac12-\sigma_\delta\right)^{-1}>2.
\]
For $1\le j\le3$, apply \eqref{eq:bega-rho-bd} with $p=\infty$,
$\mathfrak r=r_\delta$, and $s=j+1+\sigma_\delta$.  When $N=5$, the top
case contributes $\kappa^{-\sigma_\delta/2}$; for $N>5$ we may insert the
same factor since $0<\kappa<1$.  By \eqref{eq:choice-bd}, the resulting
scale factor is $\asymp \e^{c_\delta}$.
Set $U_0^\e=(\rho_0^\e+\theta_0^\e,\mathbb P^\perp u_0^\e)$.
Since $N\ge5$, \eqref{eq:decomp-bd}--\eqref{eq:tau-ac} give
$\|U_0^\e\|_{\dot B_{2,1}^{j+1+\sigma_\delta}}
\lesssim_\delta\|U_0^\e\|_{H^5}$.  The initial bounds and the iterated-log
choice of $T^\e$ grow more slowly than every negative power of $\e$.
Consequently, for all sufficiently small $\e$,
\begin{equation}\label{eq:time-bd}
 \sum_{j=1}^3
 \|\nabla_x^j((\rho^\e+\theta^\e),\mathbb P^\perp u^\e)
 \|_{L^{r_\delta}(0,T;L^\infty_x)}
 \lesssim_\delta \e^{c_\delta/2},
 \qquad 0<T\le T^\e.
\end{equation}
Since $T^\e$ is also sub-polynomial in $\e^{-1}$, H\"older's inequality
then gives
\[
 \int_0^t
 \|\nabla_x((\rho^\e+\theta^\e),\mathbb P^\perp u^\e)(s)
 \|_{L^\infty_x}^2\,ds=o(1)
\]
uniformly for $0\le t\le T^\e$.
The remaining bounds concern
$\|\nabla_x\mathbb P u^{\e}\|_{L^\infty_x}$ and
$\|\nabla_x\mathfrak{s}^{\e}\|_{L^\infty_x}$.

\medskip

\noindent\textbf{Incompressible estimate.}
Fix $\zeta_0>0$ sufficiently small.  For $N\geq5$, applying Lemma~\ref{lem:vortenergy} with $\zeta=\zeta_0$ together with
the bootstrap assumption \eqref{eq:tot-ds}, we obtain
\begin{align*}
\bega
\|\varpi^{\e}(t)\|_{L^{3,1}_x}
&\leq
\|\varpi^{\e}_0\|_{L^{3,1}_x}
+C \bigg[\kappa^{\frac{1}{2}}\int_0^t|\mathcal{D}_G^{N}(s)|^{\frac{1}{2}} ds
+ \e^{1/2-\zeta_0}(1+t) \sup_{0\leq s \leq t}|\mathcal{E}_M^{N}(s)|^{\frac{1}{2}} \cr 
&\quad \times \bigg(|\mathcal{E}_M^{N}(0)|^{\frac{1}{2}}+ \int_0^t\Big(\mathcal{E}_{tot}^{N}(s) +\kappa^{\frac{1}{2}}\Big(\mathcal{D}_{tot}^{N}(s)\Big)^{\frac{1}{2}}\Big) ds \bigg) \bigg]
 \leq
C\big(\|\varpi^{\e}_0\|_{L^{3,1}_x}+1\big).
\enda
\end{align*}
For the azimuthal vorticity $\omega^{\e,\vartheta}$, this estimate and \eqref{eq:omth-bd} give
\[
\left\|\frac{v^{\e,r}(t)}{r}\right\|_{L^\infty_x}
\lesssim
\|\varpi^\e(t)\|_{L^{3,1}_x}
\lesssim
\|\varpi_0^\e\|_{L^{3,1}_x}+1.
\]
Applying this bound in Lemma~\ref{lem:angenergy} with $\zeta=\zeta_0$, we get
{\small
\begin{align}\label{eq:bega-th}
\bega
\|\omega^{\e,\vartheta}(t)\|_{L^\infty_x}
&\le e^{Ct (\|\varpi^{\e}_0\|_{L^{3,1}_x}+1)}
\bigg[\|\omega^{\e,\vartheta}_0\|_{L^\infty_x} + 
\kappa^{\frac{1}{2}}\int_0^t |\mathcal{D}_G^{N}(s)|^{\frac{1}{2}}ds +\e^{1/2-\zeta_0}(1+t) \sup_{0\leq s \leq t}|\mathcal{E}_M^{N}(s)|^{\frac{1}{2}} \cr 
&\quad \times \bigg(|\mathcal{E}_M^{N}(0)|^{\frac{1}{2}}+ \int_0^t\Big(\mathcal{E}_{tot}^{N}(s) +\kappa^{\frac{1}{2}}\Big(\mathcal{D}_{tot}^{N}(s)\Big)^{\frac{1}{2}}\Big) ds \bigg) \bigg] \cr
&\lesssim e^{Ct (\|\varpi^{\e}_0\|_{L^{3,1}_x}+1)}
(\|\omega^{\e,\vartheta}_0\|_{L^\infty_x} +1),
\enda
\end{align} }
where the energy and dissipation terms are absorbed by the factors
$\kappa^{\frac{1}{2}}$ and $\e^{1/2-\zeta_0}$ under the bootstrap assumption \eqref{eq:tot-ds}.
Since $\omega^\e=\omega^{\e,\vartheta}e_\vartheta$ and
$|e_\vartheta|=1$, we have
$\|\omega^\e(t)\|_{L^\infty_x}
=\|\omega^{\e,\vartheta}(t)\|_{L^\infty_x}$.
Thus, combining \eqref{eq:bega-th} with the logarithmic Biot--Savart estimate
\eqref{eq:xv-bd}, we obtain
\begin{align}\label{eq:bega-est}
\bega
\|\nabla_x\mathbb{P}u^{\e}(t)\|_{L^\infty_x}
\lesssim
e^{C(\|\varpi^{\e}_0\|_{L^{3,1}_x}+1)t}
\big(\|\omega^{\e,\vartheta}_0\|_{L^\infty_x}+1\big)
\big(1+\ln^+\|\omega^{\e}(t)\|_{H^2_x}\big).
\enda
\end{align}
The $H_x^2$ norm of $\omega^\e$ is estimated by
Lemma~\ref{lem:wenergy}. Whenever
$\|\omega^{\e}(t)\|_{H^2_x}^2>1$, dividing \eqref{eq:bega-dt} by
$\|\omega^{\e}(t)\|_{H^2_x}^2$ gives
\begin{align*}
\frac{d}{dt}
\lw(1+\ln^+\|\omega^{\e}\|_{H^2_x}^2\rw)
\lesssim
\Big(\|\nabla_x\mathbb{P}u^{\e}\|_{L^\infty_x}+\|\omega^{\e}\|_{L^\infty_x} \Big)
+\mathfrak{O}(t),
\end{align*}
where
\begin{align*}
\bega
\mathfrak{O}(t) &:=\|\nabla_x \mathbb{P}^{\perp}u^{\e}\|_{L^\infty_x}+\|\nabla_x^2 \mathbb{P}^{\perp}u^{\e}\|_{L^\infty_x} +\|\nabla_x\nabla_x\cdot u^{\e}\|_{L^\infty_x}\|\omega^{\e}\|_{H^3(\R^3)} \cr 
&+ \big\|
k_B\mathrm{\Theta}^{\e}\nabla_x\theta^{\e}\times\nabla_x(\rho^{\e}+\theta^{\e})
\big\|_{H^1(\R^3)}\|\omega^{\e}\|_{H^3(\R^3)} + \bigg\|
\frac{1}{\e^2}\nabla_x \times
\bigg(
\frac{1}{\mathrm{P}^{\e}}
\nabla_x \cdot \mathbf{r}^{\e}
\bigg)
\bigg\|_{H^2_x}.
\enda
\end{align*}
By \eqref{eq:hn-en}, \eqref{eq:hn-kap}, $N\ge5$, and Sobolev product estimates,
\begin{equation}\label{eq:rho-th-bd}
 \mathfrak O(t)
 \lesssim
 \|((\rho^\e+\theta^\e),\mathbb P^\perp u^\e)(t)\|_{W^{2,\infty}}
 \bigl(1+\mathcal E_M^N(t)\bigr)
 +\kappa^{1/2}(\mathcal D_G^N(t))^{1/2}.
\end{equation}
Indeed, since $a_{\rm ac}^\e=\rho^\e+\theta^\e$ and
$\mathfrak{s}^\e=\frac32\theta^\e-\rho^\e$, one has
\[
 \theta^\e=\frac25(a_{\rm ac}^\e+\mathfrak{s}^\e),\qquad
 \nabla\theta^\e\times\nabla a_{\rm ac}^\e
 =\frac25\nabla\mathfrak{s}^\e\times\nabla a_{\rm ac}^\e.
\]
Therefore, since $N\ge5$,
\begin{align}\label{eq:th-time}
 &\left\|
 k_B\Theta^\e\nabla\theta^\e\times\nabla a_{\rm ac}^\e
 \right\|_{H^1_x}\|\omega^\e\|_{H^3_x} \lesssim
 \|a_{\rm ac}^\e\|_{W^{2,\infty}_x}
 \|\mathfrak{s}^\e\|_{H^2_x}
 \|\omega^\e\|_{H^3_x}
 \lesssim
 \|a_{\rm ac}^\e\|_{W^{2,\infty}_x}\mathcal E_M^N.
\end{align}
Moreover, $\nabla\cdot u^\e=\nabla\cdot\mathbb P^\perp u^\e$, so
\begin{equation*}
 \|\nabla\mathbb P^\perp u^\e\|_\infty
 +\|\nabla^2\mathbb P^\perp u^\e\|_\infty
 +\|\nabla\nabla\cdot u^\e\|_\infty\|\omega^\e\|_{H^3}
 \lesssim
 \|\mathbb P^\perp u^\e\|_{W^{2,\infty}}
 (1+\mathcal E_M^N).
\end{equation*}
Finally, the product estimate in $H^2$, the multiplier bounds
$\|(\mathrm P^\e)^{-1}\|_\infty+\|\nabla(\mathrm P^\e)^{-1}\|_{H^{N-1}}\lesssim1$
from \eqref{eq:hn-en}, and \eqref{eq:hn-kap} imply
\begin{align*}
 \frac1{\e^2}\left\|\nabla\times
 \left((\mathrm P^\e)^{-1}\nabla\cdot\mathbf r^\e\right)
 \right\|_{H^2_x}
 &\lesssim
 \frac{1+\e(\mathcal E_M^N)^{1/2}}{\e^2}
 \|\mathbf r^\e\|_{H^4_x} \lesssim \kappa^{1/2}(\mathcal D_G^N)^{1/2}.
\end{align*}
This curl estimate together with \eqref{eq:th-time} gives \eqref{eq:rho-th-bd}.
Apply \eqref{eq:bega-est} and \eqref{eq:bega-th} to
$\|\nabla_x \mathbb{P}u^{\e}\|_{L^\infty_x}$ and
$\|\omega^{\e}\|_{L^\infty_x}$, respectively, to obtain
\begin{align*}
\frac{d}{dt}
\lw(1+\ln^+\|\omega^{\e}\|_{H^2_x}^2\rw)
&\lesssim
e^{C(\|\varpi^{\e}_0\|_{L^{3,1}_x}+1)t}
\big(\|\omega^{\e,\vartheta}_0\|_{L^\infty_x}+1\big)
\big(1+\ln^+\|\omega^{\e}\|_{H^2_x}\big)
+\mathfrak{O}(t) \\
& \lesssim e^{C(\|\varpi^{\e}_0\|_{L^{3,1}_x}+1)t}
\big(\|\omega^{\e,\vartheta}_0\|_{L^\infty_x}+1\big)
\big(1+\ln^+\|\omega^{\e}\|^2_{H^2_x}\big)
+\mathfrak{O}(t) .
\end{align*}
Hence, Gr\"onwall's inequality yields
\begin{align*}
\bega
\ln^+\|\omega^{\e}(t)\|_{H^2_x}^2 &\leq \exp\bigg(\int_0^t Ce^{C (\|\varpi^{\e}_0\|_{L^{3,1}_x}+1)s}
(\|\omega^{\e,\vartheta}_0\|_{L^\infty_x} +1) ds\bigg) \big(1+\ln^+\|\omega^{\e}_0\|_{H^2_x}^2\big) \cr 
&+\int_0^t \exp\bigg(\int_s^t Ce^{C (\|\varpi^{\e}_0\|_{L^{3,1}_x}+1)\tau}
(\|\omega^{\e,\vartheta}_0\|_{L^\infty_x} +1) d\tau\bigg)\mathfrak{O}(s) ds.
\enda
\end{align*}
To control the last term, use \eqref{eq:rho-th-bd}.  Fix a sufficiently small $\zeta>0$ and choose
$0<\sigma\ll\zeta$.  Apply \eqref{eq:bega-rho-bd} with $p=\infty$,
$\mathfrak{r}=(1/2-\sigma)^{-1}$, and $s=j+1+\sigma$ for $j=1,2$.  Since
$s<N-1$, no top-order $\kappa$-loss occurs.  Using \eqref{eq:small-def} and
\eqref{eq:tot-ds}, H\"older's inequality gives
\begin{align*}
 \int_0^{T^\e}
 \|((\rho^\e+\theta^\e),\mathbb P^\perp u^\e)(s)\|_{W^{2,\infty}}
 (1+\mathcal E_M^N(s))\,ds
 \lesssim_\zeta \e^{1/2-\zeta}.
\end{align*}
Moreover, by Cauchy--Schwarz, \eqref{eq:choice-bd}, \eqref{eq:small-def}, and
\eqref{eq:tot-ds},
\[
 \kappa^{1/2}\int_0^{T^\e}(\mathcal D_G^N(s))^{1/2}\,ds
 \lesssim_\zeta \kappa^{1/2-\zeta}.
\]
Hence
\begin{equation*}
 \int_0^{T^\e}\mathfrak O(s)\,ds
 \lesssim_\zeta
 \e^{1/2-\zeta}+\kappa^{1/2-\zeta}=o(1).
\end{equation*}
If
\[
 A_\e(t):=\int_0^t
 Ce^{C(\|\varpi_0^\e\|_{L^{3,1}}+1)\tau}
 (\|\omega_0^{\e,\vartheta}\|_\infty+1)\,d\tau,
\]
then the inhomogeneous Gronwall term satisfies
\[
 \int_0^t e^{A_\e(t)-A_\e(s)}\mathfrak O(s)\,ds
 \le e^{A_\e(t)}\int_0^t\mathfrak O(s)\,ds.
\]
Since $1+\ln^+\|\omega_0^\e\|_{H^2_x}^2\ge1$, the last display is,
for all sufficiently small $\e$, absorbed into the homogeneous Gronwall
term.  Consequently,
\begin{align*}
\bega
\ln^+\|\omega^{\e}(t)\|_{H^2_x}^2
\lesssim
\exp\bigg(
C\big(\|\omega^{\e,\vartheta}_0\|_{L^\infty_x}+1\big)
e^{C(\|\varpi^{\e}_0\|_{L^{3,1}_x}+1)t}
\bigg)
\big(1+\ln^+\|\omega^{\e}_0\|_{H^2_x}^2\big). 
\enda
\end{align*}
Substituting this bound into \eqref{eq:bega-est}, we obtain
\begin{align}\label{eq:bega-kin}
\bega
&\|\nabla_x\mathbb{P}u^{\e}(t)\|_{L^\infty_x} \\
&\lesssim
\exp\bigg(
C\big(\|\omega^{\e,\vartheta}_0\|_{L^\infty_x}+1\big)
e^{C(\|\varpi^{\e}_0\|_{L^{3,1}_x}+1)t} + C(\Vert \varpi_0^\e\Vert_{L^{3,1}_x}+1)t
\bigg)
\big(1+\ln^+\|\omega^{\e}_0\|_{H^2_x}\big)(1+\Vert \omega_0^{\e,\vartheta}\Vert_{L^\infty_x}) \\
& \lesssim \exp\bigg(
(C+1)\big(\|\omega^{\e,\vartheta}_0\|_{L^\infty_x}+1\big)
e^{(C+1)(\|\varpi^{\e}_0\|_{L^{3,1}_x}+1)t}
\bigg)
\big(1+\ln^+\|\omega^{\e}_0\|_{H^2_x}\big) .
\enda
\end{align}
Since $\mathfrak{s}^{\e}=\frac32\theta^\e-\rho^\e$, Proposition
\ref{prop:entgrad} and \eqref{eq:kap-xu} apply on $\R^3$.  Decomposing
$\nabla_xu^{\e}
=\nabla_x\mathbb{P}u^{\e}
+\nabla_x\mathbb{P}^{\perp}u^{\e}$
and using \eqref{eq:bega-kin} and \eqref{eq:time-bd}, we obtain
\begin{align*}
	\bega
	\|\nabla_x\mathfrak{s}^{\e}(t)\|_{L^\infty_x}
	&\leq
	\exp\bigg[Ct
	\exp\bigg(
	(C+1)\big(\|\omega^{\e,\vartheta}_0\|_{L^\infty_x}+1\big)
	e^{(C+1)(\|\varpi^{\e}_0\|_{L^{3,1}_x}+1)t}
	\bigg)
	\big(1+\ln^+\|\omega^{\e}_0\|_{H^2_x}\big)
	\bigg]
	\cr
	&\quad\times
	\Big(
	\|\nabla_x\mathfrak{s}^{\e}_0\|_{L^\infty_x}
	+1
	\Big).
	\enda
\end{align*}
\medskip
\noindent\textbf{Closure.}
Combining the estimates obtained above with Theorem \ref{thm:kinenergy} and applying
Gr\"onwall's inequality, we obtain
\begin{align*}
\bega
\mathcal{E}_{tot}^{N}(t)
+\int_0^t\mathcal{D}_{tot}^{N}(s)\,ds
&\lesssim
e^{C\int_0^t
\big(1+\|\nabla_x(\rho^{\e},u^{\e},\theta^{\e})(s)\|_{L^\infty_x}^2\big)\,ds}
\big(\mathcal{E}_{tot}(0)+C\big)
 \lesssim
\Psi[\varpi^{\e}_0,\omega^{\e}_0,\mathfrak{s}^{\e}_0](t)
\big(\mathcal{E}_{tot}(0)+C\big),
\enda
\end{align*}
where $\Psi[\varpi^{\e}_0,\omega^{\e}_0,\mathfrak{s}^{\e}_0](t)$ is defined
in \eqref{eq:ct-def} with some constant $C$. 
Since $N\ge5$, \eqref{eq:diss-def} and \eqref{eq:init-bd}, together with Sobolev
embedding, give
\begin{equation*}
 \|U_0^\e\|_{H^5}
 +\|\omega_0^\e\|_{L^\infty\cap H^2}
 +\|\nabla_x\mathfrak{s}_0^\e\|_{L^\infty}
 +\|\varpi_0^\e\|_{L^{3,1}}
 \lesssim (\Lambda_5^\e)^{1/2}.
\end{equation*}
To close \eqref{eq:tot-ds}, use this initial bound and the
iterated-log hierarchy in \eqref{eq:small-def}.  We use repeatedly the
elementary fact that, for fixed exponents $A_k\ge0$ and every $\eta>0$,
\[
 \prod_{k=j+1}^5(\Lambda_k^\e)^{A_k}
 =o\big((\Lambda_j^\e)^\eta\big),
 \qquad 1\le j\le4.
\]
Since $T^\e$ is a fixed power of $\Lambda_5^\e$, the successive exponential
factors in these bounds are therefore
\[
 (\Lambda_4^\e)^{o(1)},\qquad
 (\Lambda_3^\e)^{o(1)},\qquad
 (\Lambda_2^\e)^{o(1)},\qquad
 (\Lambda_1^\e)^{o(1)}.
\]
Together with $\mathcal E_{tot}^N(0)\lesssim\Lambda_5^\e$ and
$\Lambda_5^\e=(\Lambda_1^\e)^{o(1)}$, this gives, for all sufficiently
small $\e$ and every $T<T_b^\e$,
\begin{equation}\label{eq:en-tot}
 \sup_{0\le t\le T}\mathcal E_{\rm tot}^N(t)
 +\int_0^T\mathcal D_{\rm tot}^N(s)\,ds
 \le (\Lambda_1^\e)^{1/2}<\Lambda_1^\e.
\end{equation}
The improvement is uniform in $T<T_b^\e$.  Letting $T\uparrow T_b^\e$ and
using continuity of the energy shows that equality with the bootstrap
threshold cannot occur at $T_b^\e$.  Hence
$T_b^\e=\min\{T_{\max}^\e,T^\e\}.$
If $T_{\max}^\e<T^\e$, then \eqref{eq:en-tot}, the parameter
bounds of Lemmas~\ref{lem:macrocoef} and \ref{lem:maxbridge}, and
$\Theta^\e\ge\Theta_0^\e>0$ keep the solution in a fixed perturbative
class up to $T_{\max}^\e$.  The classical local continuation theorem may
therefore be restarted at times approaching $T_{\max}^\e$, extending the
solution beyond its maximal lifespan, a contradiction.  Consequently
$T_{\max}^\e\ge T^\e$ and $T_b^\e=T^\e$.  Equation
\eqref{eq:en-tot} is stronger than \eqref{eq:ac-bd}; the acoustic estimate
\eqref{eq:bega-rho} follows from the already proved bound on every
$0<T\le T^\e$.
\end{proof}
\subsubsection{Fixed-time solenoidal energy drift}
The rough-limit argument uses only the following fixed-time consequence of the
long-time estimates.
\begin{lemma}[Fixed-time solenoidal energy drift]\label{lem:fixedenergy}
Assume the hypotheses of Theorem~\ref{thm:longtime} and
$\sup_\e\|\mathbb Pu_0^\e\|_2<\infty$.  For every fixed $T<\infty$ and every
sufficiently small $\delta>0$,
\begin{equation}\label{eq:ltwo-del}
 \sup_{0\le t\le T}
 \Big|\|\mathbb Pu^\e(t)\|_2-\|\mathbb Pu_0^\e\|_2\Big|
 \le C_{T,\delta}\left(\e^{1/4-\delta}+\kappa^{1/2-\delta}\right).
\end{equation}
\end{lemma}
\begin{proof}
Recall $a_{\rm ac}^\e,q_{\rm ac}^\e$ defined in \eqref{eq:acoustic-source-def}. Multiplying the momentum equation in \eqref{eq:tu-id} by $k_B\Theta^\e$
and applying $\mathbb P$ gives
\begin{equation}\label{eq:snd-id}
 \partial_t\mathbb Pu^\e+\mathbb P(u^\e\!\cdot\nabla u^\e)
 =-\mathbb P\!\left[
 \frac{k_B(\Theta^\e-1)}\e\nabla a_{\rm ac}^\e
 +\frac1{\e^2\mathrm P^\e}\sum_j\partial_{x_j}\mathbf r_{ij}^\e
 \right].
\end{equation}
The constant-gradient term $k_B\e^{-1}\nabla a_{\rm ac}^\e$ is annihilated by
$\mathbb P$.  The acoustic estimate \eqref{eq:bega-rho-bd}, with
\[
 (s,p,\mathfrak{r})=\left(\frac12,4,4\right),\quad
 \left(\frac32,4,4\right),\quad
 \left(\frac94,\infty,4\right),
\]
and \eqref{eq:decomp-bd}--\eqref{eq:tau-ac} imply, for each fixed $T$,
\begin{align}\label{eq:en-kap}
 &\|q_{\rm ac}^\e\|_{L^4(0,T;L^4)}
 +\|\nabla q_{\rm ac}^\e\|_{L^4(0,T;L^4)}
 +\|\nabla q_{\rm ac}^\e\|_{L^4(0,T;L^\infty)}
 +\|\nabla a_{\rm ac}^\e\|_{L^4(0,T;L^\infty)}\notag\\
 &\qquad\le C_T\e^{1/4}
 \left[\mathcal E_{tot}^N(0)^{1/2}
 +\int_0^T\left(\mathcal E_{tot}^N(s)
 +\kappa^{1/2}\mathcal D_{tot}^N(s)^{1/2}\right)ds\right].
\end{align}
No additional initial norm is used here, since $N\ge5$ and the initial
Besov norms in these three applications are bounded by the $H^5$ norm
already contained in the energy.

Taking the $L^2$ inner product of \eqref{eq:snd-id} with $\mathbb Pu^\e$
and writing $u^\e=\mathbb Pu^\e+q_{\rm ac}^\e$ gives
\begin{align*}
 \left|\int (u^\e\!\cdot\nabla u^\e)\cdot\mathbb Pu^\e\,dx\right|
 &\le \frac12\|\nabla\!\cdot q_{\rm ac}^\e\|_\infty\|\mathbb Pu^\e\|_2^2
 +\|\nabla q_{\rm ac}^\e\|_\infty\|\mathbb Pu^\e\|_2^2 +\|q_{\rm ac}^\e\|_4\|\nabla q_{\rm ac}^\e\|_4\|\mathbb Pu^\e\|_2.
\end{align*}
Moreover Lemma~\ref{lem:macrocoef} and \eqref{eq:hn-kap} give
\[
 \left\|\frac{\Theta^\e-1}{\e}\right\|_2
 \lesssim (\mathcal E_M^N)^{1/2},\qquad
 \left\|\frac1{\e^2\mathrm P^\e}\nabla\!\cdot\mathbf r^\e\right\|_2
 \lesssim \kappa^{1/2}(\mathcal D_G^N)^{1/2}.
\]
To avoid division by $\|\mathbb Pu^\e\|_2$ at its zeros, put
$Z_\lambda=(\|\mathbb Pu^\e\|_2^2+\lambda^2)^{1/2}$.  Since
$\|\mathbb Pu^\e\|_2\le(\mathcal E_{tot}^N)^{1/2}$, these estimates
give, for a.e. $t$,
\begin{align*}
 |Z_\lambda'(t)|
 &\lesssim (\mathcal E_{tot}^N(t))^{1/2}
 \bigl(\|\nabla q_{\rm ac}^\e(t)\|_\infty+\|\nabla a_{\rm ac}^\e(t)\|_\infty\bigr)
 +\|q_{\rm ac}^\e(t)\|_4\|\nabla q_{\rm ac}^\e(t)\|_4 +\kappa^{1/2}(\mathcal D_G^N(t))^{1/2}.
\end{align*}
Integrate, apply H\"older and \eqref{eq:en-kap}, and then let
$\lambda\downarrow0$.  With
\[
 \mathfrak S_{\e,T}:=(\mathcal E_{tot}^N(0))^{1/2}
 +\int_0^T\left(\mathcal E_{tot}^N(s)
 +\kappa^{1/2}(\mathcal D_{tot}^N(s))^{1/2}\right)ds,
\]
we obtain
\begin{align*}
 \sup_{0\le t\le T}
 \Big|\|\mathbb Pu^\e(t)\|_2-\|\mathbb Pu_0^\e\|_2\Big|
 &\lesssim_T
 \sup_{t\le T}(\mathcal E_{tot}^N(t))^{1/2}\e^{1/4}\mathfrak S_{\e,T}
 +\e^{1/2}\mathfrak S_{\e,T}^2 +\kappa^{1/2}\left(\int_0^T\mathcal D_G^N(s)\,ds\right)^{1/2}.
\end{align*}
By \eqref{eq:init-bd} and \eqref{eq:ac-bd}, all energy and dissipation factors above grow at most like a fixed power of $|\log\e|$ on every fixed interval.  Hence, for some $C>0$,
\[
 \sup_{0\le t\le T}
 \Big|\|\mathbb Pu^\e(t)\|_2-\|\mathbb Pu_0^\e\|_2\Big|
 \lesssim_T
 \bigl(\e^{1/4}+\e^{1/2}+\kappa^{1/2}\bigr)
 \left|\log\frac1\e\right|^C.
\]
Finally $\kappa^c\asymp\e^{1-\mathfrak y}$ and every positive power of
$\e$ or $\kappa$ dominates every logarithmic power.  Absorbing the logarithm
into an arbitrarily small loss $\delta$ proves \eqref{eq:ltwo-del}.
\end{proof}

\subsection{Strong convergence of the transported relative vorticity}\label{sec:transport}
The strong transported-relative-vorticity limit and the associated renormalized
axisymmetric Euler flow are identified under a uniform $L^p$ bound on
$\varpi_0^\e$.
\begin{theorem}[DiPerna--Lions stability]\label{thm:dlstab}
Let $1\le \mathfrak p<\infty$.  Let $V_n,V$ be divergence-free vector fields with
$V_n\in L^1(0,T;L^1_{\rm loc}(\R^3))$ such that
\begin{equation}\label{eq:lone-loc}
 V_n\to V
 \qquad\text{in }L^1(0,T;L^1_{\rm loc}(\R^3)),
\end{equation}
\begin{equation}\label{eq:lone-strg}
 V\in L^1(0,T;W^{1,1}_{\rm loc}(\R^3)),
\end{equation}
and
\begin{equation}\label{eq:lone-vort}
 \frac{|V(t,x)|}{1+|x|}
 \in L^1(0,T;L^1(\R^3))
 +L^1(0,T;L^\infty(\R^3)).
\end{equation}
Let $\rho_n$ be renormalized solutions of
\begin{equation*}
 \partial_t\rho_n+V_n\cdot\nabla_x\rho_n=0,
 \qquad
 \rho_n(0)=\rho_{0,n},
\end{equation*}
such that
\begin{equation}\label{eq:van-bd}
 \sup_n\|\rho_n\|_{L^\infty(0,T;L^\mathfrak p(\R^3))}<\infty
\end{equation}
and, for every bounded $\beta\in C(\R)$ vanishing in a neighborhood of $0$,
\begin{equation}\label{eq:van-id}
 \int_{\R^3}\beta(\rho_n(t,x))\,dx
 =
 \int_{\R^3}\beta(\rho_{0,n}(x))\,dx
 \qquad\text{for a.e. }t\in(0,T).
\end{equation}
If
\begin{equation*}
 \rho_{0,n}\to\rho_0
 \qquad\text{in }L^\mathfrak p(\R^3),
\end{equation*}
then
\begin{equation}\label{eq:rho-strg}
 \rho_n\to\rho
 \qquad\text{in }C([0,T];L^\mathfrak p(\R^3)),
\end{equation}
where $\rho$ is the unique renormalized solution of
\begin{equation*}
 \partial_t\rho+V\cdot\nabla_x\rho=0,
 \qquad
 \rho(0)=\rho_0.
\end{equation*}
\end{theorem}

\begin{proof}[Proof of Theorem~\ref{thm:dlstab}]
Apply \cite[Theorem~II.5, pp.~527--528]{diperna1989ordinary} with
$b_n=-V_n$, $b=-V$, and $c=0$.  We use its global $L^\mathfrak{p}$ alternative,
whose stated range is $1\le \mathfrak p<\infty$; in particular, the endpoint
$\mathfrak{p}=1$ is part of the theorem rather than an interpolation consequence.
Since $V_n$ and $V$ are divergence free, $\operatorname{div}b_n=\operatorname{div}b=0$, so the uniform $L^1(0,T;L^\infty)$ bound on $\operatorname{div}b_n$
and the divergence part of condition (*) in that theorem are automatic.
The remaining Sobolev part of (*) is \eqref{eq:lone-strg}, while (**) is
\eqref{eq:lone-vort}.  By Chebyshev's inequality, \eqref{eq:van-bd} implies the
required boundedness in $L^\infty(0,T;L^0)$ in the notation of
\cite[p.~521]{diperna1989ordinary}.  Finally, because
$c=\operatorname{div}b_n=0$, \eqref{eq:van-id} is exactly the global identity
\cite[(31), p.~522]{diperna1989ordinary}.  Hence Theorem~II.5 gives
\eqref{eq:rho-strg}, including the endpoint $\mathfrak p=1$.
\end{proof}

\subsubsection{Decomposition of the kinetic vorticity}
Since the microscopic dissipation is controlled only in $L^2(\R^3)$, the microscopic forcing in \eqref{eq:bega-td-id} is not controlled in $L^p(\R^3)$ for $1<p<2$.
To apply the DiPerna--Lions--Majda framework, decompose
$\varpi^\e$ into the initial-data part $\varpi_A^\e$ and the forcing part
$\varpi_B^\e$.
\begin{definition}\label{def:vortsplit}
(1) Define $\varpi_A^\e$ and $\varpi_B^\e$ by
\begin{align}\label{eq:mm-id}
\bega
\partial_t \varpi^{\e}_A
+v^{\e,r}\partial_r\varpi^{\e}_A
+v^{\e,z}\partial_z\varpi^{\e}_A
&=0,
\qquad \hspace{3mm}
\varpi^{\e}_A|_{t=0}=\varpi^{\e}_0,
\cr
\partial_t \varpi^{\e}_B
+v^{\e,r}\partial_r\varpi^{\e}_B
+v^{\e,z}\partial_z\varpi^{\e}_B
&=\bPhi^{\e}_{\varpi},
\qquad
\varpi^{\e}_B|_{t=0}=0,
\enda
\end{align}
where $v^\e=\mathbb P u^\e$.
Then $\varpi^\e=\varpi_A^\e+\varpi_B^\e$ solves \eqref{eq:bega-td-id}.
(2) Define the corresponding azimuthal vorticities by
\begin{align}\label{eq:pp-def}
\bega
\omega^{\e,\vartheta}_A:=r\varpi^{\e}_A,
\qquad
\omega^{\e,\vartheta}_B:=r\varpi^{\e}_B,
\enda
\end{align}
so that
$\omega^{\e,\vartheta}=\omega^{\e,\vartheta}_A+\omega^{\e,\vartheta}_B$.
(3) Since the solution in Theorem~\ref{thm:longtime} is classical, $v^\e$ is smooth
on every fixed interval $[0,T]\subset[0,T^\e]$.  Define $u_A^\e$ as the
unique classical solution of the Leray-projected linear equation
\begin{equation}\label{eq:xv-id}
 \partial_tu_A^\e
 +\mathbb P\!\left[
 v^\e\cdot\nabla_xu_A^\e+(\nabla_xv^\e)^Tu_A^\e
 \right]=0,
 \qquad
 u_A^\e|_{t=0}=\mathbb Pu_0^\e.
\end{equation}
The right-hand side of \eqref{eq:xv-id} is divergence free, so
$\nabla_x\cdot u_A^\e=0$.  Equivalently, there is a scalar pressure
$\pi_A^\e$, unique up to a function of time, such that
\begin{equation}\label{eq:xp-id}
 \partial_tu_A^\e+v^\e\cdot\nabla_xu_A^\e
 +(\nabla_xv^\e)^Tu_A^\e+\nabla_x\pi_A^\e=0,
 \qquad
 \nabla_x\cdot u_A^\e=0.
\end{equation}
Set
\begin{equation}\label{eq:decomp-def}
 u_B^\e:=v^\e-u_A^\e,
 \qquad
 v^\e=u_A^\e+u_B^\e.
\end{equation}
The symmetry of \eqref{eq:xp-id} and uniqueness show that $u_A^\e$ and
$u_B^\e$ are axisymmetric without swirl.
\end{definition}
\begin{lemma}
The azimuthal vorticities $\omega_A^{\e,\vartheta}$ and
$\omega_B^{\e,\vartheta}$ satisfy
\begin{align}\label{eq:cm-id}
\bega
\partial_t \omega^{\e,\vartheta}_A
+v^{\e,r}\partial_r\omega^{\e,\vartheta}_A
+v^{\e,z}\partial_z\omega^{\e,\vartheta}_A
&=
\frac{v^{\e,r}}{r}\omega^{\e,\vartheta}_A,
\qquad \hspace{1.1cm}
\omega^{\e,\vartheta}_A|_{t=0}=\omega^{\e,\vartheta}_0,
\cr
\partial_t \omega^{\e,\vartheta}_B
+v^{\e,r}\partial_r\omega^{\e,\vartheta}_B
+v^{\e,z}\partial_z\omega^{\e,\vartheta}_B
&=
\frac{v^{\e,r}}{r}\omega^{\e,\vartheta}_B
+r\bPhi^{\e}_{\varpi},
\qquad
\omega^{\e,\vartheta}_B|_{t=0}=0.
\enda
\end{align}
Moreover,
\begin{equation}\label{eq:decomp-id}
 \nabla_x\times u_A^\e=\omega_A^{\e,\vartheta}e_\vartheta,
 \qquad
 \nabla_x\times u_B^\e=\omega_B^{\e,\vartheta}e_\vartheta,
 \qquad
 \nabla_x\cdot u_A^\e=\nabla_x\cdot u_B^\e=0.
\end{equation}
For each fixed $t$, $u_A^\e(t),u_B^\e(t)\in L^2(\R^3)$, and hence the
Fourier div--curl identity gives
\begin{equation}\label{eq:gam-th}
 u_\gamma^\e
 =\nabla_x\times(-\Delta_x)^{-1}
   \bigl(\omega_\gamma^{\e,\vartheta}e_\vartheta\bigr),
 \qquad \gamma\in\{A,B\},
\end{equation}
in $L^2(\R^3)$.
The velocity fields also satisfy
\begin{align}\label{eq:fld-id}
\bega
\partial_t u_A^\e
+\mathbb{P}u^\e\cdot\nabla_x u_A^\e
+(\nabla_x \mathbb{P}u^\e)^T u_A^\e
+\nabla_x \pi_A^\e
&=0,
\qquad
u^{\e}_A|_{t=0}=\mathbb{P}u^{\e}_0,
\cr
\partial_t u_B^\e
+\mathbb{P}u^\e\cdot\nabla_x u_B^\e
+(\nabla_x \mathbb{P}u^\e)^T u_B^\e
+\nabla_x \pi_B^\e
&=
\nabla_x\times(-\Delta_x^{-1})
\big(r\bPhi^{\e}_{\varpi}e_{\vartheta}\big),
\qquad
u^{\e}_B|_{t=0}=0.
\enda
\end{align}
\end{lemma}
\begin{proof}
Substituting $\varpi_A^\e=\omega_A^{\e,\vartheta}/r$ and
$\varpi_B^\e=\omega_B^{\e,\vartheta}/r$ into \eqref{eq:mm-id} gives
\eqref{eq:cm-id}.
For each fixed $\e$ and $T<T^\e$, the coefficients in
\eqref{eq:xv-id} are smooth.  The boundedness of $\mathbb P$ on $H^m$
and the usual energy estimate for a linear first-order system give
existence and uniqueness in $C([0,T];H^m)$ for every finite $m$ allowed
by the regularity of the kinetic solution and the initial datum.  Thus the
classical solution used in Definition~\ref{def:vortsplit} is well defined.
Set $\widetilde\omega_A^\e:=\nabla_x\times u_A^\e$.  Taking the curl of
\eqref{eq:xp-id} and using $\nabla_x\cdot v^\e=0$ gives
\begin{equation*}
 \partial_t\widetilde\omega_A^\e+v^\e\cdot\nabla_x\widetilde\omega_A^\e
 =\widetilde\omega_A^\e\cdot\nabla_xv^\e.
\end{equation*}
The vector field $\omega_A^{\e,\vartheta}e_\vartheta$ satisfies the same
equation by \eqref{eq:cm-id}.  At $t=0$,
\[
 \nabla_x\times\mathbb Pu_0^\e
 =\nabla_x\times u_0^\e
 =\omega_0^{\e,\vartheta}e_\vartheta,
\]
because the Leray projection changes a vector field only by a gradient.
Uniqueness for this smooth linear equation therefore gives
the first curl identity in \eqref{eq:decomp-id}.  The second follows from
\eqref{eq:decomp-def} and
$\nabla_x\times v^\e=\omega^{\e,\vartheta}e_\vartheta$.
Taking the $L^2$ inner product of \eqref{eq:xp-id} with $u_A^\e$ yields
\[
 \frac12\frac{d}{dt}\|u_A^\e\|_2^2
 =-\int_{\R^3}u_A^\e\otimes u_A^\e: \nabla_x v^\e\,dx,
\]
so $u_A^\e(t)\in L^2$ on every fixed interval by Gronwall, because
\eqref{eq:ac-bd} and the Sobolev embedding $H^3\hookrightarrow W^{1,\infty}$ give
$\int_0^T\|\nabla_xv^\e(t)\|_\infty\,dt<\infty$.  Then
\eqref{eq:decomp-def} and $v^\e\in L^2$ give
$u_B^\e(t)\in L^2$.  For an $L^2$ divergence-free field $U$,
Plancherel gives
\[
 \widehat U(\zeta)
 =\frac{i\zeta\times\widehat{\nabla\times U}(\zeta)}{|\zeta|^2}
 \quad\text{for a.e. }\zeta\ne0.
\]
Applying this to $u_A^\e$ and $u_B^\e$ proves \eqref{eq:gam-th}.
The first equation in \eqref{eq:fld-id} is \eqref{eq:xp-id}.  For the second,
write the second equation in \eqref{eq:cm-id} in vector form,
\[
 \partial_t\omega_B^\e+v^\e\cdot\nabla_x\omega_B^\e
 =\omega_B^\e\cdot\nabla_xv^\e+r\bPhi_\varpi^\e e_\vartheta.
\]
Using
\[
 \nabla_x\times\bigl(v^\e\cdot\nabla_xu_B^\e
 +(\nabla_xv^\e)^Tu_B^\e\bigr)
 =v^\e\cdot\nabla_x\omega_B^\e-\omega_B^\e\cdot\nabla_xv^\e
\]
and \eqref{eq:gam-th}, the curl of the difference between the two sides of
the second equation in \eqref{eq:fld-id} vanishes.  Since every other term is
in the classical energy class, the Helmholtz decomposition identifies this
curl-free difference with a gradient; absorbing it into $\pi_B^\e$ proves
\eqref{eq:fld-id}.
\end{proof}
\begin{lemma}\label{lem:splitconv}
Under the assumptions of Theorem~\ref{thm:longtime}, let
$\{F^{\e}\}_{\e>0}$ be the corresponding kinetic solutions on
$[0,T^\e]$, where $T^\e\to\infty$ as $\e\to0$, and let
$\varpi_A^\e$ and $\varpi_B^\e$ be defined by \eqref{eq:mm-id}.
Then the following properties hold on $[0,T]$ for any $T\in(0,T^\e]$.
\begin{enumerate}
\item The solution $\varpi_A^\e$ in \eqref{eq:mm-id} is renormalized in the DiPerna--Lions sense.  For every bounded
$\beta\in C(\R)$ vanishing in a neighborhood of $0$,
\begin{equation}\label{eq:bet-dx}
 \int_{\R^3}\beta(\varpi_A^\e(t,x))\,dx
 =
 \int_{\R^3}\beta(\varpi_0^\e(x))\,dx,
 \qquad 0\le t\le T.
\end{equation}
Consequently, if $\varpi_0^\e\in L^p(\R^3)$, $1\le p\le\infty$, then
\begin{equation}\label{eq:strg-bd}
 \|\varpi_A^\e(t)\|_{L^p(\R^3)}
 =\|\varpi_0^\e\|_{L^p(\R^3)},
 \qquad 0\le t\le T.
\end{equation}
If $1\le p<\infty$, then moreover
\begin{equation}\label{eq:lp-decomp}
 \varpi_A^\e\in C([0,T];L^p(\R^3)).
\end{equation}
Define
\begin{equation*}
 \varpi_{A,+}^\e:=(\varpi_A^\e)_+,
 \qquad
 \varpi_{A,-}^\e:=(-\varpi_A^\e)_+,
 \qquad
 \varpi_{0,+}^\e:=(\varpi_0^\e)_+,
 \qquad
 \varpi_{0,-}^\e:=(-\varpi_0^\e)_+.
\end{equation*}
Let $X^\e$ be the flow of $v^\e$,
\begin{equation}\label{eq:strg-id}
 \partial_tX^\e(t,x)=v^\e(t,X^\e(t,x)),
 \qquad X^\e(0,x)=x.
\end{equation}
Then
\begin{align}
 \det\nabla_xX^\e(t,x)&=1,\label{eq:xx-id}\\
 \varpi_A^\e(t,X^\e(t,x))&=\varpi_0^\e(x),\label{eq:sign-decomp}\\
 \varpi_{A,\pm}^\e(t,X^\e(t,x))&=\varpi_{0,\pm}^\e(x).\label{eq:sign-psi}
\end{align}
For each sign and every
$\psi\in C_c^\infty([0,T)\times\R^3)$,
\begin{equation}\label{eq:td-psi}
 \int_0^T\!\int_{\R^3}\varpi_{A,\pm}^\e
 \bigl(\partial_t\psi+v^\e\cdot\nabla_x\psi\bigr)
 \,dx\,dt
 +\int_{\R^3}\varpi_{0,\pm}^\e\psi(0,x)\,dx=0.
\end{equation}
\item The forcing parts $\varpi_B^\e$, $\omega_B^\e$, and $u_B^\e$ defined
in Definition~\ref{def:vortsplit} satisfy, for every sufficiently small $\delta>0$,
\begin{align}
&\sup_{0\leq t\leq T}\|\varpi^{\e}_B(t)\|_{L^q(\R^3)}
 \leq C_{T,\delta}\Big(\e^{\frac12-\delta}+\kappa^{\frac12-\delta}\Big)
 \to0,
 \qquad 2\leq q\leq6, \label{eq:del-kap-bd}\\
&\sup_{0\leq t\leq T}\|\omega^{\e}_B(t)\|_{L^q(\R^3)}
 \leq C_{T,\delta}\Big(\e^{\frac12-\delta}+\kappa^{\frac12-\delta}\Big)
 \to0,
 \qquad 2\leq q\leq\infty, \label{eq:del-decomp}\\
&\sup_{0\leq t\leq T}\|u^{\e}_B(t)\|_{L^q(\R^3)}
 \leq C_{T,\delta}\Big(\e^{\frac12-\delta}+\kappa^{\frac12-\delta}\Big)
 \to0,
 \qquad 2\leq q\leq\infty. \label{eq:del-kin}
\end{align}
\end{enumerate}
\end{lemma}
\begin{proof}
Since $N\ge5$, the boundedness of $\mathbb P$ on $H^3(\R^3)$,
\eqref{eq:ac-bd}, and $H^3(\R^3)\hookrightarrow W^{1,\infty}(\R^3)$ give
\begin{equation}\label{eq:ct-bd}
 \int_0^T\|\mathbb Pu^\e(t)\|_{W^{1,\infty}}\,dt
 \le CT\sup_{0\le t\le T}\|u^\e(t)\|_{H^3}<\infty.
\end{equation}
By \eqref{eq:pq-def}, the Jacobian of \eqref{eq:strg-id} is constant and equals one;
the transport equation \eqref{eq:mm-id} then gives
\eqref{eq:xx-id}--\eqref{eq:sign-psi}.
Using \eqref{eq:xx-id} and
\eqref{eq:sign-psi}, for every
$\psi\in C_c^\infty([0,T)\times\R^3)$,
\begin{align*}
 &\int_0^T\!\int_{\R^3}\varpi_{A,\pm}^\e
 \bigl(\partial_t\psi+v^\e\cdot\nabla_x\psi\bigr)
 \,dx\,dt
 +\int_{\R^3}\varpi_{0,\pm}^\e\psi(0,x)\,dx\\
 &\qquad=
 \int_{\R^3}\varpi_{0,\pm}^\e(x)
 \int_0^T\frac{d}{dt}\psi(t,X^\e(t,x))\,dt\,dx
 +\int_{\R^3}\varpi_{0,\pm}^\e\psi(0,x)\,dx
 =0.
\end{align*}
This proves \eqref{eq:td-psi}.
For every such $\beta$ and every
$\psi\in C_c^\infty([0,T)\times\R^3)$,
\eqref{eq:xx-id}--\eqref{eq:sign-decomp} yield
\begin{align*}
 &\int_0^T\int_{\R^3}
 \beta(\varpi_A^\e)
 \bigl(\partial_t\psi+v^\e\cdot\nabla_x\psi\bigr)
 \,dx\,dt
 +\int_{\R^3}\beta(\varpi_0^\e)\psi(0)\,dx\notag\\
 &\qquad=
 \int_{\R^3}\beta(\varpi_0^\e(x))
 \int_0^T\frac{d}{dt}\psi(t,X^\e(t,x))\,dt\,dx
 +\int_{\R^3}\beta(\varpi_0^\e)\psi(0)\,dx
 =0.
\end{align*}
Thus $\varpi_A^\e$ is renormalized.  Using
\eqref{eq:sign-decomp} and the change of variables
$y=X^\e(t,x)$, whose Jacobian equals one by
\eqref{eq:xx-id}, gives \eqref{eq:bet-dx}.
For $1\le p<\infty$, define
\begin{equation*}
 \beta_{m,k}(\lambda):=\min\bigl\{(|\lambda|-k^{-1})_+^p,m\bigr\},
 \qquad m,k\ge1.
\end{equation*}
Applying \eqref{eq:bet-dx} to these truncations, then letting
$k\to\infty$ and $m\to\infty$, the monotone convergence theorem gives
\eqref{eq:strg-bd}.  The case $p=\infty$ follows directly from
\eqref{eq:sign-decomp}.
To prove \eqref{eq:lp-decomp}, fix $1\le p<\infty$ and choose
$f_n\in C_c^\infty(\R^3)$ such that
$f_n\to\varpi_0^\e$ in $L^p(\R^3)$.  Set $f_n(t):=f_n\circ (X^\e(t,\cdot))^{-1}$.
By \eqref{eq:xx-id} and \eqref{eq:sign-decomp},
\begin{equation*}
 \|\varpi_A^\e(t)-f_n(t)\|_{L^p}
 =\|\varpi_0^\e-f_n\|_{L^p},
 \qquad 0\le t\le T.
\end{equation*}
Since \eqref{eq:ct-bd} makes $X^\e$ a bi-Lipschitz flow and
$\mathbb Pu^\e\in L^1(0,T;L^\infty)$, for fixed $n$ the maps
$(X^\e(t,\cdot))^{-1}$ converge locally uniformly as $t\to s$, while
$f_n(t)$ have supports in one compact set for $0\le t\le T$.
Hence dominated convergence gives
\begin{equation*}
 \|f_n(t)-f_n(s)\|_{L^p}\to0
 \qquad\text{as }t\to s.
\end{equation*}
Therefore the isometry just proved and continuity imply
\[
 \|\varpi_A^\e(t)-\varpi_A^\e(s)\|_{L^p}
 \le
 2\|\varpi_0^\e-f_n\|_{L^p}
 +\|f_n(t)-f_n(s)\|_{L^p}.
\]
Letting $t\to s$ then $n\to\infty$ proves
\eqref{eq:lp-decomp}.
Recall $a_{\rm ac}^\e,q_{\rm ac}^\e$ defined in \eqref{eq:acoustic-source-def}. Fix sufficiently small $\delta>0$ and set
\begin{equation}\label{eq:xa-def}
 \mathcal S_\varpi^\e:=r\bPhi_\varpi^\e e_\vartheta,
 \qquad
 \mathcal G_\varpi^\e
 :=q_{\rm ac}^\e\times\omega^\e
 -k_B\frac{\Theta^\e-1}{\e}\nabla_x a_{\rm ac}^\e
 -\frac1{\e^2\mathrm P^\e}\nabla_x\cdot\mathbf r^\e.
\end{equation}
\begin{equation*}
 \nabla_x\!\left(\frac{\Theta^\e-1}{\e}\right)
 =\Theta^\e\nabla_x\theta^\e,
 \qquad
 \nabla_x\times\left(
 \frac{\Theta^\e-1}{\e}\nabla_xa^\e
 \right)
 =\Theta^\e\nabla_x\theta^\e\times\nabla_xa^\e.
\end{equation*}
Lemma~\ref{lem:sympres} shows that $q_{\rm ac}^\e$ is axisymmetric without swirl and
$\omega^\e$ is azimuthal.  Moreover, the moment definition and the same
$O(2)_z$ symmetry give
$\mathbf r^\e(t,Gx)=G\mathbf r^\e(t,x)G^T, \  G\in O(2)_z.$
Hence $(\mathrm P^\e)^{-1}\nabla_x\cdot\mathbf r^\e$ is equivariant
and has no swirl, so its curl is azimuthal.  Together with
Definition~\ref{def:vortsplit}, \eqref{eq:pq-th}, and these identities, this gives
\begin{equation}\label{eq:ctrl-id}
 \mathcal S_\varpi^\e=\nabla_x\times\mathcal G_\varpi^\e.
\end{equation}
First control the acoustic factor.  Choose $\sigma>0$ so small that
\[
 \sigma\left(1+\frac{1-\mathfrak y}{2c}\right)<\frac\delta2,
 \qquad
 r_\sigma:=\left(\frac12-\sigma\right)^{-1}.
\]
For $j=0,1,2,3$, apply \eqref{eq:bega-rho-bd} with $p=\infty$, $\mathfrak r=r_\sigma$, and $s=j+1+\sigma$; the resulting Besov index is exactly $j$, so \eqref{eq:lw-bd} yields the required $W^{3,\infty}$ norm.  The top case for $N=5$ contributes $\kappa^{-\sigma/2}$, while for $N>5$ the same factor may be inserted harmlessly because $0<\kappa<1$.  Set
\[
 \mathfrak S_{\e,T}:=
 \mathcal E_{tot}^{N}(0)^{1/2}
 +\int_0^T\left(
 \mathcal E_{tot}^{N}(s)
 +\kappa^{1/2}(\mathcal D_{tot}^{N}(s))^{1/2}
 \right)ds.
\]
By \eqref{eq:init-bd}--\eqref{eq:ac-bd},
$\mathfrak S_{\e,T}\lesssim_T|\log\e|^{15/16}$ for fixed $T$.
Using \eqref{eq:decomp-bd}--\eqref{eq:tau-ac}, we obtain
\begin{equation}\label{eq:sig-kap}
 \sum_{j=0}^3
 \|\nabla_x^j(q_{\rm ac}^\e,a_{\rm ac}^\e)\|_{L^{r_\sigma}(0,T;L^\infty_x)}
 \lesssim_{T,\sigma}
 \e^{\frac12-\sigma}\kappa^{-\frac\sigma2}
 \mathfrak S_{\e,T}.
\end{equation}
The product rule, \eqref{eq:hn-en}--\eqref{eq:hn-bd},
\eqref{eq:hn-kap} and the definitions
\eqref{eq:diss-def}, \eqref{eq:kin-def}, and \eqref{eq:tot-def} give,
for a.e. $t\in(0,T)$,
\begin{align}
 \|\mathcal S_\varpi^\e(t)\|_{H^2_x}
 +\|\mathcal G_\varpi^\e(t)\|_{L^2_x}
 &\lesssim
 \bigl(1+(\mathcal E_{tot}^{N}(t))^{1/2}\bigr)
 \sum_{j=0}^3
 \|\nabla_x^j(q_{\rm ac}^\e,a_{\rm ac}^\e)(t)\|_{L^\infty_x}
  +
 \kappa^{1/2}(\mathcal D_G^{N}(t))^{1/2}.
 \label{eq:kap-fst}
\end{align}
Indeed, the first term in \eqref{eq:xa-def} is estimated by
\begin{equation*}
 \|\nabla_x\times(q_{\rm ac}^\e\times\omega^\e)\|_{H^2_x}
 +\|q_{\rm ac}^\e\times\omega^\e\|_2
 \lesssim
 \|q_{\rm ac}^\e\|_{W^{3,\infty}_x}
 \bigl(\|\omega^\e\|_{H^3_x}+\|\omega^\e\|_2\bigr),
\end{equation*}
the second by
\begin{align*}
 &\|\Theta^\e\nabla_x\theta^\e\times\nabla_xa_{\rm ac}^\e\|_{H^2_x}
 +\left\|\frac{\Theta^\e-1}{\e}\nabla_xa_{\rm ac}^\e\right\|_2
  \lesssim
 \bigl(1+(\mathcal E_{tot}^{N}(t))^{1/2}\bigr)
 \|a_{\rm ac}^\e\|_{W^{3,\infty}_x},
\end{align*}
and the last by
\begin{align*}
 &\left\|\frac1{\e^2}\nabla_x\times\left(
 \frac1{\mathrm P^\e}\nabla_x\cdot\mathbf r^\e\right)\right\|_{H^2_x}
 +\left\|\frac1{\e^2\mathrm P^\e}\nabla_x\cdot\mathbf r^\e\right\|_2
 \lesssim
 \kappa^{1/2}(\mathcal D_G^{N}(t))^{1/2}.
\end{align*}
By \eqref{eq:kap-fst}, H\"older's inequality in time,
\eqref{eq:sig-kap}, and \eqref{eq:ac-bd}, for some fixed $C>0$,
\[
 \int_0^T\left(
 \|\mathcal S_\varpi^\e(t)\|_{H^2_x}
 +\|\mathcal G_\varpi^\e(t)\|_2
 \right)dt
 \lesssim_{T,\sigma}
 \e^{\frac12-\sigma}\kappa^{-\frac\sigma2}
 \left|\log\frac1\e\right|^C
 +\kappa^{1/2}\left|\log\frac1\e\right|^C.
\]
By the choice of $\sigma$ and the scale relation \eqref{eq:choice-bd},
the first prefactor has a strictly larger power than $\e^{1/2-\delta/2}$.
Fixed logarithmic powers can then be absorbed into the remaining $\delta/2$
loss, and similarly into an arbitrarily small power of $\kappa$.  Hence
\begin{equation}\label{eq:dt-del}
 \int_0^T\left(
 \|\mathcal S_\varpi^\e(t)\|_{H^2_x}
 +\|\mathcal G_\varpi^\e(t)\|_2
 \right)dt
 \le C_{T,\delta}
 \left(\e^{\frac12-\delta}
 +\kappa^{\frac12-\delta}\right).
\end{equation}
Next control the flow amplification.  Since $N\ge5$, \eqref{eq:ac-bd} and
$H^3(\R^3)\hookrightarrow W^{1,\infty}(\R^3)$ give, for fixed $T$,
\[
 \sup_{0\le t\le T}\|\nabla_x\mathbb Pu^\e(t)\|_\infty
 \lesssim |\log\e|^{15/32}.
\]
Hence
\begin{equation}\label{eq:pp-dt}
 \exp\left(
 \int_0^T\|\nabla_x\mathbb Pu^\e(t)\|_\infty\,dt
 \right)
 \le \exp\big(C_T|\log\e|^{15/32}\big)=\e^{-o(1)}.
\end{equation}
Combined with \eqref{eq:dt-del}, this factor preserves every strictly positive
power of $\e$ or $\kappa$ below.
For \eqref{eq:del-kap-bd}, since the flow \eqref{eq:strg-id} is
measure preserving by \eqref{eq:xx-id}, the second equation
in \eqref{eq:mm-id}, \eqref{eq:xa-def}, and
\eqref{eq:xa-bd} give, for $2\le q\le6$,
\begin{align*}
 \|\varpi_B^\e(t)\|_{L^q_x}
 &\le\int_0^t\|\bPhi_\varpi^\e(s)\|_{L^q_x}\,ds
 =\int_0^t\left\|\frac{(\mathcal S_\varpi^\e)^\vartheta}{r}(s)\right\|_{L^q_x}\,ds
 \le\int_0^t\|\nabla_x\mathcal S_\varpi^\e(s)\|_{L^q_x}\,ds.
\end{align*}
For $2\le q\le6$, interpolation between $L^2$ and $L^6$ and Sobolev
embedding give
\[
 \|\nabla_x\mathcal S_\varpi^\e\|_{L^q_x}
 \lesssim \|\mathcal S_\varpi^\e\|_{H^2_x}.
\]
Together with \eqref{eq:dt-del}, this proves \eqref{eq:del-kap-bd}.
For the azimuthal vorticity, the second equation in \eqref{eq:cm-id},
together with $|v^{\e,r}/r|\le\|\nabla_x\mathbb Pu^\e\|_\infty$
and \eqref{eq:xa-def}, gives
\begin{equation*}
 \|\omega_B^\e(t)\|_{L^q_x}
 \le
 \exp\left(\int_0^t\|\nabla_x\mathbb Pu^\e(\tau)\|_\infty d\tau\right)
 \int_0^t\|\mathcal S_\varpi^\e(s)\|_{L^q_x}\,ds,
 \qquad 2\le q\le\infty.
\end{equation*}
By $H^2(\R^3)\hookrightarrow L^\infty(\R^3)$ and interpolation,
\[
 \|\mathcal S_\varpi^\e\|_{L^q_x}
 \lesssim\|\mathcal S_\varpi^\e\|_{H^2_x},
 \qquad 2\le q\le\infty.
\]
Together with \eqref{eq:dt-del}, \eqref{eq:pp-dt}, and the vorticity estimate just obtained,
this proves \eqref{eq:del-decomp}, since the factor
in \eqref{eq:pp-dt} is $\e^{-o(1)}$.
Finally, \eqref{eq:ctrl-id} gives
\begin{equation*}
 \nabla_x\times(-\Delta_x)^{-1}
 \mathcal S_\varpi^\e
 =\nabla_x\times(-\Delta_x)^{-1}
 \nabla_x\times\mathcal G_\varpi^\e
 =\mathbb P\mathcal G_\varpi^\e,
 \qquad
 \|\mathbb P\mathcal G_\varpi^\e\|_2
 \le\|\mathcal G_\varpi^\e\|_2.
\end{equation*}
Indeed, the Fourier multiplier in this identity is
$I-\zeta\otimes\zeta/|\zeta|^2$ for $\zeta\ne0$.
Set
$Z_\lambda(t):=\bigl(\|u_B^\e(t)\|_2^2+\lambda^2\bigr)^{1/2}, \quad \lambda>0.$
Taking the $L^2$ inner product of the second equation in
\eqref{eq:fld-id} with $u_B^\e$, using
$\nabla_x\cdot\mathbb Pu^\e=0$, and applying
this projection identity, we obtain for a.e. $t$,
\begin{equation*}
 Z_\lambda'(t)
 \le
 \|\nabla_x\mathbb Pu^\e(t)\|_\infty Z_\lambda(t)
 +\|\mathcal G_\varpi^\e(t)\|_2.
\end{equation*}
Since $u_B^\e(0)=0$, applying Gronwall's inequality to this differential inequality and then letting $\lambda\downarrow0$ gives
\begin{equation*}
 \sup_{0\le t\le T}\|u_B^\e(t)\|_2
 \le
 \exp\left(\int_0^T\|\nabla_x\mathbb Pu^\e(t)\|_\infty dt\right)
 \int_0^T\|\mathcal G_\varpi^\e(t)\|_2dt.
\end{equation*}
Equations \eqref{eq:dt-del}, \eqref{eq:pp-dt}, and this $L^2$ estimate,
together with $\e^\delta|\log\e|^m,\kappa^\delta|\log\e|^m\to0$
for every $\delta,m>0$, yield
\begin{equation}\label{eq:del-kin-bd}
 \sup_{0\le t\le T}\|u_B^\e(t)\|_2
 \le C_{T,\delta}
 \left(\e^{\frac12-\delta}+\kappa^{\frac12-\delta}\right).
\end{equation}
For the $L^\infty$ bound, decompose
$\mathbf K=\mathbf K_1+\mathbf K_2$ with
$\mathbf K_1=\mathbf K\mathbf1_{\{|x|\le1\}}$ and
$\mathbf K_2=\mathbf K\mathbf1_{\{|x|>1\}}$.
Since $|\mathbf K(x)|\lesssim|x|^{-2}$,
$\mathbf K_1\in L^1(\R^3)$ and $\mathbf K_2\in L^2(\R^3)$.
By \eqref{eq:gam-th}, Young's inequality, and \eqref{eq:del-decomp},
\begin{align*}
 \|u_B^\e(t)\|_\infty
 &\le
 \|\mathbf K_1\|_1\|\omega_B^\e(t)\|_\infty
 +\|\mathbf K_2\|_2\|\omega_B^\e(t)\|_2
 \le C_{T,\delta}
 \left(\e^{\frac12-\delta}+\kappa^{\frac12-\delta}\right).
\end{align*}
Interpolating this estimate with \eqref{eq:del-kin-bd} proves
\eqref{eq:del-kin} for
$2\le q\le\infty$.
\end{proof}
\subsubsection{Compactness of the velocity}
Nobili--Seis \cite{nobili2022renormalization} give, for axisymmetric
no-swirl $u$ and $1<p<2$,
\begin{equation}\label{eq:xu-bd}
 \left\|\frac{\nabla_xu}{r}\right\|_{L^p(\R^3)}
 \lesssim \left\|\frac{\omega^\vartheta}{r}\right\|_{L^p(\R^3)}.
\end{equation}
\begin{definition}[Truncated meridional half-space]
Recall the half-space $\Hh$ from Section~\ref{sec:bscompact}.
For $R>0$, define the truncated half-space
\begin{align*}
\mathbb{H}_R:=\{(r,z)\in\R^2:r>0,\ \sqrt{r^2+z^2}<R\}.
\end{align*}
For an axisymmetric vector field $u$, the notation
$\|u\|_{W^{1,p}(\mathbb H_R)}$ below means
\[
 \|u\|_{L^p(\mathbb H_R;dr\,dz)}
 +\|\nabla_xu\|_{L^p(\mathbb H_R;dr\,dz)},
\]
where the full Cartesian gradient is evaluated on a meridional half-plane.
In particular, the angular contribution $u^r/r$ is included in this norm.
\end{definition}
We first justify the unweighted meridional norm at the axis.  By
$dx=2\pi r\,dr\,dz$, $r\le R$ on $\mathbb H_R$, \eqref{eq:xu-bd}, and
\eqref{eq:strg-bd}, for $1<p<2$,
\begin{equation}\label{eq:norm-bd}
 \|\nabla_xu_A^\e(t)\|_{L^p(\mathbb H_R)}
 \le C R^{1-1/p}
 \left\|\frac{\nabla_xu_A^\e(t)}r\right\|_{L^p(\R^3)}
 \lesssim_p R^{1-1/p}\|\varpi_0^\e\|_{L^p}.
\end{equation}
For the zero-order term, Poincar\'e's inequality on the half-disc gives
\begin{equation*}
 \|u_A^\e(t)\|_{L^p(\mathbb H_R)}
 \lesssim_p R\|\partial_{r,z}u_A^\e(t)\|_{L^p(\mathbb H_R)}
 +|\mathbb H_R|^{1/p-1}\|u_A^\e(t)\|_{L^1(\mathbb H_R)}.
\end{equation*}
It remains to estimate the last term.  By \eqref{eq:gam-th}, $u_A^\e$ is the
finite-energy axisymmetric Biot--Savart field generated by
$\omega_A^{\e,\vartheta}$.  The meridional kernel estimate
\cite[Proposition~2.3, (2.11)]{gallay2015remarks} gives
\begin{equation}\label{eq:sav-bd}
 |u_A^\e(t,r,z)|
 \lesssim
 \int_{\Hh}
 \frac{|\omega_A^{\e,\vartheta}(t,\bar r,\bar z)|}
 {|r-\bar r|+|z-\bar z|}\,d\bar r\,d\bar z.
\end{equation}
Split the integral in \eqref{eq:sav-bd} into $\mathbb H_{2R}$ and its
complement.  On the near part, Young's inequality in $\R^2$ and
$\int_{|Y|\le3R}|Y|^{-1}\,dY\lesssim R$ give
\begin{equation*}
 \|u_{A,\mathrm{near}}^\e(t)\|_{L^1(\mathbb H_R)}
 \lesssim R\|\omega_A^{\e,\vartheta}(t)\|_{L^1(\Hh)}.
\end{equation*}
If $(r,z)\in\mathbb H_R$ and $(\bar r,\bar z)\notin\mathbb H_{2R}$,
then $|(r,z)-(\bar r,\bar z)|\ge R$; hence
\begin{equation*}
 \|u_{A,\mathrm{far}}^\e(t)\|_{L^1(\mathbb H_R)}
 \lesssim \frac{|\mathbb H_R|}{R}
 \|\omega_A^{\e,\vartheta}(t)\|_{L^1(\Hh)}
 \lesssim R\|\omega_A^{\e,\vartheta}(t)\|_{L^1(\Hh)}.
\end{equation*}
Since $\omega_A^{\e,\vartheta}=r\varpi_A^\e$,
\begin{equation}\label{eq:comp-id}
 \|\omega_A^{\e,\vartheta}(t)\|_{L^1(\Hh)}
 =\frac1{2\pi}\|\varpi_A^\e(t)\|_{L^1(\R^3)}
 =\frac1{2\pi}\|\varpi_0^\e\|_{L^1(\R^3)}.
\end{equation}
Combining \eqref{eq:norm-bd}--\eqref{eq:comp-id} yields
\begin{equation}\label{eq:lone-lp}
 \sup_{0\le t\le T}\|u_A^\e(t)\|_{W^{1,p}(\mathbb H_R)}
 \lesssim_{p,R}\|\varpi_0^\e\|_{L^1\cap L^p}.
\end{equation}
The two-dimensional Sobolev embedding on $\mathbb H_R$ then gives
\begin{equation}\label{eq:lone-lp-bd}
 \sup_{0\le t\le T}\|u_A^\e(t)\|_{L^q(\mathbb H_R)}
 \lesssim_{p,q,R}\|\varpi_0^\e\|_{L^1\cap L^p},
 \qquad 2<q<\frac{2p}{2-p}.
\end{equation}
For $u\in L^2(\R^3)$ with $\nabla\cdot u=0$ and
$\omega=\nabla\times u$, Fourier inversion gives
\begin{equation}\label{eq:four-id}
 u=\nabla\times(-\Delta)^{-1}\omega
 \qquad\text{in }\mathcal S'(\R^3).
\end{equation}
\begin{lemma}\label{lem:locreg}
Let $2\le p<\infty$ and $R>0$. Suppose
$u\in L^2(\R^3;\R^3)$ satisfies
\begin{equation}\label{eq:comp-def}
 \nabla\cdot u=0,
 \qquad
 \omega:=\nabla\times u\in L^p(B_{2R}).
\end{equation}
Then
\begin{equation}\label{eq:ltwo-om}
 \|u\|_{W^{1,p}(B_R)}
 \le C_{p,R}
 \left(\|u\|_{L^2(\R^3)}+\|\omega\|_{L^p(B_{2R})}\right).
\end{equation}
\end{lemma}
The proof is recorded in Appendix~\ref{sec:auxana}.

We use Simon's compactness criterion
\cite[Corollary~4]{simon1987compact}: if $X\Subset B\hookrightarrow Y$,
then bounded subsets of
$L^\infty(0,T;X)\cap W^{1,\infty}(0,T;Y)$ are relatively compact in
$C([0,T];B)$.
\begin{lemma}
For every $R>0$ there exists $s>5/2$ such that
\begin{equation}\label{eq:td-comp}
 \sup_{\e>0}
 \|\partial_tu_A^\e\|_{L^\infty(0,T;H^{-s}(B_R))}<\infty.
\end{equation}
\end{lemma}
\begin{proof}
Recall that
\begin{equation}\label{eq:vel-def}
 v^\e=\mathbb Pu^\e=u_A^\e+u_B^\e.
\end{equation}
By \eqref{eq:fld-id}, $\nabla\cdot v^\e=0$, and
$\nabla\cdot u_A^\e=0$. Since
\[
 ((\nabla u_A^\e)^Tu_A^\e)_i
 =\partial_i u_{A,j}^\e\,u_{A,j}^\e
 =\frac12\partial_i|u_A^\e|^2,
\]
application of the Leray projector to $\eqref{eq:fld-id}_1$ gives
\begin{equation}\label{eq:leray-id}
 \partial_tu_A^\e
 =
 -\mathbb P\left[
 \partial_j\!\left(v_j^\e u_A^\e\right)
 +(\nabla u_B^\e)^Tu_A^\e
 \right]
 \qquad\text{in }\mathcal D'(\R^3).
\end{equation}
For $\psi\in H^s(\R^3;\R^3)$ and $s>5/2$, the Sobolev embedding
$H^s(\R^3)\hookrightarrow W^{1,\infty}(\R^3)$ yields
\begin{align*}
 \left|
 \left\langle
 \partial_j(v_j^\e u_A^\e),\psi
 \right\rangle
 \right|
 &\le
 \|v^\e\|_2\|u_A^\e\|_2
 \|\nabla\psi\|_\infty, \notag \\
 \left|
 \left\langle
 (\nabla u_B^\e)^Tu_A^\e,\psi
 \right\rangle
 \right|
 &\le
 \|\nabla u_B^\e\|_2\|u_A^\e\|_2
 \|\psi\|_\infty.
\end{align*}
Since $\mathbb P$ is bounded on $H^{-s}(\R^3)$,
\eqref{eq:leray-id} and these estimates imply
\begin{equation}\label{eq:td-pp}
 \|\partial_tu_A^\e(t)\|_{H^{-s}(\R^3)}
 \le
 C_s
 \left(
 \|v^\e(t)\|_2+\|\nabla u_B^\e(t)\|_2
 \right)
 \|u_A^\e(t)\|_2.
\end{equation}
By \eqref{eq:decomp-id}, $\nabla\cdot u_B^\e=0$ and
$\nabla\times u_B^\e=\omega_B^\e$. Hence Plancherel's identity gives
\begin{align}
 \|\nabla u_B^\e(t)\|_2^2
 &=
 \|\nabla\times u_B^\e(t)\|_2^2
 +\|\nabla\cdot u_B^\e(t)\|_2^2 =
 \|\omega_B^\e(t)\|_2^2. \notag 
\end{align}
Therefore \eqref{eq:del-decomp} yields
\begin{equation*}
 \sup_{0\le t\le T}\|\nabla u_B^\e(t)\|_2
 \to0
 \qquad\text{as }\e\to0.
\end{equation*}
Moreover, \eqref{eq:ltwo-del}, \eqref{eq:del-kin}, and \eqref{eq:vel-def} give
\begin{equation*}
 \sup_{\e>0}\sup_{0\le t\le T}
 \left(
 \|v^\e(t)\|_2+\|u_A^\e(t)\|_2+\|u_B^\e(t)\|_2
 \right)<\infty.
\end{equation*}
Equation \eqref{eq:td-pp}, together with the convergence of $\nabla u_B^\e$ just established and the uniform $L^2$ bound, proves the uniform bound in $H^{-s}(\R^3)$, hence
\eqref{eq:td-comp}.
\end{proof}
\subsubsection{Identification of the rough Euler limit}\label{sec:proof-rough-limit}

\begin{proof}[Proof of Theorem \ref{thm:roughlimit}]

By exact moment matching \eqref{eq:xd-id}, $G^\e=\Pi_M^\perp F^\e$.  Hence
\eqref{eq:diss-def} and \eqref{eq:ac-bd} give
\[
 \sup_{0\le t\le T^\e}
 \left\|\frac{G^\e(t)}{\sqrt{M^\e(t)}}\right\|_{L^2_{x,\xi}}^2
 \le \e^2\sup_{0\le t\le T^\e}\mathcal E_G^N(t)
 \le \e^2\left|\log\frac1\e\right|^{15/16},
\]
which proves \eqref{eq:mic-maxw}.

The thermodynamic mode is treated below by the exact
transport--forcing splitting \eqref{eq:rough-def}--\eqref{eq:dt-ineq}, which uses only
the data convergence \eqref{eq:mic-conv} and the kinetic forcing estimates.  We first
focus on the strong convergence of the transported component $\varpi_A^\e$.  The convergence of the forcing
parts $\varpi_B^\e$ and $u_B^\e$ has already been established in
Lemma~\ref{lem:splitconv}, so only the initial-data part remains.
Use a separate exponent for velocity compactness and the target
exponent $p$.  If $1<p<2$, set $p_0=p$; if $p\ge2$, fix any
$p_0\in(1,2)$.  In the latter case the uniform $L^{p_0}$ bound follows by
interpolation between the uniform $L^1$ and $L^p$ bounds.  Hence, in both
ranges, \eqref{eq:strg-bd} gives
\begin{equation*}
 \sup_{\e>0}\sup_{0\le t\le T}
 \|\varpi_A^\e(t)\|_{L^1\cap L^{p_0}}<\infty.
\end{equation*}
\paragraph{Compactness of $u_A^\e$.}
Fix $R>0$. \eqref{eq:lone-lp}--\eqref{eq:lone-lp-bd}, with $p=p_0$, and
the uniform $L^1\cap L^{p_0}$ bound give, for every $2<q_0<\frac{2p_0}{2-p_0},$
\begin{equation}\label{eq:lim-bd}
 \sup_{\e>0}\sup_{0\le t\le T}
 \left(
 \|u_A^\e(t)\|_{W^{1,p_0}(B_R)}
 +\|u_A^\e(t)\|_{L^{q_0}(B_R)}
 \right)
 \le C_R.
\end{equation}
Indeed, $dx=2\pi r\,dr\,dz$ and $r\le R$ on $B_R$ give
\begin{align}
 \|u_A^\e(t)\|_{W^{1,p_0}(B_R)}
 &\le C_R\|u_A^\e(t)\|_{W^{1,p_0}(\mathbb H_R)}, \notag \\
 \|u_A^\e(t)\|_{L^{q_0}(B_R)}
 &\le C_R\|u_A^\e(t)\|_{L^{q_0}(\mathbb H_R)}. \notag
\end{align}
Equations \eqref{eq:lone-lp}--\eqref{eq:lone-lp-bd} then yield
\eqref{eq:lim-bd}.
Choose
\[
 1<r_0<\min\left\{2,\frac{3p_0}{3-p_0}\right\}.
\]
The Rellich embedding gives
\begin{equation*}
 W^{1,p_0}(B_R)\Subset L^{r_0}(B_R).
\end{equation*}
If $v_j$ is bounded in
$W^{1,p_0}(B_R)\cap L^{q_0}(B_R)$, then, after extraction,
$v_j\to v$ in $L^{r_0}(B_R)$ by Rellich compactness and
$v_j\rightharpoonup w$ in $L^{q_0}(B_R)$. Both convergences hold in
$\mathcal D'(B_R)$, hence $w=v$.  Since $r_0<2<q_0$, interpolation gives
$v_j\to v$ in $L^2(B_R)$.  Therefore
\begin{equation*}
 W^{1,p_0}(B_R)\cap L^{q_0}(B_R)
 \Subset L^2(B_R).
\end{equation*}
Apply \cite[Corollary~4]{simon1987compact} on $B_R$ with
\[
 X=W^{1,p_0}(B_R)\cap L^{q_0}(B_R),
 \qquad
 B=L^2(B_R),
 \qquad
 Y=H^{-s}(B_R).
\]
This compact embedding, together with the continuous embedding
$L^2(B_R)\hookrightarrow H^{-s}(B_R)$, the regularity of the kinetic
solution in Theorem~\ref{thm:longtime}, and the bounds
\eqref{eq:lim-bd} and \eqref{eq:td-comp} verify the hypotheses of Simon's
criterion.  Hence, after a
diagonal extraction in $R$,
\begin{equation}\label{eq:ltwo-loc}
 u_A^\e\to u^E
 \qquad\text{in }C([0,T];L^2_{\rm loc}(\R^3)).
\end{equation}
Equations \eqref{eq:rough-bd}, \eqref{eq:ltwo-del}, and \eqref{eq:del-kin}, together with
$u_A^\e=\mathbb Pu^\e-u_B^\e$, give a uniform
$L^\infty(0,T;L^2)$ bound.  After the same extraction,
\begin{equation}\label{eq:ltwo-unif}
 u_A^\e\overset{\ast}{\rightharpoonup}u^E
 \qquad\text{in }L^\infty(0,T;L^2(\R^3)).
\end{equation}
Finally, \eqref{eq:strg-bd} with the target exponent $p$ gives, after a
further extraction,
\begin{equation}\label{eq:lp-exp}
 \varpi_A^\e\overset{\ast}{\rightharpoonup}\varpi^E
 \qquad\text{in }L^\infty(0,T;L^p(\R^3)).
\end{equation}
We next identify the limiting velocity and relative vorticity.  Since
$\omega_A^\e=r\varpi_A^\e e_\vartheta$ and $r e_\vartheta=(-x_2,x_1,0)$, for every $\Phi\in C_c^\infty((0,T)\times\R^3;\R^3)$, the function
$r e_\vartheta\cdot\Phi$ belongs to
$C_c^\infty((0,T)\times\R^3)$. Hence, by the weak-$\star$ convergence
\eqref{eq:lp-exp},
\begin{align*}
\int_0^T\int_{\R^3}\omega_A^\e\cdot\Phi\,dx\,dt
&=
\int_0^T\int_{\R^3}
\varpi_A^\e(r e_\vartheta\cdot\Phi)\,dx\,dt
\to
\int_0^T\int_{\R^3}
\varpi^E(r e_\vartheta\cdot\Phi)\,dx\,dt.
\end{align*}
Thus
\begin{equation}\label{eq:ee-def}
 \omega_A^\e\rightharpoonup
 \omega^E:=r\varpi^Ee_\vartheta
 \qquad
 \text{in }\mathcal D'((0,T)\times\R^3).
\end{equation}
On the other hand, since $\omega_A^\e=\nabla_x\times u_A^\e$, for every
$\Phi\in C_c^\infty((0,T)\times\R^3;\R^3)$, integration by parts and
\eqref{eq:ltwo-loc} give
\begin{align}
 \int_0^T\int_{\R^3}\omega_A^\e\cdot\Phi\,dx\,dt
 &=
 \int_0^T\int_{\R^3}
 u_A^\e\cdot(\nabla_x\times\Phi)\,dx\,dt \to
 \int_0^T\int_{\R^3}
 u^E\cdot(\nabla_x\times\Phi)\,dx\,dt. \notag 
\end{align}
Hence
\begin{equation*}
 \omega_A^\e\rightharpoonup\nabla_x\times u^E
 \qquad
 \text{in }\mathcal D'((0,T)\times\R^3).
\end{equation*}
Equation \eqref{eq:ee-def} and this distributional convergence imply
\begin{equation}\label{eq:ee-id}
 \nabla_x\times u^E
 =
 \omega^E
 =
 r\varpi^Ee_\vartheta
 \qquad
 \text{in }\mathcal D'((0,T)\times\R^3).
\end{equation}
Since $\nabla_x\cdot u_A^\e=0$, \eqref{eq:ltwo-loc} also gives
\begin{equation}\label{eq:eul-id}
 \nabla_x\cdot u^E=0
 \qquad
 \text{in }\mathcal D'((0,T)\times\R^3).
\end{equation}
By \eqref{eq:ltwo-unif},
$u^E\in L^\infty(0,T;L^2(\R^3))$. Testing
\eqref{eq:ee-id}--\eqref{eq:eul-id} with separated space--time test
functions shows that, for a.e. $t\in(0,T)$, both identities hold in
$\mathcal D'(\R^3)$. \eqref{eq:four-id} therefore gives
\begin{equation*}
 u^E=\nabla_x\times(-\Delta_x)^{-1}\omega^E
 ={\bf K}\ast\omega^E,
 \qquad
 \omega^E=r\varpi^Ee_\vartheta.
\end{equation*}
First verify the regularity and growth assumptions in
Theorem~\ref{thm:dlstab}.  Suppose $1<p<2$.  By
\eqref{eq:lim-bd}, the sequence fixed in \eqref{eq:ltwo-loc} is bounded in
$L^\infty(0,T;W^{1,p}(B_R))$.  Let $\widetilde u$ be any weak-$\ast$
cluster point in this space.  Since weak-$\ast$ convergence in
$L^\infty(0,T;W^{1,p}(B_R))$ implies distributional convergence on
$(0,T)\times B_R$, while \eqref{eq:ltwo-loc} gives distributional convergence
to $u^E$, one has $\widetilde u=u^E$.  Hence every weak-$\ast$ cluster point is $u^E$, so
\[
 u_A^\e\overset{\ast}{\rightharpoonup}u^E
 \qquad\text{in }L^\infty(0,T;W^{1,p}(B_R)).
\]
By weak lower semicontinuity and \eqref{eq:lim-bd},
\begin{equation*}
 \|u^E\|_{L^\infty(0,T;W^{1,p}(B_R))}
 \le
 \liminf_{\e\to0}
 \|u_A^\e\|_{L^\infty(0,T;W^{1,p}(B_R))}
 \le C_R.
\end{equation*}
Suppose $p\ge2$.
By \eqref{eq:ltwo-unif}, $u^E\in L^\infty(0,T;L^2(\R^3))$.
Equations \eqref{eq:ee-id} and \eqref{eq:lp-exp} give, for every $R>0$,
\[
 \|\omega^E(t)\|_{L^p(B_{2R})}
 \le 2R\|\varpi^E(t)\|_{L^p(\R^3)}
 \qquad\text{for a.e. }t.
\]
Applying Lemma~\ref{lem:locreg} to this bound yields
\begin{equation*}
 \|u^E\|_{L^\infty(0,T;W^{1,p}(B_R))}
 \le C_{p,R}
 \left(
 \|u^E\|_{L^\infty_tL^2_x}
 +\|\varpi^E\|_{L^\infty_tL^p_x}
 \right).
\end{equation*}
Thus, for every $1<p<\infty$,
\begin{equation}\label{eq:loc-lp}
 u^E\in L^\infty(0,T;W^{1,p}_{\rm loc}(\R^3)).
\end{equation}
To verify \eqref{eq:lone-vort}, set $g(t,x):=|u^E(t,x)|/(1+|x|)$.
By \eqref{eq:ltwo-unif},
\begin{equation*}
 \|g\|_{L^\infty(0,T;L^2)}
 \le
 \|u^E\|_{L^\infty(0,T;L^2)}<\infty.
\end{equation*}
Define $g_1:=g\mathbf 1_{\{g>1\}}$ and $g_\infty:=g\mathbf 1_{\{g\le1\}}$.
Then
\begin{align}
 \int_0^T\|g_1(t)\|_{L^1}\,dt
 &\le
 \int_0^T\|g(t)\|_{L^2}^2\,dt
 \le
 T\|u^E\|_{L^\infty_tL^2_x}^2, \qquad
 \int_0^T\|g_\infty(t)\|_{L^\infty}\,dt
 \le T. \notag
\end{align}
Thus
\begin{equation}\label{eq:lone-kin}
 \frac{|u^E|}{1+|x|}
 \in
 L^1(0,T;L^1(\R^3))
 +
 L^1(0,T;L^\infty(\R^3)).
\end{equation}
By \eqref{eq:ltwo-loc}, \eqref{eq:del-kin}, and
$\mathbb Pu^\e=u_A^\e+u_B^\e$, \eqref{eq:pp-conv} holds.  Hence, for every
compact $K\subset\R^3$,
\begin{align}
 \int_0^T
 \|\mathbb Pu^\e(t)-u^E(t)\|_{L^1(K)}\,dt
 &\le
 T|K|^{1/2}
 \|\mathbb Pu^\e-u^E\|_{C([0,T];L^2(K))}
 \to0.
 \label{eq:pp-lone}
\end{align}
Thus \eqref{eq:lone-loc} holds with
$V_\e=\mathbb Pu^\e$ and $V=u^E$.  By \eqref{eq:loc-lp},
\begin{align}
 \int_0^T\|u^E(t)\|_{W^{1,1}(K)}\,dt
 &\le
 T|K|^{1-\frac1p}
 \|u^E\|_{L^\infty(0,T;W^{1,p}(K))}<\infty,
 \label{eq:dt-free}
\end{align}
so \eqref{eq:lone-strg} holds.  The divergence-free conditions follow from
\eqref{eq:pq-def} and \eqref{eq:eul-id}, and \eqref{eq:lone-vort} is \eqref{eq:lone-kin}.
Furthermore, \eqref{eq:strg-bd} and \eqref{eq:lp-lone} give
\begin{equation}\label{eq:lone-cond}
 \sup_{\e>0}\sup_{0\le t\le T}
 \left(\|\varpi_A^\e(t)\|_{L^1}+\|\varpi_A^\e(t)\|_{L^p}\right)
 =
 \sup_{\e>0}
 \left(\|\varpi_0^\e\|_{L^1}+\|\varpi_0^\e\|_{L^p}\right)<\infty.
\end{equation}
Relabel the subsequence fixed in \eqref{eq:ltwo-loc}--\eqref{eq:lp-exp} as
$\e_n\to0$.  Apply Theorem~\ref{thm:dlstab} first with
\begin{equation*}
 \mathfrak p=p,\qquad
 V_n=v^{\e_n},\qquad
 V=u^E,\qquad
 \rho_n=\varpi_A^{\e_n},\qquad
 \rho_{0,n}=\varpi_0^{\e_n}.
\end{equation*}
The hypotheses follow from
\eqref{eq:pp-lone}, \eqref{eq:dt-free}, \eqref{eq:pq-def}, \eqref{eq:eul-id},
\eqref{eq:lone-kin}, Lemma~\ref{lem:splitconv}, \eqref{eq:bet-dx},
\eqref{eq:lone-cond}, and \eqref{eq:kin-conv}.  Thus, for the
unique renormalized solution $\varpi^{(p)}$ of
\begin{equation*}
 \partial_t\varpi^{(p)}+u^E\cdot\nabla_x\varpi^{(p)}=0,
 \qquad
 \varpi^{(p)}(0)=\varpi_0^E,
\end{equation*}
we have
\begin{equation}\label{eq:lp-strg}
 \varpi_A^{\e_n}\to\varpi^{(p)}
 \qquad\text{in }C([0,T];L^p(\R^3)).
\end{equation}
Comparison with \eqref{eq:lp-exp} gives
\begin{equation}\label{eq:lp-strg-id}
 \varpi^E=\varpi^{(p)}
 \qquad\text{a.e. on }(0,T)\times\R^3.
\end{equation}
Apply Theorem~\ref{thm:dlstab} again with
\begin{equation*}
 \mathfrak p=1,\qquad
 V_n=v^{\e_n},\qquad
 V=u^E,\qquad
 \rho_n=\varpi_A^{\e_n},\qquad
 \rho_{0,n}=\varpi_0^{\e_n}.
\end{equation*}
The velocity hypotheses are \eqref{eq:pp-lone}, \eqref{eq:dt-free},
\eqref{eq:pq-def}, \eqref{eq:eul-id}, and \eqref{eq:lone-kin}.  The remaining
hypotheses follow from Lemma~\ref{lem:splitconv}, \eqref{eq:bet-dx},
\eqref{eq:lone-cond}, and \eqref{eq:kin-conv}.  Hence there is a
unique renormalized solution $\varpi^{(1)}\in C([0,T];L^1)$ such that
\begin{equation*}
 \varpi_A^{\e_n}\to\varpi^{(1)}
 \qquad\text{in }C([0,T];L^1(\R^3)).
\end{equation*}
For every $t\in[0,T]$ and every $\phi\in C_c^\infty(\R^3)$,
\eqref{eq:lp-strg} and the $C_tL^1$ convergence just proved give
\begin{equation*}
 \int_{\R^3}\varpi^{(p)}(t)\phi\,dx
 =
 \lim_{n\to\infty}\int_{\R^3}\varpi_A^{\e_n}(t)\phi\,dx
 =
 \int_{\R^3}\varpi^{(1)}(t)\phi\,dx.
\end{equation*}
Hence
\begin{equation*}
 \varpi^{(1)}(t)=\varpi^{(p)}(t)
 \qquad\text{a.e. in }\R^3,\quad 0\le t\le T.
\end{equation*}
By \eqref{eq:lp-strg-id}, this common representative agrees with the weak-$\ast$ limit
$\varpi^E$ a.e. on $(0,T)\times\R^3$.  We henceforth denote this
representative by $\varpi^E$.  Returning to the notation $\e$ for the
extracted sequence,
\begin{equation*}
 \varpi_A^\e\to\varpi^E
 \qquad\text{in }C([0,T];L^1(\R^3))\cap C([0,T];L^p(\R^3)),
 \qquad 1<p<\infty.
\end{equation*}
In particular, \eqref{eq:vel-conv} holds. When $p\geq 6/5$, \eqref{eq:vel-conv} with directly application of Lemma \ref{lem:cylcurl} and interpolation gives \eqref{eq:loc-conv_r}. Together with \eqref{eq:del-kin}, \eqref{eq:pp-conv_r} follows.

For the remaining assertions of Theorem~\ref{thm:roughlimit}, fix an
admissible pair $(\mathfrak{r},\bar p)$ and $0<s<4$ as in assertion~(4).
By \eqref{eq:bega-rho}, \eqref{eq:decomp-bd}--\eqref{eq:tau-ac}, and \eqref{eq:ac-bd},
for fixed $T$,
\[
 \|\mathbb P^\perp u^\e\|_{L^\mathfrak{r}_T\dot B_{\bar p,1}^{s+3(1/\bar p-1/2)+1/\mathfrak{r}}}
 +\|\rho^\e+\theta^\e\|_{L^\mathfrak{r}_T\dot B_{\bar p,1}^{s+3(1/\bar p-1/2)+1/\mathfrak{r}}}
 \lesssim_{T,s,\mathfrak{r},\bar p}
 \e^{1/\mathfrak{r}}\left|\log\frac1\e\right|^{15/16}.
\]
Since $\e^\delta|\log\e|^m\to0$ for every $\delta,m>0$,
this gives \eqref{eq:pq-rho}.
For \eqref{eq:therm-conv}, decompose the thermodynamic mode as
\begin{equation}\label{eq:rough-def}
 \mathfrak{s}^\e=\mathfrak{s}_A^\e+\mathfrak{s}_B^\e,
 \qquad
 \mathfrak{s}_0^\e:=\frac32\theta_0^\e-\rho_0^\e.
\end{equation}
With $S^\e$ defined by \eqref{eq:par-def},
\begin{align}
 (\partial_t+u^\e\cdot\nabla_x)\mathfrak{s}_A^\e&=0,
 &\mathfrak{s}_A^\e(0)&=\mathfrak{s}_0^\e,
 \label{eq:td-strg}\\
 (\partial_t+u^\e\cdot\nabla_x)\mathfrak{s}_B^\e&=-S^\e,
 &\mathfrak{s}_B^\e(0)&=0.
 \label{eq:td-strg-id}
\end{align}
By \eqref{eq:par-def}, \eqref{eq:hn-en}--\eqref{eq:hn-bd},
\eqref{eq:hn-kap}, and \eqref{eq:ac-bd},
\begin{align}
 \int_0^T\|S^\e(t)\|_2\,dt
 &\lesssim_T
 \kappa^{1/2}
 \left(\int_0^T\mathcal D_G^{N}(t)\,dt\right)^{1/2}
 \lesssim_T
 \kappa^{1/2}\left|\log\frac1\e\right|^{15/32}
 \to0.
 \label{eq:kap-strg}
\end{align}
Apply \eqref{eq:bega-rho} with $(s,\bar p,\mathfrak{r})=(9/4,\infty,4)$.
{Using \eqref{eq:decomp-bd}--\eqref{eq:tau-ac}, \eqref{eq:ac-bd}, and
$\dot B_{\infty,1}^1\hookrightarrow \dot{W}^{1,\infty}$ gives
\[
 \|q_{\rm ac}^\e\|_{L^4(0,T;\dot{W}^{1,\infty})}
 \lesssim_T
 \e^{1/4}\left|\log\frac1\e\right|^{15/16}
 \to0.
\]
Here $q_{\rm ac}^\e$ is defined in \eqref{eq:acoustic-source-def}. Hence
\begin{equation}\label{eq:dt-def}
 d_\e(T):=
 \int_0^T\|\nabla_x\cdot u^\e(t)\|_\infty\,dt
 =\int_0^T\|\nabla_x\cdot q_{\rm ac}^\e(t)\|_\infty\,dt
 \to0,
\end{equation}
Bernstein's inequality gives
\begin{align*}
    &  \Vert q_{\rm ac}^\e\Vert_\infty \leq \sum_{j\leq J} \Vert \dot{\Delta}_j q_{\rm ac}^\e\Vert_\infty + \sum_{j>J} \Vert \dot{\Delta}_jq_{\rm ac}^\e \Vert_\infty \lesssim 2^{3J/2}\Vert q_{\rm ac}^\e\Vert_2 + 2^{-J} \Vert q_{\rm ac}^\e\Vert_{\dot{B}_{\infty,1}^1}.
\end{align*}
Therefore, interpolation yields
\[
 q_{\rm ac}^\e\to0
 \qquad\text{in }L^1(0,T;L^\infty(\R^3)).
\]}
Together with \eqref{eq:pp-lone}, this yields
\begin{equation}\label{eq:lone-lim}
 u^\e\to u^E
 \qquad\text{in }L^1(0,T;L^1_{\rm loc}(\R^3)).
\end{equation}
For every bounded $\beta\in C^1(\R)$ vanishing near zero, the smooth
solution of \eqref{eq:td-strg} satisfies
\begin{equation}\label{eq:sol-id}
 \frac{d}{dt}\int_{\R^3}\beta(\mathfrak{s}_A^\e(t,x))\,dx
 =
 \int_{\R^3}(\nabla_x\cdot u^\e)(t,x)
 \beta(\mathfrak{s}_A^\e(t,x))\,dx.
\end{equation}
Thus \eqref{eq:sol-id} is the global renormalization
identity required in \cite[Theorem~II.5, pp.~527--528]{diperna1989ordinary}.
Let $X^\e$ be the classical flow of $u^\e$.  Equation
\eqref{eq:td-strg} and the Jacobian equation for $X^\e$ give
\begin{align*}
 \mathfrak{s}_A^\e(t,X^\e(t,x))
 &=\mathfrak{s}_0^\e(x),\\
 \det \nabla_xX^\e(t,x)
 &=\exp\left(\int_0^t
 (\nabla_x\cdot u^\e)(\tau,X^\e(\tau,x))\,d\tau\right).
\end{align*}
Hence \eqref{eq:dt-def} and these two flow identities imply
\begin{equation}\label{eq:strg-vort}
 \sup_{0\le t\le T}\|\mathfrak{s}_A^\e(t)\|_2
 \le
 e^{d_\e(T)/2}\|\mathfrak{s}_0^\e\|_2.
\end{equation}
Apply \cite[Theorem~II.5, pp.~527--528]{diperna1989ordinary} with
$b_\e=-u^\e$, $b=-u^E$, $c=0$, and exponent $2$.  The fields $u^\e$
are smooth, hence $\mathfrak{s}_A^\e$ is a renormalized solution of \eqref{eq:td-strg}. Moreover,
\eqref{eq:dt-def} gives
\[
 \sup_{\e\ll1}\|\operatorname{div}b_\e\|_{L^1(0,T;L^\infty)}<\infty,
 \qquad
 \operatorname{div}b=0,
\]
and \eqref{eq:sol-id} is precisely the global identity
\cite[(31), p.~522]{diperna1989ordinary} for $b_\e=-u^\e$.
The convergence hypothesis is \eqref{eq:lone-lim}; the limit-field conditions
(*) and (**) are \eqref{eq:dt-free}, \eqref{eq:eul-id}, and \eqref{eq:lone-kin}; the
data convergence is \eqref{eq:mic-conv}; and \eqref{eq:strg-vort} supplies the uniform
$L^2$ bound.  Hence, for the unique renormalized solution
$\mathfrak{s}^E$ of
\begin{equation*}
 \partial_t\mathfrak{s}^E+u^E\cdot\nabla_x\mathfrak{s}^E=0,
 \qquad
 \mathfrak{s}^E(0)=\frac32\theta_0^E-\rho_0^E,
\end{equation*}
one has
\begin{equation}\label{eq:ltwo-rho}
 \mathfrak{s}_A^\e\to\mathfrak{s}^E
 \qquad\text{in }C([0,T];L^2(\R^3)).
\end{equation}
Set $\frac32\theta^E-\rho^E:=\mathfrak{s}^E$.
Taking the $L^2$ inner product of \eqref{eq:td-strg-id} with
$\mathfrak{s}_B^\e$ and using \eqref{eq:dt-def} gives
\begin{equation*}
 \frac{d}{dt}\|\mathfrak{s}_B^\e(t)\|_2
 \le
 \frac12\|\nabla_x\cdot u^\e(t)\|_\infty
 \|\mathfrak{s}_B^\e(t)\|_2
 +\|S^\e(t)\|_2
\end{equation*}
for a.e. $t$ with $\|\mathfrak{s}_B^\e(t)\|_2>0$; the standard
regularization
$(\|\mathfrak{s}_B^\e\|_2^2+\lambda^2)^{1/2}$ gives the same estimate
through its zeros.  Thus \eqref{eq:kap-strg},
\eqref{eq:dt-def}, and Gronwall's inequality imply
\begin{equation}\label{eq:dt-ineq}
 \sup_{0\le t\le T}\|\mathfrak{s}_B^\e(t)\|_2
 \le
 e^{d_\e(T)/2}\int_0^T\|S^\e(t)\|_2\,dt
 \to0.
\end{equation}
Equations \eqref{eq:rough-def},
\eqref{eq:ltwo-rho}, and \eqref{eq:dt-ineq} prove
\eqref{eq:therm-conv}.

\end{proof}

\section{Global energy and radial-moment transfer in the kinetic limit}\label{sec:globlimit}
The all-$p$ theorem gives strong transported-relative-vorticity convergence and local velocity convergence for every $p>1$.  If $p\ge4/3$, Theorem~\ref{thm:mfreecomp} upgrades this to global $C_tL_x^2$ compactness without moment assumptions.  If $1<p<4/3$, finite absolute impulse together with the mixed sign-component $L^{p^\sharp}$ condition provides the radial control needed for the same global conclusion.  The first-radial-moment and absolute-impulse estimates are stated separately because they require different hypotheses.
\subsection{Absolute impulse}
\begin{corollary}[Uniform absolute impulse]\label{cor:unifimp}
Fix $T<\infty$ and let $F^\e$ be the solutions of
Theorem~\ref{thm:longtime}.  Let $\varpi_A^\e,u_A^\e,u_B^\e$ be defined by
Definition~\ref{def:vortsplit}.  Assume, for some $1<p<\infty$,
\begin{equation}\label{eq:mtwo-sol}
 \sup_{\e>0}\left(
 \|\varpi_0^\e\|_1+\|\varpi_0^\e\|_p+\mathfrak M_2(\varpi_0^\e)
 \right)<\infty.
\end{equation}
Then
\begin{equation}\label{eq:lsup-abs}
 \limsup_{\e\to0}\sup_{0\le t\le T}
 \mathfrak M_2(\varpi_A^\e(t))<\infty
\end{equation}
under either of the following conditions.
\begin{enumerate}
\item $p\ge5/3$.
\item $1<p<5/3$ and, with
$p^\sharp=4p/(3p-1)$,
\begin{equation}\label{eq:sharp-abs}
 \sup_{\e>0}
 \min\left\{
 \|(\varpi_0^\e)_+\|_{p^\sharp},
 \|(\varpi_0^\e)_-\|_{p^\sharp}
 \right\}<\infty.
\end{equation}
\end{enumerate}
\end{corollary}
\begin{proof}
By \eqref{eq:small-def}, $T<T^\e$ for all sufficiently small $\e$.  Recall $v^\e=\mathbb Pu^\e=u_A^\e+u_B^\e$, and set
$Y_\e(t):=\mathfrak M_2(\varpi_A^\e(t)).$
Choose
\[
 p_*:=
 \begin{cases}
  5/3,& p\ge5/3,\\
  p,& 1<p<5/3,
 \end{cases}
 \qquad
 p_*^\sharp:=\frac{4p_*}{3p_*-1}.
\]
In case~(1), interpolation between $L^1$ and $L^p$, together with
\eqref{eq:strg-bd} and \eqref{eq:mtwo-sol}, gives
\[
 \sup_{\e>0}\sup_{0\le t\le T}
 \|\varpi_A^\e(t)\|_{5/3}<\infty.
\]
Since $(5/3)^\sharp=5/3$, both sign parts are uniformly bounded in
$L^{p_*^\sharp}$.  In case~(2), \eqref{eq:xx-id} and \eqref{eq:sign-psi} give, for each sign,
\[
 \|(\varpi_A^\e(t))_\pm\|_{p^\sharp}
 =\|(\varpi_0^\e)_\pm\|_{p^\sharp},
 \qquad 0\le t\le T.
\]
Hence \eqref{eq:strg-bd}, \eqref{eq:mtwo-sol}, and \eqref{eq:sharp-abs} imply
\[
 \sup_{\e>0}\sup_{0\le t\le T}
 \left(
 \|\varpi_A^\e(t)\|_1+\|\varpi_A^\e(t)\|_{p_*}
 +\min_{\sigma\in\{+,-\}}
 \|(\varpi_A^\e(t))_\sigma\|_{p_*^\sharp}
 \right)<\infty.
\]
Thus the same bound holds in both cases.
Let $\Psi_L$ be the truncation in \eqref{eq:rad-def} and set
\[
 Y_{\e,L}(t):=\frac1{2\pi}\int_{\R^3}
 \Psi_L(r^2)|\varpi_A^\e(t)|\,dx.
\]
The measure-preserving flow identities \eqref{eq:xx-id}--\eqref{eq:sign-psi} give
\[
 Y_{\e,L}(t)-Y_{\e,L}(0)
 =\frac1\pi\int_0^t\!\int_{\R^3}
 \Psi_L'(r^2)r(v^\e)^r|\varpi_A^\e|\,dx\,ds.
\]
Since $u_A^\e=v^\e-u_B^\e$, \eqref{eq:ac-bd} and \eqref{eq:del-decomp}
show that
\[
 \sup_{0\le t\le T}\|r\varpi_A^\e(t)\|_2
 =\sup_{0\le t\le T}\|\nabla\times u_A^\e(t)\|_2<\infty.
\]
Consequently,
\[
 \int_0^T\!\int_{\R^3}|v^\e|r|\varpi_A^\e|\,dx\,dt
 \le T\|v^\e\|_{L_t^\infty L_x^2}
 \|\nabla\times u_A^\e\|_{L_t^\infty L_x^2}<\infty.
\]
Monotone convergence on the left and dominated convergence on the right yield
$Y_\e\in W^{1,1}(0,T)$ and
\[
 \pi Y_\e'(t)
 =\int_{\R^3}(v^\e)^r r|\varpi_A^\e|\,dx
 \qquad\text{for a.e. }t\in(0,T).
\]
Write $\varpi_\pm=(\varpi_A^\e)_\pm$ and suppress $\e,t$.
This bound gives $r\varpi_\pm\in L^2$, while
$\mathfrak M_2(\varpi_\pm)\le Y_\e(t)$.  Together with the
uniform $L^{p_*}$ bound, Proposition~\ref{prop:radmom} and \eqref{eq:gam-th} give
\[
 u_A^\e=\Bcal\varpi_+-\Bcal\varpi_-.
\]  The diagonal pairings vanish by \eqref{eq:bs-dx}.  More explicitly,
using the convention
$\mathfrak X(\alpha,\beta)=\int(\Bcal\beta)^r r\alpha\,dx$ from
\eqref{eq:axis-def} and applying the condition $r \varpi_{\pm} \in L^2 $ we have
\begin{align*}
 \int_{\R^3}(u_A^\e)^r r|\varpi_A^\e|\,dx
 &=\mathfrak X(\varpi_+,\varpi_+)
   +\mathfrak X(\varpi_-,\varpi_+)
   -\mathfrak X(\varpi_+,\varpi_-)
   -\mathfrak X(\varpi_-,\varpi_-)\notag\\
 &=2\mathfrak X(\varpi_-,\varpi_+)
 =-2\mathfrak X(\varpi_+,\varpi_-),
\end{align*}
where the second line uses both \eqref{eq:bs-dx} and the skew identity
\eqref{eq:bs-id}.  Choose $\sigma$ so that
$\|\varpi_\sigma\|_{p_*^\sharp}$ is minimal and let $\bar\sigma$ be the
opposite sign.  This expansion then gives
\[
 \left|\int_{\R^3}(u_A^\e)^r r|\varpi_A^\e|\,dx\right|
 =2|\mathfrak X(\varpi_{\bar\sigma},\varpi_\sigma)|.
\]
The first factor lies in $L^1\cap L^{p_*}$ and the second in
$L^1\cap L^{p_*^\sharp}$.  Proposition~\ref{prop:crossint}
therefore gives
\[
 \left|\int_{\R^3}(u_A^\e)^r r|\varpi_A^\e|\,dx\right|
 \lesssim_{p_*}
 \|\varpi_A^\e\|_1^2
 +\|\varpi_A^\e\|_{p_*}^{1/4}
 Y_\e^{3/4}
 \min_{\sigma\in\{+,-\}}\|\varpi_\sigma\|_{p_*^\sharp}.
\]
The $u_B^\e$ contribution is quantitatively negligible.  Indeed,
\eqref{eq:decomp-id}, \eqref{eq:ac-bd}, and \eqref{eq:del-decomp} give, uniformly on
$[0,T]$,
\begin{align}
 \|\nabla\times u_A^\e(t)\|_2
 &\le \|\nabla\times\mathbb Pu^\e(t)\|_2
      +\|\omega_B^\e(t)\|_2 \lesssim
 \left|\log\frac1\e\right|^{15/32}
 +C_{T,\delta}\bigl(\e^{1/2-\delta}+\kappa^{1/2-\delta}\bigr). \notag
\end{align}
Hence \eqref{eq:del-kin} implies
\begin{align}
 \sup_{0\le t\le T}
 \frac1\pi\left|\int_{\R^3}(u_B^\e)^r r|\varpi_A^\e|\,dx\right|
 &\le \frac1\pi
 \|u_B^\e\|_{L_t^\infty L_x^2}
 \|\nabla\times u_A^\e\|_{L_t^\infty L_x^2}\notag\\
 &\lesssim_{T,\delta}
 \bigl(\e^{1/2-\delta}+\kappa^{1/2-\delta}\bigr)
 \left(
  1+\left|\log\frac1\e\right|^{15/32}
 \right)
 \to0, \notag
\end{align}
where the last limit follows from \eqref{eq:choice-bd} and the elementary fact
that every positive power dominates every logarithm.  Consequently,
\[
 |Y_\e'(t)|\le C_{p,T}(1+Y_\e(t)^{3/4})+o_\e(1)
 \qquad\text{for a.e. }t\in(0,T).
\]
Since $Y_\e(0)$ is uniformly bounded by \eqref{eq:mtwo-sol}, integration of
$(1+Y_\e)^{-3/4}Y_\e'$ gives \eqref{eq:lsup-abs}.
\end{proof}
\subsection{First radial moment}
\begin{proposition}[Uniform first radial moment]\label{prop:radfirst}
Fix $T<\infty$ and $3/2\le p<\infty$.  Let $F^\e$ be the solutions of
Theorem~\ref{thm:longtime}, and let $\varpi_A^\e,u_A^\e,u_B^\e$ be defined by
Definition~\ref{def:vortsplit}.  Assume
\begin{equation}\label{eq:mone-sol}
 \sup_{\e>0}
 \left(
 \|\varpi_0^\e\|_1+\|\varpi_0^\e\|_p+\mathfrak M_1(\varpi_0^\e)
 \right)<\infty.
\end{equation}
Then
\begin{equation}\label{eq:lsup-fst}
 \limsup_{\e\to0}\sup_{0\le t\le T}
 \mathfrak M_1(\varpi_A^\e(t))<\infty.
\end{equation}
\end{proposition}
\begin{proof}
Set $f_\e:=|\varpi_A^\e|$ and $h_\e:=rf_\e$.  The flow formula
\eqref{eq:sign-decomp}, measure preservation, and
\[
 r(X^\e(t,x))
 \le r(x)+\int_0^t
 |(\mathbb Pu^\e)^r(s,X^\e(s,x))|\,ds
\]
first show that $\mathfrak M_1(\varpi_A^\e(t))<\infty$ for every fixed
$\e$ and $t$.  Keeping the pointwise radial velocity gives
\[
 \mathfrak M_1(\varpi_A^\e(t))
 \le \mathfrak M_1(\varpi_0^\e)
 +\frac1{2\pi}\int_0^t\int_{\R^3}
 f_\e |(\mathbb Pu^\e)^r|\,dx\,ds.
\]
By \eqref{eq:decomp-def}, split the last integral into the $u_A^\e$ and
$u_B^\e$ contributions.  The div--curl identity \eqref{eq:gam-th}
identifies $u_A^\e$ with the Biot--Savart field generated by
$r\varpi_A^\e e_\vartheta$.  Hence the radial kernel estimate
\cite[Proposition~2.3, (2.11)]{gallay2015remarks} and
Lemma~\ref{lem:hls} give
\[
 \frac1{2\pi}\int f_\e |(u_A^\e)^r|\,dx
 \lesssim \|h_\e\|_{L^{4/3}(\R^2)}^2.
\]
For $p>3/2$, the weighted factorization used in
Theorem~\ref{thm:firstmom} yields
\[
 \|h_\e(t)\|_{4/3}^2
 \lesssim_p
 \|\varpi_A^\e(t)\|_p^{\frac{p}{2(p-1)}}
 \|\varpi_A^\e(t)\|_1^{\frac{2p-3}{2(p-1)}}
 \mathfrak M_1(\varpi_A^\e(t))^{1/2},
\]
with the $L^1$ factor absent at $p=3/2$.  Transport preserves the
$L^1\cap L^p$ norms.  For the forcing velocity,
\eqref{eq:strg-bd} and \eqref{eq:del-kin} give, uniformly on $[0,T]$,
\[
 \frac1{2\pi}\int f_\e |(u_B^\e)^r|\,dx
 \lesssim \|\varpi_0^\e\|_1\|u_B^\e\|_\infty=o_\e(1).
\]
Therefore, with
$Z_\e(t):=\sup_{0\le s\le t}\mathfrak M_1(\varpi_A^\e(s))^{1/2}$,
\eqref{eq:mone-sol} gives
\[
 Z_\e(T)^2\lesssim_{p,T}1+Z_\e(T)+o_\e(1).
\]
The resulting uniform bound for $Z_\e(T)$ is exactly
\eqref{eq:lsup-fst}.
\end{proof}
\subsection{Approximate Euler equation}
\begin{lemma}[Approximate Euler equation for the transported velocity]\label{lem:approxeul}
Assume the hypotheses of Theorem~\ref{thm:roughlimit}, fix $T<\infty$, and let
$F^\e$ be the corresponding kinetic solutions.  For all sufficiently small
$\e$, the transported
velocity $a^\e:=u_A^\e$ satisfies
\begin{equation}\label{eq:vel-id}
 \partial_ta^\e+a^\e\cdot\nabla a^\e+\nabla\widetilde\pi_A^\e
 =f_A^\e,
 \qquad \nabla\cdot a^\e=\nabla\cdot f_A^\e=0,
 \qquad \nabla\times a^\e=r\varpi_A^\e e_\vartheta,
\end{equation}
where
\begin{equation}\label{eq:ta-id}
 f_A^\e
 =-\mathbb P\left[
 u_B^\e\cdot\nabla a^\e+(\nabla u_B^\e)^Ta^\e
 \right],\qquad
 \widetilde\pi_A^\e=\pi_A^\e+\frac{|a^\e|^2}{2}
 +\Delta^{-1}\nabla\cdot\left(u_B^\e\cdot\nabla a^\e+(\nabla u_B^\e)^Ta^\e\right).
\end{equation}
The fields $a^\e$ and $f_A^\e$ are axisymmetric without swirl.
Moreover,
\begin{align}
 \sup_{\e\ll1}\|a^\e\|_{L^\infty(0,T;L^2)}&<\infty,
 \label{eq:approx-bd}\\
 \|f_A^\e\|_{L^1(0,T;L^2)}&\to0.
 \label{eq:lone-ltwo}
\end{align}
\end{lemma}
\begin{proof}
Set $a^\e=u_A^\e$ and $b^\e=u_B^\e$.  Since
$\mathbb Pu^\e=a^\e+b^\e$, the first equation in \eqref{eq:fld-id} and
$(\nabla a^\e)^Ta^\e=\nabla(|a^\e|^2/2)$ give
\[
 \partial_ta^\e+a^\e\cdot\nabla a^\e
 +\nabla\left(\pi_A^\e+\frac{|a^\e|^2}{2}\right)
 =-b^\e\cdot\nabla a^\e-(\nabla b^\e)^Ta^\e.
\]
Applying the Leray projection to the right-hand side and absorbing its
gradient part into the pressure proves
\eqref{eq:vel-id}--\eqref{eq:ta-id}.  By Definition~\ref{def:vortsplit}, both
$a^\e$ and $b^\e$ are $O(2)_z$-equivariant.  The vector field inside
the brackets in \eqref{eq:ta-id} is therefore $O(2)_z$-equivariant, and
the Leray projector commutes with orthogonal changes of variables.  Hence
$f_A^\e$ is $O(2)_z$-equivariant; reflection equivariance excludes an
azimuthal component, so $f_A^\e$ is axisymmetric without swirl.
The curl identity in \eqref{eq:vel-id} is
\eqref{eq:decomp-id}.  The construction in Definition~\ref{def:vortsplit} and the $H^m$ energy
estimate recorded in the proof of \eqref{eq:decomp-id} give, for each fixed
$\e$ and every finite $m$ allowed by the classical kinetic solution,
\[
 u_A^\e\in C([0,T];H^m(\R^3)),
 \qquad
 \partial_tu_A^\e
 =-\mathbb P\!\left[v^\e\cdot\nabla u_A^\e
 +(\nabla v^\e)^Tu_A^\e\right]
 \in C([0,T];H^{m-1}(\R^3)).
\]
Taking $m=3$ verifies the pointwise regularity hypothesis of
Theorem~\ref{thm:mfreecomp}.  No uniform $H^3$ bound is used in that theorem.
By \eqref{eq:diss-def}, \eqref{eq:tot-def}, \eqref{eq:ac-bd}, and $N\ge5$,
\begin{equation*}
 \sup_{0\le t\le T}\|\mathbb Pu^\e(t)\|_{H^2}
 \lesssim \left|\log\frac1\e\right|^{15/32}.
\end{equation*}
Since $b^\e$ is divergence free, Plancherel and \eqref{eq:del-decomp} with $q=2$
give
\begin{equation}\label{eq:curl-del}
 \sup_{0\le t\le T}\|\nabla b^\e(t)\|_2
 =\sup_{0\le t\le T}\|\nabla\times b^\e(t)\|_2
 \lesssim_{T,\delta}
 \e^{1/2-\delta}+\kappa^{1/2-\delta}.
\end{equation}
Equation \eqref{eq:del-kin}, with $q=\infty$, gives the same bound for
$\|b^\e\|_{L^\infty_tL^\infty_x}$.  Hence Sobolev embedding and
$a^\e=\mathbb Pu^\e-b^\e$ imply
\begin{equation*}
 \sup_{0\le t\le T}
 \left(\|a^\e(t)\|_\infty+\|\nabla a^\e(t)\|_2\right)
 \lesssim_T
 \left|\log\frac1\e\right|^{15/32}+1.
\end{equation*}
The $L^2$ bound \eqref{eq:approx-bd} follows instead from
\eqref{eq:rough-bd}, \eqref{eq:ltwo-del}, and \eqref{eq:del-kin} with $q=2$.
The $L^2$ boundedness of $\mathbb P$, followed by
\eqref{eq:curl-del} and the Sobolev bound just obtained yields
\begin{align*}
 \|f_A^\e\|_{L^1(0,T;L^2)}
 &\le
 \int_0^T\left(
 \|b^\e\|_\infty\|\nabla a^\e\|_2
 +\|\nabla b^\e\|_2\|a^\e\|_\infty
 \right)dt \lesssim_{T,\delta}
 \left(\e^{1/2-\delta}+\kappa^{1/2-\delta}\right)
 \left(1+\left|\log\frac1\e\right|^{15/32}\right).
\end{align*}
After fixing $0<\delta<1/2$, \eqref{eq:choice-bd} and the elementary
power--logarithm comparison show that the last expression tends to zero.  This proves \eqref{eq:lone-ltwo}.
\end{proof}

\subsection{Global energy upgrade and radial-moment corollary}\label{sec:proof-conservation}
\begin{proof}[Proof of Theorem~\ref{thm:globenergy}]
Equations \eqref{eq:ltwo-del} and \eqref{eq:comp-conv} give \eqref{eq:pp-fin}.  Work along a
subsequence furnished by Theorem~\ref{thm:roughlimit}; in particular,
\eqref{eq:vel-conv}--\eqref{eq:loc-conv} hold.
Assume first that $p\ge4/3$.  Apply Theorem~\ref{thm:mfreecomp} to
$a_n=u_A^{\e_n}$, $\varpi_n=\varpi_A^{\e_n}$ along the chosen subsequence.
The approximate Euler equation, the uniform $L^2$ bound, and the vanishing
forcing are Lemma~\ref{lem:approxeul}; the strong
$C_t(L^1\cap L^p)$ convergence is \eqref{eq:vel-conv}; and the initial $L^2$
convergence follows from the initial condition in \eqref{eq:fld-id} and
\eqref{eq:comp-conv}.  Therefore
\begin{equation}\label{eq:bs-ltwo}
 u_A^\e\to\Bcal\varpi^E
 \qquad\text{in }C([0,T];L^2(\R^3)),
\end{equation}
and Theorem~\ref{thm:mfreecomp} also gives
\begin{equation*}
 \|\Bcal\varpi^E(t)\|_2=\|u_0^E\|_2,
 \qquad0\le t\le T.
\end{equation*}
Comparing \eqref{eq:bs-ltwo} with the local convergence \eqref{eq:loc-conv} yields
\begin{equation}\label{eq:bs-loc}
 u^E(t)=\Bcal\varpi^E(t)
 \qquad\text{for a.e. }x\in\R^3,\quad0\le t\le T.
\end{equation}
Thus \eqref{eq:bs-ltwo} and \eqref{eq:bs-loc} prove \eqref{eq:eul-conv}, while
\eqref{eq:decomp-def} and \eqref{eq:del-kin} prove \eqref{eq:rough-conv}.  This energy identity gives \eqref{eq:pp-bd}.
Assume next that $1<p<4/3$.  Equations \eqref{eq:lp-lone}, \eqref{eq:fld-bd}, and
\eqref{eq:lp-sharp} verify \eqref{eq:mtwo-sol} and assertion~(2) of
Corollary~\ref{cor:unifimp}.  Hence
\begin{equation}\label{eq:lsup-div}
 \limsup_{\e\to0}\sup_{t\in[0,T]}\mathfrak M_2(\varpi_A^\e(t))<\infty.
\end{equation}
For every $\e$ and $t$, the div--curl identity \eqref{eq:decomp-id} gives
\[
 \nabla\cdot u_A^\e=0,
 \qquad
 \nabla\times u_A^\e=r\varpi_A^\e e_\vartheta.
\]
After discarding finitely many members of the chosen subsequence,
\eqref{eq:lsup-div} gives a uniform bound for
$\mathfrak M_2(\varpi_A^\e(t))$, so $\varpi_A^\e(t)\in\mathcal X_p$ for every
$t\in[0,T]$.  The $L^2$ bound for $u_A^\e$ follows from \eqref{eq:rough-bd},
\eqref{eq:ltwo-del}, and \eqref{eq:del-kin}.  The uniqueness assertion in
Proposition~\ref{prop:radmom} therefore gives
\begin{equation*}
 u_A^\e(t)=\Bcal\varpi_A^\e(t),
 \qquad 0\le t\le T.
\end{equation*}
Now \eqref{eq:lsup-div}, \eqref{eq:vel-conv}, Corollary~\ref{cor:unifstab}, and this identification give \eqref{eq:bs-ltwo}.  Comparing \eqref{eq:bs-ltwo} with
\eqref{eq:loc-conv} again gives \eqref{eq:bs-loc}, and hence \eqref{eq:eul-conv}.
Equation \eqref{eq:del-kin}
then gives \eqref{eq:rough-conv}.  Finally,
\begin{align*}
 \sup_{t\in[0,T]}\left|\|u^E(t)\|_2-\|u_0^E\|_2\right|
 &\le
 \|u^E-\mathbb Pu^\e\|_{L^\infty(0,T;L^2)} +
 \sup_{t\in[0,T]}
 \left|\|\mathbb Pu^\e(t)\|_2-\|u_0^E\|_2\right|,
\end{align*}
so \eqref{eq:rough-conv} and \eqref{eq:pp-fin} give \eqref{eq:pp-bd}.
\end{proof}
\begin{proof}[Proof of Corollary~\ref{cor:radcrit}]
Suppose first that $p\ge3/2$ and \eqref{eq:cond-bd} holds.  Equations
\eqref{eq:lp-lone} and \eqref{eq:cond-bd} verify \eqref{eq:mone-sol}; hence
Proposition~\ref{prop:radfirst} gives \eqref{eq:lsup-mone}.  This proves the
sign-free first-radial-moment assertion.
Now assume \eqref{eq:fld-bd}.  Equations \eqref{eq:lp-lone} and \eqref{eq:fld-bd}
verify \eqref{eq:mtwo-sol}.  If $p\ge5/3$, assertion~(1) of
Corollary~\ref{cor:unifimp} gives \eqref{eq:lsup-mtwo}.  If $1<p<5/3$ and
\eqref{eq:lp-sharp} also holds, assertion~(2) of Corollary~\ref{cor:unifimp} gives the
same conclusion.  Thus \eqref{eq:lsup-mtwo} holds throughout the conditional
absolute-impulse range stated in the corollary.
Finally, conservation of the transported $L^1$ norm in \eqref{eq:strg-bd}
and the elementary moment inequality \eqref{eq:mominterp} give, for every
$t\in[0,T]$,
\begin{align*}
 \mathfrak M_1(\varpi_A^\e(t))^2
 &\le \frac1{2\pi}
 \|\varpi_A^\e(t)\|_1\mathfrak M_2(\varpi_A^\e(t)) =\frac1{2\pi}
 \|\varpi_0^\e\|_1\mathfrak M_2(\varpi_A^\e(t)).
\end{align*}
Combining this inequality with \eqref{eq:lp-lone} and \eqref{eq:lsup-mtwo} proves
\eqref{eq:lsup-mone} throughout that conditional range, including
$1<p<3/2$.

\end{proof}

\subsection{Weak--strong identification}\label{sec:weakstrong}

\begin{lemma}[Velocity formulation of the global-energy limit]\label{lem:weak-euler-limit}
Assume the hypotheses of Theorem~\ref{thm:globenergy}, and let
$(u^E,\varpi^E)$ be a limiting Euler pair along a subsequence furnished by
Theorem~\ref{thm:roughlimit}.  Then
\[
    u^E\in C([0,T];L^2(\R^3;\R^3))
\]
and, for every divergence-free
$\phi\in C_c^\infty([0,T)\times\R^3;\R^3)$,
\begin{equation}\label{eq:weak-euler-velocity}
    \int_0^T\!\int_{\R^3}
    \Bigl(u^E\cdot\partial_t\phi
      +(u^E\otimes u^E):\nabla\phi\Bigr)\,dx\,dt
    +\int_{\R^3}u_0^E\cdot\phi(0)\,dx=0.
\end{equation}
In particular, the global-energy limit is a finite-energy weak Euler
solution in the standard divergence-free velocity formulation.
\end{lemma}
\begin{proof}
Set $a^\e:=u_A^\e$.  By Lemma~\ref{lem:approxeul},
\[
 \partial_ta^\e+a^\e\cdot\nabla a^\e
 +\nabla\widetilde\pi_A^\e=f_A^\e,
 \qquad \nabla\cdot a^\e=\nabla\cdot f_A^\e=0,
\]
with $\|f_A^\e\|_{L^1(0,T;L^2)}\to0$.  Testing against a divergence-free
$\phi\in C_c^\infty([0,T)\times\R^3)$ gives
\begin{align*}
 &\int_0^T\!\int_{\R^3}
 \Bigl(a^\e\cdot\partial_t\phi
 +(a^\e\otimes a^\e):\nabla\phi\Bigr)\,dx\,dt
 +\int_{\R^3}a^\e(0)\cdot\phi(0)\,dx  
 =-\int_0^T\!\int_{\R^3}f_A^\e\cdot\phi\,dx\,dt.
\end{align*}
Along the subsequence in Theorem~\ref{thm:globenergy},
\eqref{eq:eul-conv} gives
$a^\e\to u^E$ in $C([0,T];L^2)$, while
$a^\e(0)=\mathbb Pu_0^\e\to u_0^E$ in $L^2$ by
\eqref{eq:comp-conv}.  Moreover,
\[
 \|a^\e\otimes a^\e-u^E\otimes u^E\|_{L^\infty(0,T;L^1)}
 \le
 \bigl(\|a^\e\|_{L^\infty_tL^2_x}
      +\|u^E\|_{L^\infty_tL^2_x}\bigr)
 \|a^\e-u^E\|_{L^\infty_tL^2_x}
 \longrightarrow0.
\]
Passing to the limit proves \eqref{eq:weak-euler-velocity}.  The same
$C_tL^2_x$ convergence shows $u^E\in C([0,T];L^2)$.
\end{proof}

\begin{proposition}[Weak--strong identification under bounded strain]
\label{prop:weakstrong}
Let $u\in C([0,T];L^2(\R^3;\R^3))$ be a finite-energy weak Euler
solution with initial datum $u_0\in L^2(\R^3)$, and assume
\begin{equation}\label{eq:weak-energy-equality}
    \|u(t)\|_2=\|u_0\|_2,
    \qquad 0\le t\le T.
\end{equation}
Let $0<\tau\le T$, and let
$U\in L^\infty(0,\tau;L^2(\R^3;\R^3))$ be another finite-energy weak
Euler solution with the same initial datum.  Assume that
\begin{equation}\label{eq:weakstrong-strain}
    \nabla^{\rm s}U
    \in L^1(0,\tau;L^\infty(\R^3)),
    \qquad
    \nabla^{\rm s}U:=\frac12\bigl(\nabla U+(\nabla U)^{\mathsf T}\bigr).
\end{equation}
Then $u=U$ almost everywhere on $(0,\tau)\times\R^3$.  If in addition
$U\in C([0,\tau];L^2(\R^3))$, then
\[
    u(t)=U(t)
    \qquad\text{in }L^2(\R^3),\quad 0\le t\le\tau.
\]
\end{proposition}
\begin{proof}
The energy identity \eqref{eq:weak-energy-equality} makes $u$ an
admissible weak Euler solution.  Under \eqref{eq:weakstrong-strain},
$U$ conserves its kinetic energy by
\cite[Proposition~1.7]{DeRosaInversiStefani2023}; hence $U$ is admissible
as well.  The standard weak--strong uniqueness theorem
\cite[Theorem~1.6]{DeRosaInversiStefani2023} gives $u=U$ almost everywhere
on $(0,\tau)\times\R^3$.  If both fields are strongly continuous with
values in $L^2$, the identity extends to every time in $[0,\tau]$.
\end{proof}

\section{Prescribed rough data and kinetic realizations}\label{sec:roughdata}
Smooth kinetic initial data can converge to rough finite-energy Euler data.
Proposition~\ref{prop:graphdens} gives smooth axisymmetric approximation in
the finite-energy graph norm, Proposition~\ref{prop:diagreal} lifts these profiles
to Landau and non-cutoff Boltzmann data with prescribed thermodynamic and
acoustic components, and Proposition~\ref{prop:signchg} exhibits a sign-changing
low-integrability class beyond a symmetric $L^{p^\sharp}$ assumption.
{Singularity-passage consequences of these realizations are treated
in Section~\ref{sec:singularpassage}.}
\subsection{Finite-energy graph class and graph-norm density}
\begin{definition}[Finite-energy graph class]
For $1<p<\infty$, let $\mathcal Z_p$ be the set of axisymmetric
$\varpi\in L^1(\R^3)\cap L^p(\R^3)$ for which the div--curl system
\begin{equation*}
 \nabla\cdot u=0,\qquad
 \nabla\times u=r\varpi e_\vartheta
 \quad\text{in }\mathcal D'(\R^3)
\end{equation*}
has a solution $u\in L^2(\R^3)$.  The solution is unique and is denoted
$u=\Bcal\varpi$.
\end{definition}
\begin{proposition}[Moment-free graph-norm density]\label{prop:graphdens}
Let $6/5<p<\infty$.  For every $\varpi\in\mathcal Z_p$ there are
axisymmetric $\varpi_n\in C_c^\infty(\R^3)$ such that
\begin{equation}\label{eq:bs-axis}
 \|\varpi_n-\varpi\|_1+\|\varpi_n-\varpi\|_p
 +\|\Bcal\varpi_n-\Bcal\varpi\|_2\to0.
\end{equation}
No sign or spatial moment assumption is used.
\end{proposition}
\begin{proof}
Set $u=\Bcal\varpi$ and $\omega=r\varpi e_\vartheta=\nabla\times u$.
Choose a spherically radial $\chi\in C_c^\infty(\R^3)$ with
$0\le\chi\le1$, $\chi=1$ on $B_1$, and $\chi=0$ on $B_2^c$, and set
\[
 \chi_R(x):=\chi(x/R),\qquad
 \varpi^R:=\chi_R\varpi,\qquad
 \omega^R:=\chi_R\omega.
\]
The toroidal field $\omega^R=r\varpi^Re_\vartheta$ is divergence free:
since $\chi_R$ is spherically radial and $e_\vartheta$ is tangent to
spheres,
\[
 \nabla\cdot(\chi_R\omega)
 =\nabla\chi_R\cdot\omega+\chi_R\nabla\cdot\omega=0.
\]
Let $f_R=1-\chi_R$.  From
$f_R\omega=\nabla\times(f_Ru)+\nabla\chi_R\times u$ and
Lemma~\ref{lem:divcurl},
\[
 \|\omega^R-\omega\|_{\dot H^{-1}}
 \le\|f_Ru\|_2+C\|\nabla\chi_R\times u\|_{6/5}.
\]
The second term is supported in $A_R=B_{2R}\setminus B_R$ and satisfies
\[
 \|\nabla\chi_R\times u\|_{6/5}
 \lesssim R^{-1}|A_R|^{1/3}\|u\|_{L^2(A_R)}
 \lesssim\|u\|_{L^2(A_R)}.
\]
Since $u\in L^2$, the curl isomorphism in Lemma~\ref{lem:divcurl} gives
\[
 \|\Bcal\varpi^R-\Bcal\varpi\|_2\to0
 \qquad(R\to\infty),
\]
while dominated convergence gives
\[
 \|\varpi^R-\varpi\|_1+\|\varpi^R-\varpi\|_p\to0.
\]

Fix $R$ and let $\varphi_\nu$ be a spherically radial standard mollifier.
Set $\varpi^{R,\nu}=\varphi_\nu*\varpi^R$.  Then
$\varpi^{R,\nu}\in C_c^\infty$ is axisymmetric and
\[
 \|\varpi^{R,\nu}-\varpi^R\|_1
 +\|\varpi^{R,\nu}-\varpi^R\|_p\to0
 \qquad(\nu\downarrow0).
\]
For $0<\nu\le1$, the difference
$\eta^{R,\nu}:=\varpi^{R,\nu}-\varpi^R$ is supported in $B_{2R+1}$, so
\[
 \|r\eta^{R,\nu}e_\vartheta\|_{6/5}
 \le(2R+1)\|\eta^{R,\nu}\|_{6/5}.
\]
Since $p>6/5$, interpolation between $L^1$ and $L^p$ gives
$\|\eta^{R,\nu}\|_{6/5}\to0$.  Lemma~\ref{lem:divcurl} therefore yields
\[
 \|\Bcal\varpi^{R,\nu}-\Bcal\varpi^R\|_2\to0
 \qquad(\nu\downarrow0).
\]
Choose $R_n\uparrow\infty$ and then $\nu_n\downarrow0$ diagonally so
that all four error terms tend to zero, and set
$\varpi_n:=\varpi^{R_n,\nu_n}$.  Then \eqref{eq:bs-axis} follows.
\end{proof}

\subsection{Slow diagonal kinetic realization and prescribed data}
\begin{proposition}[Slow diagonal kinetic realization of smooth profiles]\label{prop:diagreal}
Fix the integer $N\ge5$ and the kinetic scales of Theorem~\ref{thm:longtime}.
Let $1<p<\infty$.  Let
\[
 \varpi_0\in L^1(\R^3)\cap L^p(\R^3),\qquad
 u_0\in L^2(\R^3),\qquad
 \mathfrak{s}_0\in L^2(\R^3)
\]
be axisymmetric, and suppose that there are axisymmetric
$\varpi_{0,n}\in C_c^\infty(\R^3)$ such that, with
$u_{0,n}^E:=\Bcal\varpi_{0,n}$,
\begin{equation}\label{eq:bs-slow}
 \|\varpi_{0,n}-\varpi_0\|_1+\|\varpi_{0,n}-\varpi_0\|_p
 +\|u_{0,n}^E-u_0\|_2\to0.
\end{equation}
Let $\mathfrak{s}_{0,n},a_{0,n},\phi_{0,n}\in C_c^\infty(\R^3)$ be
axisymmetric scalar functions with
\begin{equation}\label{eq:axis-phi}
 \|\mathfrak{s}_{0,n}-\mathfrak{s}_0\|_2\to0,
\end{equation}
and put
\begin{equation}\label{eq:slow-def}
 q_{0,n}:=\nabla\phi_{0,n},\qquad
 u_{0,n}:=u_{0,n}^E+q_{0,n},\qquad
 \rho_{0,n}:=\frac35a_{0,n}-\frac25\mathfrak{s}_{0,n},\qquad
 \theta_{0,n}:=\frac25(a_{0,n}+\mathfrak{s}_{0,n}).
\end{equation}
No smallness, vanishing, or convergence assumption is imposed on the smooth
acoustic profiles $(a_{0,n},q_{0,n})$; their norms are allowed to enter the
$n$-dependent constants in the slow diagonal construction.
For either the Landau equation or the non-cutoff Boltzmann equation, there
is a family of smooth nonnegative axisymmetric no-swirl kinetic initial data
$F_0^\e$ satisfying the initial-data hypothesis \eqref{eq:init-bd} of
Theorem~\ref{thm:longtime} for all sufficiently small $\e$, and whose macroscopic
initial fields satisfy
\begin{align}
 \varpi_0^\e &\to \varpi_0
 \quad\text{in }L^1(\R^3)\cap L^p(\R^3),
 \label{eq:init-conv}\\
 \mathbb Pu_0^\e &\to u_0
 \quad\text{in }L^2(\R^3),
 \label{eq:slow-conv}\\
 \frac32\theta_0^\e-\rho_0^\e &\to \mathfrak{s}_0
 \quad\text{in }L^2(\R^3).
 \label{eq:real-conv}
\end{align}
More precisely, if $\e_{n+1}<\e\le\e_n$ in the diagonal constructed below,
then
\begin{equation}\label{eq:slow-id}
 \rho_0^\e+\theta_0^\e=a_{0,n},\qquad
 \mathbb P^\perp u_0^\e=q_{0,n},\qquad
 \frac32\theta_0^\e-\rho_0^\e=\mathfrak{s}_{0,n}.
\end{equation}
Thus neither acoustic component is required to vanish.  If, in addition,
$a_{0,n}\to a_0$ and $q_{0,n}\to q_0$ in $L^2$, then the first two
quantities in \eqref{eq:slow-id} converge in $L^2$ to
$a_0$ and $q_0$, respectively.  In particular,
\eqref{eq:lp-lone}--\eqref{eq:mic-conv} hold with
$(\varpi_0^E,u_0^E,\frac32\theta_0^E-\rho_0^E)
=(\varpi_0,u_0,\mathfrak{s}_0)$.
The diagonal choice preserves the following uniform bounds whenever they
hold for the smooth relative-vorticity profiles:
\begin{align}
 \sup_n\mathfrak M_1(\varpi_{0,n})<\infty
 &\quad\Longrightarrow\quad \eqref{eq:cond-bd},
 \label{eq:prof-bd}\\
 \sup_n\mathfrak M_2(\varpi_{0,n})<\infty
 &\quad\Longrightarrow\quad \eqref{eq:fld-bd},
 \label{eq:slow-bd}\\
 \sup_n\min\{\|(\varpi_{0,n})_+\|_{p^\sharp},
                   \|(\varpi_{0,n})_-\|_{p^\sharp}\}<\infty
 &\quad\Longrightarrow\quad \eqref{eq:lp-sharp},
 \label{eq:real-bd}
\end{align}
whenever the displayed quantities are relevant.
\end{proposition}
\begin{proof}
Since
$r\varpi_{0,n}e_\vartheta=\varpi_{0,n}(-x_2,x_1,0)$ is smooth and compactly
supported, the div--curl system \eqref{eq:divcurl-sys} and the
$L^2$ boundedness of the Riesz transforms give, for $1\le k\le N+1$,
\begin{equation*}
 \|\nabla^k u_{0,n}^E\|_2
 \lesssim_k
 \|\nabla^{k-1}(r\varpi_{0,n}e_\vartheta)\|_2<\infty.
\end{equation*}
Hence $u_{0,n}^E\in H^{N+1}$.  Since
$a_{0,n},\mathfrak{s}_{0,n},\phi_{0,n}$ are smooth and compactly supported,
\begin{equation}\label{eq:htop-def}
 Z_{0,n}:=(\rho_{0,n},u_{0,n},\theta_{0,n})\in H^{N+1}(\R^3),
 \qquad
 C_{Z,n}:=1+\|Z_{0,n}\|_{H^{N+1}}+\|Z_{0,n}\|_{W^{2,\infty}}<\infty.
\end{equation}
Moreover $q_{0,n}=\nabla\phi_{0,n}$ is irrotational, so
\begin{equation}\label{eq:htop-id}
 \mathbb Pu_{0,n}=u_{0,n}^E,\qquad
 \mathbb P^\perp u_{0,n}=q_{0,n},\qquad
 \frac{(\nabla\times u_{0,n})^\vartheta}{r}=\varpi_{0,n}.
\end{equation}
Lift these macroscopic fields to kinetic initial data.  For fixed
$n$ and sufficiently small $\e>0$, set
\begin{equation}\label{eq:data-def}
 M_{0,n}^\e(x,\xi)
 :=M_{[e^{\e\rho_{0,n}(x)},\,\e u_{0,n}(x),\,e^{\e\theta_{0,n}(x)}]}(\xi),
 \qquad
 F_{0,n}^\e:=M_{0,n}^\e.
\end{equation}
By the parametrization \eqref{eq:param-id}, the macroscopic fields of
$F_{0,n}^\e$ are exactly
\begin{equation*}
 \rho_{0,n}^\e=\rho_{0,n},\qquad
 u_{0,n}^\e=u_{0,n},\qquad
 \theta_{0,n}^\e=\theta_{0,n}.
\end{equation*}
The identities in \eqref{eq:slow-def} therefore give
\begin{equation}\label{eq:ana-id}
 \rho_{0,n}^\e+\theta_{0,n}^\e=a_{0,n},\qquad
 \frac32\theta_{0,n}^\e-\rho_{0,n}^\e=\mathfrak{s}_{0,n}.
\end{equation}
Since the scalar profiles are axisymmetric and both $u_{0,n}^E$ and
$q_{0,n}$ are $O(2)_z$-equivariant without swirl,
\eqref{eq:loc-id} gives
\begin{equation*}
 F_{0,n}^\e(Gx,G\xi)=F_{0,n}^\e(x,\xi),
 \qquad G\in O(2)_z.
\end{equation*}
By \eqref{eq:dist-def},
\begin{equation}\label{eq:mp-id}
 \Pi_{M_{0,n}^\e}^\perp F_{0,n}^\e=0.
\end{equation}
We next verify the top microscopic term without imposing
$\rho_{0,n}=\theta_{0,n}=0$.  Define the parameterized Maxwellian
$\mathcal M_{\rm par}(a,b,c):=M_{[e^a,b,e^c]}.$
If $h=(h_\rho,h_u,h_\theta)$, direct differentiation of \eqref{eq:loc-id}
gives
\begin{equation}\label{eq:par-rho}
 D\mathcal M_{\rm par}(a,b,c)[h]
 =\mathcal M_{\rm par}(a,b,c)\left[
 h_\rho+\frac{\xi-b}{k_Be^c}\cdot h_u
 +\left(\frac{|\xi-b|^2}{2k_Be^c}-\frac32\right)h_\theta
 \right].
\end{equation}
At $(a,b,c)=\e Z_{0,n}(x)$, the three brackets in
\eqref{eq:par-rho} are precisely linear combinations of the five basis
functions defining $\mathcal N_{M_{0,n}^\e}$.  Hence
\begin{equation}\label{eq:par-ran}
 D\mathcal M_{\rm par}(\e Z_{0,n})[h]
 \in\operatorname{Ran}\Pi_{M_{0,n}^\e}
 \qquad\text{for every }h\in\R^5.
\end{equation}
For every spatial multi-index $\gamma$ with $1\le|\gamma|\le N+1$,
Fa\`a di Bruno's formula, grouped according to the number of parameter
derivatives of $\mathcal M_{\rm par}$, yields
\begin{equation*}
 \partial_x^\gamma F_{0,n}^\e
 =\e D\mathcal M_{\rm par}(\e Z_{0,n})[\partial_x^\gamma Z_{0,n}]
 +\e^2R_{\gamma,n}^\e.
\end{equation*}
Every term in $R_{\gamma,n}^\e$ contains at least two positive spatial
derivatives of components of $Z_{0,n}$ and a parameter derivative of
$\mathcal M_{\rm par}$ of order at least two.  After reducing $\e_n^*>0$, we have
for $0<\e\le\e_n^*$,
$\frac12\le e^{\e\rho_{0,n}},e^{\e\theta_{0,n}}\le2, |\e u_{0,n}|\le1.$
Every parameter derivative of $\mathcal M_{\rm par}$ is therefore a polynomial in
$\xi-\e u_{0,n}$ times $M_{0,n}^\e$, with coefficients bounded by $C_n$.
Using \eqref{eq:htop-def}, Sobolev multiplication, and Gaussian integration,
we obtain
\[
 \sup_{0<\e\le\e_n^*}
 \|R_{\gamma,n}^\e(M_{0,n}^\e)^{-1/2}\|_{L^2_{x,\xi}}
 \le C_{\gamma,n}.
\]
Together with \eqref{eq:par-ran}, this implies
\begin{equation}\label{eq:xd-bd}
 \|\Pi_{M_{0,n}^\e}^\perp
 \partial_x^\gamma F_{0,n}^\e\,(M_{0,n}^\e)^{-1/2}\|_2
 \le C_{\gamma,n}\e^2,
 \qquad1\le|\gamma|\le N+1.
\end{equation}
For the Landau expansion in \eqref{eq:dist-id}, define
\begin{equation*}
 M_{0,n}^{\e,\tau}:=\mathcal M_{\rm par}(\tau\e Z_{0,n})
 =M_{[e^{\tau\e\rho_{0,n}},\,\tau\e u_{0,n},\,e^{\tau\e\theta_{0,n}}]},
 \qquad0\le\tau\le1.
\end{equation*}
The fundamental theorem of calculus gives the exact identity
\begin{equation}\label{eq:mu-par}
 f_{0,n}^\e
 =\frac{M_{0,n}^\e-\mu}{\e\sqrt\mu}
 =\int_0^1
 \frac{D\mathcal M_{\rm par}(\tau\e Z_{0,n})[Z_{0,n}]}{\sqrt\mu}
 \,d\tau.
\end{equation}
For $0<\e\le\e_n^*$, reduce $\e_n^*$ further so that
$e^{\tau\e\theta_{0,n}}\in[3/4,5/4]$ and
$|\tau\e u_{0,n}|\le1$ for all $(\tau,x)$.  Expanding the exponent and
using Young's inequality gives constants $c_*>0$ and $C_n<\infty$ such
that
\begin{align*}
 -\frac{|\xi-\tau\e u_{0,n}|^2}{2k_Be^{\tau\e\theta_{0,n}}}
 +\frac{|\xi|^2}{4k_B}
 &\le -c_*|\xi|^2+C_n,
 \qquad0\le\tau\le1 .
\end{align*}
After the fixed velocity dilation used in the Landau collision estimates,
which normalizes $k_B=1$, the quadratic coefficient in the previous bound tends to $1/4$ as $\e\to0$.  Since the Landau weight
was fixed with $q<1/8$, we may reduce $\e_n^*$ so that $c_*>q$ in the
normalized variable.  Differentiating
\eqref{eq:mu-par} in $x$ and at most once in $\xi$, and using
\eqref{eq:htop-def}, \eqref{eq:par-rho}, this Gaussian bound, and Sobolev
multiplication, yields
\begin{align}
 \sup_{0<\e\le\e_n^*}
 \Bigg[
 &\sum_{|\gamma|\le N+1}
 \|w\partial_x^\gamma f_{0,n}^\e\|_2
 +\sum_{\substack{|\beta|=1\\|\gamma|\le\lfloor N/2\rfloor}}
 \|w^{1/2}\partial_\xi^\beta\partial_x^\gamma
 (I-\Pi)f_{0,n}^\e\|_2
 \Bigg]\le C_n.
 \label{eq:der-xd}
\end{align}
For the Boltzmann expansion in \eqref{eq:dist-id}, \eqref{eq:data-def} gives
\begin{equation}\label{eq:bolt-id}
 g_{0,n}^\e=0,
 \qquad
 f_{0,n}^\e
 =\frac{M_{0,n}^\e-\mu_0}{\e\sqrt{\mu_0}}.
\end{equation}
Set
$T_0^\e:=1-\e^{1-c_0}, \  \Theta_{0,n}^\e(x):=e^{\e\theta_{0,n}(x)}.$
Since $|e^{\e\theta_{0,n}}-1|\le C_n\e$ for $0<\e\le\e_n^*$,
and $\e=o(\e^{1-c_0})$, after reducing $\e_n^*$ we have
\[
 \Theta_{0,n}^\e-T_0^\e
 \ge \e^{1-c_0}-C_n\e
 \ge\frac12\e^{1-c_0},
 \qquad
 \frac34\le T_0^\e,\Theta_{0,n}^\e\le\frac54.
\]
Moreover $|\e u_{0,n}|\le1$.  Expanding the exponent gives
\[
 -\frac{|\xi-\e u_{0,n}(x)|^2}{2k_B\Theta_{0,n}^\e(x)}
 +\frac{|\xi|^2}{4k_BT_0^\e}
 =
 -\left(\frac1{2k_B\Theta_{0,n}^\e}
          -\frac1{4k_BT_0^\e}\right)|\xi|^2
 +\frac{\e u_{0,n}\cdot \xi}{k_B\Theta_{0,n}^\e}
 -\frac{\e^2|u_{0,n}|^2}{2k_B\Theta_{0,n}^\e}.
\]
After decreasing $\e_n^*$ once more,
$\frac1{2\Theta_{0,n}^\e}-\frac1{4T_0^\e}\ge\frac1{15}$.
Young's inequality in the linear term then yields constants $c_B>0$ and $C_n$ such
that
\begin{equation*}
 -\frac{|\xi-\e u_{0,n}(x)|^2}{2k_B\Theta_{0,n}^\e(x)}
 +\frac{|\xi|^2}{4k_BT_0^\e}
 \le-c_B|\xi|^2+C_n\e^2.
\end{equation*}
For every spatial multi-index $\gamma\ne0$,
\[
 \partial_x^\gamma f_{0,n}^\e
 =\frac{\partial_x^\gamma M_{0,n}^\e}{\e\sqrt{\mu_0}}.
\]
By Fa\`a di Bruno's formula, every term in
$\partial_x^\gamma M_{0,n}^\e$ contains at least one factor
$\e\partial_x^{\gamma'}Z_{0,n}$ with $|\gamma'|\ge1$.  The Gaussian bound just established, \eqref{eq:htop-def}, and the derivative identity above,
together with the fact
that the Boltzmann weight $w_0$ grows only polynomially times
$e^{q\langle \xi/a_\e\rangle}$, give for every fixed $m\ge0$
\begin{align}
 \sup_{0<\e\le\e_n^*}
 \Bigg[
 &\sum_{1\le|\gamma|\le N+1}
 \|\langle \xi\rangle^m w_0\partial_x^\gamma f_{0,n}^\e
 \|_{L_x^2H_\xi^1}
+\sum_{\substack{|\beta|=1\\1\le|\gamma|\le\lfloor N/2\rfloor}}
 \|\langle \xi\rangle^m w_0^{1/2}
 \partial_\xi^\beta\partial_x^\gamma f_{0,n}^\e
 \|_{L_x^2H_\xi^1}
 \Bigg]
 \le C_{m,n}.
 \label{eq:null-bd}
\end{align}
For $3/4\le T_0^\e\le1$, choose an $L_\xi^2$-orthonormal basis
$\{\phi_j^\e\}_{j=0}^4$ of $\operatorname{Null}(\mathcal L_0)$ obtained
from the five functions in \eqref{eq:pr-conv}.  Each $\phi_j^\e$ is a polynomial
of degree at most two times $\sqrt{\mu_0}$, with coefficients uniformly
bounded for $T_0^\e\in[3/4,1]$.  Thus
\[
 \Pi_0h=\sum_{j=0}^4\langle h,\phi_j^\e\rangle_{L_\xi^2}\phi_j^\e,
 \qquad
 |\langle h,\phi_j^\e\rangle|\le\|h\|_2,
\]
and Gaussian integration, together with $a_\e\sim1$, gives, for every
fixed $m\ge0$ and $|\beta|\le1$,
\begin{equation}
 \|\langle \xi\rangle^m w_0
 \partial_\xi^\beta\Pi_0 h\|_{H_\xi^1}
 \le C_m\|h\|_{L_\xi^2},
\end{equation}
uniformly in $\e$.  Moreover, for $m$ sufficiently large,
\begin{equation}\label{eq:gam-large}
 \|h\|_{H^s_{2s+\gamma}}
 \lesssim_m
 \|\langle \xi\rangle^m h\|_{H_\xi^1}.
\end{equation}
To verify this embedding directly from the fractional norm, put $a:=2s+\gamma$.  For $|\eta|\le1$,
\[
 h(\xi+\eta)-h(\xi)=\int_0^1 \eta\cdot\nabla_\xi h(\xi+\tau \eta)\,d\tau,
\]
and hence, by Cauchy--Schwarz in $\tau$ and the comparability
$\langle \xi\rangle^a\lesssim_a\langle \xi+\tau \eta\rangle^a$,
\begin{align}
 &\iint_{|\eta|\le1}\langle \xi\rangle^a
 \frac{|h(\xi+\eta)-h(\xi)|^2}{|\eta|^{3+2s}}\,d\eta\,d\xi\notag\\
 &\quad\lesssim_a
 \int_{|\eta|\le1}|\eta|^{-1-2s}
 \int_0^1\!\int_{\R^3}
 \langle \xi+\tau \eta\rangle^a|\nabla_\xi h(\xi+\tau \eta)|^2
 \,d\xi\,d\tau\,d\eta \lesssim_{a,s}
 \|\langle \xi\rangle^{a/2}\nabla_\xi h\|_2^2, \notag
\end{align}
where the last step uses
$\int_{|\eta|\le1}|\eta|^{-1-2s}\,d\eta <\infty \ (0<s<1).$
The weighted $L^2$ part is bounded in the same way, and choosing
$2m\ge\max\{a,0\}$ proves \eqref{eq:gam-large}.  Thus
\eqref{eq:null-bd}--\eqref{eq:gam-large} control all positive-spatial-order
Boltzmann high-moment terms in \eqref{eq:xi-def}; the zero spatial order is
measured by $g_{0,n}^\e=0$ in \eqref{eq:bolt-id}.
We now verify the complete initial energy.  Equation \eqref{eq:mp-id} gives
$\mathcal E_G^N(0)=0$.  The macroscopic part is bounded by
\eqref{eq:htop-def}, and \eqref{eq:xd-bd} yields for $|\gamma|=N+1$
\begin{equation*}
 \frac{\kappa}{\e^2}
 \|\Pi_{M_{0,n}^\e}^\perp
 \partial_x^\gamma F_{0,n}^\e(M_{0,n}^\e)^{-1/2}\|_2^2
 \le C_{\gamma,n}\kappa\e^2\le C_{\gamma,n}.
\end{equation*}
For Landau, \eqref{eq:der-xd} bounds \eqref{eq:land-id}.  For Boltzmann,
\eqref{eq:bolt-id} annihilates $\mathcal E_{H,1}^N(0)$, while
\eqref{eq:null-bd}--\eqref{eq:gam-large} bound $\mathcal E_{H,2}^N(0)$.
Consequently, for the corresponding collision model and after reducing
$\e_n^*$ if necessary,
\begin{equation}\label{eq:en-coll}
 \sup_{0<\e\le\e_n^*}
 \mathcal E_{tot}^{N}(F_{0,n}^\e)\le C_n.
\end{equation}
Since $\varpi_{0,n}\in C_c^\infty(\R^3)$,
\begin{equation}\label{eq:ana-def}
 L_n:=C_n+\|\varpi_{0,n}\|_{L^{3,1}}^2<\infty.
\end{equation}
By \eqref{eq:small-def}, after fixing a structural number
$0<\bar\e<e^{-e^{e^{e}}}$ sufficiently small, all five iterated logarithms
$\Lambda_j^\e$ are positive on $(0,\bar\e]$ and
$\Lambda_5^\e$ is nondecreasing as $\e\downarrow0$, with
$\Lambda_5^\e\to\infty$.  We choose $\e_n\downarrow0$ recursively.  First
choose
\[
 0<\e_1\le\min\{\bar\e,\e_1^*,1\}
 \quad\text{so that}\quad
 \Lambda_5^\e\ge L_1\quad(0<\e\le\e_1).
\]
Having chosen $\e_{n-1}$, choose
\begin{equation}\label{eq:aux}
 0<\e_n\le
 \min\{\bar\e,\e_n^*,n^{-1},\tfrac12\e_{n-1}\},
 \qquad
 \Lambda_5^\e\ge L_n
 \quad\text{for every }0<\e\le\e_n.
\end{equation}
Such a choice is possible because $L_n<\infty$ is fixed when $n$ is fixed
and $\Lambda_5^\e\to\infty$ as $\e\downarrow0$.
For $\e_{n+1}<\e\le\e_n$, define
\begin{equation}\label{eq:choice-def}
 F_0^\e:=F_{0,n}^\e,
 \qquad
 \varpi_0^\e:=\varpi_{0,n},
 \qquad
 u_0^\e:=u_{0,n},
 \qquad
 (\rho_0^\e,\theta_0^\e):=(\rho_{0,n},\theta_{0,n}).
\end{equation}
Then \eqref{eq:en-coll}, \eqref{eq:ana-def}, and \eqref{eq:aux} give
\eqref{eq:init-bd}.  Equations \eqref{eq:htop-id}, \eqref{eq:ana-id}, and
\eqref{eq:choice-def} give \eqref{eq:slow-id}.  Since
$\varpi_0^\e=\varpi_{0,n}$ and $\mathbb Pu_0^\e=u_{0,n}^E$ on each interval
$(\e_{n+1},\e_n]$, \eqref{eq:bs-slow} gives
\eqref{eq:init-conv}--\eqref{eq:slow-conv}; it also
implies the uniform bounds \eqref{eq:lp-lone}--\eqref{eq:rough-bd}.  Finally,
\eqref{eq:axis-phi}, \eqref{eq:ana-id}, and
\eqref{eq:choice-def} give \eqref{eq:real-conv}, hence
\eqref{eq:mic-conv}.  The implications
\eqref{eq:prof-bd}--\eqref{eq:real-bd}
follow from the same piecewise definition \eqref{eq:choice-def}.  This proves the
proposition.
\end{proof}

\begin{corollary}[Prescribed rough finite-energy, thermodynamic, and acoustic data]\label{cor:roughreal}
Let $1<p<\infty$.  Let $\varpi_0\in L^1(\R^3)\cap L^p(\R^3)$ and
$\mathfrak{s}_0,a_0\in L^2(\R^3)$ be axisymmetric.  Let
$q_0\in L^2(\R^3;\R^3)$ be axisymmetric without swirl and assume
\begin{equation}\label{eq:data-id}
 \mathbb Pq_0=0.
\end{equation}
Assume one of the following.
\begin{enumerate}
\item If $p\ge4/3$, there is $u_0\in L^2(\R^3)$ such that
\begin{equation}\label{eq:div-curl-id}
 \nabla\cdot u_0=0,
 \qquad
 \nabla\times u_0=r\varpi_0e_\vartheta
 \quad\text{in }\mathcal D'(\R^3).
\end{equation}
\item If $1<p<4/3$, then
\begin{equation}\label{eq:mtwo-lp}
 \mathfrak M_2(\varpi_0)<\infty,
 \qquad
 \min\left\{
 \|(\varpi_0)_+\|_{L^{p^\sharp}},
 \|(\varpi_0)_-\|_{L^{p^\sharp}}
 \right\}<\infty.
\end{equation}
\end{enumerate}
In case (2), set $u_0:=\Bcal\varpi_0$; Proposition~\ref{prop:radmom} gives
$u_0\in L^2(\R^3)$.
For either the Landau equation or the non-cutoff Boltzmann equation, there
exists a family of smooth nonnegative axisymmetric no-swirl kinetic initial
data satisfying the hypotheses of Theorem~\ref{thm:longtime} such that
\begin{align}
 \varpi_0^\e \to\varpi_0
 &\text{ in }L^1(\R^3)\cap L^p(\R^3),
 \label{eq:data-conv}\\
 \mathbb Pu_0^\e\to u_0
 &\text{ in }L^2(\R^3), \notag
 \\
 \frac32\theta_0^\e-\rho_0^\e\to\mathfrak{s}_0
 & \text{ in }L^2(\R^3).
 \label{eq:ac-conv}
\end{align}
The realization may moreover be chosen so that the acoustic initial data
converge to the prescribed targets:
\begin{align}
 \rho_0^\e+\theta_0^\e\to a_0
 &\text{ in }L^2(\R^3),
 \label{eq:presc-conv}\\
 \mathbb P^\perp u_0^\e\to q_0
 &\text{ in }L^2(\R^3).
 \label{eq:pq-conv}
\end{align}
Thus no vanishing acoustic preparation is imposed.
For every sequence $\e_j\downarrow0$ and every finite $T>0$, a subsequence
admits a renormalized axisymmetric Euler limit $(u^E,\varpi^E)$ on $[0,T]$
with initial relative vorticity $\varpi_0$, together with the transported
thermodynamic limit $\mathfrak{s}^E$, and along that subsequence
\begin{align}
 \varpi_A^{\e_j} \to\varpi^E
 &\text{ in }C([0,T];L^1\cap L^p),
 \label{eq:lim-conv}\\
 \mathbb Pu^{\e_j}\to u^E
 &\text{ in }C([0,T];L^2(\R^3)),
 \label{eq:pp-ltwo}\\
 \frac32\theta^{\e_j}-\rho^{\e_j}\to\mathfrak{s}^E
 &\text{ in }C([0,T];L^2(\R^3)),
 \label{eq:th-rho}
\end{align}
where $\mathfrak{s}^E$ is the unique renormalized solution of
\begin{equation}\label{eq:real-id}
 \partial_t\mathfrak{s}^E+u^E\cdot\nabla\mathfrak{s}^E=0,
 \qquad
 \mathfrak{s}^E(0)=\mathfrak{s}_0.
\end{equation}
Moreover,
\begin{equation}\label{eq:presc-bd}
 \|u^E(t)\|_2=\|u_0\|_2,
 \qquad0\le t\le T.
\end{equation}
For $p\ge4/3$, no sign or spatial moment assumption is required.
\end{corollary}
\begin{proof}[Proof of Corollary~\ref{cor:roughreal}]
Suppose first that $p\ge4/3$.  The hypothesis
\eqref{eq:div-curl-id} gives $\varpi_0\in\mathcal Z_p$, and
Lemma~\ref{lem:divcurl} gives $u_0=\Bcal\varpi_0$.
Proposition~\ref{prop:graphdens} provides axisymmetric
$\varpi_{0,n}\in C_c^\infty(\R^3)$ such that
\begin{equation}\label{eq:bs-kin}
 \|\varpi_{0,n}-\varpi_0\|_1+\|\varpi_{0,n}-\varpi_0\|_p
 +\|\Bcal\varpi_{0,n}-u_0\|_2\to0.
\end{equation}
Now let $1<p<4/3$ and assume \eqref{eq:mtwo-lp}.  Set
$\eta_+:=(\varpi_0)_+, \quad \eta_-:=(\varpi_0)_-.$
Then $\eta_\pm\in L^1\cap L^p$ and
$\mathfrak M_2(\eta_+)+\mathfrak M_2(\eta_-)=\mathfrak M_2(\varpi_0)<\infty$.  Hence
Proposition~\ref{prop:radmom} gives $u_0=\Bcal\varpi_0\in L^2$.
Choose $\sigma\in\{+,-\}$ so that
$\eta_\sigma\in L^{p^\sharp}$.  Let $0\le\chi_R\le1$ be spherically radial,
$\chi_R=1$ on $B_R$, $\chi_R=0$ on $B_{2R}^c$, and let
$\varphi_\nu$ be a nonnegative spherically radial standard mollifier supported in
$B_\nu$.  By dominated convergence, choose $R_n\uparrow\infty$ so that,
for both signs,
\[
 \|(1-\chi_{R_n})\eta_\pm\|_1
 +\|(1-\chi_{R_n})\eta_\pm\|_p
 +\mathfrak M_2((1-\chi_{R_n})\eta_\pm)\to0.
\]
For each $n$, choose $0<\nu_n\le1$ sufficiently small and set
$\eta_{\pm,n}:=\varphi_{\nu_n}*(\chi_{R_n}\eta_\pm).$
If $g_{\pm,n}:=\eta_{\pm,n}-\chi_{R_n}\eta_\pm$, then
$\operatorname{supp}g_{\pm,n}\subset B_{2R_n+1}$ and
$\mathfrak M_2(g_{\pm,n})\le \frac{(2R_n+1)^2}{2\pi}\|g_{\pm,n}\|_1.$
Hence the mollification error is controlled by its $L^1$ norm.
Thus $\nu_n$ may be chosen so that
\begin{equation*}
 \|\eta_{\pm,n}-\eta_\pm\|_1
 +\|\eta_{\pm,n}-\eta_\pm\|_p
 +\mathfrak M_2(\eta_{\pm,n}-\eta_\pm)\to0.
\end{equation*}
Moreover, Young's inequality and $0\le\chi_{R_n}\le1$ give
\begin{equation*}
 \|\eta_{\sigma,n}\|_{p^\sharp}
 \le\|\chi_{R_n}\eta_\sigma\|_{p^\sharp}
 \le\|\eta_\sigma\|_{p^\sharp}.
\end{equation*}
Set
$\varpi_{0,n}:=\eta_{+,n}-\eta_{-,n}.$
Then $\eta_{\pm,n}\ge0$ and
$\varpi_{0,n}\in C_c^\infty(\R^3)$ is axisymmetric.  Moreover,
this approximation bound gives
\begin{align}
 &\|\varpi_{0,n}-\varpi_0\|_1+\|\varpi_{0,n}-\varpi_0\|_p\to0,
 \label{eq:axis-mtwo}\\
& \mathfrak M_2(\varpi_{0,n})
 \le\mathfrak M_2(\eta_{+,n})+\mathfrak M_2(\eta_{-,n})
 \le\mathfrak M_2(\varpi_0)+o(1).
 \label{eq:mtwo-eta-bd}
\end{align}
Since
$(\varpi_{0,n})_+\le\eta_{+,n},$ $(\varpi_{0,n})_-\le\eta_{-,n},$
this $L^{p^\sharp}$ bound yields
\begin{equation}\label{eq:sharp-eta}
 \sup_n\min\left\{
 \|(\varpi_{0,n})_+\|_{p^\sharp},
 \|(\varpi_{0,n})_-\|_{p^\sharp}
 \right\}
 \le\|\eta_\sigma\|_{p^\sharp}<\infty.
\end{equation}
Finally, Proposition~\ref{prop:radmom} and \eqref{eq:axis-mtwo}--\eqref{eq:mtwo-eta-bd} imply
\begin{align}
 \|\Bcal\varpi_{0,n}-u_0\|_2^2
 &\lesssim_p
 \|\varpi_{0,n}-\varpi_0\|_p^{A_p}
 \bigl(\mathfrak M_2(\varpi_{0,n})+\mathfrak M_2(\varpi_0)\bigr)^{B_p}
 \to0. \notag
\end{align}
Hence \eqref{eq:bs-kin} holds in the low-$p$ regime as
well.
Choose axisymmetric
$\mathfrak{s}_{0,n},a_{0,n}\in C_c^\infty(\R^3)$ such that
\begin{equation}\label{eq:axis-slow}
 \|\mathfrak{s}_{0,n}-\mathfrak{s}_0\|_2
 +\|a_{0,n}-a_0\|_2\to0.
\end{equation}
We next approximate the irrotational acoustic velocity.  By
\eqref{eq:data-id}, there is
$\phi_0\in\dot H^1(\R^3)$, unique modulo constants, such that
$q_0=\nabla\phi_0, \quad \|\nabla\phi_0\|_2=\|q_0\|_2.$
Indeed, in Fourier variables one may take
$\widehat{\phi_0}(\eta)=-i\,\eta\cdot\widehat q_0(\eta)/|\eta|^2$;
then $\mathbb Pq_0=0$ gives
$i\eta\widehat{\phi_0}=\widehat q_0$ almost everywhere.
Choose a spherically radial cutoff $\chi_R$ with $\chi_R=1$ on $B_R$ and
$\chi_R=0$ on $B_{2R}^c$.  Sobolev's inequality
$\dot H^1(\R^3)\hookrightarrow L^6(\R^3)$ yields
\begin{align}
 \|\nabla(\chi_R\phi_0)-\nabla\phi_0\|_2
 &\le \|(1-\chi_R)\nabla\phi_0\|_2
 +C R^{-1}\|\phi_0\|_{L^2(B_{2R}\setminus B_R)}\notag\\
 &\le \|(1-\chi_R)\nabla\phi_0\|_2
 +C\|\phi_0\|_{L^6(B_{2R}\setminus B_R)}\to0. \notag 
\end{align}
Mollifying $\chi_R\phi_0$ by a spherically radial mollifier and then averaging over
rotations about the $z$-axis produces
$\phi_{0,n}\in C_c^\infty(\R^3)$ axisymmetric such that, with
$q_{0,n}:=\nabla\phi_{0,n}$,
\begin{equation}\label{eq:axis-rad}
 \|q_{0,n}-q_0\|_2\to0.
\end{equation}
To see that the averaging does not increase the error, define
$\mathcal A\psi(x):=(2\pi)^{-1}\int_0^{2\pi}
\psi(R_\vartheta x)\,d\vartheta$.  Since the axisymmetric vector field
$q_0$ is equivariant, $q_0(R_\vartheta x)=R_\vartheta q_0(x)$, and
\[
 \nabla\mathcal A\psi(x)-q_0(x)
 =\frac1{2\pi}\int_0^{2\pi}
 R_\vartheta^{\mathsf T}
 \bigl(\nabla\psi(R_\vartheta x)-q_0(R_\vartheta x)\bigr)
 \,d\vartheta.
\]
Jensen's inequality and the invariance of Lebesgue measure under rotations
therefore give
$\|\nabla\mathcal A\psi-q_0\|_2 \le \|\nabla\psi-q_0\|_2.$
Thus each $q_{0,n}$ is smooth, compactly supported, axisymmetric without
swirl, and irrotational.
Apply Proposition~\ref{prop:diagreal} with
$(\mathfrak{s}_{0,n},a_{0,n},q_{0,n})$.  Equations
\eqref{eq:slow-id},
\eqref{eq:axis-slow}, and
\eqref{eq:axis-rad} give
\eqref{eq:presc-conv}--
\eqref{eq:pq-conv}, while the same proposition gives
Theorems~\ref{thm:longtime}--\ref{thm:roughlimit} and
\eqref{eq:data-conv}--\eqref{eq:ac-conv}.  If
$1<p<4/3$, equations \eqref{eq:mtwo-eta-bd} and
\eqref{eq:sharp-eta}, together with
\eqref{eq:slow-bd}--\eqref{eq:real-bd},
give \eqref{eq:fld-bd}--\eqref{eq:lp-sharp}.
Given any sequence $\e_j\downarrow0$, Theorem~\ref{thm:roughlimit} supplies a
subsequence and a renormalized axisymmetric Euler pair $(u^E,\varpi^E)$ with
initial relative vorticity $\varpi_0$ and the strong convergence
\eqref{eq:lim-conv}.  In the high-$p$ regime Theorem~\ref{thm:globenergy}
requires no sign or spatial moment assumption, while in the low-$p$ regime
these bounds verify its additional hypotheses.  Therefore
\eqref{eq:pp-ltwo} and \eqref{eq:presc-bd} follow from
\eqref{eq:rough-conv}--\eqref{eq:pp-bd} in both cases.  The convergence
\eqref{eq:th-rho} and the transport equation
\eqref{eq:real-id} are exactly assertion~(3) and the
identification in the proof of Theorem~\ref{thm:roughlimit}, namely
\eqref{eq:ltwo-rho} with initial datum $\mathfrak{s}_0$.\end{proof}

{
The construction above supplies the kinetic approximations used for the
singularity-passage results in Section~\ref{sec:singularpassage}.  We conclude
the present section with a sign-changing realization in the rough
integrability range.
}

\subsection{A rough sign-changing kinetic class}
\begin{proposition}[Rough sign-changing data]\label{prop:signchg}
Let $1<p<5/3$ and let $p^\sharp$ be defined by
\eqref{eq:psharp}.  For either the Landau equation or the non-cutoff Boltzmann equation there
is a family of smooth nonnegative axisymmetric no-swirl kinetic initial data
satisfying the hypotheses of Theorems~\ref{thm:longtime}--\ref{thm:globenergy}
such that the limiting relative vorticity satisfies
\begin{equation}\label{eq:lone-sharp}
 \varpi_0^E\in L^1(\R^3)\cap L^p(\R^3),\qquad
 \varpi_0^E\notin L^{p^\sharp}(\R^3),\qquad
 \varpi_0^E\notin L^3(\R^3),
\end{equation}
and
\begin{equation}\label{eq:sharp-lp}
 (\varpi_0^E)_+\not\equiv0,
 \qquad
 (\varpi_0^E)_-\not\equiv0.
\end{equation}
Moreover,
\begin{align}
    &  (\varpi_0^E)_-\in L^{p^\sharp},
 \qquad
 (\varpi_0^E)_+\notin  L^{p^\sharp}.\label{oneside-Lpsharp}
\end{align}
For every finite $T>0$ and every limiting Euler field furnished by
Theorem~\ref{thm:roughlimit},
\begin{equation}\label{eq:fld-eul}
 \|(\varpi^E(t))_+\|_p=\|(\varpi_0^E)_+\|_p>0,
 \qquad
 \|(\varpi^E(t))_-\|_p=\|(\varpi_0^E)_-\|_p>0,
 \qquad 0\le t\le T.
\end{equation}
For $1<p<4/3$, condition \eqref{eq:lp-sharp} does not force
either sign component to vanish.  For $4/3\le p<5/3$, the same construction
lies in the moment-free energy class of Theorem~\ref{thm:globenergy}.
\end{proposition}
\begin{proof}
Since $1<p<5/3$, one has $p<p^\sharp<2$.  Fix
\begin{equation}\label{eq:sharp-alp}
 \frac{2}{p^\sharp}<\alpha<\frac{2}{p}.
\end{equation}
Choose nonnegative $\chi,\psi\in C_c^\infty(\R)$ such that
\begin{equation*}
 \chi(s)=1\quad(0\le s\le1),
 \qquad
 \operatorname{supp}\chi\subset(-1,4),
 \qquad
 \psi(z)=1\quad(|z|\le1/2),
 \qquad
 \operatorname{supp}\psi\subset(-1,1).
\end{equation*}
For $r>0$, set $\zeta(r,z):=r^{-\alpha}\chi(r^2)\psi(z)$.
By cylindrical coordinates, for $1\le s<\infty$,
\[
 \|\zeta\|_{L^s(\R^3)}^s
 =2\pi\|\psi\|_{L^s(\R)}^s
 \int_0^\infty r^{1-\alpha s}\chi(r^2)^s\,dr.
\]
The choice of $\alpha$ gives $\alpha p<2<\alpha p^\sharp$ and, since $p^\sharp<2$, also $3\alpha>2$.  Hence
\begin{equation}\label{eq:zet-lone}
 \zeta\in L^1\cap L^p,
 \qquad
 \zeta\notin L^{p^\sharp},
 \qquad
 \zeta\notin L^3.
\end{equation}
Moreover, by \eqref{eq:radmom-def},
\begin{equation}\label{eq:mtwo-zet}
 \mathfrak M_2(\zeta)
 =\|\psi\|_{L^1(\R)}
 \int_0^\infty r^{3-\alpha}\chi(r^2)\,dr<\infty,
\end{equation}
since \eqref{eq:sharp-alp} implies $\alpha<2$.
Choose a nonzero nonnegative axisymmetric $\eta\in C_c^\infty(\R^3)$ with
$\operatorname{supp}\eta\subset\{x:\ 3<r<4,\ |z|<1\}$.  Its support is disjoint from that of $\zeta$.  Define $\varpi_0^E:=\zeta-\eta$.  Then $(\varpi_0^E)_+=\zeta$ and
$(\varpi_0^E)_-=\eta$, so \eqref{eq:zet-lone} proves
\eqref{eq:lone-sharp} and \eqref{eq:sharp-lp}. $\eta \in C_c^\infty(\R^3)$ proves \eqref{oneside-Lpsharp}.
For $n\ge1$, define
\begin{equation}\label{eq:zet-alp}
 \zeta_n(r,z)
 :=(r^2+n^{-2})^{-\alpha/2}\chi(r^2)\psi(z),
 \qquad
 \varpi_{0,n}:=\zeta_n-\eta.
\end{equation}
The axisymmetric extensions of $\zeta_n$ and $\varpi_{0,n}$ belong to
$C_c^\infty(\R^3)$.  Since
$0\le\zeta_n\le\zeta$ and $\zeta_n\to\zeta$ a.e.,
\eqref{eq:zet-lone} and dominated convergence give
\begin{equation}\label{eq:lone-axis}
 \varpi_{0,n}\to\varpi_0^E
 \qquad\text{in }L^1(\R^3)\cap L^p(\R^3).
\end{equation}
The disjoint supports also give $(\varpi_{0,n})_+=\zeta_n$ and
$(\varpi_{0,n})_- =\eta$.  Moreover, $0\le\zeta_n\le\zeta$ and
\eqref{eq:zet-lone} imply
\begin{equation*}
 \sup_{n\ge1}\|\varpi_{0,n}\|_{L^1\cap L^p}
 \le
 \|\zeta\|_{L^1\cap L^p}+\|\eta\|_{L^1\cap L^p}<\infty.
\end{equation*}
Consequently, 
\begin{equation}\label{eq:sharp-lp-bd}
 \sup_{n\ge1}
 \min\left\{
 \|(\varpi_{0,n})_+\|_{p^\sharp},
 \|(\varpi_{0,n})_-\|_{p^\sharp}
 \right\}
 \le\|\eta\|_{p^\sharp}<\infty.
\end{equation}
Furthermore, \eqref{eq:mtwo-zet} and
\eqref{eq:zet-alp} imply
\begin{equation}\label{eq:mtwo-zet-bd}
 \sup_{n\ge1}\mathfrak M_2(\varpi_{0,n})
 \le\mathfrak M_2(\zeta)+\mathfrak M_2(\eta)<\infty.
\end{equation}
Let
$u_{0,n}:=\Bcal\varpi_{0,n}, \  u_0^E:=\Bcal\varpi_0^E.$
By Proposition~\ref{prop:radmom}, these fields belong to
$L^2(\R^3)$.  Equations \eqref{eq:bs-xstab},
\eqref{eq:lone-axis}, and
\eqref{eq:mtwo-zet-bd} yield
\begin{equation*}
 \|u_{0,n}-u_0^E\|_2^2
 \le C_p
 \|\varpi_{0,n}-\varpi_0^E\|_p^{A_p}
 \bigl(\mathfrak M_2(\varpi_{0,n})+\mathfrak M_2(\varpi_0^E)\bigr)^{B_p}
 \to0.
\end{equation*}
Apply Proposition~\ref{prop:diagreal} to the sequence
$\{\varpi_{0,n}\}$ and these $L^1\cap L^p$ and $L^2$ convergences,
with the auxiliary profiles chosen as
$\mathfrak{s}_0=0, \ \mathfrak{s}_{0,n}=a_{0,n}=\phi_{0,n}=0.$
The proposition then produces a smooth nonnegative kinetic family satisfying
Theorems~\ref{thm:longtime}--\ref{thm:roughlimit} and \eqref{eq:comp-conv}.  If $1<p<4/3$,
\eqref{eq:sharp-lp-bd}--\eqref{eq:mtwo-zet-bd} and
\eqref{eq:slow-bd}--\eqref{eq:real-bd}
verify \eqref{eq:fld-bd}--\eqref{eq:lp-sharp}; if $4/3\le p<5/3$, no additional
moment hypothesis is required for the energy conclusion of
Theorem~\ref{thm:globenergy}.  Thus in either range the family satisfies the relevant
hypotheses of Theorem~\ref{thm:globenergy}, while \eqref{eq:lone-sharp}--\eqref{eq:sharp-lp} hold.
To verify \eqref{eq:fld-eul}, \eqref{eq:xx-id} and \eqref{eq:sign-psi} give
\[
 \|(\varpi_A^\e(t))_\pm\|_p=\|(\varpi_0^\e)_\pm\|_p,
 \qquad 0\le t\le T.
\]
Since $|a_\pm-b_\pm|\le|a-b|$, this identity together with
\eqref{eq:vel-conv} and \eqref{eq:kin-conv} gives
\begin{align*}
 \sup_{0\le t\le T}
 \left|
 \|(\varpi^E(t))_\pm\|_p-\|(\varpi_0^E)_\pm\|_p
 \right|
 &\le
 \sup_{0\le t\le T}\|\varpi_A^\e(t)-\varpi^E(t)\|_p
 +\|\varpi_0^\e-\varpi_0^E\|_p
 \to0.
\end{align*}
Together with \eqref{eq:sharp-lp}, this convergence proves \eqref{eq:fld-eul}.
\end{proof}

\section{Kinetic passage through Euler singularities}\label{sec:singularpassage}
The statements below concern the projected macroscopic velocity and the
transported relative vorticity.  For each fixed admissible kinetic family, no
subsequence extraction is needed on $[0,T_*]$: $\mathbb P u^\e$ converges
strongly in $C([0,T_*];L^2)$ to the regular Euler velocity extended by its
terminal trace, while $\varpi_A^\e$ converges strongly in
$C([0,T_*];L^1\cap L^p)$ to the corresponding trace of the transported
relative vorticity.  The same trace is reached from both sides in shrinking
kinetic time windows.  In the $p=2$ regime, $u_A^\e$ is uniformly almost
Lipschitz in local $L^2$, and every Euler cluster limit inherits the same
modulus.  The kinetic $L^{3,1}$ control also yields a local Lipschitz estimate
with coefficient controlled by $1+\|\varpi_0^\e\|_{L^{3,1}}$.

\subsection{Continuation of a regular Euler solution}
We begin with the pre-breakdown weak--strong identification.  This is the only
place in the passage argument where weak--strong uniqueness is used.

\begin{corollary}[Kinetic continuation past a regular Euler lifespan]
\label{cor:beyondbreakdown}
Assume the hypotheses of Corollary~\ref{cor:roughreal} for some $p>1$, and
let $u_0$ be the corresponding finite-energy Euler velocity.  Suppose that for some
$0<T_*<\infty$ there is a finite-energy axisymmetric no-swirl Euler solution
\[
    u^{\rm reg}\in C([0,T_*);L^2(\R^3;\R^3))
\]
with initial datum $u_0$ such that, for every $\tau<T_*$,
\begin{equation}\label{eq:prebreak-strain}
    \nabla^{\rm s}u^{\rm reg}
    \in L^1(0,\tau;L^\infty(\R^3)).
\end{equation}
Then, for every $T>T_*$, the kinetic solutions furnished by
Corollary~\ref{cor:roughreal} exist on $[0,T]$ for all sufficiently small
$\e$.  From every sequence $\e_j\downarrow0$ one may extract a subsequence
and an energy-conserving axisymmetric no-swirl weak Euler solution
$u^E\in C([0,T];L^2(\R^3))$ such that
\[
    \mathbb Pu^{\e_j}\to u^E
    \qquad\text{in }C([0,T];L^2(\R^3)),
\]
and
\begin{equation}
    u^E(t)=u^{\rm reg}(t)
    \qquad 0\le t<T_*.
\end{equation}
In particular,
\begin{equation}\label{eq:break-ltwo-limit}
    u^{\rm reg}(t)\longrightarrow u^E(T_*)
    \qquad\text{in }L^2(\R^3)\quad\text{as }t\uparrow T_*.
\end{equation}
Thus $u^E$ gives an energy-conserving Euler continuation of the regular
solution along the selected kinetic subsequence.
\end{corollary}
\begin{proof}[Proof sketch]
Fix $T>T_*$ and a sequence $\e_j\downarrow0$.  Since $T^\e\to\infty$,
the corresponding kinetic solutions are defined on $[0,T]$ for all large
$j$.  Corollary~\ref{cor:roughreal} yields, after extraction, an
energy-conserving renormalized Euler pair $(u^E,\varpi^E)$ with
$\mathbb Pu^{\e_j}\to u^E$ in $C([0,T];L^2)$; by
Lemma~\ref{lem:weak-euler-limit}, $u^E$ is a weak Euler solution in the
standard velocity formulation.  For every $\tau<T_*$, the energy equality
for $u^E$, the common initial datum, and \eqref{eq:prebreak-strain} allow us
to apply Proposition~\ref{prop:weakstrong} on $[0,\tau]$, giving
$u^E=u^{\rm reg}$ there.  Since $\tau<T_*$ is arbitrary,
$u^E=u^{\rm reg}$ on $[0,T_*)$.  Finally,
$u^E\in C([0,T];L^2)$ implies
$u^{\rm reg}(t)=u^E(t)\to u^E(T_*)$ strongly in $L^2$ as
$t\uparrow T_*$, proving \eqref{eq:break-ltwo-limit} and the claimed
continuation.
\end{proof}

\subsection{A common singular trace and quantitative kinetic passage}
The continuation may remain subsequential at fixed positive times beyond
$T_*$, but the projected velocity and transported relative vorticity of one
fixed admissible kinetic family converge without extraction on $[0,T_*]$.
Compactness together with the pre-singular identification also forces every
post-singular cluster path to take the same value at $T_*$.

\begin{corollary}[Global passage, full-family singular trace, and transported invariants]
\label{cor:global-passage-invariants}
Assume the hypotheses of Corollary~\ref{cor:beyondbreakdown} and fix one
kinetic family furnished by Corollary~\ref{cor:roughreal}.  Let $p>1$ be the
relative-vorticity exponent in that corollary.  Then the following hold.

\smallskip
\noindent
\emph{(i) Global subsequential passage.}
For every sequence $\e_j\downarrow0$ there are a subsequence
$\e_{j_k}\downarrow0$ and a single renormalized axisymmetric no-swirl Euler
pair $(u^E,\varpi^E)$ defined for all $t\ge0$ such that, for every finite
$T>0$,
\begin{equation}\label{eq:global-single-subsequence}
\begin{aligned}
 \mathbb P u^{\e_{j_k}}&\longrightarrow u^E
       &&\text{in }C([0,T];L^2(\R^3)),\\
 \varpi_A^{\e_{j_k}}&\longrightarrow\varpi^E
       &&\text{in }C([0,T];L^1(\R^3)\cap L^p(\R^3)).
\end{aligned}
\end{equation}
Moreover,
\begin{equation}
 \|u^E(t)\|_2=\|u_0\|_2\quad(t\ge0),
 \qquad u^E(t)=u^{\rm reg}(t)\quad(0\le t<T_*).
\end{equation}

\smallskip
\noindent
\emph{(ii) Full-family convergence to a common singular trace.}
There is a unique
\[
 \varpi^{\rm reg}\in C([0,T_*);L^1(\R^3)\cap L^p(\R^3))
\]
such that
\[
 \nabla\times u^{\rm reg}(t)=r\varpi^{\rm reg}(t)e_\vartheta
 \qquad\text{in }\mathcal D'(\R^3),\qquad 0\le t<T_*.
\]
There are also unique traces
\[
 u_*\in L^2(\R^3;\R^3),
 \qquad
 \varpi_*\in L^1(\R^3)\cap L^p(\R^3),
\]
such that
\begin{equation}\label{eq:singular-vorticity-trace}
 u^{\rm reg}(t)\to u_*\quad\text{in }L^2,
 \qquad
 \varpi^{\rm reg}(t)\to\varpi_*\quad\text{in }L^1\cap L^p
 \qquad(t\uparrow T_*).
\end{equation}
Moreover, $\nabla\times u_*=r\varpi_*e_\vartheta$ in distributions.  If we set
$u^{\rm reg}(T_*):=u_*$ and $\varpi^{\rm reg}(T_*):=\varpi_*$, then the
projected velocity and transported relative vorticity of the
entire fixed kinetic family satisfy
\begin{equation}\label{eq:full-family-to-trace}
 \lim_{\e\downarrow0}\sup_{0\le t\le T_*}
 \left(
   \|\mathbb P u^\e(t)-u^{\rm reg}(t)\|_2
  +\|\varpi_A^\e(t)-\varpi^{\rm reg}(t)\|_{L^1\cap L^p}
 \right)=0.
\end{equation}
The same trace is reached from both sides in shrinking kinetic time windows:
\begin{equation}\label{eq:two-sided-singular-trace}
 \lim_{\rho\downarrow0}\limsup_{\e\downarrow0}
 \sup_{\substack{0\le t\le T^\e\\ |t-T_*|\le\rho}}
 \left(
   \|\mathbb P u^\e(t)-u_*\|_2
  +\|\varpi_A^\e(t)-\varpi_*\|_{L^1\cap L^p}
 \right)=0.
\end{equation}
Since $T^\e\to\infty$, for each fixed $\rho>0$ the restriction
$t\le T^\e$ is eventually inactive on the whole interval
$[\max\{0,T_*-\rho\},T_*+\rho]$.  Hence it does not affect the
$\limsup_{\e\downarrow0}$ in \eqref{eq:two-sided-singular-trace}.
Consequently every global kinetic Euler cluster limit has the same two-sided
velocity trace $u_*$ and the same trace $\varpi_*$ of the transported relative
vorticity at $T_*$.  This does not imply that two such limits coincide at any fixed time
$t>T_*$, and it does not compare distinct microscopic preparations.

\smallskip
\noindent
\emph{(iii) Transported distribution.}
For every globally Lipschitz function $\beta:\R\to\R$ with $\beta(0)=0$,
\begin{equation}\label{eq:global-kinetic-renormalized-invariant}
 \int_{\R^3}\beta(\varpi^E(t,x))\,dx
   =\int_{\R^3}\beta(\varpi_0(x))\,dx,
 \qquad t\ge0.
\end{equation}
In particular, for $q\in[1,p]$,
\begin{equation}\label{eq:global-signed-vorticity-norm}
 \|(\varpi^E(t))_\pm\|_{L^q}
 =\|(\varpi_0)_\pm\|_{L^q},\qquad t\ge0,
\end{equation}
and the positive and negative superlevel-set distributions are preserved.
\end{corollary}
\begin{proof}
For (i), start with an arbitrary sequence $\e_j\downarrow0$ and apply
Corollary~\ref{cor:roughreal} on the nested intervals
$[0,T_*+m]$, $m=1,2,\ldots$.  A diagonal extraction converges in both
topologies of \eqref{eq:global-single-subsequence} on every finite time
interval; uniqueness of limits on overlaps gives a single global pair
$(u^E,\varpi^E)$.  Corollary~\ref{cor:roughreal} and
Theorem~\ref{thm:globenergy} give the renormalized Euler structure and the
energy identity, while Corollary~\ref{cor:beyondbreakdown} gives
$u^E=u^{\rm reg}$ on $[0,T_*)$.

For (ii), fix $T>T_*$ and consider any cluster pair $(u,\varpi)$ on
$[0,T]$.  Corollary~\ref{cor:beyondbreakdown} gives
$u=u^{\rm reg}$ on $[0,T_*)$.  The relation
$\nabla\times u=r\varpi e_\vartheta$, initially valid for almost every time,
in fact holds for every time: after testing against
$\Phi\in C_c^\infty(\R^3;\R^3)$, both sides are continuous in time by
$u\in C_tL^2$ and $\varpi\in C_tL^1$.  Hence any two cluster vorticities
$\varpi_1,\varpi_2$ satisfy
$r(\varpi_1-\varpi_2)e_\vartheta=0$ on $[0,T_*)$.  Testing on
$\Omega\Subset\{r>0\}$ against
$\psi\,r e_\vartheta/r^2$, $\psi\in C_c^\infty(\Omega)$, gives
$\varpi_1=\varpi_2$ on $\Omega$; since the symmetry axis has zero
three-dimensional measure, the equality holds almost everywhere in $\R^3$.
This defines the unique
$\varpi^{\rm reg}\in C([0,T_*);L^1\cap L^p)$.

Choose one cluster pair and set
$(u_*,\varpi_*):=(u(T_*),\varpi(T_*))$.  Its strong time continuity and the
preceding pre-singular uniqueness give \eqref{eq:singular-vorticity-trace};
any other cluster pair has the same trace by the same argument.  Passing to
the curl identity at $T_*$ yields
$\nabla\times u_*=r\varpi_*e_\vartheta$.
If \eqref{eq:full-family-to-trace} failed, a sequence
$\e_j\downarrow0$ would stay a fixed positive distance from the extended
pair $(u^{\rm reg},\varpi^{\rm reg})$ somewhere on $[0,T_*]$.  Extracting a
cluster pair on $[0,T]$ contradicts the strong uniform-in-time convergence,
because every such cluster pair equals the extended regular pair on
$[0,T_*]$.  This proves \eqref{eq:full-family-to-trace}.

The two-sided statement is similar but uses moving times.  If
\eqref{eq:two-sided-singular-trace} failed, there would be
$\rho_j\downarrow0$, $\e_j\downarrow0$, and times $t_j$ with
$|t_j-T_*|\le\rho_j$ such that the distance of
$(\mathbb Pu^{\e_j}(t_j),\varpi_A^{\e_j}(t_j))$ from
$(u_*,\varpi_*)$ is bounded below by a fixed positive constant.  Using
$T^\e\to\infty$, choose the sequence so that all large $j$ are defined on
$[0,T_*+1]$, and extract a cluster pair there.  Then $t_j\to T_*$, while
strong convergence on the fixed interval and continuity of the cluster pair
at $T_*$ force that distance to tend to zero, a contradiction.  Thus
\eqref{eq:two-sided-singular-trace} holds.

For (iii), Lemma~\ref{lem:splitconv} gives, for each $\e$, a
measure-preserving flow and therefore
\[
 \int_{\R^3}\beta(\varpi_A^\e(t))
 =\int_{\R^3}\beta(\varpi_0^\e)
\]
for every globally Lipschitz $\beta$ with $\beta(0)=0$.  The strong
$C_tL^1$ convergence in \eqref{eq:global-single-subsequence}, together with
$\varpi_0^\e\to\varpi_0$ in $L^1$, permits passage to the limit on every
finite interval and proves
\eqref{eq:global-kinetic-renormalized-invariant}.  Exact transport also
preserves the signed $L^q$ norms.  Since strong convergence in
$L^1\cap L^p$ implies strong convergence in $L^q$ for $1\le q\le p$ by
interpolation, passing to the limit gives
\eqref{eq:global-signed-vorticity-norm}.  Finally, applying
\eqref{eq:global-kinetic-renormalized-invariant} to
$\beta_n(z)=\min\{1,n(z-a)_+\}$ and to
$\min\{1,n(-z-a)_+\}$, then letting $n\to\infty$, gives preservation of the
positive and negative superlevel-set distributions.
\end{proof}

\begin{proposition}[Quantitative temporal modulus and singular trace]
\label{prop:quantitative-singular-passage}
Assume the hypotheses of Corollary~\ref{cor:global-passage-invariants} and,
in addition, the $p=2$ prescribed-data case.  Fix $R,T>0$.
For every $s,t\in[0,T]$ with $s\ne t$ and all sufficiently small coupled
parameters (so in particular $T\le T^\e$),
\begin{equation}\label{eq:kinetic-temporal-envelope}
 \|u_A^\e(t)-u_A^\e(s)\|_{L^2(B_R)}
 \le C_{R,T}|t-s|
 \min\left\{
   \sqrt{1+\log^+\!\frac1{|t-s|}},\,
   1+\|\varpi_0^\e\|_{L^{3,1}}
 \right\}.
\end{equation}
For every sufficiently small $\delta>0$,
\begin{align}\label{eq:projected-temporal-envelope}
 &\|\mathbb P u^\e(t)-\mathbb P u^\e(s)\|_{L^2(B_R)}\notag\\
 &\qquad\le C_{R,T}|t-s|
 \min\left\{
   \sqrt{1+\log^+\!\frac1{|t-s|}},\,
   1+\|\varpi_0^\e\|_{L^{3,1}}
 \right\}
 +C_{T,\delta}\bigl(\e^{1/2-\delta}+\kappa^{1/2-\delta}\bigr).
\end{align}
Consequently every global kinetic Euler cluster limit satisfies, for
$s\ne t$,
\begin{equation}\label{eq:limit-almost-lipschitz}
 \|u^E(t)-u^E(s)\|_{L^2(B_R)}
 \le C_{R,T}|t-s|
 \sqrt{1+\log^+\!\frac1{|t-s|}}.
\end{equation}
If $T\ge T_*$, then for the common trace $u_*$ of
Corollary~\ref{cor:global-passage-invariants},
\begin{equation}\label{eq:two-sided-almost-lipschitz-trace}
 \|u^E(t)-u_*\|_{L^2(B_R)}
 \le C_{R,T}|t-T_*|
 \sqrt{1+\log^+\!\frac1{|t-T_*|}},
 \qquad t\ne T_*,\quad t,T_*\in[0,T].
\end{equation}
Thus every post-singular cluster continuation satisfies the same local
almost-Lipschitz estimate relative to the common velocity trace.

The initial bound \eqref{eq:init-bd} gives
$\|\varpi_0^\e\|_{L^{3,1}}\le\sqrt{\Lambda_5^\e}$, and hence
\begin{equation}\label{eq:iterated-log-temporal-envelope}
 \|u_A^\e(t)-u_A^\e(s)\|_{L^2(B_R)}
 \le C_{R,T}|t-s|
 \min\left\{
   \sqrt{1+\log^+\!\frac1{|t-s|}},\,
   1+\sqrt{\Lambda_5^\e}
 \right\}.
\end{equation}
If
$\sup_\e\|\varpi_0^\e\|_{L^{3,1}}<\infty$, then $u_A^\e$ is uniformly
locally Lipschitz in time with values in $L^2_{\rm loc}$, and every kinetic
Euler cluster limit inherits the same local Lipschitz bound.
\end{proposition}
\begin{proof}
The $p=2$ data convergence, Lemma~\ref{lem:splitconv}, and
Lemma~\ref{lem:approxeul} give
\begin{equation*}
 \sup_{0\le \tau\le T}
 \bigl(\|\varpi_A^\e(\tau)\|_2+\|u_A^\e(\tau)\|_2\bigr)\le C_T.
\end{equation*}
We first derive the estimate that is uniform in the kinetic parameters. Choose a smooth radial cutoff \(\chi_R\), equal to $1$ on \(B_{12R}\) and supported in \(B_{16R}\), and set
\[
w^\varepsilon
:=
\nabla\times(-\Delta)^{-1}
\bigl(\chi_R r\varpi_A^\varepsilon e_\vartheta\bigr),
\qquad
h^\varepsilon:=u_A^\varepsilon-w^\varepsilon.
\]
Weighted Calder\'on--Zygmund estimate with $r^{-1}\in A_2(\R^3)$ in \eqref{As} gives
\[
\begin{aligned}
\int_{\mathbb R^3}\frac{|\nabla w^\varepsilon|^2}{r}\,dx
&\lesssim
\int_{\mathbb R^3}
\frac{|\chi_R r\varpi_A^\varepsilon e_\vartheta|^2}{r}\,dx =
\int_{\mathbb R^3}
\chi_R^2r|\varpi_A^\varepsilon|^2\,dx
\le C_R\|\varpi_A^\varepsilon\|_2^2.
\end{aligned}
\] 
Moreover, the Hardy–Littlewood–Sobolev inequality yields
\[
\|w^\varepsilon\|_2
\lesssim
\|\chi_R r\varpi_A^\varepsilon\|_{6/5}
\le C_R\|\varpi_A^\varepsilon\|_2.
\]
$-\Delta u_A^\varepsilon
=\nabla\times(r\varpi_A^\varepsilon e_\vartheta)$ implies that $h^\e$ is harmonic in \(B_{12R}\). The interior harmonic estimates therefore give
\[
\begin{aligned}
\|h^\varepsilon\|_{L^\infty(B_{8R})}
+\|\nabla h^\varepsilon\|_{L^\infty(B_{8R})}
&\le C_R\|h^\varepsilon\|_{L^2(B_{12R})} \le C_R\left(
\|u_A^\varepsilon\|_2+\|\varpi_A^\varepsilon\|_2
\right)
\le C_{R,T}.
\end{aligned}
\]
Combining these estimates we obtain
\begin{equation}\label{eq:weighted-meridional-gradient}
 \sup_{0\le\tau\le T}
 \int_{B_{8R}}\frac{|\nabla u_A^\e(\tau,x)|^2}{r}\,dx
 \le C_{R,T}.
\end{equation}

In cylindrical coordinates the Jacobian cancels the weight in
\eqref{eq:weighted-meridional-gradient}. The global \(L^2_x\) bound and an anchored Poincaré inequality consequently give, 
\[
\sup_{0\le\tau\le T}
\left\|
\bigl(u_A^{\varepsilon,r}(\tau),u_A^{\varepsilon,z}(\tau)\bigr)
\right\|_{H^1(
(0,4R)\times(-4R,4R)
)}
\le C_{R,T}.
\]
Applying the two-dimensional Moser–Trudinger inequality on $(0,4R)\times(-4R,4R)$ yields
\begin{equation}\label{eq:local-q-growth}
 \sup_{0\le\tau\le T}
 \left(
   \|\nabla u_A^\e(\tau)\|_{L^2(B_{3R})}
   +q^{-1/2}\|u_A^\e(\tau)\|_{L^q(B_{3R})}
 \right)
 \le C_{R,T},
 \qquad q\ge4.
\end{equation}

We next estimate the nonlinear term. Let \(\eta_R\in C_c^\infty(B_{3R})\) equal one on \(B_{2R}\). Since \(\nabla\cdot u_A^\e=0\),
\[
\mathbb{P}\operatorname{div}(u_A^\e\otimes u_A^\e)
={} \mathbb{P}\left[
\eta_R(u_A^\e\cdot\nabla)u_A^\e+(u_A^\e\cdot\nabla\eta_R)u_A^\e
\right] + \mathbb{P}\operatorname{div}\left[(1-\eta_R)u_A^\e\otimes u_A^\e\right].
\]
Hölder’s inequality and \eqref{eq:local-q-growth} give
\[
\left\|
\eta_R(u_A^\e\cdot\nabla)u_A^\e+(u_A^\e\cdot\nabla\eta_R)u_A^\e
\right\|_{L^{2q/(q+2)}(\mathbb R^3)} \le
C_R\|u_A^\e\|_{L^q(B_{3R})}
\left(
\|\nabla u_A^\e\|_{L^2(B_{3R})}+\|u_A^\e\|_2
\right)
\le C_{R,T}\sqrt q.
\]
For the remaining term, the source is separated from \(B_R\), and the off-diagonal kernel of \(P\operatorname{div}\) has order \(|x-y|^{-4}\). Consequently,
\[
\left\|
\mathbb{P}\operatorname{div}\left[(1-\eta_R)u_A^\e\otimes u_A^\e\right]
\right\|_{L^\infty(B_R)}
\lesssim
\sup_{x\in B_R}
\int_{\mathbb R^3\setminus B_{2R}}
\frac{|u_A^\e(y)|^2}{|x-y|^4}\,dy
\le C_R\|u_A^\e\|_2^2.
\]
Hence, uniformly for $q\ge4$,
\begin{equation}\label{eq:local-projected-nonlinearity}
 \|\mathbb P\operatorname{div}(u_A^\e\otimes u_A^\e)
   \|_{L^{2q/(q+2)}(B_R)}
 \le C_{R,T}\sqrt q.
\end{equation}
Moreover, \eqref{eq:ta-id}, \eqref{eq:del-kin}, \eqref{eq:curl-del},
and \eqref{eq:ac-bd}, together with Lemma~\ref{lem:scalecons}, give for every
fixed sufficiently small $\delta>0$
\begin{equation}\label{eq:kinetic-forcing-small}
 \|f_A^\e\|_{L^\infty(0,T;L^2)}
 \le C_{T,\delta}
 \bigl(\e^{1/2-\delta}+\kappa^{1/2-\delta}\bigr)
 \bigl(1+|\log\e|^{15/32}\bigr)=o(1).
\end{equation}
Since $L^2(B_R)\hookrightarrow L^{2q/(q+2)}(B_R)$ uniformly for $q\ge4$,
the approximate Euler equation and
\eqref{eq:local-projected-nonlinearity}--\eqref{eq:kinetic-forcing-small} imply
\[
 \|\partial_tu_A^\e\|_{L^\infty(0,T;L^{2q/(q+2)}(B_R))}
 \le C_{R,T}\sqrt q.
\]
Together with \eqref{eq:local-q-growth}, the fundamental theorem of
calculus and interpolation give
\begin{equation}\label{eq:time-interpolation-q}
 \|u_A^\e(t)-u_A^\e(s)\|_{L^2(B_R)}
 \le C_{R,T}\sqrt q\,|t-s|^{1-2/q},
 \qquad q\ge4.
\end{equation}
For $0<|t-s|\le e^{-2}$ choose
$q=2\log(1/|t-s|)$; for larger increments use the uniform $L^2$ bound.
This proves the first, parameter-uniform estimate in
\eqref{eq:kinetic-temporal-envelope}.

For the second estimate, measure preservation in Lemma~\ref{lem:splitconv}
gives
\[
 \|\varpi_A^\e(t)\|_{L^{3,1}}
 =\|\varpi_0^\e\|_{L^{3,1}}.
\]
The endpoint Riesz-potential estimate $I_1:L^{3,1}(\R^3)\to L^\infty(\R^3)$, applied to the
localized Biot--Savart field yields
\[
\|w^\varepsilon(t)\|_\infty
\lesssim
\|\chi_R r\varpi_A^\varepsilon(t)\|_{L^{3,1}}
\le C_R\|\varpi_0^\varepsilon\|_{L^{3,1}}.
\]
Together with the interior harmonic bound for \(h^\varepsilon\), this gives
\begin{equation}\label{eq:lorentz-local-velocity}
 \|u_A^\e(t)\|_{L^\infty(B_{3R})}
 \le C_{R,T}\bigl(1+\|\varpi_0^\e\|_{L^{3,1}}\bigr).
\end{equation}
Using the same cutoff decomposition as above, now in $L^2(B_R)$, and
\eqref{eq:local-q-growth}, we obtain
\[
 \|\mathbb P\operatorname{div}(u_A^\e\otimes u_A^\e)\|_{L^2(B_R)}
 \le C_{R,T}\bigl(1+\|\varpi_0^\e\|_{L^{3,1}}\bigr).
\]
The approximate Euler equation and \eqref{eq:kinetic-forcing-small} therefore
imply
\[
 \|\partial_tu_A^\e\|_{L^\infty(0,T;L^2(B_R))}
 \le C_{R,T}\bigl(1+\|\varpi_0^\e\|_{L^{3,1}}\bigr),
\]
and hence
\[
 \|u_A^\e(t)-u_A^\e(s)\|_{L^2(B_R)}
 \le C_{R,T}\bigl(1+\|\varpi_0^\e\|_{L^{3,1}}\bigr)|t-s|.
\]
Taking the better of this estimate and the bound obtained from
\eqref{eq:time-interpolation-q} proves
\eqref{eq:kinetic-temporal-envelope}.

Finally, Lemma~\ref{lem:splitconv} gives
$\sup_{\tau\le T}\|u_B^\e(\tau)\|_2
 \lesssim_{T,\delta}\e^{1/2-\delta}+\kappa^{1/2-\delta}$.
Since $\mathbb Pu^\e=u_A^\e+u_B^\e$, this yields
\eqref{eq:projected-temporal-envelope}.  Strong
$C([0,T];L^2)$ convergence along every Euler cluster subsequence gives
\eqref{eq:limit-almost-lipschitz}, and evaluation at the common trace from
Corollary~\ref{cor:global-passage-invariants} gives
\eqref{eq:two-sided-almost-lipschitz-trace}.  The estimate
\eqref{eq:iterated-log-temporal-envelope} follows from
$\|\varpi_0^\e\|_{L^{3,1}}\le\sqrt{\Lambda_5^\e}$.  If the initial
$L^{3,1}$ norms are uniformly bounded, then
$u_A^\e=\mathbb Pu^\e-u_B^\e\to u^E$ in $C([0,T];L^2)$ along each cluster
subsequence, so the local Lipschitz estimate passes to the Euler limit.
\end{proof}

\subsection{Three finite-time singularity mechanisms}
The preceding structure applies to three qualitatively different no-swirl
axisymmetric singularity constructions: the locally self-similar construction
in \cite{ElgindiGhoulMasmoudi2021}, the non-self-similar construction in
\cite{CordobaMartinezZoroaZheng2025}, and the asymptotically self-similar
construction in \cite{Chen2026EulerII}.  The cited constructions supply finite
energy and regularity on every compact
pre-singular time interval.  To apply the $p=2$ kinetic realization, it remains
to verify that the initial relative vorticity belongs to $L^1\cap L^2$.

\begin{lemma}[H\"older vorticity and relative vorticity]
\label{lem:holder-vorticity-relative-lp}
Let $\omega_0\in C_c^\alpha(\R^3;\R^3)$ be axisymmetric and purely
azimuthal, with $0<\alpha<1$.  For $r>0$ write
$\omega_0=\omega_0^\vartheta(r,z)e_\vartheta$ and define
$\varpi_0=\omega_0^\vartheta/r$ there; extend $\varpi_0$ by zero on the
symmetry axis, which does not affect its Lebesgue class.  Then
$\varpi_0\in L^1\cap L^q$ for every $1\le q<2/(1-\alpha)$.
\end{lemma}
\begin{proof}
Axisymmetry and pure azimuthality imply $\omega_0=0$ on the symmetry axis:
rotation invariance makes the axial value parallel to $e_z$, while continuity
from $r>0$ makes its $e_z$ component zero.  Hence H\"older continuity gives
$|\omega_0^\vartheta(r,z)|\le C r^\alpha$ and therefore
$|\varpi_0(r,z)|\le C r^{\alpha-1}$.  Since $\omega_0$ is compactly
supported, cylindrical coordinates reduce the only possible singularity to
$\int_0^1 r^{1-q(1-\alpha)}\,dr$, which is finite exactly when
$q<2/(1-\alpha)$.  This proves the claim.
\end{proof}
\iffalse
\begin{proof}
Fix $z\in\mathbb R$.  Rotation invariance implies that the Cartesian
vector $\omega_0(0,0,z)$ is fixed by every rotation about the $z$-axis and
hence is parallel to $e_z$.  On the other hand, for $r>0$ pure azimuthality
gives $\omega_0\cdot e_z=0$, so continuity implies that the $e_z$ component
also vanishes on the axis.  Therefore $\omega_0(0,0,z)=0$.  For $r>0$ and
any $\vartheta$, set
$x=(r\cos\vartheta,r\sin\vartheta,z)$ and $x_0=(0,0,z)$.  Since
$|\omega_0^\vartheta(r,z)|=|\omega_0(x)|$ and
$\omega_0\in C^\alpha$, the H\"older seminorm gives
\begin{equation}
 |\omega_0^\vartheta(r,z)|
 =|\omega_0(x)-\omega_0(x_0)|
 \le [\omega_0]_{C^\alpha}|x-x_0|^\alpha
 =[\omega_0]_{C^\alpha}r^\alpha.
\end{equation}
This estimate is uniform in $z$ and $\vartheta$.  Therefore
\[
 |\varpi_0(r,z)|
 =\frac{|\omega_0^\vartheta(r,z)|}{r}
 \le C r^{\alpha-1}.
\]
Choose $R>0$ so that the support is contained in
$\{0\le r\le R,\ |z|\le R\}$.  In cylindrical coordinates,
$dx=r\,dr\,d\vartheta\,dz$, and hence for $1\le q<\infty$,
\begin{align*}
 \|\varpi_0\|_q^q
 &=2\pi\int_{-R}^R\int_0^R
   |\varpi_0(r,z)|^q r\,dr\,dz\\
 &\le C_R\int_0^R r^{q(\alpha-1)}r\,dr
 =C_R\int_0^R r^{1-q(1-\alpha)}\,dr.
\end{align*}
The last integral is finite precisely when
$1-q(1-\alpha)>-1$, equivalently
$q(1-\alpha)<2$, or
$q<2/(1-\alpha)$.  This proves the asserted range.  In particular,
$q=1$ is allowed for every $\alpha>0$, and $q=2$ is allowed for every
$\alpha>0$ because $2<2/(1-\alpha)$.
\end{proof}
\fi

\begin{corollary}[Kinetic passage through three axisymmetric Euler singularity mechanisms]
\label{cor:three-singularity-passage}
Let $u^{\rm reg}$ be any one of the following finite-energy axisymmetric
no-swirl Euler solutions, with regular lifespan $[0,T_*)$:
\begin{itemize}
\item[\rm (i)] the invariant no-swirl subfamily of the construction in
\cite{ElgindiGhoulMasmoudi2021}, whose initial vorticity is compactly supported
and belongs to $C^\alpha(\R^3)$ for some $\alpha>0$;
\item[\rm (ii)] after translating the pre-singular interval of the
construction in \cite{CordobaMartinezZoroaZheng2025} to $[0,T_*)$, the
axisymmetric no-swirl finite-energy solution with compactly supported initial
vorticity in $C^\alpha(\R^3)$ for some $\alpha>0$;
\item[\rm (iii)] for some $0<\alpha<1/3$, an asymptotically self-similar
axisymmetric no-swirl solution from \cite{Chen2026EulerII}, with
$u_0\in C^{1,\alpha}(\R^3)\cap L^2(\R^3)$ and
$\omega_0=\nabla\times u_0\in C_c^\alpha(\R^3)$.
\end{itemize}
Then in each case
\[
 u_0=u^{\rm reg}(0)\in L^2(\R^3),\qquad
 \varpi_0=\frac{(\nabla\times u_0)^\vartheta}{r}
 \in L^1(\R^3)\cap L^2(\R^3),
\]
where the quotient is understood almost everywhere off the symmetry axis and
extended arbitrarily on the axis.  Consequently, for either collision model,
the $p=2$ kinetic realization furnished by
Corollary~\ref{cor:roughreal} satisfies all conclusions of
Corollaries~\ref{cor:beyondbreakdown} and
\ref{cor:global-passage-invariants} and
Proposition~\ref{prop:quantitative-singular-passage}.  In particular, for each
fixed admissible kinetic preparation,
$\mathbb P u^\e\to u^{\rm reg}$ and
$\varpi_A^\e\to\varpi^{\rm reg}$ without subsequence extraction on
$[0,T_*)$, both converge to the common trace pair $(u_*,\varpi_*)$ at
$T_*$, and the pair $(\mathbb P u^\e,\varpi_A^\e)$ approaches
$(u_*,\varpi_*)$ in the two-sided shrinking-window sense of
\eqref{eq:two-sided-singular-trace}.  The velocity fields $u_A^\e$ and
$\mathbb P u^\e$ satisfy the quantitative temporal estimates of
Proposition~\ref{prop:quantitative-singular-passage}, with the split error in
the estimate for $\mathbb P u^\e$.  Every global Euler cluster continuation
conserves energy and the distribution of the transported relative vorticity.
\end{corollary}
\begin{proof}[Proof sketch]
The required source properties are contained in the cited constructions.  For
\cite{ElgindiGhoulMasmoudi2021}, Theorem~1 and Section~2.6 provide
finite-energy data with compactly supported $C^\alpha$ vorticity, and the
homogeneous swirl equation (2.17) leaves the zero-swirl subfamily invariant.
For \cite{CordobaMartinezZoroaZheng2025}, Theorem~2 together with
Remarks~3 and~5 gives the finite-energy axisymmetric no-swirl solution with
compactly supported vorticity; \cite[Theorem~1.2]{Chen2026EulerII} gives the
corresponding properties in the third case.  Each solution is classical on
compact subintervals of $[0,T_*)$, so the strain hypothesis of
Corollary~\ref{cor:beyondbreakdown} holds there.  Lemma~\ref{lem:holder-vorticity-relative-lp}
with $q=1,2$ yields $\varpi_0\in L^1\cap L^2$ in all three cases.  Thus the
$p=2$ hypotheses of Corollary~\ref{cor:roughreal} are satisfied (with, for
example, $\mathfrak s_0=a_0=q_0=0$), and no sign or spatial-moment condition
is needed because $p=2\ge4/3$.  Applying
Corollaries~\ref{cor:beyondbreakdown} and
\ref{cor:global-passage-invariants}, followed by
Proposition~\ref{prop:quantitative-singular-passage}, gives all stated
conclusions.
\end{proof}

\appendix

\section{Auxiliary analytic arguments}\label{sec:auxana}
\begin{proof}[Proof of Lemma~\ref{lem:divcurl}]
For smooth compactly supported $\Omega$,
\[
 \widehat{\nabla\times(-\Delta)^{-1}\Omega}(k)
 =\frac{ik\times\widehat\Omega(k)}{|k|^2}.
\]
Equation \eqref{eq:curl-hm} follows from
$|k\times\widehat\Omega|\le|k||\widehat\Omega|$ and density.  If
$\Omega=\nabla\times u$ and $\nabla\cdot u=0$, then
$\widehat\Omega=ik\times\widehat u$ and
$ik\times\widehat\Omega=|k|^2\widehat u$ for almost every $k\ne0$,
which proves \eqref{eq:curl-isom}.  Finally, for
$\Phi\in C_c^\infty$,
\[
 |\langle G,\Phi\rangle|
 \le\|G\|_{6/5}\|\Phi\|_6
 \lesssim\|G\|_{6/5}\|\nabla\Phi\|_2,
\]
which gives \eqref{eq:hminus-lp} by duality.
\end{proof}

\begin{proof}[Proof of Lemma \ref{lem:cylcurl}]
Set $\omega=r\eta e_\vartheta$ and decompose it as
$\omega_{\rm in}:=\mathbf1_{\{r<2R\}}\omega$ and
$\omega_{\rm out}:=\mathbf1_{\{r>2R\}}\omega$.  Both fields are divergence free in $\mathcal D'(\R^3)$.  To see this
for the sharp cutoff, test against $\phi\in C_c^\infty(\R^3)$ and use
cylindrical coordinates:
\[
 \langle\nabla\cdot\omega_{\rm in},\phi\rangle
 =-\int_{r<2R}r\eta(r,z)e_\vartheta\cdot\nabla\phi\,dx
 =-\int_{r<2R}\eta(r,z)\,\partial_\vartheta\phi\,dx=0.
\]
Equivalently, the jump interface has normal $e_r$, while
$\omega\cdot e_r=0$.  Define
$v_{\rm in}:=\nabla\times(-\Delta)^{-1}\omega_{\rm in}, \quad v_{\rm out}:=v-v_{\rm in}.$
Since $\omega_{\rm in}\in L^{6/5}$ is divergence free, the Fourier
identity
$\nabla\times\nabla\times(-\Delta)^{-1}\omega_{\rm in}
=\omega_{\rm in}$ gives
$\nabla\cdot v_{\rm in}=0$ and
$\nabla\times v_{\rm in}=\omega_{\rm in}$.
Since $\|\omega_{\rm in}\|_{6/5}\le2R\|\eta\|_{6/5}$, the
three-dimensional Hardy--Littlewood--Sobolev estimate gives
$\|v_{\rm in}\|_2\lesssim R\|\eta\|_{6/5}$.  The field $v_{\rm out}\in L^2$ is divergence free and has curl
$\omega_{\rm out}$.  The Fourier div--curl identity therefore identifies it
with the Biot--Savart field of $\omega_{\rm out}$.
Fix the target angle.  If $x=(r,0,z)$ with $0<r<R$ and
$y=(\rho\cos\varphi,\rho\sin\varphi,\zeta)$ with $\rho>2R$, then
$|x-y|^2\gtrsim\rho^2+(z-\zeta)^2$.  The three-dimensional Biot--Savart
formula and $dy=\rho\,d\rho\,d\varphi\,d\zeta$ give
\begin{equation*}
 |v_{\rm out}(r,z)|
 \lesssim\int_{2R}^\infty
 (K_\rho*_z|\eta(\rho,\cdot)|)(z)\,d\rho,
 \qquad K_\rho(s):=\frac{\rho^2}{\rho^2+s^2}.
\end{equation*}
Choose $a>1$ by $1+\frac12=\frac1a+\frac1q$.
Since $\|K_\rho\|_{L^a(\R)}\lesssim_a\rho^{1/a}$, Young's inequality
in $z$ gives
\begin{equation*}
 \|v_{\rm out}(r,\cdot)\|_{L^2_z}
 \lesssim_q\int_{2R}^\infty
 \rho^{1/a}\|\eta(\rho,\cdot)\|_{L^q_z}\,d\rho.
\end{equation*}
Write the integrand as
$\rho^{1/q}\|\eta(\rho,\cdot)\|_{L^q_z}\rho^{1/a-1/q}$.
H\"older's inequality in $\rho$ and the cylindrical $L^q$ identity give
\begin{align}
 \|v_{\rm out}(r,\cdot)\|_{L^2_z}
 &\lesssim_q
 \|\eta\|_q
 \left(\int_{2R}^\infty
 \rho^{(1/a-1/q)q'}\,d\rho\right)^{1/q'} \lesssim_q R^{5/2-3/q}\|\eta\|_q, \notag
\end{align}
where the last integral converges precisely in the present range
$q<6/5$.
Integrating against $2\pi r\,dr$ on $0<r<R$ gives
\begin{equation*}
 \|\mathbf1_{\mathfrak C_R}v_{\rm out}\|_2
 \lesssim_q R^{7/2-3/q}\|\eta\|_q.
\end{equation*}
Together with the estimate for $v_{\rm in}$, this proves
\eqref{eq:eta-qr}.
\end{proof}

\begin{proof}[Proof of Lemma~\ref{lem:axiskern}]
For $0<s\le1$, the inequality
$1-\cos\alpha\gtrsim\alpha^2$ on $[0,\pi]$ gives
\[
 |\mathcal F(s)|
 \lesssim \int_0^\pi(s+\alpha^2)^{-1/2}\,d\alpha
 \lesssim 1+|\log s|
 \lesssim_\sigma s^{-\sigma}.
\]
For $s\ge1$, using $\int_0^\pi\cos\alpha\,d\alpha=0$ and the
mean-value theorem gives
\[
 \mathcal F(s)
 =\int_0^\pi\cos\alpha
 \left[(s+2(1-\cos\alpha))^{-1/2}-s^{-1/2}\right]d\alpha,
 \qquad
 |\mathcal F(s)|\lesssim s^{-3/2}.
\]
Since $\sigma<3/2$, these estimates prove \eqref{eq:kern-bound}.

Differentiating \eqref{eq:kern-def},
\[
 \mathcal F'(s)
 =-\frac12\int_0^\pi
 \frac{\cos\alpha}{(s+2(1-\cos\alpha))^{3/2}}\,d\alpha.
\]
For $0<s\le1$, the same lower bound on $1-\cos\alpha$ and the change
$\alpha=s^{1/2}y$ give
\[
 |\mathcal F'(s)|
 \lesssim\int_0^\pi(s+\alpha^2)^{-3/2}\,d\alpha
 \lesssim s^{-1},
\]
whereas for $s\ge1$ the differentiated formula gives
$|\mathcal F'(s)|\lesssim s^{-3/2}$.  Hence
\eqref{eq:kern-der} follows for $0\le\sigma'\le1/2$.
\end{proof}

\begin{proof}[Proof of Lemma \ref{lem:bsformula}]
Since $re_\vartheta=(-x_2,x_1,0)$, axisymmetry gives
$\nabla\cdot(r\varpi e_\vartheta)=0$.  Moreover,
\eqref{eq:vortlift} and axisymmetry imply
$\int_{\R^3}\omega\,dx=0$.  By \eqref{eq:bs-def},
\begin{equation}\label{eq:bs-kernel}
 u(x)=\frac1{4\pi}\int_{\R^3}
 \omega(y)\times\frac{x-y}{|x-y|^3}\,dy.
\end{equation}
If $Q$ is a rotation about the $z$-axis, then
$\omega(Qx)=Q\omega(x)$ and hence $u(Qx)=Qu(x)$.  If
$S=\operatorname{diag}(1,-1,1)$, then
$\omega(Sx)=-S\omega(x)$ and
$(Sa)\times(Sb)=-S(a\times b)$, so \eqref{eq:bs-kernel} gives
$u(Sx)=Su(x)$.  Thus $u_2=0$ on $\{x_2=0\}$, and rotational symmetry
shows that $u$ has no swirl.

If $\operatorname{supp}\omega\subset B_{R_0}$, the zero mean of $\omega$
and \eqref{eq:bs-kernel} give
\[
 |u(x)|\le C|x|^{-3}\int_{\R^3}|y|\,|\omega(y)|\,dy,
 \qquad |x|\ge4R_0.
\]
Since the Biot--Savart kernel is locally integrable and
$\omega\in C_c^\infty$, this proves $u\in L^2$.
Because $\nabla\cdot\omega=0$, Plancherel's theorem yields
\begin{equation}\label{eq:energy-id}
 \|u\|_2^2
 =\langle\omega,(-\Delta)^{-1}\omega\rangle
 =\frac1{4\pi}\iint_{\R^3\times\R^3}
 \frac{\omega(x)\cdot\omega(y)}{|x-y|}\,dx\,dy.
\end{equation}
The double integral is absolutely convergent.  Write
\[
 x=(r\cos\theta,r\sin\theta,z),\qquad
 y=(\rho\cos\varphi,\rho\sin\varphi,\zeta),\qquad
 \alpha=\theta-\varphi.
\]
Then $e_\vartheta(\theta)\cdot e_\vartheta(\varphi)=\cos\alpha$ and
\[
 |x-y|^2=(r-\rho)^2+(z-\zeta)^2+2r\rho(1-\cos\alpha).
\]
For every even $2\pi$-periodic $h$,
\[
 \int_0^{2\pi}\!\int_0^{2\pi}h(\theta-\varphi)
 \,d\theta\,d\varphi
 =4\pi\int_0^\pi h(\alpha)\,d\alpha.
\]
Substitution of \eqref{eq:vortlift} and the cylindrical volume elements
into \eqref{eq:energy-id} now gives \eqref{eq:bs-energy}.
\end{proof}

\begin{proof}[Proof of Lemma~\ref{lem:raddense}]
Let $\chi_R\in C_c^\infty(\R^3)$ be spherically radial, with
$0\le\chi_R\le1$, $\chi_R=1$ on $B_R$, and $\chi_R=0$ on $B_{2R}^c$.
For the $\mathcal X_p$ approximation set $m=2$ and
$\mathfrak M_m:=\mathfrak M_2$; for the $\mathcal Y_p$ approximation
set $m=1$ and $\mathfrak M_m:=\mathfrak M_1$.  Dominated convergence gives
\[
 \|(1-\chi_R)\varpi\|_p\to0,
 \qquad
 \mathfrak M_m((1-\chi_R)\varpi)\to0,
\]
and, if $\varpi\in L^1$, also
$\|(1-\chi_R)\varpi\|_1\to0$.

For fixed $R$, let $\varphi_\nu$ be a spherically radial standard mollifier
supported in $B_\nu$ and set
$g_{R,\nu}:=\varphi_\nu*(\chi_R\varpi)-\chi_R\varpi$.  Then
$g_{R,\nu}$ is axisymmetric, is supported in $B_{2R+1}$ for
$0<\nu\le1$, and satisfies $\|g_{R,\nu}\|_p\to0$ as $\nu\downarrow0$.
Since $p>1$ and the support is fixed,
\[
 \mathfrak M_m(g_{R,\nu})
 \lesssim_R\|g_{R,\nu}\|_1
 \lesssim_R\|g_{R,\nu}\|_p\to0.
\]
If $\varpi\in L^1$, standard $L^1$ mollification also gives
$\|g_{R,\nu}\|_1\to0$.
Choose $R_n\uparrow\infty$ and then $\nu_n\downarrow0$ diagonally and set
$\varpi_n:=\varphi_{\nu_n}*(\chi_{R_n}\varpi)$.  These bounds
then give \eqref{eq:mtwo-axis}--\eqref{eq:lone-mone}, with $m=2$ for
$\mathcal X_p$ and $m=1$ for $\mathcal Y_p$.
\end{proof}

\begin{proof}[Proof of Lemma~\ref{lem:scalecons}]
The first relation in \eqref{eq:choice-bd} gives
$\kappa^c\asymp\e^{1-c_0}$, which is \eqref{eq:kap-choice}, while the
second gives $\e\kappa^{-1/2}\lesssim\e^{c_0}$.  Since $1/2<c\le1$ and
$c_0$ is fixed sufficiently small, the same relation gives
$\kappa^2\lesssim\e$.  Finally, \eqref{eq:boot-bd} yields
$\kappa^{2c-1}\lesssim\e^{6c_0}$; together with
$\e^2/\kappa=(\e/\kappa^c)^2\kappa^{2c-1}$, this gives $\e^2/\kappa\lesssim\e^{8c_0}$.
\end{proof}

\begin{proof}[Proof of Lemma \ref{lem:macrocoef}]
Since $N\ge4$, $H^N(\mathbb R^3)$ is an algebra and
$H^2(\mathbb R^3)\hookrightarrow L^\infty$.  The bootstrap bound
\eqref{eq:tot-bd}, together with $\mathcal E_M^N\le \mathcal E_{tot}^{N}$,
gives
\[
 \e\| (\rho^\e,u^\e,\theta^\e)\|_{H^N_x}
 \lesssim \e(\mathcal E_M^N)^{1/2}
 \lesssim \e^{1-c_0/2},
\]
which proves \eqref{eq:loc-bd}.  For a smooth scalar function $\Phi$ on a fixed neighborhood of the origin,
the chain rule and the $H^N$ product estimate give the standard Moser bound
\begin{equation}
 \|\nabla_x\Phi(h)\|_{H^{N-1}_x}
 \le C_\Phi\|\nabla_x h\|_{H^{N-1}_x}
 \bigl(1+\|h\|_{H^N_x}^{N-1}\bigr),
 \qquad \|h\|_{L^\infty_x}\le c_*.
\end{equation}
Indeed, each derivative $\partial^\alpha\Phi(h)$, $1\le|\alpha|\le N$,
is a finite sum of terms
$\Phi^{(m)}(h)\prod_{j=1}^m\partial^{\alpha_j}h$ with
$\alpha_j\ne0$ and $\sum_j\alpha_j=\alpha$; one factor is kept in
$L^2_x$ and all remaining factors are put in $L^\infty_x$ by
$H^2_x\hookrightarrow L^\infty_x$.  Applying this composition estimate to
$\Phi(h)=e^{\pm h}$ with $h=\e\rho^\e$ and $h=\e\theta^\e$, and then
using \eqref{eq:loc-bd}, gives the first four terms in \eqref{eq:hn-en};
the last follows from $\mathrm U^\e=\e u^\e$.
Finally,
\[
 \frac{\Theta^\e-1}{\e}
 =\theta^\e\int_0^1e^{\tau\e\theta^\e}\,d\tau,
 \qquad
 \frac{\nabla_x\Theta^\e}{\e}=\Theta^\e\nabla_x\theta^\e.
\]
Applying the same product/composition estimates, together with
\eqref{eq:loc-bd}, gives \eqref{eq:hn-bd}.
\end{proof}

\begin{proof}[Proof of Lemma \ref{lem:maxbridge}]
From \eqref{eq:tot-bd}, $c_0=\mathfrak y$, and \eqref{eq:choice-bd},
\begin{equation*}
 (\mathcal E^N(t))^{1/2}
 \le (\mathcal E_{\rm tot}^N(t))^{1/2}
 \lesssim \e^{-c_0/2},
 \qquad 0\le t\le T_*.
\end{equation*}
This proves \eqref{eq:glob-bd}.  Since $N\ge4$, Sobolev embedding and
\eqref{eq:param-id} give
\begin{align}\label{eq:embed-bd}
 \|\mathrm U^\e\|_{L_x^\infty}
 &\lesssim \e(\mathcal E_M^N)^{1/2}
 \lesssim \e^{1-c_0/2},\notag\\
 \|\mathrm P^\e-1\|_{L_x^\infty}
 +\|\Theta^\e-1\|_{L_x^\infty}
 &\lesssim \e(\mathcal E_M^N)^{1/2}
 \exp\!\big(C\e(\mathcal E_M^N)^{1/2}\big)
 \lesssim \e^{1-c_0/2}.
\end{align}
Here we used $|e^\lambda-1|\le |\lambda|e^{|\lambda|}$.  Since
$\Theta_0^\e=1-\e^{1-c_0}$, \eqref{eq:bridge-bd} follows.
The first lower bound in \eqref{eq:assm-bd} and $e^{-\lambda}\ge1-\lambda$ yield
\[
 \Theta^\e-\Theta_0^\e
 \ge \e^{1-c_0}-C\e^{1-c_0/2}
 =\e^{1-c_0}(1-C\e^{c_0/2}),
\]
which proves \eqref{eq:pos-bd}.  Expanding the two quadratic forms and
completing the square gives \eqref{eq:xi-loc} exactly.  By
\eqref{eq:embed-bd} and \eqref{eq:pos-bd},
\begin{equation*}
 \frac{|\mathrm U^\e|^2}{\Theta^\e-\Theta_0^\e}
 \lesssim\frac{\e^{2-c_0}}{\e^{1-c_0}}
 =\e.
\end{equation*}
Using \eqref{eq:hn-en}, $\Theta_0^\e\to1$, and
\eqref{eq:xi-loc} and this bound in the exact quotient
\[
 \frac{\mu_0}{M^\e}
 =\frac1{\mathrm P^\e}
 \left(\frac{\Theta^\e}{\Theta_0^\e}\right)^{3/2}
 \exp\!\left(
 -\frac{|\xi|^2}{2k_B\Theta_0^\e}
 +\frac{|\xi-\mathrm U^\e|^2}{2k_B\Theta^\e}\right)
\]
proves \eqref{eq:mu-maxw}.  Since $a_\e=(k_B\Theta_0^\e)^{1/2}$
and both $a_\e$ and $a_\e^{-1}$ are uniformly bounded, there is a fixed
$c_a>0$ such that $\left\langle \xi/{a_\e}\right\rangle\ge c_a\langle\xi\rangle.$
By the definition of $w_0$,
\[
 \langle\xi\rangle^m w_0(\xi)^{-2}
 =\langle\xi\rangle^m
 \left\langle\frac\xi{a_\e}\right\rangle^{-2\ell}
 \exp\!\left(-2q\left\langle\frac\xi{a_\e}\right\rangle\right).
\]
For each fixed $m$, the polynomial factor in this display is bounded by
$C_m e^{q c_a\langle\xi\rangle}$, uniformly in $\e$.  Therefore
\eqref{eq:mu-maxw} gives \eqref{eq:poly-bd}, for example with
$c_q=q c_a$.
It remains to prove \eqref{eq:vd-bd}.  For
$Z=(P,U,\Theta)$ write $M_Z=M_{[P,U,\Theta]}$.  The parameter derivatives
are
\begin{align*}
 \partial_P M_Z=P^{-1}M_Z, \quad 
 \partial_{U_j}M_Z=\frac{\xi_j-U_j}{k_B\Theta}M_Z, \quad 
 \partial_\Theta M_Z
 =\left(-\frac3{2\Theta}
 +\frac{|\xi-U|^2}{2k_B\Theta^2}\right)M_Z.
\end{align*}
Let
$Z_\tau=(1,0,\Theta_0^\e)+\tau(\mathrm P^\e-1,
\mathrm U^\e,\Theta^\e-\Theta_0^\e)$, $0\le\tau\le1$.
By \eqref{eq:glob-bd}--\eqref{eq:bridge-bd}, $P_\tau$ and $\Theta_\tau$
stay in a fixed compact subset of $(0,\infty)$ and $|U_\tau|=o(1)$.
For $|\beta|\le2$, differentiating the three displayed parameter formulas
in $\xi$ shows that
\begin{equation}\label{eq:par-bd}
 \left|\partial_\xi^\beta D_Z
 \left(\frac{M_{Z_\tau}}{\sqrt{\mu_0}}\right)\right|
 \le C\langle\xi\rangle^{m_\beta}e^{-c|\xi|^2}
 \le C\mu_0^{1/3},
\end{equation}
with fixed $C,c>0$ independent of $\e,\tau,t,x$.  The fundamental theorem
of calculus in the three Maxwellian parameters gives
\[
 \partial_\xi^\beta A^\e
 =\int_0^1
 \partial_\xi^\beta D_Z\left(\frac{M_{Z_\tau}}{\sqrt{\mu_0}}\right)
 [\mathrm P^\e-1,\mathrm U^\e,
   \Theta^\e-\Theta_0^\e] \,d\tau.
\]
Equations \eqref{eq:bridge-bd} and \eqref{eq:par-bd} prove
\eqref{eq:vd-bd}.

The homogeneous correction $H_\infty^\e$ is independent of $x$.  For
$\mu_T=(2\pi k_BT)^{-3/2}e^{-|\xi|^2/(2k_BT)}$,
\[
 H_\infty^\e=\frac1\e\int_{\Theta_0^\e}^{1}
 \frac{\partial_T\mu_T}{\sqrt{\mu_0}}\,dT,
 \qquad
 \partial_T\mu_T=\mu_T\left(-\frac3{2T}+\frac{|\xi|^2}{2k_BT^2}\right).
\]
For $T\in[\Theta_0^\e,1]$ and $|\beta|\le1$,
$J_{T,\beta}:=w_0\partial_\beta(\partial_T\mu_T/\sqrt{\mu_0})$ satisfies
$|J_{T,\beta}|+|\nabla J_{T,\beta}|\le Ce^{-c|\xi|^2}$ uniformly.  Thus the
mean-value formula gives, on $|\eta|\le1$,
$|J_{T,\beta}(\xi)-J_{T,\beta}(\xi-\eta)|\le C|\eta|e^{-c'|\xi|^2}$, and therefore
$\int_{|\eta|\le1}|\eta|^{-3-2s}|\eta|^2\,d\eta <\infty.$
The $|\eta|>1$ part and all polynomial velocity weights are Gaussian
integrable.  Hence, by the upper bound in \eqref{eq:xi-bd},
\[
 \sup_{\Theta_0^\e\le T\le1}\left(
 \left\|w_0\partial_\beta\frac{\partial_T\mu_T}{\sqrt{\mu_0}}
 \right\|_{L^2_\xi\cap L^2_\Bb}
 +\left\|w_0\frac{\partial_T\mu_T}{\sqrt{\mu_0}}
 \right\|_{L^2_{2s+\gamma}}\right)\le C.
\]
Since $1-\Theta_0^\e=\e^{1-c_0}$, this proves \eqref{eq:hinfty-bd}.

The local-Maxwellian correction $\widetilde H^\e$ in \eqref{eq:maxw-def} is spatially localized through the
macroscopic perturbation.  With
$Z_\tau^\e:=(1,0,1)+\tau (\mathrm P^\e-1,\mathrm U^\e,\Theta^\e-1), \ 0\le\tau\le1,$
the fundamental theorem of calculus in the Maxwellian parameters gives
\begin{equation}\label{eq:maxw-path}
 \widetilde H^\e
 =\int_0^1
 \frac{D_ZM_{Z_\tau^\e}}{\sqrt{\mu_0}}
 \left[
 \frac{\mathrm P^\e-1}{\e},
 \frac{\mathrm U^\e}{\e},
 \frac{\Theta^\e-1}{\e}
 \right]d\tau.
\end{equation}
By \eqref{eq:param-id}, Lemma~\ref{lem:macrocoef}, and \eqref{eq:loc-bd},
\begin{equation}\label{eq:maxw-param-bd}
 \left\|\left(
 \frac{\mathrm P^\e-1}{\e},
 \frac{\mathrm U^\e}{\e},
 \frac{\Theta^\e-1}{\e}\right)\right\|_{H_x^N}
 \lesssim (\mathcal E_M^N)^{1/2},
 \qquad
 \e(\mathcal E_M^N)^{1/2}=o(1).
\end{equation}
The compact parameter bounds already proved in \eqref{eq:embed-bd}--\eqref{eq:par-bd}
imply, for every $|\alpha|\le2$ and $|\beta|\le1$, that each coefficient
obtained by applying $\partial_x^\alpha\partial_\xi^\beta$ to
$D_ZM_{Z_\tau^\e}/\sqrt{\mu_0}$ is a finite sum of polynomial--Gaussian
functions whose velocity norms
\[
 \|w_0(\cdot)\|_{L_\xi^2\cap L_\Bb^2\cap L^2_{2s+\gamma}}
 \quad\hbox{and}\quad
 \|w_0^{1/2}\partial_\beta(\cdot)\|_{L_\xi^2}
\]
are bounded uniformly in $(\e,\tau,t,x)$.  Indeed, every positive
$x$-derivative of such a coefficient contains a factor
$\e\partial^{\gamma}Z^\e$; the remaining factors are bounded in
$L_x^\infty$ by $H_x^2\hookrightarrow L_x^\infty$, and
\eqref{eq:maxw-param-bd} absorbs every product containing two or more positive
spatial derivatives.  Applying the $H_x^2$ product estimate to
\eqref{eq:maxw-path} therefore gives
\begin{align*}
 &\|w_0\widetilde H^\e\|_{H_x^2L_\xi^2}
 +\|w_0\widetilde H^\e\|_{H_x^2L_\Bb^2}
 +\|w_0\widetilde H^\e\|_{H_x^2L^2_{2s+\gamma}} +
 \sum_{|\beta|=1}
 \|w_0^{1/2}\partial_\beta\widetilde H^\e\|_{H_x^2L_\xi^2}
 \lesssim (\mathcal E_M^N)^{1/2}
 \le (\mathcal E^N)^{1/2},
\end{align*}
which is \eqref{eq:xi-gam}.  Notice that this conclusion uses only
Lemma~\ref{lem:macrocoef} and the parameter bounds established above in the
present proof; in particular, it does not invoke Lemma~\ref{lem:maxbridge}
recursively.

\end{proof}

\begin{proof}[Proof of Lemma \ref{lem:fluctxfer}]
Set $Z^\e=(\rho^\e,u^\e,\theta^\e)$ and
\[
 \mathcal M_{\rm par}(a,\mathbf{b},c):=M_{[e^{a},\mathbf{b},e^{c}]},
 \qquad
 M^\e=\mathcal M_{\rm par}(\e\rho^\e,\e u^\e,\e\theta^\e).
\]
Since $\mu_0$ is independent of $x$, every positive spatial derivative
contains at least one differentiated macroscopic parameter.  The multivariable
chain rule gives, for $1\le|\alpha|\le N$,
\begin{equation}\label{eq:maxw-id}
 \frac{\partial^\alpha M^\e}{\e\sqrt{\mu_0}}
 =\sum_{m=1}^{|\alpha|}\e^{m-1}
 \sum_{\substack{\alpha_1+\cdots+\alpha_m=\alpha\\|\alpha_j|\ge1}}
 C_{\alpha_1,\ldots,\alpha_m}\,
 \mathcal G_m^\e(\xi)
 \prod_{j=1}^m\partial^{\alpha_j}Z^\e,
\end{equation}
where $\mathcal G_m^\e$ is a finite linear combination of parameter
derivatives of $\mathcal M_{\rm par}/\sqrt{\mu_0}$ evaluated at
$(\e\rho^\e,\e u^\e,\e\theta^\e)$.  Lemma~\ref{lem:maxbridge} implies that the
parameters range in a fixed compact set.  Hence, for $|\beta|\le1$ and each
fixed $m\le N$,
\begin{equation}\label{eq:der-bet}
 \sup_{t,x}\left(
 \|w_0^{1/2}\partial_\beta\mathcal G_m^\e\|_{L^2_\xi}
 +\|w_0\mathcal G_m^\e\|_{H^s_{2s+\gamma}}
 \right)\le C_{N,\ell,q}.
\end{equation}
Indeed these coefficients are uniformly polynomial--Gaussian.  With
$H=w_0\mathcal G_m^\e$, both $|H|$ and $|\nabla H|$ are bounded by
$Ce^{-c|\xi|^2}$; therefore the local fractional seminorm is bounded by
$C\int_0^1 \tau^{1-2s}\,d \tau<\infty$, while the complementary region is
Gaussian integrable.  This proves \eqref{eq:der-bet}, and \eqref{eq:xi-bd} gives
the corresponding $\Bb$ bounds.
By the $H_x^2$ algebra property and the standard product estimate, each
term in \eqref{eq:maxw-id} satisfies
\begin{align*}
 \e^{m-1}
 \Big\|\prod_{j=1}^m\partial^{\alpha_j}Z^\e\Big\|_{L_x^2}
 &\lesssim
 \|Z^\e\|_{H_x^N}
 \bigl(\e\|Z^\e\|_{H_x^N}\bigr)^{m-1} \lesssim (\mathcal E_M^N)^{1/2},
\end{align*}
where the last step uses \eqref{eq:bolt-def}--\eqref{eq:glob-bd}, which imply
$\e(\mathcal E_M^N)^{1/2}=o(1)$.  Combining this with
\eqref{eq:der-bet} yields
\begin{equation}\label{eq:der-maxw}
 \left\|w_0\frac{\partial^\alpha M^\e}{\e\sqrt{\mu_0}}\right\|_{L^2_{x,\xi}}
 +\left\|w_0\frac{\partial^\alpha M^\e}{\e\sqrt{\mu_0}}\right\|_{L^2_{x,\Bb}}
 +\left\|w_0\frac{\partial^\alpha M^\e}{\e\sqrt{\mu_0}}\right\|_{L^2_xL^2_{2s+\gamma}}
 \lesssim(\mathcal E_M^N)^{1/2}.
\end{equation}
Differentiating \eqref{eq:maxw-id} once in $\xi$, using
\eqref{eq:der-bet}, and applying the same spatial product estimate gives,
for $|\beta|=1$,
\begin{equation*}
 \left\|w_0^{1/2}\partial_\beta
 \left(\frac{\partial^\alpha M^\e}{\e\sqrt{\mu_0}}\right)
 \right\|_{L^2_{x,\xi}}
 \lesssim
 \sum_{m=1}^{|\alpha|}
 \e^{m-1}
 \sum_{\substack{\alpha_1+\cdots+\alpha_m=\alpha\\|\alpha_j|\ge1}}
 \left\|\prod_{j=1}^m\partial^{\alpha_j}Z^\e\right\|_{L_x^2}
 \lesssim(\mathcal E_M^N)^{1/2}.
\end{equation*}
The coefficient bounds used in \eqref{eq:der-maxw} also hold with $w_0$
replaced by $1$; hence
\begin{equation*}
 \left\|\frac{\partial^\alpha M^\e}{\e\sqrt{\mu_0}}\right\|_{L^2_{x,\xi}}
 +\left\|\frac{\partial^\alpha M^\e}{\e\sqrt{\mu_0}}\right\|_{L^2_{x,\Bb}}
 \lesssim(\mathcal E_M^N)^{1/2}.
\end{equation*}
Finally,
\[
 \partial^\alpha g^\e
 =\partial^\alpha f^\e-
 \frac{\partial^\alpha M^\e}{\e\sqrt{\mu_0}},
 \qquad |\alpha|\ge1,
\]
and \eqref{eq:der-alp}--\eqref{eq:der-alp-bd} follow from \eqref{eq:der-maxw} and the
triangle inequality.
\end{proof}

\begin{proof}[Proof of Lemma \ref{lem:collest}]
We separate fixed-global estimates from the local-Maxwellian transfer.
For Landau, \eqref{eq:lin-bd}--\eqref{eq:xi-pr} and the weighted coercivities
\eqref{eq:top-bd}--\eqref{eq:alp-bet} follow from
\cite[Lemmas~2.2--2.3]{lei2024global} after the fixed dilation normalizing
$k_B$.  In the weighted estimates use $(\ell,4q)$ for $|\beta|=0$ and
$(1+\ell/2,2q)$ for $|\beta|=1$; the corresponding test/coercive weights
are $w^2$ and $(w,w^{1/2})$.  Choosing the source small parameter first
and summing over $|\alpha|\le m$, $|\beta|=1$ leaves exactly the lower
velocity-order remainder in \eqref{eq:alp-bet}.  The parameters are admissible
because $q<1/8$.

For non-cutoff Boltzmann, the normalized coercive estimates are
\cite[Lemma~2.1, (2.1)--(2.3)]{xu2026diffusiveVMB}.  To transfer them to
$\mu_0$, put
$(T_\e h)(y):=a_\e^{3/2}h(a_\e y)$,  $\mu_{\rm st}(y):=(2\pi)^{-3/2}e^{-|y|^2/2}.$
Since $a_\e+a_\e^{-1}\lesssim1$ and
$a_\e^3\mu_0(a_\e y)=\mu_{\rm st}(y)$,
\begin{align*}
 T_\e(\mathcal L_0h) =a_\e^\gamma\mathcal L_{\rm st}(T_\e h), \ \ \ 
 \|h\|_{L^2_\Bb}^2 =a_\e^\gamma\|T_\e h\|_{D}^2,\\
 \langle\Gamma_0(h,k),m\rangle_\xi
 =a_\e^\gamma\langle\Gamma_{\rm st}(T_\e h,T_\e k),T_\e m\rangle_y, \ \ \ 
 T_\e\partial_\xi^\beta=a_\e^{-|\beta|}\partial_y^\beta T_\e,
\end{align*}
and $T_\e\Pi_0=\Pi_{\rm st}T_\e$.  Moreover
$w_0(a_\e y)=\langle y\rangle^\ell e^{q\langle y\rangle}$ exactly.
Thus the normalized estimates pull back uniformly to \eqref{eq:top-bd}--
\eqref{eq:alp-bet}; the compact remainder in the one-velocity-derivative
estimate is controlled, by \eqref{eq:xi-bd}, by the lower velocity-order
$w_0^{1/2}L^2_{\mathbf D}$ term already displayed in \eqref{eq:alp-bet}.
The same norm comparison gives \eqref{eq:xi-pr} because
$2s+\gamma>-3/2$.

We next treat the local Maxwellian.  Fix $(t,x)$ and set
\[
 a=(k_B\Theta^\e)^{1/2},\qquad y=\frac{\xi-\mathrm U^\e}{a},\qquad
 M^\e(\xi)=\mathrm P^\e a^{-3}\mu_{\rm st}(y),
\]
Write $h(\xi)=\sqrt{\mathrm P^\e}\,a^{-3/2}\sqrt{\mu_{\rm st}(y)}\,g(y)$.
The five vanishing local moments of $h$ are precisely
$g\perp\operatorname{Null}(\mathcal L_{\rm st})$.  Homogeneity gives
\[
 \int \mathcal L_Mh\,\frac h{M^\e}\,d\xi
 =\begin{cases}
 \mathrm P^\e\langle\mathcal L_{\rm st}g,g\rangle,
       &\text{Landau},\\
 \mathrm P^\e a^{\gamma+3}
 \langle\mathcal L_{\rm st}g,g\rangle,
       &\text{Boltzmann}.
 \end{cases}
\]
All scalar factors are uniformly comparable to one.  It remains only to
compare the normalized collision norm of $g$ with the physical local norm.
This comparison is uniform under every affine map $\xi=V+by$ with
$b+b^{-1}+|V|\lesssim1$.  For Landau it follows from the classical
Coulomb diffusion-matrix structure
\[
 \eta \cdot\sigma(\xi)\eta\asymp
 \langle \xi\rangle^{-1}|P_{\xi^\perp}\eta|^2
 +\langle \xi\rangle^{-3}|P_\xi \eta|^2, \ \eta \in \R^3
\]
because $\langle V+by\rangle\asymp\langle y\rangle$, the radial
directions differ by $O(|V|\langle y\rangle^{-1})$, and the remaining
compact region is uniformly elliptic.

For non-cutoff Boltzmann, use the Gressman--Strain geometric norm
$N^{s,\gamma}$, equivalent to the normalized Dirichlet norm.  With
$\xi=V+by$ one has the exact lifted relation
\[
 \binom{\xi-\xi'}{(|\xi|^2-|\xi'|^2)/2}
 =\begin{pmatrix}bI_3&0\\ bV^{\mathsf T}&b^2\end{pmatrix}
 \binom{y-y'}{(|y|^2-|y'|^2)/2}.
\]
The matrix and its inverse are uniformly bounded, so the lifted distance
and polynomial weights are comparable.  Replacing the cutoff
$d(y,y')\le1$ by $d(y,y')\le C$ creates only a fixed annulus; its
contribution is controlled by the weighted $L^2_{\gamma+2s}$ part because
$|\{y':d(y,y')\le C\}|\lesssim_C\langle y\rangle^{-1}$.
Applying the inverse affine map gives the reverse bound.  Thus, in both
models,
\[
 \left\|\frac h{\sqrt{M^\e}}\right\|_{L^2_{\mathbf D}}
 \asymp \|g\|_{D_{\rm st}}
\]
uniformly.  The normalized coercivity therefore proves \eqref{eq:mh-bd},
and symmetry plus null-space annihilation gives \eqref{eq:mq-bd}.

For \eqref{eq:op-bd}, define
$f(y)=F(U+ay)/\sqrt{M^\e(U+ay)}$ and similarly $g,h$.  The collision
pairing scales as
\[
 \int Q(F,G)H(M^\e)^{-1}\,d\xi
 =\begin{cases}
 (\mathrm P^\e)^{1/2}a^{3/2}
 \langle\Gamma_{\rm st}(f,g),h\rangle,&\text{Landau},\\
 (\mathrm P^\e)^{1/2}a^{\gamma+9/2}
 \langle\Gamma_{\rm st}(f,g),h\rangle,&\text{Boltzmann}.
 \end{cases}
\]
The factors are uniformly bounded.  Combining the affine norm comparison
above with the normalized local-Maxwellian trilinear estimates
\cite[Lemmas~2.4 and~3.2]{lei2024global}, exactly as in the original
detailed proof, and applying the ordered estimate to both $(F,G)$ and
$(G,F)$ proves \eqref{eq:op-bd}.

It remains to justify the fixed-global trilinear statements.  The Landau
forms \eqref{eq:landau-tri} and \eqref{eq:landau-tri-wt} are precisely the
refinements of \cite[Lemma~10]{strain2008exponential} used in the original
detailed proof: bounded-velocity and residual-Gaussian pieces give the
$L^2_{-1}$ gain, while the convolution term leaves respectively the
harmless $\mu^{1/32}$ factor and the polynomial factor
$\vartheta_L=e^{-q\langle\xi/a_L\rangle^2}w$.  For one velocity
derivative use \cite[Lemma~2.5(ii)]{lei2024global} with
$(\ell_s,Q)=(1+\ell/2,2q)$.  After the fixed thermal dilation, its first
slot carries the polynomial weights $A_0=\langle y\rangle^{1+\ell/2}$ and
$A_1=\langle y\rangle^{\ell/2}$.  Since $Q=2q$ and
$\mu_{\rm st}(y)\sim e^{-|y|^2/2}$, the weighted source factors satisfy
\[
 \langle y\rangle^{1+\ell/2}\mu_{\rm st}^{-q}\lesssim w,
 \qquad
 \langle y\rangle^{\ell/2}\mu_{\rm st}^{-q}\asymp w^{1/2},
\]
so the second and testing slots are bounded by $W_0=w$ and $W_1=w^{1/2}$.
The source sum has $\beta_1+\beta_2\le\beta$, while its term involving
spatial derivatives of the Maxwellian parameters vanishes because the
Maxwellian is constant.  Repeating the bounded-velocity and
residual-Gaussian split used for \eqref{eq:landau-tri-wt} yields the same
$L^2_{-1}$ refinement in the second branch: on $|\xi|\lesssim1$ the two
norms are equivalent, while on the complementary Gaussian pieces one may
insert $\langle\xi\rangle^{-1/2}$ at the cost of weakening the Gaussian.
It remains only to verify \eqref{eq:landau-slot-wt}.  Explicitly,
\[
 r_0:=A_0/W_0=\langle y\rangle^{1-\ell/2}e^{-q\langle y\rangle^2},
 \qquad
 r_1:=A_1/W_1=e^{-q\langle y\rangle^2/2}.
\]
Thus both $r_j$ are radial Gaussian-decaying multipliers and
\[
 \sup_\xi\langle\xi\rangle^{1/2}|r_j(\xi)|
 +\sup_\xi\langle\xi\rangle^{-1/2}|\nabla_\xi r_j(\xi)|<\infty.
\]
Therefore \eqref{eq:xi-bd} gives
$\|r_jW_jh\|_2\lesssim\|W_jh\|_{L^2_{-1}}\lesssim\|W_jh\|_{L^2_{\Lb}}$.
For the derivative part of the Landau norm, radiality gives
$\partial_i r_j=b_j(\xi)\xi_i r_j$ with $b_j$ bounded; hence
\[
 \sigma^{ik}\partial_i r_j\partial_k r_j|W_jh|^2
 \lesssim r_j^2\sigma^{ik}\xi_i\xi_k|W_jh|^2,
\]
which is controlled by the zeroth-order part of
$\|W_jh\|_{L^2_{\Lb}}^2$.  This proves \eqref{eq:landau-slot-wt}.

Finally, \eqref{eq:boltz-tri} is
\cite[Theorem~1.2]{alexandre2012boltzmann}, while
\eqref{eq:boltz-tri-wt} follows from
\cite[Lemma~2.2, (2.4)--(2.5)]{xu2026diffusiveVMB}; these estimates cover
$0<s<1$ and $\gamma>\max\{-3,-3/2-2s\}$, include the lower
velocity-order terms, and their $(\ell/2,q/2)$ specialization gives the
one-velocity-derivative weight $w_0^{1/2}$.  To prove \eqref{eq:boltz-wt-hier},
write $m(\xi)=w_0^{-1/2}(\xi)$.  Since $a_\e+a_\e^{-1}\lesssim1$,
$|m|+|\nabla_\xi m|\le Ce^{-c\langle\xi\rangle}$.  By \eqref{eq:xi-bd}, it
suffices to show
\[
 \|mH\|_{H^s_{2s+\gamma}}
 \lesssim \|H\|_{H^s_\gamma}+\|H\|_{L^2_{2s+\gamma}}.
\]
For $|\eta|\le1$ use
$m(\xi)H(\xi)-m(\xi-\eta)H(\xi-\eta) =m(\xi)[H(\xi)-H(\xi-\eta)] +[m(\xi)-m(\xi-\eta)]H(\xi-\eta).$
The first term is controlled by the $H^s_\gamma$ seminorm because
$\langle\xi\rangle^{2s}|m(\xi)|^2\lesssim1$.  For the second,
$\langle\xi\rangle\asymp\langle\xi-\eta\rangle$ and
$|m(\xi)-m(\xi-\eta)|\lesssim|\eta|e^{-c\langle\xi\rangle}$; hence its
contribution is bounded by
$C\|H\|_{L^2_{2s+\gamma}}^2\int_0^1\rho^{1-2s}\,d\rho$, which is finite
for $s<1$.  The $L^2_{2s+\gamma}$ part is immediate from $|m|\le C$.
Thus
$\|mH\|_{L^2_{\Bb}}\lesssim\|H\|_{L^2_{\Bb}}$ by both sides of
\eqref{eq:xi-bd}.  Apply this first to $H=w_0^{1/2}h$ and then to
$H=w_0h$ to obtain \eqref{eq:boltz-wt-hier}.  This completes the proof.
\end{proof}

\begin{proof}[Proof of Lemma \ref{lem:transsrc}]
Set $c^\e:=\xi-\mathrm U^\e$ and $D_u:=\partial_t+u^\e\cdot\nabla_x$.
The Maxwellian identity is
\begin{align*}
 \frac{\mathcal T_\e M^\e}{M^\e}
 &=\mathcal T_\e\log\mathrm P^\e-\frac32\mathcal T_\e\log\Theta^\e
 +\frac{c^\e\cdot\mathcal T_\e\mathrm U^\e}{k_B\Theta^\e}
 +\frac{|c^\e|^2}{2k_B(\Theta^\e)^2}\mathcal T_\e\Theta^\e,\\
 \mathcal T_\e\log\mathrm P^\e&=-\nabla\cdot u^\e+c^\e\cdot\nabla\rho^\e,\\
  \mathcal T_\e\log\Theta^\e&=\e D_u\theta^\e+c^\e\cdot\nabla\theta^\e,\\
 \mathcal T_\e\mathrm U^\e&=\e D_u u^\e+c^\e\cdot\nabla u^\e.
\end{align*}
By \eqref{eq:tu-id},
\[
 \e D_u u_i^\e=-k_B\Theta^\e\partial_i(\rho^\e+\theta^\e)
 -\frac1{\e\mathrm P^\e}\sum_j\partial_j\mathbf r_{ij}^\e,
 \quad
 \e D_u\theta^\e=-\frac23\nabla\cdot u^\e
 -\frac{2(\nabla\cdot\mathfrak q^\e+\nabla\mathrm U^\e:\mathbf r^\e)}
 {3\e k_B\mathrm P^\e\Theta^\e}.
\]
Using $\nabla\mathrm U^\e=\e\nabla u^\e$, Lemma~\ref{lem:macrocoef}, and
\eqref{eq:hn-kap},
\begin{equation*}
 \|\e D_uu^\e\|_2+\|\e D_u\theta^\e\|_2
 \lesssim(\mathcal E^N)^{1/2}
 +\e\kappa^{1/2}(\mathcal D^N)^{1/2}.
\end{equation*}
The compact parameter bounds imply, for every fixed $m$ and $0\le a\le1$,
\[
 \sup_{t,x}\|\langle\xi\rangle^mM^\e\mu_0^{-3/4}\|_{L^2_\xi}
 +\|\langle\xi\rangle^m w_0^a\mu_0^{1/4}\|_{L^\infty_\xi}\le C_{m,a}.
\]
Substitution in the first display therefore proves the first term of
\eqref{eq:mu-ltwo} and both terms of \eqref{eq:xi-mu}.  For $|\beta|=1$, it remains
to note the exact identity
\begin{align*}
 \partial_{\xi_j}(\mu_0^{-1/2}\mathcal T_\e M^\e)
 &=\left(\frac{\xi_j}{2k_B\Theta_0^\e}-\frac{c_j^\e}{k_B\Theta^\e}\right)
 \mu_0^{-1/2}\mathcal T_\e M^\e\\
 &\quad+\frac{M^\e}{\sqrt{\mu_0}}
 \left(\frac{\mathcal T_\e\mathrm U_j^\e}{k_B\Theta^\e}
 +\frac{c_j^\e\mathcal T_\e\Theta^\e}{k_B(\Theta^\e)^2}\right)
 +\frac1\e\frac{\partial_{x_j}M^\e}{\sqrt{\mu_0}},
\end{align*}
where
\[
 \frac1\e\frac{\partial_{x_j}M^\e}{\sqrt{\mu_0}}
 =\frac{M^\e}{\sqrt{\mu_0}}\left[
 \partial_j\rho^\e-\frac32\partial_j\theta^\e
 +\frac{c^\e\cdot\partial_ju^\e}{k_B\Theta^\e}
 +\frac{|c^\e|^2}{2k_B\Theta^\e}\partial_j\theta^\e\right].
\]
Every term is thus a uniformly polynomial--Gaussian velocity factor times a
coefficient controlled above; factoring
$M^\e/\sqrt{\mu_0}=(M^\e\mu_0^{-3/4})\mu_0^{1/4}$ and taking $a=1/2$
gives the velocity-derivative term in \eqref{eq:mu-ltwo}.
\end{proof}

\begin{proof}[Proof of Lemma~\ref{lem:acdisp}]
Set
\[
 b:=|D|^{-1}\nabla\cdot q,
 \qquad
 \lambda:=\sqrt{AB},
 \qquad
 W_\pm:=\sqrt B\,a\pm i\sqrt A\,b.
\]
Because $q=\mathbb P^\perp q$, one has the exact Fourier-multiplier identity
$q=-\nabla|D|^{-1}b$.  Hence the Riesz transforms give
$\|q\|_{\dot B_{p,1}^\sigma}\simeq
 \|b\|_{\dot B_{p,1}^\sigma}$ for $1<p<\infty$.  At $p=\infty$, this multiplier identity is used after dyadic localization: on the fixed
annulus supporting $\dot\Delta_j$, the multiplier
$-i\xi/|\xi|$ has an $L^1$ convolution kernel with norm independent of
$j$, so
$\|\dot\Delta_j q\|_{L^\infty}\lesssim
 \|\dot\Delta_j b\|_{L^\infty}$ uniformly in $j$.
Taking $|D|^{-1}\nabla\cdot$ of the second equation in \eqref{eq:pq-id}
gives
\[
 \partial_tb-\frac{B}{\e}|D|a
 =|D|^{-1}\nabla\cdot F_q.
\]
Hence
\begin{equation}\label{eq:wave-id}
 \partial_tW_\pm\mp i\frac{\lambda}{\e}|D|W_\pm
 =\sqrt B F_a
 \pm i\sqrt A|D|^{-1}\nabla\cdot F_q=:G_\pm.
\end{equation}
We prove the required dyadic half-wave bound.  Let $\dot\Delta_j$ be
supported in a fixed annulus.  The kernel of
$\dot\Delta_j e^{\pm it\lambda|D|/\e}$ is
\[
 K_j(t,x)
 =2^{3j}\int_{\R^3}
 e^{i(2^jx\cdot\zeta\pm 2^jt\lambda|\zeta|/\e)}
 \varphi(\zeta)\,d\zeta.
\]
On the fixed annulus, integration in polar coordinates followed by one
integration by parts away from the stationary cone, and the one-dimensional
stationary-phase estimate on the cone, give
\begin{equation*}
 \|K_j(t,\cdot)\|_{L^\infty_x}
 \lesssim
 2^{3j}\left(1+\frac{2^j|t|}{\e}\right)^{-1}.
\end{equation*}
Thus the propagator is unitary on $L^2$ and, by interpolation,
\begin{align*}
 \|\dot\Delta_j e^{\pm it\lambda|D|/\e}h\|_{L^p_x}
 \lesssim
 2^{3j(1-2/p)}
 \left(1+\frac{2^j|t|}{\e}\right)^{-(1-2/p)}
 \|\dot\Delta_j h\|_{L^{p'}_x}.
\end{align*}
The $TT^*$ argument applied to this dispersive bound, followed by the
one-dimensional Hardy--Littlewood--Sobolev inequality in time, yields the desired estimate on the sharp admissible line; interpolation with the energy estimate
gives the full region \eqref{eq:pq-bd}.  After the rescaling
$t=(\e2^{-j}/\lambda)\tau$, the result is
\begin{equation*}
 \|\dot\Delta_j e^{\pm it\lambda|D|/\e}h\|_{L_t^\mathfrak{r} L_x^p}
 \lesssim
 \e^{1/\mathfrak{r}}
 2^{j(3(1/2-1/p)-1/\mathfrak{r})}
 \|\dot\Delta_jh\|_2.
\end{equation*}
The excluded pair $(2,\infty)$ is precisely the three-dimensional wave
endpoint not supplied by the $TT^*$ estimate.
Duhamel's formula in \eqref{eq:wave-id}, Minkowski's inequality in the source time, and this homogeneous estimate give
\begin{align*}
 \|\dot\Delta_jW_\pm\|_{L_t^\mathfrak{r}L_x^p}
 &\lesssim \e^{1/\mathfrak{r}}
 2^{j(3(1/2-1/p)-1/\mathfrak{r})}
 \left(\|\dot\Delta_jW_\pm(0)\|_2
 +\int_0^T\|\dot\Delta_jG_\pm(\tau)\|_2\,d\tau\right).
\end{align*}
Multiply by
$2^{j(s+3(1/p-1/2)+1/\mathfrak{r})}$ and sum in $j$.  The powers cancel exactly,
leaving $2^{js}$.  Since
$|D|^{-1}\nabla\cdot$ is an order-zero multiplier on every dyadic
annulus, this proves \eqref{eq:fin-ac}.
\end{proof}

\begin{proof}[Proof of Lemma \ref{lem:sympres}]
For $G\in O(2)_z$ set $F_G(t,x,\xi)=F^\e(t,Gx,G\xi)$.  Orthogonality gives $\xi\cdot\nabla_xF_G=(G\xi)\cdot\nabla_yF^\e(t,y,G\xi)|_{y=Gx}$.
For Landau,
$\Phi(G\mathbf k)=G\Phi(\mathbf k)G^T$; for Boltzmann,
$B(G\mathbf k,G\mathbf n)=B(\mathbf k,\mathbf n)$ and the post-collisional velocities are
mapped to $G\xi'$ and $G\xi_*'$.  Thus in both cases $Q(F_G,F_G)(t,x,\xi)=Q(F^\e,F^\e)(t,Gx,G\xi)$.
Hence $F_G$ solves the same Cauchy problem as $F^\e$ and equals $F^\e$
by uniqueness.  Taking moments gives
$\mathrm U^\e(t,Gx)=G\mathrm U^\e(t,x)$.  The Leray projection commutes
with orthogonal transformations, so $\mathbb Pu^\e$ is equivariant.
Reflections in $O(2)_z$ eliminate the azimuthal component.
\end{proof}

\begin{proof}[Proof of Lemma \ref{lem:logbs}]
Since $\nabla_x\cdot v=0$, $-\Delta v=\nabla_x\times\omega$ and
$\nabla_xv=T\omega$, where $T:=\nabla_x\nabla_x\times(-\Delta)^{-1}$ is a matrix of
homogeneous Fourier multipliers of degree zero.  Let $\{\dot\Delta_j\}_{j\in\mathbb Z}$
be a homogeneous Littlewood--Paley decomposition.  On each dyadic annulus,
the multiplier of $\dot\Delta_jT$ has an $L^1$ convolution kernel with
norm bounded independently of $j$.  Hence
\begin{equation}\label{eq:jt-bd}
 \|\dot\Delta_jT\omega\|_\infty
 \lesssim \|\omega\|_\infty,
 \qquad j\in\mathbb Z.
\end{equation}
Bernstein's inequality and the $L^2$ boundedness of $T$ also give
\begin{equation}\label{eq:bern-bd}
 \|\dot\Delta_jT\omega\|_\infty
 \lesssim 2^{3j/2}\|\dot\Delta_j\omega\|_2.
\end{equation}
Set $A:=e+\|\omega\|_{H^2}$.  Choose integers $J_-,J_+\ge1$ so that
$2^{3J_-/2}\simeq A$ and $2^{J_+/2}\simeq A$.
For the very low frequencies, \eqref{eq:bern-bd} and Cauchy--Schwarz yield
\begin{align*}
 \sum_{j<-J_-}\|\dot\Delta_jT\omega\|_\infty
 \lesssim
 \left(\sum_{j<-J_-}2^{3j}\right)^{1/2}
 \left(\sum_{j<-J_-}\|\dot\Delta_j\omega\|_2^2\right)^{1/2} 
 \lesssim 2^{-3J_-/2}\|\omega\|_2
 \lesssim1.
\end{align*}
For the middle frequencies, \eqref{eq:jt-bd} gives
\begin{align*}
 \sum_{-J_-\le j\le J_+}
 \|\dot\Delta_jT\omega\|_\infty
 \lesssim (J_-+J_++1)\|\omega\|_\infty 
 \lesssim (1+\log A)\|\omega\|_\infty.
\end{align*}
For the high frequencies, \eqref{eq:bern-bd} and Cauchy--Schwarz give
\begin{align*}
 \sum_{j>J_+}\|\dot\Delta_jT\omega\|_\infty
  \lesssim
 \left(\sum_{j>J_+}2^{-j}\right)^{1/2}
 \left(\sum_{j>J_+}2^{4j}\|\dot\Delta_j\omega\|_2^2\right)^{1/2} 
 \lesssim 2^{-J_+/2}\|\omega\|_{\dot H^2}
 \lesssim1.
\end{align*}
The three series are absolutely summable in $L^\infty$, so their sum is
$T\omega=\nabla_xv$.  Since
$1+\log(e+x)\lesssim1+\ln^+x$, these estimates prove
\eqref{eq:xv-bd}.
\end{proof}

\begin{proof}[Proof of Lemma \ref{lem:locreg}]
Choose $\chi\in C_c^\infty(B_{2R})$ with $\chi=1$ on $B_{3R/2}$ and set
$U:=\nabla\times(-\Delta)^{-1}(\chi\omega)$ and $h:=u-U$.
Calder\'on--Zygmund gives
$\|\nabla U\|_{L^p(\R^3)}\le C_p\|\omega\|_{L^p(B_{2R})}$.
Since $p\ge2$ and $\chi\omega$ is supported in $B_{2R}$,
$\|\chi\omega\|_2\lesssim_{p,R}\|\omega\|_{L^p(B_{2R})}$; hence the
$L^2$ Calder\'on--Zygmund estimate and Sobolev inequality give
\[
 \|U\|_{L^2(B_{3R/2})}
 \le C_R\|U\|_6
 \le C_R\|\nabla U\|_2
 \le C_{p,R}\|\omega\|_{L^p(B_{2R})}.
\]
By \eqref{eq:comp-def}, $-\Delta u=\nabla\times\omega$ in
$\mathcal D'(B_{2R})$, so $h$ is harmonic in $B_{3R/2}$.  The interior
estimate for harmonic functions therefore yields
\[
 \|\nabla h\|_{L^p(B_R)}
 \le C_{p,R}\|h\|_{L^2(B_{3R/2})}
 \le C_{p,R}\left(\|u\|_2+\|\omega\|_{L^p(B_{2R})}\right).
\]
Combining the last two gradient bounds gives the same estimate for
$\|\nabla u\|_{L^p(B_R)}$.  Finally, Poincar\'e's inequality gives
\[
 \|u\|_{L^p(B_R)}
 \le C_R\|\nabla u\|_{L^p(B_R)}
 +|B_R|^{\frac1p-\frac12}\|u\|_{L^2(B_R)},
\]
which proves \eqref{eq:ltwo-om}.
\end{proof}

\begin{lemma}[Hardy--Littlewood--Sobolev inequality in $\R^2$]\label{lem:hls}
Let $0<\lambda<2$ and $1<a,b<\infty$ satisfy
\begin{equation}\label{eq:hls-cond}
 \frac1a+\frac1b+\frac\lambda2=2.
\end{equation}
Then, for every $f\in L^a(\R^2)$ and $g\in L^b(\R^2)$,
\begin{equation}\label{eq:hls-bound}
 \iint_{\R^2\times\R^2}
 \frac{|f(X)g(Y)|}{|X-Y|^\lambda}\,dX\,dY
 \lesssim_{\lambda,a,b}
 \|f\|_{L^a(\R^2)}\|g\|_{L^b(\R^2)}.
\end{equation}
In particular, for equal exponents,
\begin{equation}\label{eq:hls-equal}
 a=b=\frac{4}{4-\lambda}.
\end{equation}
For the Riesz potential $I_1f(X):=\int_{\R^2}|X-Y|^{-1}f(Y)\,dY$,
\begin{equation}\label{eq:riesz-bound}
 \|I_1f\|_{L^4(\R^2)}
 \lesssim \|f\|_{L^{4/3}(\R^2)}.
\end{equation}
\end{lemma}
\begin{proof}
The bilinear estimate \eqref{eq:hls-bound} is Theorem~4.3 of
\cite[p.~106]{lieb2001analysis} with $n=2$.
Equation \eqref{eq:hls-equal} is the specialization of
\eqref{eq:hls-cond} to $a=b$.  The mapping estimate
\eqref{eq:riesz-bound} is the corresponding Hardy--Littlewood--Sobolev
mapping theorem with $n=2$, $\lambda=1$, and exponent $4/3$.
\end{proof}

\makeatletter
\let\sectionname\@empty
\makeatother
\section*{Acknowledgments}
G.C.B. was supported by the National Research Foundation of Korea (NRF) grant
funded by the Korea government (MSIT) (No.~2021R1C1C2094843). H.C. is supported by grant from the National Natural Science Foundation of China (Project No.~12601434).
C.K. is partially supported by NSF CAREER Grant No.~2047681.  This material is
partly based upon work supported by the National Science Foundation under Grant
No.~DMS-2424139, while C.K. was in residence at the Simons Laufer Mathematical
Sciences Institute in Berkeley, California, during the Fall 2025 semester.
C.K. thanks Kyudong Choi and In-Jee Jeong for their explanations of their
results and ideas, including those in \cite{ChoiJeong2025}.

\bibliographystyle{siam}
%\bibliography{citation}

\begin{thebibliography}{10}

\bibitem{AbidiHmidiKeraani2010}
{\sc H.~Abidi, T.~Hmidi, and S.~Keraani}, {\em On the global well-posedness for the axisymmetric {Euler} equations}, Math. Ann., 347 (2010), pp.~15--41.

\bibitem{alexandre2012boltzmann}
{\sc R.~Alexandre, Y.~Morimoto, S.~Ukai, C.-J. Xu, and T.~Yang}, {\em {The Boltzmann equation without angular cutoff in the whole space: I, Global existence for soft potential}}, Journal of Functional Analysis, 262 (2012), pp.~915--1010.

\bibitem{bae2026quantitative}
{\sc G.-C. Bae and C.~Kim}, {\em Quantitative closure analysis toward ideal fluids}, arXiv preprint,  (2026).
\newblock arXiv:2603.14671v2.

\bibitem{bardosgolselevermore1991}
{\sc C.~Bardos, F.~Golse, and C.~D. Levermore}, {\em {Fluid dynamic limits of kinetic equations. {I}. Formal derivations}}, J. Statist. Phys., 63 (1991), pp.~323--344.

\bibitem{bardosgolselevermore1993}
\leavevmode\vrule height 2pt depth -1.6pt width 23pt, {\em {Fluid dynamic limits of kinetic equations. {II}. Convergence proofs for the {B}oltzmann equation}}, Comm. Pure Appl. Math., 46 (1993), pp.~667--753.

\bibitem{brkic2024vanishing}
{\sc P.~Brki{\'c} and E.~Wiedemann}, {\em On the vanishing viscosity limit for {3D} axisymmetric flows without swirl}, Nonlinear Anal. Real World Appl., 77 (2024), p.~104066.

\bibitem{CaoFanQin2026}
{\sc D.~Cao, J.~Fan, and G.~Qin}, {\em Linear growth of the vorticity maximum for axisymmetric {E}uler flows without swirl}, arXiv preprint,  (2025).
\newblock arXiv:2511.03171v4.

\bibitem{Chen2026EulerII}
{\sc J.~Chen}, {\em {Asymptotically Self-Similar Blowup for 3D Incompressible Euler with {$C^{1,1/3-}$} Velocity II: 3D Profiles, Blowup, and Limiting Behavior}}, 2026.
\newblock arXiv:2605.15130v2.

\bibitem{ChoiJeong2025}
{\sc K.~Choi and I.-J. Jeong}, {\em On vortex stretching for anti-parallel axisymmetric flows}, American Journal of Mathematics, 147 (2025), pp.~1251--1284.

\bibitem{ConstantinETiti1994}
{\sc P.~Constantin, W.~E, and E.~S. Titi}, {\em {Onsager's conjecture on the energy conservation for solutions of Euler's equation}}, Comm. Math. Phys., 165 (1994), pp.~207--209.

\bibitem{CordobaMartinezZoroaZheng2025}
{\sc D.~C{\'o}rdoba, L.~Mart{\'i}nez-Zoroa, and F.~Zheng}, {\em {Finite time singularities to the 3D incompressible {Euler} equations for solutions in $C^{\infty}(\mathbb{R}^3\setminus\{0\})\cap C^{1,\alpha}\cap L^2$}}, Ann. PDE, 11 (2025), p.~Paper No. 19.

\bibitem{danchin2007axisymmetric}
{\sc R.~Danchin}, {\em Axisymmetric incompressible flows with bounded vorticity}, Russian Math. Surveys, 62 (2007), pp.~475--496.

\bibitem{demasiespositolebowitz1989}
{\sc A.~De~Masi, R.~Esposito, and J.~L. Lebowitz}, {\em {Incompressible Navier--Stokes and Euler Limits of the Boltzmann Equation}}, Comm. Pure Appl. Math., 42 (1989), pp.~1189--1214.

\bibitem{DeRosaInversiStefani2023}
{\sc L.~De~Rosa, M.~Inversi, and G.~Stefani}, {\em {Weak-strong uniqueness and vanishing viscosity for incompressible Euler equations in exponential spaces}}, Journal of Differential Equations, 366 (2023), pp.~833--861.

\bibitem{Delort1992}
{\sc J.-M. Delort}, {\em Une remarque sur le probl\`eme des nappes de tourbillon axisym\'etriques sur {$\mathbb R^3$}}, J. Funct. Anal., 108 (1992), pp.~274--295.

\bibitem{diperna1989cauchy}
{\sc R.~J. DiPerna and P.-L. Lions}, {\em {On the Cauchy problem for Boltzmann equations: global existence and weak stability}}, Annals of Mathematics,  (1989), pp.~321--366.

\bibitem{diperna1989ordinary}
\leavevmode\vrule height 2pt depth -1.6pt width 23pt, {\em {Ordinary differential equations, transport theory and Sobolev spaces}}, Inventiones mathematicae, 98 (1989), pp.~511--547.

\bibitem{EgamberganovYao2025}
{\sc K.~Egamberganov and Y.~Yao}, {\em Growth estimates for axisymmetric {E}uler equations without swirl}, arXiv preprint,  (2025).
\newblock arXiv:2512.13456v1.

\bibitem{Elgindi2021}
{\sc T.~M. Elgindi}, {\em {Finite-time singularity formation for $C^{1,\alpha}$ solutions to the incompressible {Euler} equations on $\mathbb{R}^3$}}, Ann. of Math. (2), 194 (2021), pp.~647--727.

\bibitem{ElgindiGhoulMasmoudi2021}
{\sc T.~M. Elgindi, T.-E. Ghoul, and N.~Masmoudi}, {\em {On the stability of self-similar blow-up for $C^{1,\alpha}$ solutions to the incompressible {Euler} equations on $\mathbb{R}^3$}}, Cambridge J. Math., 9 (2021), pp.~1035--1075.

\bibitem{gallay2015remarks}
{\sc T.~Gallay and V.~{\v{S}}ver{\'a}k}, {\em Remarks on the {Cauchy} problem for the axisymmetric {Navier--Stokes} equations}, Confluentes Math., 7 (2015), pp.~67--92.

\bibitem{golsesaintraymond2004}
{\sc F.~Golse and L.~Saint-Raymond}, {\em {The {N}avier--{S}tokes limit of the {B}oltzmann equation for bounded collision kernels}}, Invent. Math., 155 (2004), pp.~81--161.

\bibitem{golsesaintraymond2009}
\leavevmode\vrule height 2pt depth -1.6pt width 23pt, {\em {The incompressible {N}avier--{S}tokes limit of the {B}oltzmann equation for hard cutoff potentials}}, J. Math. Pures Appl., 91 (2009), pp.~508--552.

\bibitem{guo2006diffusive}
{\sc Y.~Guo}, {\em Boltzmann diffusive limit beyond the {Navier--Stokes} approximation}, Comm. Pure Appl. Math., 59 (2006), pp.~626--687.

\bibitem{GustafsonMillerTsai2026}
{\sc S.~Gustafson, E.~Miller, and T.-P. Tsai}, {\em Growth rates for anti-parallel vortex tube {E}uler flows in three and higher dimensions}, Journal of Mathematical Fluid Mechanics, 28 (2026), p.~Paper No. 19.

\bibitem{JiuLiuNiu2021}
{\sc Q.~Jiu, J.~Liu, and D.~Niu}, {\em Global existence of weak solutions to the incompressible axisymmetric {E}uler equations without swirl}, J. Nonlinear Sci., 31 (2021), p.~36.

\bibitem{JiuWuYang2015}
{\sc Q.~Jiu, J.~Wu, and W.~Yang}, {\em Viscous approximation and weak solutions of the 3d axisymmetric {Euler} equations}, Math. Methods Appl. Sci., 38 (2015), pp.~548--558.

\bibitem{JiuXin2004}
{\sc Q.~Jiu and Z.~Xin}, {\em Viscous approximations and decay rate of maximal vorticity function for 3-{D} axisymmetric {E}uler equations}, Acta Math. Sin. (Engl. Ser.), 20 (2004), pp.~385--404.

\bibitem{JiuXin2006}
\leavevmode\vrule height 2pt depth -1.6pt width 23pt, {\em On strong convergence to 3-{D} axisymmetric vortex sheets}, J. Differential Equations, 223 (2006), pp.~33--50.

\bibitem{lei2024global}
{\sc Y.~Lei, S.~Liu, Q.~Xiao, and H.~Zhao}, {\em {Global Hilbert expansion for some nonrelativistic kinetic equations}}, Journal of the London Mathematical Society, 110 (2024), p.~e70016.

\bibitem{lieb2001analysis}
{\sc E.~H. Lieb and M.~Loss}, {\em Analysis}, vol.~14, American Mathematical Soc., 2001.

\bibitem{LimJeong2025}
{\sc D.~Lim and I.-J. Jeong}, {\em On the optimal rate of vortex stretching for axisymmetric {E}uler flows without swirl}, Archive for Rational Mechanics and Analysis, 249 (2025), p.~Paper No. 32.

\bibitem{lionsmasmoudi2001I}
{\sc P.-L. Lions and N.~Masmoudi}, {\em From the {B}oltzmann equations to the equations of incompressible fluid mechanics. {I}}, Arch. Ration. Mech. Anal., 158 (2001), pp.~173--193.

\bibitem{lionsmasmoudi2001II}
\leavevmode\vrule height 2pt depth -1.6pt width 23pt, {\em From the {B}oltzmann equations to the equations of incompressible fluid mechanics. {II}}, Arch. Ration. Mech. Anal., 158 (2001), pp.~195--211.

\bibitem{nobili2022renormalization}
{\sc C.~Nobili and C.~Seis}, {\em Renormalization and energy conservation for axisymmetric fluid flows}, Mathematische Annalen, 382 (2022), pp.~1--36.

\bibitem{Onsager1949}
{\sc L.~Onsager}, {\em Statistical hydrodynamics}, Nuovo Cimento (9), 6 (1949), pp.~279--287.

\bibitem{saintraymond2003euler}
{\sc L.~Saint-Raymond}, {\em {Convergence of solutions to the {B}oltzmann equation in the incompressible {E}uler limit}}, Arch. Ration. Mech. Anal., 166 (2003), pp.~47--80.

\bibitem{saintraymond2009entropy}
\leavevmode\vrule height 2pt depth -1.6pt width 23pt, {\em {Hydrodynamic limits: some improvements of the relative entropy method}}, Ann. Inst. H. Poincar\'e C Anal. Non Lin\'eaire, 26 (2009), pp.~705--744.

\bibitem{ShaoWeiZhang2026}
{\sc F.~Shao, D.~Wei, and Z.~Zhang}, {\em {Global Regularity of Axisymmetric Euler Equations Without Swirl in Higher Dimensions}}, Acta Mathematica Sinica, English Series, 42 (2026), pp.~663--679.

\bibitem{ShirotaYanagisawa1994}
{\sc T.~Shirota and T.~Yanagisawa}, {\em Note on global existence for axially symmetric solutions of the {Euler} system}, Proc. Japan Acad. Ser. A Math. Sci., 70 (1994), pp.~299--304.

\bibitem{simon1987compact}
{\sc J.~Simon}, {\em Compact sets in the space {$L^p(0,T;B)$}}, Annali di Matematica Pura ed Applicata, 146 (1987), pp.~65--96.

\bibitem{strain2008exponential}
{\sc R.~M. Strain and Y.~Guo}, {\em {Exponential decay for soft potentials near Maxwellian}}, Archive for Rational Mechanics and Analysis, 187 (2008), pp.~287--339.

\bibitem{strainzhu2013vpl}
{\sc R.~M. Strain and K.~Zhu}, {\em {The {V}lasov--{P}oisson--{L}andau system in $\mathbb{R}^3_x$}}, Arch. Ration. Mech. Anal., 210 (2013), pp.~615--671.

\bibitem{UkhovskiiIudovich1968}
{\sc M.~R. Ukhovskii and V.~I. Iudovich}, {\em Axially symmetric flows of ideal and viscous fluids filling the whole space}, J. Appl. Math. Mech., 32 (1968), pp.~52--62.

\bibitem{xu2026diffusiveVMB}
{\sc Y.~Xu, F.~Zhou, W.~Gong, and W.~Wu}, {\em {Diffusive Limit of the Noncutoff Vlasov--Maxwell--Boltzmann System}}, SIAM Journal on Mathematical Analysis, 58 (2026), pp.~1187--1231.

\end{thebibliography}

\end{document}